\documentclass{amsart}

\usepackage{amsmath,amsrefs,bbold,mathrsfs,amscd,lineno,hyperref}
\usepackage{mathtools,afterpage}
\usepackage{tikz}
\usepackage{mathrsfs}
\usepackage{yfonts,amsbsy}

\usepackage{wrapfig}

\usepackage{tipa}

\usetikzlibrary{matrix, patterns,snakes,calc,intersections}
\usetikzlibrary{math,arrows,decorations.markings}

\newcommand{\proj}{\operatorname{proj}}
\newcommand{\Dom}{\operatorname {Dom}}
\newcommand{\NE}{\operatorname {NE}}

\newcommand{\SL}{\operatorname {SL}}

\renewcommand{\sp}{{\ \ }}
\newcommand{\1}{{\mathbb 1 }}

\renewcommand{\Re}{\operatorname {Re}}
\renewcommand{\Im}{\operatorname {Im}}

\newcommand\hC{{\widehat{\mathbb C}}}
\newcommand\LL{{\mathcal L}}

\newcommand\NN{\mathcal{N}}

\newcommand{\Mandel}{{\mathcal{M}}}
\newcommand\MM{\mathscr{M}}
\newcommand\bMM{{\pmb {\MM}}}

\newcommand\binfty{{\infty}}

\newcommand\PT{\mathbb{ T}}

\newcommand\mbbS{{\mathbb S}}

\newcommand\mI{{\mathfrak I}}

\newcommand\M{\mathbb M}

\newcommand\YY{\mathcal{Y}}
\newcommand\HH{\boldsymbol{\Delta}}
\newcommand\Hors{{\mathcal H}} 
\newcommand\mHors{{\boldsymbol H}} 
\newcommand{\bDelta}{{\mathbb{\Delta}}}

\newcommand{\bPi}{{\mathbb{\Pi}}}

\newcommand{\skel}{{\mathfrak T}}

\newcommand{\bA}{{\mathbb{A}}}
\newcommand{\bB}{{\mathbb{B}}}
\newcommand{\bO}{{\mathbb{O}}}

\newcommand\aRR{\mathcal{R}^-}

\newcommand\RRR{{\mathfrak R}}
\newcommand\RR{\mathcal{R}}
\newcommand\cRR{R}
\newcommand\mRR{\boldsymbol R}
\newcommand\prm{{\operatorname{prm}}}
\newcommand\RRc{\mathcal{R}_{\prm}}
\newcommand\cRRc{R_{\prm}}

\newcommand\mRRc{ \boldsymbol R_{\prm}}

\newcommand\mF{\mathfrak{F}}
\newcommand\intr{\operatorname{int}}

\newcommand\seq{{\operatorname{\mathbf s}}}

\newcommand\Jul{{\mathfrak J}}
\newcommand\Kfilled{{\mathfrak K}}
\newcommand\Post{{\mathfrak P}}
\newcommand\ovPost{{\overline{\mathfrak P}}}

\newcommand\frakX{{\mathfrak X}}

\newcommand\frakY{{\mathfrak Y}}
\newcommand\frakZ{{\mathfrak Z}}

\newcommand{\bUpsilon}{{\mathbf \Upsilon}}

\newcommand{\bM}{{\mathbf M}}

\newcommand{\bF}{{\mathbf F}}
\newcommand{\tF}{{\mathfrak F}}

\newcommand{\bZ}{{\mathbf Z}}
\newcommand{\bL}{{\mathbf L}}
\newcommand{\bX}{{\mathbf X}}
\newcommand{\bY}{{\mathbf Y}}
\newcommand{\bW}{{\mathbf W}}
\newcommand{\bQ}{{\mathbf Q}}
\newcommand{\bR}{{\mathbf R}}
\newcommand{\bE}{{\mathbf E}}

\newcommand{\bD}{{\mathbf D}}
\newcommand{\bT}{{\mathbf T}}
\newcommand{\bJ}{{\mathbf J}}
\newcommand{\bS}{{\mathbf S}}
\newcommand{\bU}{{\mathbf U}}
\newcommand{\bP}{{\mathbf P}}
\newcommand{\bH}{{\mathbf H}}
\newcommand{\bG}{{\mathbf G}}

\newcommand{\bK}{{\mathbf K}}
\newcommand{\bI}{{\mathbf I}}
\newcommand{\bbf}{{\mathbf f}}
\newcommand{\bbn}{{\mathbf n}}
\newcommand{\bbk}{{\mathbf k}}
\newcommand{\bbs}{{\mathbf s}}

\newcommand{\bbh}{{\mathbf h}}
\newcommand{\btau}{{ \boldsymbol {\tau}}}
\newcommand{\bbg}{{\mathbf g}}
\newcommand{\Rect}{{\mathfrak R }}

\newcommand{\str}{{\star}}

\newcommand{\bxx}{{\mathbb{x}}}
\newcommand{\byy}{{\mathbb{y}}}

\newcommand{\mm}{\mathbb{m}}
\newcommand{\pp}{\mathfrak{p}}
\newcommand{\ee}{\mathbf e}
\renewcommand{\aa}{\mathfrak{a}}
\newcommand{\bb}{{\mathbf b}}

\newcommand{\vv}{\mathfrak{v}}
\newcommand{\ww}{\mathfrak{w}}
\renewcommand{\qq}{\mathfrak{q}}
\renewcommand{\tt}{\mathfrak{t}}
\newcommand{\ff}{\mathfrak{f}}
\newcommand{\rr}{\mathfrak{r}}
\renewcommand{\ss}{\mathfrak{s}}
\newcommand{\kk}{\mathfrak{k}}

\newcommand{\A}{\mathbf A}
\newcommand{\B}{\mathbf B}

\newcommand\wall{{Q}}

\newcommand\inn{{\Omega}}

\newcommand\binn{{\mathbb \inn}}

\newcommand\wgamma{{\widetilde \gamma}}

\def\loc{{\mathrm{loc}}}

\newcommand\strai{{\boldsymbol \chi}}

\newcommand\WW{\mathcal{W}}
\newcommand\QQ{\mathcal{Q}}

\newcommand\Unst{{\boldsymbol{\mathcal{W}}^u}}

\newcommand\UnstLoc{{\boldsymbol{\mathcal{W}}^u_\loc}}

\newcommand\UU{\mathcal{U}}
\newcommand\bUU{{ \boldsymbol{\mathcal U}}}
\newcommand\bVV{{ \boldsymbol{\mathcal V}}}
\newcommand\bWW{{ \boldsymbol{\mathcal W}}}
\newcommand\bOO{{ \boldsymbol{\mathcal O}}}

\newcommand\BB{\mathcal{B}}
\renewcommand\AA{\mathcal{A}}

\newcommand\scrP{\mathscr{P}}

\newcommand\PP{\mathcal{P}}

\newcommand\out{{\operatorname{out}}}
\newcommand\up{{\operatorname{up}}}

\newcommand\ext{{\operatorname{ext}}}
\newcommand\new{{\operatorname{new}}}

\newcommand\cp{{}}
\newcommand\Sieg{{\operatorname{Sieg}}}
\newcommand\Pacm{{\operatorname{Pacm}}}

\newcommand\frb{{\operatorname{frb}}}

\newtheorem{claim}{Claim}
\newtheorem{claim2}{Claim}

\newcommand{\wC}{\widehat{\mathbb{C}}}
\newcommand{\C}{\mathbb{C}}
\newcommand{\Q}{\mathbb{Q}}
\newcommand{\R}{\mathbb{R}}
\newcommand{\N}{\mathbb{N}}
\newcommand{\Left}{\mathfrak{L}}
\newcommand{\Right}{\mathfrak{R}}

\renewcommand{\H}{\mathbb{H}}
\newcommand{\oH}{\overline {\mathbb{H}}}

\newcommand{\Z}{\mathbb{Z}}
\newcommand{\T}{\mathbb{T}}
\newcommand{\Circle}{\mathbb{S}^1}
\newcommand{\Disk}{\mathbb{D}}
\newcommand{\ovDisk}{{\overline \Disk}}
\newcommand{\Parab}{\mathbb{P}}
\newcommand{\Lbb}{{\mathbb L}}
\newcommand{\Ibb}{{\mathbb I}}

\newcommand{\Sbb}{{\mathbb X}}
\newcommand{\Sby}{{\mathbb Y}}

 \newcommand{\Fol}{{\mathcal F}}
  \newcommand{\bFol}{{\boldsymbol{ \mathcal F}}}

\newtheorem{thm}{Theorem}[section]
\newtheorem{cor}[thm]{Corollary}
\newtheorem{lem}[thm]{Lemma}

\newtheorem{prop}[thm]{Proposition}
\newtheorem{rem}[thm]{Remark}
\newtheorem{property}[thm]{Property}

\newtheorem{conj}[thm]{Conjecture}

\theoremstyle{remark}

\numberwithin{equation}{section}

\theoremstyle{definition}

\font\nt=cmr7

\def\be{\begin{equation}}

\newcommand{\area}{{\operatorname{area}\ }}

\newcommand{\Fat}{{\mathfrak F}}
\newcommand{\Esc}{{\boldsymbol{\operatorname{Esc}}}}
\newcommand{\ParEsc}{{\boldsymbol{\mathcal E}}}
\newcommand{\DEsc}{{\boldsymbol{\mathcal D}}}

\newcommand{\bnd}{{\mathrm{bnd}}}

\newcommand{\ra}{\rightarrow}

\def\QL{{\text{QL}}}

\newcommand{\QG}{{\mathcal {QL}}}
\newcommand{\Conn}{{\textgoth M}}

\newcommand{\diam}{\operatorname{diam}}
\newcommand{\dist}{\operatorname{dist}}

\renewcommand{\mod}{\operatorname{mod}}

\newcommand{\orb}{\operatorname{orb}}

\newcommand{\id}{\operatorname{id}}

\newcommand{\CV}{{\operatorname{CV}}}
\newcommand{\CP}{{\operatorname{CP}}}

\newcommand{\balpha}{{\boldsymbol \alpha}}

\newcommand{\bgamma}{{\boldsymbol{ \gamma}}}

\newcommand{\bdelta}{{\boldsymbol{ \delta}}}
\newcommand{\brho}{{\boldsymbol{ \rho}}}
\newcommand{\bchi}{{\boldsymbol{ \chi}}}

\DeclareFontFamily{U}{mathb}{\hyphenchar\font45}
\DeclareFontShape{U}{mathb}{m}{n}{ <5> <6> <7> <8> <9> <10> gen * mathb <10.95> mathb10 <12> <14.4> <17.28> <20.74> <24.88> mathb12 }{}
\DeclareSymbolFont{mathb}{U}{mathb}{m}{n}
\DeclareFontSubstitution{U}{mathb}{m}{n}
\DeclareMathSymbol{\selfmap}{3}{mathb}{"FD}

\newcommand{\RN}[1]{\MakeUppercase{\romannumeral#1}}%

\def\note#1
{\marginpar
{\nt $\leftarrow$
\par
\hfuzz=20pt \hbadness=9000 \hyphenpenalty=-100 \exhyphenpenalty=-100
\pretolerance=-1 \tolerance=9999 \doublehyphendemerits=-100000
\finalhyphendemerits=-100000 \baselineskip=6pt
#1}\hfuzz=1pt}

\usepackage{amssymb}


\theoremstyle{nonumberbreak}

\begin{document}

\title[MLC at satellite parameters]{Local connectivity of the Mandelbrot set\\ at some satellite parameters of bounded type}
\author{Dzmitry Dudko}
\author{Mikhail Lyubich}

\begin{abstract}

We explore geometric properties of the Mandelbrot set $\Mandel$, 
and the corresponding Julia sets $\Jul_c$,   near the main cardioid.
Namely, we establish that:  a)  $\Mandel$ is  locally connected  at certain infinitely renormalizable parameters $c$ of \emph{bounded satellite} type,  providing first examples of this kind; b) The Julia sets $\Jul_c$ are also locally connected
and have positive area; c) $\Mandel$ is self-similar  near  Siegel parameters of periodic type.  We approach these problems by analyzing the unstable manifold of the pacman renormalization operator constructed in \cite{DLS} as a global
transcendental family. It is the first occasion when external rays and puzzles of limiting transcendental maps are applied to study the Polynomial dynamics.

\end{abstract}
\maketitle

\setcounter{tocdepth}{1}

\tableofcontents

\section{Introduction}

\subsection{Main themes }
This paper touches upon several central themes of Holomorphic Dynamics:
Rigidity and MLC,  local connectivity and area of Julia sets, in their interplay with  Renormalization Theory. Developing further 
the {\em  Pacman Renormalization Theory} 
designed in \cite{DLS} (jointly with Nikita Selinger), we demonstrate that near-neutral dynamics can be studied as transcendental dynamics on renormalization unstable manifolds. As an application, we produce parameters of new kind, previously unaccessible,
where the Mandelbrot set is locally connected and whose Julia sets 
are locally connected and have positive area. 

The MLC problem (of local connectivity of the Mandelbrot set) goes back to the classical work 
of Douady and Hubbard from the 1980s. The original motivation was to produce a precise topological model for 
the Mandelbrot set $\Mandel$. Soon afterwards a deep connection to the Mostow Rigidity phenomenon was revealed, due to insights by Thurston and Sullivan.  
Around 1990s, due to the work of Yoccoz, the problem was closely linked to 
the Quadratic-like Renormalization, 
by reducing it to the case of  {\em infinitely renormalizable parameters}. 

{\em Quadratic-like Renormalization} appears in two flavors:
{\em primitive}  and {\em satellite}. 
The former generates stronger  expansion allowing for  a better control,
which has led to substantial progress in the past 20 years. 
In particular, MLC was established by Jeremy Kahn 
and the second author \cite{L:acta,KL1,KL2}
under the {\em molecule condition} when all the renormalizations
stay  uniformly away from the satellite type. 
This covers, in particular,  all parameters of {\em bounded primitive type}
\cite{K}.   

The satellite renormalization  is delicately related to the 
non-expanding rotation regime
near the main cardioid of $\Mandel$, and progress in 
understanding of this kind of phenomenon has been slower. 
Recently, Cheraghi and Shishikura \cite{CS},
using the Inou-Shishikura {\em  Almost Parabolic Renormalization}  \cite{IS},  
have constructed a certain set of parameters of {\em unbounded satellite type} where
MLC holds. In this paper we construct, by completely different methods,  first examples of  parameters of 
{\em bounded satellite type} where MLC holds.    
Moreover,  Julia sets are also locally connected at these parameters (``JLC'').

 Problem of area of Julia sets is intimately related to the MLC and JLC problems,
and progress in these three problem has been made in parallel. 
First examples of Julia sets of
positive area were constructed by Buff and Cheritat around 2006 \cite{BC}. 
Their machinery  has produced  examples of three types: 
{\em Cremer, Siegel,} and  {\em infinitely renormalizable of unbounded type}
(probably, all {\em non-locally-connected}).
In a more recent  work by Artur Avila and the second author \cite{AL-posmeas}, 
 {\em infinitely renormalizable } examples of {\em bounded primitive type}
were constructed (all {\em locally connected}). The machinery developed there applies to  maps constructed in this paper
giving first examples of Julia sets of positive area for 
 {\em infinitely renormalizable } maps of {\em bounded satellite type}
 (also locally connected).

 Our main tool is {\em Pacman Renormalization} developed in \cite{DLS}.
It combines features of two classical Renormalization Theories: {\em  Quadratic-like}  and {\em Siegel}.  
The latter   
originated in the 1980s in  physics literature, which yielded a  long-standing renormalization conjecture. 
In the 1990s  McMullen constructed a Siegel renormalization periodic point  and described its maximal analytic extension for 
any rotation number of periodic type \cite{McM3}. It was  proven in \cite{DLS} that this point is hyperbolic with one-dimensional unstable manifold $\WW^u$. (Let us note that in the mid 2000’s Inou and Shishikura proved the existence and hyperbolicity of Siegel renormalization fixed points of sufficiently high combinatorial type using a completely different approach, based upon the parabolic perturbation theory~\cite{IS}. On the other hand, Gaidashev and Yampolsky gave a computer assistant proof of hyperbolicity for the golden mean rotation number~\cite{GY}.)

In this paper we study the above unstable manifold $\WW^u$ as a one-parameter transcendental family. 

It was shown in~\cite{DLS} that every map in $\WW^u$ admits a maximal analytic extension to a $\sigma$-proper map onto $\C$. Using ideas of Transcendental Dynamics (compare~\cites{DK,  EL,E, SZ, RRRS,Re,BL,BR}), we construct ``external rays" and describe the associated 
``puzzle structure" for this family (\S\S \ref{s:Dyn F_str}--\ref{s:HolMot Esc}). This allows us to construct an appropriate quadratic-like family inside  $\WW^u$
(\S \ref{s:par pacm}). Using hyperbolicity of the pacman renormalization established in \cite{DLS}, we transfer this family along the associated {\em hybrid lamination} in the space of pacmen,
from $\WW^u$ to the quadratic family,  yielding desired parameters (\S \ref{s:proof:main thms}). Let us note that even though similarity between neutral and transcendental dynamics has long been observed (see~\cites{Ep,Sh:HD,SY,CS}), to the best of our knowledge, the external and puzzle structure of the associated transcendental families has never before been explored and applied to the polynomial dynamics; see~\S\ref{ss:Rem Transc Dynam} for a further discussion.

We remark in conclusion that the unstable manifold of a quadratic-like renormalization operator can be described in a similar way as it is done in this paper for a pacman renormalization operator. In fact, some steps are simpler for quadratic-like maps thanks to their nice external structure and a simpler algebraic structure of the associated cascade.

\begin{figure}[t!]
\[\begin{tikzpicture}[scale=1.3]

\coordinate  (w0) at (-3.9,1.5);

\node (v0) at (-2.9,1.4)  {};

\draw  (v0) ellipse (3.5 and 2);

\node[shift={(0.15,0)}]  at (w0)  {$\alpha$};

\coordinate (v1) at (-2.2,1.3) {};
\coordinate (v2) at (-0.6,2.1) {} {};
\coordinate (v3) at (-0.6,0.6) {} {};

  \draw[ line width=0.5pt,red] 
            (v3) 
            .. controls (v3) and (v1) ..
            (v1)
            .. controls (v1) and (v2) ..
            (v2)
            .. controls (-7.3,4.4) and (-7.1,-1.5) ..
            (v3);

\coordinate (w1) at (-4.72,2.42) {} {};

\draw[line width=0.5pt,red]  (w0)--(w1);

\coordinate (w2) at (-6.1,0.6)  {};
\draw  [line width=0.5pt]  (w0)--(w2);

\draw[line  width=0.5pt,-latex] (-3.1,1.6) .. controls (-2.6,2.1) and (-3.1,3.2) .. (-3.8,2.8);

\node at (-5.5,0.7) {$\gamma_1$};
\node at (-2.8,2.7) {$f$};
\node[red] at (-4.5,1.9) {$\gamma_0$};

\node[red] at (-1.2,1.6) {$\gamma_+$};
\node[red] at (-1.1,0.9) {$\gamma_-$};
\node at (-1.9,4.6) {};
\node[red] at (-3.5,0.8) {$U$};
\node at (-1.6,0.1) {$V$};
\node at (-2.3,1.3) {$\alpha'$};
\end{tikzpicture}\]
\caption{A (full) pacman is a $2:1$ map $f:U
\to V$ such that the critical arc $\gamma_1$ has $3$ preimages: $\gamma_0$, $\gamma_+$ and $\gamma_-$. }
\label{Fg:Pacman}
\end{figure}

\subsection{Statement of the results}
\label{ss:Stat of Results}
Let us pass to a more technical description. 
Let $c(\theta),\theta\in \R/\Z$, be the parameterization of the main cardioid $\partial \HH$ by the rotation number $\theta$. Consider the molecule map $\mRRc\colon \Mandel \dashrightarrow \Mandel$ (see~\cite[Appendix C]{DLS}); its action on $\partial \HH$ is given by
\begin{equation}
\label{eq:ActOnRotNumb}
\theta\longrightarrow \frac{\theta}{1-\theta}\sp \mbox{ if }0\le \theta \le \frac{1}{2};\quad \sp \theta\longrightarrow\frac{2\theta-1}{\theta} \sp \mbox{if }\frac{1}{2}\le \theta\le 1. 
\end{equation}
 Let us fix an  $\mRRc$-periodic point $c(\theta)\in \partial \HH$ with period $\mm$. Note that $z^2+c(\theta)$ is a Siegel polynomial.

Let $\Mandel_0$ be a small copy of the Mandelbrot set centered at the main molecule such that $\Mandel_0$ is close to $c(\theta)$. By the Yoccoz inequality, $\Mandel_0$ is contained in a small neighborhood of  $c(\theta)$. Define inductively $\Mandel_{n-1}$ to be the unique preimage of $\Mandel_{n}$ under $\mRRc^\mm$ so that $\Mandel_{n-1}$  is also in a small neighborhood of  $c(\theta)$; i.e.~the $\Mandel_n$ shrink to $c(\theta)$. For $n\le 0$ denote by $\mRR_n\colon \Mandel_n\to \Mandel$ the Douady-Hubbard straightening map.

A map $g\colon \C\dashrightarrow \C$ is $1+\varepsilon$-\emph{conformal} at $z_0\in \Dom g$ if $g$ has the derivative $g'(z_0)$ at $z_0$ such that the corresponding linear map approximates $g$ with an error term:
\[g(z_0+z)=g(z_0)+ g'(z_0) z +O(|z|^{1+\varepsilon}),\sp\sp z+z_0\in \Dom g.\]

\begin{thm}
\label{thm:main}
 There is a small copy $\Mandel_0$ of the Mandelbrot set close to $c(\theta)$ and centered at the main molecule such that the preimages $\Mandel_n$ of $\Mandel_0$ as above {\bf scale} linearly at $c(\theta)$: the map
\begin{equation}
\label{eq:thm:main}
\mRRc^\mm\colon \{c(\theta)\}\cup \bigcup_{n<0} \Mandel_n  \to \{c(\theta)\}\cup  \bigcup_{n\le 0} \Mandel_n
\end{equation}
is $1+\varepsilon$-conformal at $c(\theta)$. Moreover, for every $n\le 0$ we have
\begin{itemize}
\item {\bf rigidity}: the set $\displaystyle{\bigcap _{i\ge 0}\Dom (\mRR^i_n)}=\{c_n\}$ is a singleton; 
\item {\bf JLC}: the Julia set of $z^2+c_n$ is locally connected;
\item  {\bf positive measure:} the Julia set of $z^2+c_n$ has positive measure.
\end{itemize}
\end{thm}

 In fact, we construct a horseshoe of parameters where local connectivity holds. We also show that $p_n(z)\coloneqq z^2+c_n$ has a forward invariant \emph{valuable flower} $X_n$ containing the postcritical set of $p_n$ such that $X_n$ is in a small neighborhood of the closed Siegel disk of $z^2+c(\theta)$. This is a partial case of the conjecture on the upper semi-continuity of the mother hedgehog, see~\cite{DLS}*{Appendix C}.

There are examples of infinitely renormalizable polynomials with non-locally connected Julia sets \cite{Mi}. The examples are based on near-parabolic effects when small Julia sets are forced not to shrink. This may actually be the only mechanism for non-locally connectivity of the Julia sets in the infinitely renormalizable case. See~\cites{Hu, Mi, J,McM2, L:acta, KL1, KL2} for classes of locally connected Julia sets. Our results demonstrate that the Julia sets behave nicely near Siegel parameters of bounded type. There is also  substantial progress in understanding near-parabolic Julia sets in the Inou-Shishikura class, see  \cites{CS,SY, Che }. 

It was shown in~\cite{Yarrington} that the Julia set of an infinitely renormalizable polynomial $p$ has measure $0$ if the renormalizations of $p$ stay ``sufficiently far away'' from the main molecule.  Our results indicate that if the renormalizations of $p$ are close to the Siegel maps, then the Julia set of $p$ inherits a positive mass from Siegel filled-in Julia sets. Thus, one may expect a certain monotonicity of the measure depending on how far the renormalizations of $p$ are away from the main hyperbolic component. This is also consistent with the Hubbard conjecture stating that the measure of the filled-in Julia set of every parameter in the main hyperbolic component is at least some universal $\varepsilon>0$. It is recently shown in~\cite{DS} that the classical Feigenbaum polynomial has Hausdorff dimension less than $2$, and consequently it has measure~$0$.

\subsection{Outline of the paper} {\bf Section~\ref{s:SectRenorm}} reviews combinatorial aspects of the pacman and maximal prepacman renormalizations. It starts by discussing the model case of disk rotations $\Lbb_\theta\coloneqq [z\mapsto \ee(\theta) z]\colon \overline \Disk\selfmap$ and the associated commuting pairs
\[(T_{\vv^-}, \sp T_{\ww})=(T_{-\vv}, \sp T_{\ww}) \coloneqq [z\mapsto z-\vv,\sp z\mapsto z+\ww] \colon \overline \H_-\selfmap,\]
where $\Lbb_\theta$ is the quotient of $(T_{\vv^-}, \sp T_{\ww})$ under $z\mapsto z+\vv+\ww$. We define the \emph{sector renormalization} acting on rotations of $\overline \Disk$ and the renormalization on commuting pairs, see~\S\ref{sss:renorm:pair of rotat}; the latter is a rescaled iteration of a given pair. Lemma~\ref{lem:PT:rot} summarizes basic properties of arising antirenormalization matrices.

In~\S\ref{s:sector renorm of homeo}, we extend the discussion to the cases of self-homeomorphisms and partial self-homeomorphisms of the disk $\overline \Disk$. Propositions~\ref{prop:beta:FunDom} and~\ref{prop:quot of fund domain} relate the dynamical planes of (partial) self-homeomorphisms of $\Disk$ to the dynamical planes of associated commuting pairs.

In {\bf Section~\ref{s:PacmanRenorm}}, we collect the background information on the pacman renormalization from~\cite{DLS}. A \emph{pacman} is an almost $2:1$ map with a covering structure illustrated on Figure~\ref{Fg:Pacman}, while a \emph{prepacman} (Figure~\ref{Fg:Prepacman}) is a commuting pair obtained by cutting a pacman along its ``critical arc $\gamma_1$''. For every rotational number of periodic type  $\theta=\mRRc^\mm(\theta)$, there is an associated analytic pacman renormalization operator $\RR\colon \BB\dashrightarrow \BB$ in a suitable Banach space $\BB$ with a hyperbolic fixed point $f_\str=\RR (f_\str)$, where $f_\str$ is a pacman that has a Siegel disk with rotation number $\theta$. We denote by $\WW^s$ and $\WW^u$ the stable and unstable manifolds of $f_\str$ respectively.

\emph{Maximal prepacmen.} A key fact is that every $f\in \WW^u$ has a \emph{maximal prepacman} (see~\eqref{eq:Max Prep U pm}): a prepacman of $f$ with an embedding into $\C$ such that both prepacman maps admit maximal extensions 
\[\bF= (\bbf_-\colon  \bX_-\to \C,\sp \bbf_+\colon  \bX_+\to \C),
\] 
as \emph{$\sigma$-proper coverings} of $\C$, i.e.~$\bbf_\pm \mid \bX_\pm$ is an increasing union of proper maps. The construction of a maximal prepacman goes as follows. Every $f\in \WW^u$ can be antirenormalized infinitely many times \[f=f_0,\sp f_{-1},\sp f_{-2},\dots,\sp \sp\sp\sp \RR f_{n-1}=f_n.\]
As Figure~\ref{Fig:Sf1dash} illustrates, we cut $f_0\colon U_0\to V_0$ along its critical arc $\gamma_1$ and then embed the sector $S_0\coloneqq V_0\setminus \gamma_0$ into the dynamical plane of $f_{-1}\colon U_{-1}\to V_{-1}$. Then we cut $f_{-1}\colon U_{-1}\to V_{-1}$ along $\gamma_1$ and embed $S_{-1}\coloneqq V_{-1}\setminus \gamma_{1}$ into the dynamical plane of $f_{-2}$. Continuing (and linearizing) this process, we construct $\bF=\bF_0$ as a ``direct union'' of $f_n\colon U_n\to V_n,\sp n\le 0$, cut along $\gamma_1$.

In the course of the construction of $\bF$, the $\alpha$-fixed point of $f$ ``goes to infinity''. To better  relate the dynamical planes of $\bF$ and $f$, we formally add the fixed point $\balpha$ to the dynamical plane of $\bF$ and introduce an appropriate ``wall topology'' for $\C\cup \{\balpha\}$ so that small neighborhoods of $\balpha(\bF)$ are ``full lifts (lifts followed by spreading around)'' of small neighborhoods of $\alpha(f)$, see~\S\ref{ss:balpha:wall topoloy}. We also introduce a ``renormalization triangulation'' for $\bF$ to control its dynamics near $\balpha$ (see~\S\ref{ss:RenTriang}) and to project dynamical objects from the $\bF$-plane to the $f$-plane, see Theorem~\ref{thm:quot of fund domain} (which is an application of Proposition~\ref{prop:quot of fund domain}).

\emph{Global unstable manifold.} Let $\UnstLoc =\{\bF\colon f\in \WW^u\}\simeq \WW^u$ be the space of maximal prepacmen. The operator $\RR$ acts on $\UnstLoc$ as the multiplication by $\lambda_\str$, where $\lambda_\str>1$. In the dynamical planes, $\RR(\bF)$ is a rescaled iteration of $\bF$ (see~\eqref{eq:iter of max prep}). Therefore, we can globalize $\UnstLoc\subset\Unst\simeq \C$. We can also view $\Unst$ as the space of rescaled limits of quadratic polynomials, see~\S\ref{ss:geom pict}. Therefore, a zoomed picture of the Mandelbrot set near ${c(\theta)}$ gives a good approximation to $\Unst$, see Figure~\ref{Fig:unst man+max sieg}.

We follow up with a discussion of  basic dynamical properties of maximal prepacmen in {\bf Section~\ref{s:max prep}.}

\emph{Cascade.} Consider a maximal prepacman $\bF\in \Unst$, and set $\bF_n\coloneqq \RR^n(\bF)$. For $n\le 0$ we denote by $\bF^\#_n=\big(\bbf^\#_{n,\pm}\big)$ the rescaled version of $\bF_n$ so that $\bbf_\pm=\bbf_{0,\pm}$ are iterates of $\bbf^\#_{n,\pm}$. The cascade \[\bF^{\ge 0}\sp\coloneqq \sp \big<
\id, \sp \bbf_{n,\pm}^\#:\sp n\le 0\big>\] is the semigroup generated by all $\bF^{\#}_n$; see~\S\ref{ss:PowerTriples} for an equivalent definition. It turns out that every map in $\bF^{\ge 0}$ can be written as $\bF^P$ with the usual power law $\bF^{P}\circ \bF^Q=\bF^{P+Q}$, where $P,Q$ belong to a dense semigroup $\PT$ of $\R_{\ge0}$.

\emph{Fatou, Julia, and escaping sets.} 
The \emph{Fatou set} $\Fat(\bF)$ of $\bF$ is the set where the family $\{\bF^P\}$ is normal. The \emph{Julia set} $\Jul( \bF)$ is the complement of the Fatou set. The cascade $\bF^{\ge0}$ has a single critical orbit: the set of critical values of $\bF^P$ is exactly $\displaystyle{\bigcup_{Q<P} \bF^Q\{0\}}$. The postcritical set of $\bF^P$ is $\Post(\bF)=\displaystyle{\bigcup_{Q\in \PT} \bF^Q\{0\}}$. The dynamics of $\bF^{\ge0 }$ is proper discontinuous in an appropriate sense; in particular, for every $P\in \PT$ and every $x\in \C$ the set $\displaystyle{\bigcup_{Q<P}\bF^{Q}\{x\}}$ is discrete, see Lemma~\ref{lem:discr of dyn}.

 Given $P\in \PT_{>0}$, the \emph{$P$-th escaping set}
 is
\[\Esc_P(\bF)\coloneqq \C\setminus \Dom(\bF^P).\] 
The \emph{escaping set} is
$\Esc(\bF)\coloneqq \bigcup_{P>0} \Esc_P(\bF).$ 
We have $\overline { \Esc(\bF)}=\Jul(\bF)$, see Corollary~\ref{cor:Jul is ovl Esc}.

In {\bf Section~\ref{s:Dyn F_str}}, we study the dynamical plane of the renormalization fixed point $\bF_\str$.
Since $\RR\bF_\str= \bF_\str$, the dynamical self-similarity ${A_\str\colon z\mapsto \mu_\str z}$ conjugates $\bF_\str ^P$ to $\bF_\str^{P\tt}$, where $P\in \PT$ and $\tt>1$, see~\S\ref{ss:PowerTriples}. This allows us to give an explicit description of the dynamical plane of $\bF_\str$. It has an invariant unbounded Siegel disk $\bZ_\str$ -- the rescaled limit of $Z_\str$, see Figure~\ref{Fig:unst man+max sieg}. Every Fatou component $\bZ_i$ of $\bF_\str$ is either $\bZ_\str$ or its preimage under a certain iterate $\bF_\str^T$ with $T\in \PT_{>0}.$ We prove in~\S\ref{ss:alpha-points} that each $\bZ_i$ is a bounded subset of $\C$. Moreover, if $T$ is minimal, then $\overline \bZ_i\cap \Esc_T(\bF_\str)=\{\alpha_i\}$ is a singleton, and $\bF_\str^T\colon \bZ_i\to \bZ_\str$ extends continuously to $\bF_\str^T\colon \overline \bZ_i\to \overline \bZ_\str\cup \{\balpha\}$ so that $\bF_\str^T(\alpha_
i)=\balpha$. We say that $\alpha_i$ is an \emph{alpha-point of generation $|\alpha_i|\coloneqq T$}.

Alpha-points are cut points of $\Esc(\bF_\str)$: each set $\Esc(\bF_\str)\setminus \{\alpha_i\}$ has two bounded components and one unbounded, see Figure~\ref{Fig:GeomScal}. Moreover, there is a unique curve in $\Esc(\bF_\str)$ connecting any two alpha-points. We write $\alpha_i\succ\alpha_j$ if $\alpha_i$ is in one of the bounded components of $\Esc(\bF_\str)\setminus\{\alpha_j\}$. It follows that $|\alpha_i|> |\alpha_j|$ and, moreover, there is a unique simple arc $[\alpha_i,\alpha_j]\subset \Esc (\bF_\str)$ connecting $\alpha_i$ and $\alpha_j$. We say that $[\alpha_i,\alpha_j]$ is a \emph{ray segment}. An \emph{external ray} is a maximal concatenation of ray segments, see~\S\ref{ss:ExtRayFstr}. We show that external rays have a tree structure: every two external rays eventually meet.

In {\bf Section~\ref{s:HolMot Esc}}, applying the Fatou and Riesz Theorems,  we show that the escaping set $\Esc_P(\bF)$ is the set of accumulation points of $\bF^{-P}(x)$ for all $x\in\C$. Then we deduce that $\Esc(\bF)$ moves locally holomorphically unless it hits an iterated preimage of $0$. Therefore, $\bF$ has the same external structure as $\bF_\str$ with appropriate adjustments when $0\in\Esc(\bF)$.

In {\bf Section~\ref{s:par pacm}}, we show that $\Unst$ contains certain ternary satellite small copies of the Mandelbrot set. The argument is illustrated on Figure~\ref{Fig:Rec Small copy} and goes as follows.  The renormalization Siegel fixed point $\bF_\str$ belongs to the boundary of the main hyperbolic component $\HH\subset \Unst$. Let $\bF_\rr\in \partial \HH$ be a parabolic prepacman close to $\bF_\str$. Then there is a satellite hyperbolic component $\HH_\rr\subset \Unst$ attached at $\bF_\rr$. In a small neighborhood of $\bF_\rr$, there is another parabolic prepacman $\bF_{\rr,\ss}$; let $\HH_{\rr,\ss}\subset \Unst$ be the secondary satellite hyperbolic component attached at $\bF_{\rr,\ss}$. For every $\bG\in \partial \HH_{\rr,\ss}\setminus \{\bF_{\rr,\ss}\}$, appropriate periodic rays in $\Esc(\bG)$ land together and form a quadratic-like domain (after a thickening) for the partial small copy $\bMM(\HH_{\rr,\ss})$ centered at $\HH_{\rr,\ss}$. We {\bf do not know} whether $\bMM(\HH_{\rr,\ss})$ is bounded or even complete -- this is related to the realization of parameter rays, see~\S\ref{ss:param rays}. However, there is a locally continuous straightening map $\chi\colon \bMM(\HH_{\rr,\ss})\to \Mandel$; using the Yoccoz inequality for $\Mandel$ and then $\chi^{-1}$, we can find a parabolic prepacman $\bF_{\rr,\ss,\tt}\in \partial \HH_{\rr,\ss}$ in a small neighborhood of $\bF_{\rr,\ss}$ together with a ternary small copy of the Mandelbrot set $\bMM_0= \bMM_{\rr,\ss,\tt}\subset \Unst$ attached to $\bF_{\rr,\ss,\tt}$. 

We set $\bMM_n\coloneqq \RR^n (\bMM_0)\subset \Unst, n\in \Z$, to be the renormalization orbit of $\bMM_0$.

In {\bf Section~\ref{s:val flow}}, we prove the Valuable Flower Theorem: for $n\ll 0$, if $\bG\in \bMM_n$, then the associated pacman $g\in \WW^u$ has a \emph{valuable flower} $X(g)$ around $\alpha(g)$ in a small neighborhood of $\overline Z(f_\str)$ defined as the ``combinatorial connected hull'' of the cycle of secondary small filled-in Julia sets, see~\S\ref{ss:ValFlow:pacman} and Figure~\ref{Fig:thm:proj of val flow}. In particular, $X(g)$ contains the postcritical set of $g$. We construct first the valuable flower $\bX(\bG)$ in the dynamical plane of $\bG$  (where external rays of $\bG$ help to control the location of $\bX(\bG)$) and then project to the $g$-plane using Theorem~\ref{thm:quot of fund domain}. We denote by $\MM_n\subset \WW^u$ the set of pacmen $g$ with $\bG\in \bMM_n$.

The valuable flower $X(g)$ labels the hybrid class of $g$: there is a unique quadratic polynomial $p\in \Mandel$ such that $p$ has a valuable flower $X(p)$ and $g$ and $p$ are hybrid conjugate in neighborhoods of their valuable flowers $X(g)$ and $X(p)$. It provides us with the straightening map from $\MM_n$ to the associated small copy $\Mandel_n\subset \Mandel$ containing all $p=p(g)$ with $g\in \MM_n$.

The proofs of the main results are collected in {\bf Section~\ref{s:proof:main thms}.} We first construct a \emph{stable lamination} in the space of pacmen as follows. For $\bbn\ll 0$, there is a local complex codimension-one lamination $\bFol_\bbn$ in a small neighborhood of $\MM_\bbn$ characterized by the property that pacmen in the same leaf are hybrid conjugate in neighborhoods of their valuable flowers, see~\S\ref{ss:StabLamin}. For $m\le \bbn$ we define $\bFol_m$ to be the pullback of $\bFol_\bbn$ under $\RR^{\bbn-m}$. By hyperbolicity of $\RR$,
\[\bFol\coloneqq \bigcup_{m\le \bbn} \bFol_m\cup \{\WW^s\}\] 
forms a codimension-one lamination.

\emph{Proof of the main results (rough outline).}  Theorem~\ref{thm:main} essentially follows from the hyperbolicity of $\RR$ combined with the holonomy along $\bFol$. The stable manifold $\WW^s$ intersects the slice of quadratic polynomials $\QQ$ at $c(\theta)$. For $n\ll 0$, the intersection of $\bFol_n$ with the (renormalized) slice $\QQ$ is the ternary copy $\Mandel_n$ of the Mandelbrot set. The map $\mRR\coloneqq \mRRc^\mm$ from~\eqref{eq:thm:main} factorizes as $\RR\mid \QQ$ postcomposed with the holonomy along $\bFol$ bringing $\RR(\QQ)\cap \bFol$ back to $\QQ$. Since the holonomy is  asymptotically conformal, the hyperbolicity of $\RR$ implies the scaling theorem.

For $k< n$, we have $\mRR_k=\mRR_n\circ \mRR^{n-k}$. Since \[\mRR^{n-k}\big( c(\theta) +z\big)\approx c(\theta) +\lambda_\str ^{n-k} z,\sp\sp \lambda_\str>1\] is expanding, the composition $\mRR_n\circ \mRR^{n-k}=\mRR_k$ is expanding for fixed $n<0$ and a sufficiently big $n-k \gg 0$. Therefore, the non-escaping set of $\mRR_n\colon \Mandel_n\to \Mandel$ consists of a single parameter $c_n$. This implies the rigidity part of Theorem~\ref{thm:main}.  


Let $g_n\colon U_n\to V_n$ be the quadratic-like renormalization fixed point associated with $\mRR_n\colon (\Mandel_n,c_n)\to (\Mandel,c_n)$, where $g_n$ is hybrid conjugate to $p_n(z)=z^2+c_n$. Consequently, $g_n$ has a quadratic-like restriction $g_n^{m(n)}\colon U'_n\to V'_n$ affinely conjugate to $g_n\mid U_n$. Such a structure already implies the JLC. In fact, a priori bounds for the $g_n$ are controlled by a quadratic-like domain of the copy $\bMM_0\subset \Unst$ discussed above; i.e. the a priori bounds for $g_n$ are uniform over $n$.

As $n\to -\infty$, the map $g_n\colon  U_n\to V_n$ tends to a Siegel quadratic-like map $g_\str\colon U_\str\to V_\str$ so that the valuable flower $X(g_n)$ approximates the Siegel disk $\overline Z(g_\str)$. For such an approximation, the Koebe-type estimates of~\cite{AL-posmeas}*{\S\S6.6--6.8} imply that for $n\ll 0$ the probability of the $g_n$-orbit of $z\in \overline Z(g_\str)$ to enter $U'_n$ is much higher than the probability of escaping $U_n$. By~\cite{AL-AMS} the Julia set of $g_n$ (and hence of $p_n$) has positive area. Since our notations are different, we will recap the argument in~\S\ref{ss:PositArea} -- see its beginning for a short outline and comparison with \cite{AL-posmeas}.

\subsection{Remarks about Transcendental and Neutral Dynamics} \label{ss:Rem Transc Dynam}It has been long observed that the dynamics near indifferent periodic points is  strikingly similar to the transcendental dynamics. For instance, the Julia sets at Cremer points look similar to ``Cantor bouquets'' -- the Julia sets of certain exponential maps \cite{dg}. This was a guiding idea for Shishikura's seminal work~\cite{Sh:HD}. The similarity between Neutral and Transcendental dynamics was broadly advertised in numerous talks by Shishikura, Epstein, Rempe, and others. It has become clear over time that the Cremer and Siegel phenomena are not as mysterious as they once were viewed. However, only recently they were rigorously explained in the Inou-Shishikura class~\cite{SY,Che}.

A substantial difficulty in the neutral renormalization theory is that arising maps (such as pacmen, see Figure~\ref{Fg:Pacman}) are {\bf not} genuine branched covering. It appears that some arising issues can be resolved by considering transcendental extensions on the unstable manifolds. In the 1990s, McMullen observed that renormalization periodic points (in both, quadratic-like and Siegel cases) admit maximal extensions as $\sigma$-proper branched coverings of the complex plane.\footnote{The domain of analyticity for the Feigenbaum renormalization fixed point was first studied by H. Epstein~\cites{E1,E2}.} In~\cite{DLS} we extended this result to every map on the unstable manifold $\WW^u$ of a Siegel renormalization periodic point. This was a key ingredient in our proof that $\dim \WW^u=1$. Roughly: since maximal prepacmen have a single critical orbit, they naturally form a one-dimensional space.

In the current paper, we construct external rays and develop the puzzle theory for the limiting transcendental family $\WW^u$. We also introduce a machinery to transfer results from $\WW^u$ to the dynamical planes of rational maps.  We believe that with further advancements in the Transcendental Dynamics (the theory of parameter rays, see \S\ref{ss:param rays}) and the Neutral Renormalization Theory (the full hyperbolicity over all combinatorics), the understanding of the near-Neutral Dynamics could be brought to an essentially complete form. For illustration, let us discuss below two central ideas from Sections~\ref{s:val flow} and \ref{s:proof:main thms}. 



\emph{Valuable flowers.} An essential ingredient in the constructions of positive area Julia sets in~\cites{BC, AL-posmeas} is the Buff-Cheritat lemma asserting that certain perturbations of a Siegel map $f$ have the postcritical sets in a small neighborhood of the Siegel disk $Z(f)$. The Buff-Cheritat lemma allows Koebe-type area estimates (see~\S\ref{ss:PositArea}) and is an application of the Almost Parabolic Renormalization Theory~\cite{IS}. 

A valuable flower is roughly the ``combinatorial connected hull'' of the postcritical set and is a near-neutral analogy of the filled Julia set of a polynomial. Just like filled Julia sets of polynomials depend upper semicontinuously on the parameter, we conjectured in~\cite{DLS}*{Appendix C.4} that Siegel disks/hedgehogs/valuable flowers depend upper semicontinuously on the parameter in (at least) the main molecule of the Mandelbrot set. In  Section~\ref{s:val flow}, we designed a soft argument for the upper-semicontinuity: once a valuable flower $ \bX(\bG)$ is recognized in the dynamical plane of a maximal prepacman $\bG$, the flowers $ \bX(\bG_n), \sp \bG_0=\bG,\sp n\le 0$, converge to the Siegel disk $\overline \bZ(\bF_\str)$ under the antirenormalization. The argument is based on the convergence in $\wC\setminus \bZ_\str$ of wakes of $\bG_n$ to the wakes of $\bF_\str$, see Lemma~\ref{lem:bW:bG to bFstr}.

We emphasize that the positive area property is almost automatic, once a valuable flower is constructed and sufficiently many antirenormalizations are taken, see Remark~\ref{rem:posit area}.


\emph{Hybrid lamination.}
It was shown in~\cite{L:FCT} that hybrid classes foliate the connectedness locus of the complex space of all quadratic-like maps. The argument can not be adopted to the space of pacmen -- they are not branched coverings. 

To construct hybrid lamination in the space of pacmen, we employ the renormalization. We first recognize maps on the unstable manifold labeled by their valuable flowers. This leads to a local lamination in a small neighborhood of recognized parameters. Pulling back the local lamination using the renormalization, we obtain a fairly dense hybrid lamination in a neighborhood of the stable manifold. 

If the full hyperbolic renormalization horseshoe is constructed for neutral renormalization, then the renormalization will be much more efficient in creating hybrid lamination -- it can be pulled back along various branches. Combined with Conjecture~\ref{conj:P-P relation}, this may lead a compete theory, see Remark~\ref{rem: pacman str map}.



\subsection*{Acknowledgments} The first author was partially supported by Simons Foundation grant of the IMS,  the ERC grant ``HOLOGRAM'', and the NSF grant DMS
2055532. The second author has been partly supported by the NSF, the Hagler and Clay Fellowships, the Institute for Theoretical Studies at ETH (Zurich), and MSRI (Berkeley).

We also thank Artur Avila for stimulating the discussion of the area problem.

Results of this paper were first announced at the conference in memory of Yoccoz, May 2017, Coll\`ege de France (Paris).

A number of our pictures are made with W. Jung's program \emph{Mandel}.

We thank the referee for carefully reading the paper and making many useful comments.

\section{Sector renormalization}
\label{s:SectRenorm}
In this section we refine the discussion of the sector renormalization from~\cite[Appendices A and B]{DLS}.

\subsection{Sector renormalization of rotations}
\label{ss:ren of rotat}

\begin{figure}[tp!]
\begin{tikzpicture}[  scale=1.3]

  \begin{scope}[shift={(0,0)}, scale =0.5]

\draw(0.,0.) circle (4.cm);

\draw[red ,shift={(0.,0.)}, fill=red, fill opacity=0.2]
 plot[domain=2.557797598845188:3.8669922442899227,variable=\t]({1.*4.*cos(\t r)+0.*4.*sin(\t r)},{0.*4.*cos(\t r)+1.*4.*sin(\t r)});

\draw[draw opacity=0, red, fill, fill opacity=0.2] (0.,0.)-- (-3.337507402085898,2.204777617135533)
  -- (-2.9929375178397235,-2.653737932484554)--(0.,0.);
  \draw[red ] (0.,0.)-- (-3.337507402085898,2.204777617135533)
   (-2.9929375178397235,-2.653737932484554)--(0.,0.);
    \draw[blue] 
   (1.26, 3.8)--(0,0);
   \draw (-0.42, 2.44) node {$\Sbb_-$};
   \draw (-2, 0) node {$\Sby$}; 
    \draw (1.5, -1.5) node {$\Sbb_+$}; 
    \draw (-0.62, 2.14) edge[->,bend right] node[left]{$\Lbb_\theta$}(-1.8,0.3);
    \draw (-2,-0.7) node {(delete)};
    \draw  (-1.8,-0.9) edge[->,bend right] node[left]{$\Lbb_\theta$}(-0.32, -2.5);
    \draw (1.94, 0.68) edge[->,bend right] node[right]{$\Lbb_\theta$}(-0., 2.24);
    \node at (-3.6, -2.94){$\Lbb_\theta(\Ibb)$};
    \node at (-3.6, 2.64){$\Ibb$};
\end{scope}

 \begin{scope}[ shift={(5,0)}, scale =0.5, rotate around={180:(0,0)}]
\draw(0.,0.) circle (4.cm);

\begin{scope}[rotate around={7:(0,0)}]
\draw[blue ,shift={(0.,0.)}, fill=blue, fill opacity=0.2]
 plot[domain=2.557797598845188:3.8669922442899227,variable=\t]({1.*4.*cos(\t r)+0.*4.*sin(\t r)},{0.*4.*cos(\t r)+1.*4.*sin(\t r)});

\draw[draw opacity=0, blue, fill, fill opacity=0.2] (0.,0.)-- (-3.337507402085898,2.204777617135533)
  -- (-2.9929375178397235,-2.653737932484554)--(0.,0.);
    \draw[blue ] (0.,0.)-- (-3.337507402085898,2.204777617135533)
   (-2.9929375178397235,-2.653737932484554)--(0.,0.);
\end{scope}  
  
  \node at (-4.3, 0){$1$};

   \draw[blue] (0,0)--(-4, -0);
   \draw (-3, -1.3) node {$\Sbb_+$};
   \draw (-2.92, 0.86) node {$\Sbb_-$};

    \draw (2,-0.7) node {\LARGE (delete)};
      \draw (-0.62, 2.14) edge[->,bend right] node[right]{$\Lbb_\theta$}(-1.8,0.3);

    \draw  (-1.8,-0.9) edge[->,bend right] node[right]{$\Lbb_\theta$}(-0.32, -2.5);
     \draw (1.94, 0.68) edge[->,bend right] node[left]{$\Lbb_\theta$}(-0., 2.24);
     \draw[fill]
     (0.22, -2.6) circle (0.05cm)
     (0.9, -2.5) circle (0.05cm)
     (1.6, -2.1) circle (0.05cm);
\end{scope}    
    
\end{tikzpicture}
\caption{Left: the prime renormalization deletes the smallest sector $\Sby$ between $\Ibb$ and $\Lbb_\theta(\Ibb)$ and projects $(\Lbb^2_\theta \mid \Sbb_-~,\sp \Lbb_\theta \mid \Sby_+)$ to a new rotation. Right: more generally, a sector renormalization realizes the first return map to a sector $\Sbb_-\cup\Sbb_+$}
\label{Fg:RenOfRotDisc}
\end{figure}

 Consider the rotation of the unit disk 
\begin{equation}
\label{eq:defn:Lbb}
  \Lbb_\theta\colon  \ovDisk\to \ovDisk,\sp z\to \ee(\theta) z
 \end{equation}
 by an angle $\theta\in \R/\Z$. Choose a closed internal ray $\Ibb$ of $\ovDisk$, and let $\Sby\subset\overline {\Disk}$ be the smallest closed sector between $\Ibb$ and $\Lbb_\theta(\Ibb)$, see the left side of Figure~\ref{Fg:RenOfRotDisc}. Consider \[\Sbb_-\coloneqq \Lbb_{\theta}^{-1}(\Sby)\sp \text{ and }\sp\Sbb_+\coloneqq \overline {\Disk\setminus\left( \Sby\cup \Sbb_-\right)}.\]
 Then 
\begin{equation}
\label{eq:ap:FRM:DeletSect}
(\Lbb_\theta\mid \Sbb_+,\sp  \Lbb^2_\theta\mid \Sbb_-)
\end{equation}is the first return of points in $\Sbb_-\cup \Sbb_+$ back to $\Sbb_-\cup \Sbb_+$. Let $\omega\in [0,1/2]$ be the angle of $\Sby$ at the vertex $0$. Assuming $1\not\in \Sby$, the map $z\to z^{1/(1-\omega)}$ projects~\eqref{eq:ap:FRM:DeletSect} to a new rotation $\Lbb_{\cRRc(\theta)}\colon \ovDisk\to \ovDisk$, called the \emph{prime renormalization} of $\Lbb_\theta$.  Direct calculations (see~\cite{DLS}*{Lemma A.1}) show that
 
\begin{equation}
\label{eq:R_prm}
\cRRc(\theta)  = \begin{cases} \frac{\theta}{1-\theta}& \mbox{if }0\le \theta \le \frac{1}{2}, \\
\frac{2\theta-1}{\theta} & \mbox{if }\frac{1}{2}\le \theta\le 1.
\end{cases}
\end{equation}
 
  More generally, a \emph{sector renormalization} $\RR$ of $\Lbb_\theta$ is (see Figure~\ref{Fg:RenOfRotDisc})
\begin{itemize}
\item a \emph{renormalization} sector $\Sbb$ presented as a union of two subsectors  $\Sbb_-\cup \Sbb_+$ normalized so that $1\in \Sbb_-\cap \Sbb_+$;
\item  a pair of iterates, called a sector \emph{pre-renormalization}, 
\begin{equation}
\label{eq:app:FirstReturnToSector}
\left(\Lbb_\theta ^\aa\mid \Sbb_- , \sp\sp \Lbb_\theta ^\bb\mid \Sbb_+\right)
\end{equation}
realizing the first return of points in $\Sbb_-\cup \Sbb_+$ back to $\Sbb$; and 
\item the gluing map 
\[\psi\colon \Sbb_-\cup \Sbb_+\to \ovDisk,\sp\sp z\to z^{1/\omega},\]
projecting~\eqref{eq:app:FirstReturnToSector} to a new rotation $\Lbb_\mu$,
where $\omega$ is the angle of $\Sbb$ at $0$.
\end{itemize}
  
  A sector renormalization is an iteration of the prime renormalization (see~\cite[Lemma A.2]{DLS}); in particular, $ \mu=\cRRc^m(\theta)$ for some $m\ge1$.

\subsubsection{Renormalization of pairs of translations}
\label{sss:renorm:pair of rotat}
 Let us now lift these renormalizations of rotations to the universal cover of $\overline {\Disk}^*\coloneqq \overline {\Disk}\setminus\{0\}$. Write 
\[ \H_-\coloneqq\{z\mid \Im(z)< 0\}.\] 
 Writing $\vv^-=-\vv$, the translations
\[T_{\vv^-}=T_{-\vv}\coloneqq T_{ -\theta}\colon z\mapsto z-\theta, \sp \sp T_{\ww}\coloneqq  T_{ 1-\theta }\colon z\mapsto z +1-\theta\]  are two lifts of $\Lbb_{\theta}\colon \overline {\Disk}^*\to \overline {\Disk}^*$ 
under 
\begin{equation}
\label{eq:sect ren:unv cover}
\ee_-\colon \sp z\mapsto e^{-2\pi i  z } \colon\sp \oH_- \to \overline {\Disk}^*.
\end{equation}
Since $\chi\coloneqq T_\ww\circ T_{\vv^-}^{-1}$ is a deck transformation of~\eqref{eq:sect ren:unv cover}, we have $\Lbb_\theta\simeq T_{-\vv}/\langle \chi\rangle $ as a map on $\ovDisk^*$.

\begin{figure}[tp!]
\begin{tikzpicture}
\draw  (-6,0) edge[->]   (2,0); 
\node[above] at  (2,0){$\vv^-$};
\draw  (0,4) edge[<-] (0,-1);
\node[right] at (0,4){$\ww$};

\draw  (-4,4)--  (1,-1); 
\node[right] at (1,-1){$\vv=\ww$}; 
\draw[blue] (-2.5,1) edge[bend left,->] node[below left] {$\left(\begin{matrix}
        1 &1\\0 &1
        \end{matrix}\right)$ }(-3.5,3.8); 
        
\draw[red] (-0.5,1) edge[bend right,->] node[above right] {$\left(\begin{matrix}
        1 &0\\1 &1
        \end{matrix}\right)$ }(-2.5,2.2);         

\end{tikzpicture}
\caption{The map $\cRRc$ is $2$-to-$1$ on $\R_{\le0}\times \R_{\ge0}$; it maps two $1/8$-subsectors to $\R_{\le0}\times \R_{\ge0}$.}
\label{Fig:act of prim ren}
\end{figure}

For a non-zero vector $\left(\begin{matrix}
        \vv^- \\
  \ww
        \end{matrix} \right)=\left(\begin{matrix}
        -\vv \\
  \ww
        \end{matrix} \right)$ in $\R_{\le0}\times \R_{\ge0}$ 
write, (see Figure~\ref{Fig:act of prim ren})

 \begin{equation}
\label{eq:R_prm:proj}
\left(\begin{matrix}
        -\vv_1 \\
  \ww_1
        \end{matrix} \right) =  
  \left(\begin{matrix}
        \vv^-_1 \\
  \ww_1
        \end{matrix} \right) = \cRRc \left(\begin{matrix}
        \vv^- \\
  \ww
        \end{matrix} \right)\coloneqq  
        \left\{ \begin{matrix}
          \left(\begin{matrix}
        \vv^-+\ww \\
         \ww
        \end{matrix} \right) &\sp \text{ if $\vv\ge \ww$},\\&&\\
          \left(\begin{matrix}
        \vv^- \\
  \ww+\vv^-
        \end{matrix} \right) &\sp \text{ if $\ww> \vv$}.
        \end{matrix} 
       \right.
 \end{equation}
and observe that~\eqref{eq:R_prm} is the \emph{projectivization} of~\eqref{eq:R_prm:proj} via 
 \begin{equation}
\label{eq:v w to theta}
\left(\begin{matrix}
        -\vv \\
  \ww
        \end{matrix} \right) \to \frac{\vv}{\vv+\ww}\eqqcolon\theta\in \R/\Z.        
\end{equation}

The \emph{prime pre-renormalization} of the commuting pair $T_{\vv^-},~T_{\ww}$ with $\vv^-,\ww\in \R_{\le 0}\times \R_{\ge0}$ is the commuting pair \[(T_{\vv^-_1},~T_{\ww_1})=\RRc(T_{\vv^-},~T_{\ww});\] it is obtained (see Figure~\ref{Fig:prime pre ren}) by replacing $T_{\vv^-}$ with $T_{\vv^-}\circ T_\ww$ if $\vv\ge \ww$, and by replacing $T_{\ww}$ with $T_{\vv^-}\circ T_\ww$ otherwise. We denote by $\chi_1\coloneqq T_{\ww_1}\circ T_{\vv^-_1}^{-1}$ the new deck transformation. If $\Lbb_\theta\simeq T_{\vv^-}/\langle \chi \rangle $ (i.e.~\eqref{eq:v w to theta} holds), then $\Lbb_{\cRRc(\theta)}\simeq T_{\vv^-_1}/\langle \chi_1 \rangle$ as a map on $\ovDisk^*$.

Let us now consider an iteration of~\eqref{eq:R_prm:proj}. Recall that matrices  $\left(\begin{matrix}
        1 &1\\0 &1
        \end{matrix}\right)$ and $\left(\begin{matrix}
        1 &0\\1 &1
        \end{matrix}\right)$ generate the modular group $\SL_2(\Z)$. We write  \[\SL^{(\mm)}_2(\Z)\coloneqq \left\{I_1I_2,\dots, I_m\mid I_i \in \left\{\left(\begin{matrix}
        1 &1\\0 &1
        \end{matrix}\right), \left(\begin{matrix}
        1 &0\\1 &1
        \end{matrix}\right) \right\}\right\},\] and we denote by $\SL^+_2(\Z)=\displaystyle \bigcup_{\mm\ge 0}\SL_2^{(\mm)}(\Z)$ the ``positive'' sub-semigroup of $\SL_2(\Z)$. The quadrant $\R_{\le 0}\times\R_{\ge 0}$ splits into $2^\mm$ equal closed sectors so that on each sector $S$ the map $\cRRc^\mm$ is equal to 
\begin{equation}
  \label{eq:sect to quarter}      
          x\mapsto \M x \colon\sp  S \twoheadrightarrow\R_{\le0}\times \R_{\ge 0}        
\end{equation} for a certain $\M\in \SL_2^{(\mm)}(\Z)$. (As a consequence, $\SL^+_2(\Z)$ is a free semigroup.)

Since $S$ is a proper subsector of $ \R_{\le0}\times \R_{\ge 0} $, the operator ~\eqref{eq:sect to quarter} has an eigenvector $\left(\begin{matrix}
        \vv^- \\ \ww
        \end{matrix}\right)\in S$,  unique up to scaling. Note that $\vv=-\vv^-,\ww>0$ unless $S$ is a boundary sector containing either $\left(\begin{matrix}
        -1\\0
        \end{matrix}\right)$ or $\left(\begin{matrix}
        0\\1
        \end{matrix}\right)$; in that case $\M \in \left\{ \left(\begin{matrix}
        1 &n\\0 &1
        \end{matrix}\right), \left(\begin{matrix}
        1 &0\\n &1
        \end{matrix}\right)\right\}$.

We assume that $S$ is not a boundary sector. Then all the entries of $\M$ are positive and $\M$ has two eigenvalues $\tt>1$ and $1/\tt<1$ so that
\begin{equation}
\label{eq:vv1 ww1}
 \left(\begin{matrix}
        \vv^-_1 \\
  \ww_1
        \end{matrix} \right) \coloneqq 1/\tt\left(\begin{matrix}
        \vv^- \\ \ww
        \end{matrix}\right) = \M \left(\begin{matrix}
        \vv^- \\ \ww
        \end{matrix}\right) .
  \end{equation}

\begin{figure}[tp!]
\begin{tikzpicture}
\draw (-7,0) -- (3.5,0);
\draw[red] (0,0 ) -- (0,-3)
(-6,0 ) -- (-6,-3)
(2.5,0 ) -- (2.5,-3);
\draw[red,dashed]  (-3.5,0 ) -- (-3.5,-3);
\draw[blue,bend right =20] (-0.1,-1) edge[->] node[below]{$T_{\vv^-}$}(-5.9,-1);
\draw[blue,bend left =30] (0.1,-1) edge[->] node[below]{$T_{\ww}$}(2.4,-1);
\draw[blue,bend right =10] (-0.1,-2.8) edge[->] node[below]{$T_{\vv^-_1}$}(-3.4,-2.8);
\draw[blue,bend left =10] (-3.4,-2) edge[->] node[below]{$\chi_1$}(2.4,-2);
\end{tikzpicture}
\caption{The prime pre-renormalization of a pair of translations $T_{\vv^-},~T_{\ww}$ replaces $T_{\vv^-}$ with $T_{\vv_1}\coloneqq T_{\vv^-}\circ T_{\ww}$ if $\vv\ge \ww$; in this case the new deck transformation is $\chi_1=T_{\vv}=T_{\ww_1}\circ T_{\vv^-_1}^{-1}$.}
\label{Fig:prime pre ren}
\end{figure}

Writing $\theta=\frac{\vv}{\vv+\ww}$, we see that $\theta =\cRRc^\mm(\theta)$, and we say that $\M$ is the \emph{antirenormalization matrix} associated with $\theta =\cRRc^\mm(\theta)$. It is easy to see that all periodic points of $\cRRc$ arise from the above construction.

\begin{lem} 
\label{lem:lambda:t*t}
For $\theta$ and $\tt$ as above, we have
 \[\left(\cRRc^\mm\right)'(\theta)=\tt^2.\]
\end{lem}        
 \begin{proof}
We can view $\theta$ as a fixed point of the M\"obius transformation induced by $\M$; its derivative $\left(\cRRc^\mm\right)'(\theta)$ is equal to $\tt^2$. 

Equivalently, direct calculations show that if $ \left(\begin{matrix}
        \vv^-_1 \\
  \ww_1
        \end{matrix} \right) = \cRRc \left(\begin{matrix}
        \vv^- \\
  \ww
        \end{matrix} \right)$, then
\[ \cRRc'\left(\frac{\vv}{\vv+\ww}\right)=\left(\frac{\vv+\ww}{\vv_1+\ww_1}\right)^2.\] If~\eqref{eq:vv1 ww1} holds, then $\frac{\vv+\ww}{\vv_1+\ww_1}=\tt$ and the claim follows.
 \end{proof}
 
\subsubsection{Cascade $\big(T^P\big)_{P\in \PT}$}
\label{sss:power-triples}
 Let us fix $\vv,\ww,\theta,\mm, \tt, \M$ as above so that, in particular,~\eqref{eq:vv1 ww1} holds. Observe that $\tt\not\in \Q$, because $\tt>1$, $\det \M=1$, but the entries of $\M$ are positive integers. We set $\cRR\coloneqq \cRRc^\mm$. For $n\in \Z$, write \[\vv_n\coloneqq \tt^{-n}\vv,\sp\sp \ww_n \coloneqq \tt^{-n} \ww;\]
then $(T_{\vv^-_n},\sp T_{\ww_n} )_{n\in \Z}$ is the full pre-renormalization tower: \begin{equation}
\label{eq:defn RR on transl}
\RR \left(T_{\vv^-_n},\sp T_{\ww_n} \right)=\left(T_{\vv^-_{n+1}},\sp T_{\ww_{n+1}} \right),
\end{equation}
where $\RR=\RRc^\mm$.

For an abelian semigroup
 \[\PT^{n}\coloneqq \{(n,a,b)\mid a,b \in\Z_{\ge 0}\}\simeq\Z_{\ge0}^2\]
we define the monomorphism
\[\sigma_n\colon \PT^{n}\hookrightarrow \PT^{n-1}\colon (n,(a,b))\mapsto (n-1,(a,b)\M).\]

 For a \emph{power-triple} $(n,a,b)\in \Z\times \Z_{\ge0 }\times \Z_{\ge0}$ we write 
\begin{equation} 
 \label{eq:T n a b}
 T^{(n,a,b)}\coloneqq T^a_{\vv^-_n}\circ T^b_{\ww_n}=T_{\tt^{-n}(b\ww-a\vv)},
\end{equation}
(Later on the non-invertible maximal prepacman $\big(\bbf_{n,-}^\#,~ \bbf_{n,+}^\#\big)$ will play the role of $(T_{\vv^-_n},~T_{\ww_n})$.)

\begin{figure}[tp!]

{\begin{tikzpicture}[description/.style={fill=white,inner sep=2pt}]
    \matrix (m) [matrix of math nodes, row sep=4.5em,
    column sep=10em, text height=1.5ex, text depth=0.25ex]
{     & {\PT^{n-1}}&\\
     { \Z^2_{\ge0}}& & {\R_{\ge 0}}\\
   &{\PT^{n}}&\\};
    \path[->,font=\scriptsize]
 (m-2-1) edge[dashed, ->] node[description] {$\simeq$} (m-1-2) 
 (m-2-1)  edge[dashed,->] node[description] {$\simeq$} (m-3-2) 
  (m-3-2)  edge[->] node[description] {$\M\simeq\sigma_n$}  (m-1-2) 
  (m-1-2)  edge[right hook->]  node[description] {$ \tt ^{n-1}\proj_\tt \simeq \iota $} (m-2-3) 
  (m-3-2)  edge[right hook->]  node[description] {$ \tt ^{n}\proj_\tt\simeq \iota  $} (m-2-3) ;

  \end{tikzpicture}}
\caption{Every $\PT^n$ is a copy of $\Z^2_{\ge0}$ and can be consistently embedded into $\R_{\ge 0}$. This induces the embedding $\iota$ of $\displaystyle{\PT\coloneqq \lim_{\longrightarrow }\PT^{n}}$ into $\R_{\ge 0}$. 
Note that dashed arrows do not commute with solid arrows.}
\label{Fig:Diagram T^n}
\end{figure}

Observe that 
\begin{equation}
\label{eq:T n a b=n-1 c d}
T^{(n,a,b)}=T^{\sigma_n(n,a,b)}.
\end{equation}
Indeed, $T^{(n,a,b)}$ is the translation by
\[(a,b) \left(\begin{matrix} v_n^- \\ w_n  \end{matrix} \right)=(a,b) \left(\M  \left(\begin{matrix} v_{n-1}^- \\ w_{n-1}  \end{matrix} \right)\right)=
\left( (a,b) \M \right)  \left(\begin{matrix} v_{n-1}^- \\ w_{n-1}  \end{matrix}\right);\]
thus $T^{(n,a,b)}= T^{(n-1, (a,b)\M)}$. 

We define the \emph{semi-group of power-triples} as the direct limit
\begin{equation}
\label{eq:defn:PT}
\PT\coloneqq \lim_{\longrightarrow }\PT^{n} =\{(x_i)_{i\le k}\mid k\in \Z, \sigma_i(x_i)=x_{i-1}\}.
\end{equation}
We also write $\PT=\displaystyle\bigcup_{n\in \Z} \PT^n=\{(n,a,b)\in \Z\times \Z^2_{\ge 0}\}/\sim$. By~\eqref{eq:T n a b=n-1 c d}, $\PT$ acts naturally on $\R$ by translations; i.e.~$(T^P)_{P\in \PT}$ is a \emph{cascade}. 

\begin{lem} 
\label{lem:PT:rot}
The action of $\PT$ on $\R$ is free: $T^{(n,a,b)}=T^{(m,c,d)}$ if and only if 
\begin{equation}
\label{eq:lem:PT:rot}
(a,b)\M^n = (c,d)\M^{m} .
\end{equation}

 Let $e_{\tt} \in \R^2_{>0},e_{1/\tt} \in \R_{<0}\times  \R_{>0}$ be the eigen-covectors (viewed as rows) of the eigenvalues $\tt$ and $1/\tt$ of $\M$:
\[e_{\tt} \M =\tt e_\tt,\sp\sp\sp  e_{1/\tt} \M =  (1 /\tt )e_{1/\tt}.\]  
   Decompose every covector $\omega \in \R^2_{\ge 0}$ as
\[\omega= \proj_\tt(\omega) e_\tt+\proj_{1/\tt}(\omega) e_{1/\tt},\sp \sp \proj_\tt(\omega)\in \R_{\ge 0},\sp \proj_{1/\tt}(\omega)\in \R.\] Then 
\[\iota(n,a,b)\coloneqq  \tt ^n\proj_\tt  \left(\begin{matrix}
          a \\ b
        \end{matrix}\right)\]
  induces an embedding of $\PT$ into $\R_{\ge 0}$  (see also Figure~\ref{Fig:Diagram T^n}) such that
\[
\iota(n-1,a,b)= \iota(n,a,b)/\tt.
\]

View now $\PT$ as a sub-semigroup of $\R_{\ge0}$. This turns $\PT$ into a linearly ordered semi-group with subtraction:
\[\text {if }\sp P\ge T,\sp\text{ then }\sp  P-T\in \PT,  \sp \sp \sp \sp P,T\in \PT \subset \R_{\ge 0} ;\]
 $P\mapsto \tt P\colon (n,a,b)\mapsto (n+1,a,b)$ is an automorphism of $\PT$; and:
\begin{equation}
\label{eq:T^P vs T^tP}
 T^P=A_{ \tt^{n}} \circ \left(T^{\tt^n P} \right) \circ A_{\tt^{-n}} ,
\end{equation}
where $A_{\tt^n}\colon z\mapsto \tt^n z$ is scaling.
\end{lem}
\begin{proof}
If the action of $\PT$ on $\R$ is not free, then there are $(n,a,b)\not= (n,c,d)$ such that
$T^{(n,a,b)}=T^{(n,c,d)}$; this is equivalent to $(a-c)\vv^-+(b-d)\ww=0$. Since $\tt\not\in \Q$, the coordinates $\vv^-$ and $\ww$ are rationally independent. Therefore, $a=c$ and $b=d$. 

Clearly, $\iota\colon \PT\to \R_{\ge0}$ is a homomorphism such that $\iota(n-1,a,b)= \iota(n,a,b)/\tt$. If $\iota$ was not an embedding, then there would be $(a,b)\not=(c,d)\in \Z^2_{\ge 0}$ with 
$\proj_\tt (
          a , b)= \proj_\tt  (
          c , d
        )$; i.e.~ $(
          a-c , b-d)$ is a non-zero integer covector parallel to $e_{1/\tt}$. This is impossible because coordinates of $e_{1/\tt}$ are rationally independent.

If $\iota(n,a,b)> \iota(n,c,d)$, then for sufficiently big $m\gg 0$ the covector
  \begin{align*}
 \left(
          a_m , b_m \right) \coloneqq & \left(
          a-c , b-d \right)\M^m \\ =& \tt^m\proj_\tt \left(
          a-c , b-d \right)   e_\tt +  \tt^{-m}\proj_{1/\tt} \left(
          a-c , b-d \right)  e_{1/\tt} 
\end{align*}
   has positive coordinates because $\tt>1$ and $\proj_\tt \left(
          a-c , b-d \right)  =\iota(n,a,b)- \iota(n,c,d)>0$. Therefore, $ (n,a,b)-(n,c,d)\simeq (n-m, a_m, b_m) \in \PT$. The remaining claims follow immediately from the definitions.
\end{proof}  

\subsubsection{Renormalization tilings}
\label{sss:ren tiling:transl}
Consider the following closed intervals
\[J_{0}(0)\coloneqq [-\vv,0],\sp\sp J_{0}(1)=[0,\ww].\]
Note that $\widetilde J_0\coloneqq J_{0}(0)\cup J_{0}(1)$ is a fundamental domain for the deck transformation $\chi$ and that
\[T_{\vv^-}\colon  J_{0}(1) \to \widetilde J_0,\sp T_{\ww}\colon  J_{0}(0) \to \widetilde J_0 \]
 realizes the first return of points in $\widetilde J_0$ back to $\widetilde J_0$ under the cascade $T^{\ge0}\coloneqq \big(T^P\big)_{P\in \PT}$.

\begin{figure}
\begin{tikzpicture}
  \tikzmath{
\b1=10*0.61803398875;
\a1=(0.61803398875-1)*10;
\x1=2.6180339928;
\x2=2.6180339928*2.6180339928;
}

\draw (\a1-0.3,0)--(\b1+0.3,0);

\draw (0,-0.2)-- (0,0.3);
\node[above] at (0,0.3){$0$};

\draw (\a1, -0.2)-- (\a1, 0.3);
\node[above] at (\a1,0.3){$-\vv_{-2}$};

\draw (\b1, -0.2)-- (\b1, 0.3);
\node[above] at (\b1,0.3){$\ww_{-2}$};

\node[red,above] at (-1.8,0){$J_{-2}(0)$};
\draw (-1.8,0.6) edge[->,bend left] node[below]{$T_{\ww_{-2}}$}  (-1.8+\b1,0.6);
\node[blue,above] at (3,0){$J_{-2}(1)$};
\draw (3,-0.1) edge[->,bend left=20] node[below]{$T_{-\vv_{-2}}$}  (3+\a1,-0.1);

\node[left] at(\a1-0.3,0){$\mathbb {J}_{-2}$};

\begin{scope}[shift={(0,-1)}]

\draw (\a1-0.3,-1.5)--(\b1+0.3,-1.5);
\node[left] at(\a1-0.3,-1.5){$\mathbb {J}_{-1}$};

\begin{scope}[shift={(0,-1.5)}]
\draw (\b1/\x1, -0.15)-- (\b1/\x1, 0.2);
\draw (0,-0.15)-- (0,0.2);
\node[blue,above] at (3/\x1,0){$J_{-1}(1)$};
\draw (3/\x1,-0.1) edge[->,bend left=20] node[below]{$T_{-\vv_{-2}}$}  (3/\x1+\a1,-0.1);
\end{scope}

\begin{scope}[shift={(0,-1.5)}]
\draw (\a1/\x1, -0.15)-- (\a1/\x1, 0.2);
\draw (0,-0.15)-- (0,0.2);
\node[red,above] at (-1.8/\x1,0){$J_{-1}(0)$};
\draw (-1.8/\x1,0.6) edge[->,bend left=15] node[below]{$T_{\ww_{-2}}$}  (-1.8/\x1+\b1,0.6);
\end{scope}

\begin{scope}[shift={(\a1,-1.5)}]
\draw (\b1/\x1, -0.15)-- (\b1/\x1, 0.2);
\draw (0,-0.15)-- (0,0.2);
\node[blue,above] at (3/\x1,0){$J_{-1}(-1)$};
\end{scope}

\begin{scope}[shift={(\a1+\b1,-1.5)}]
\draw (\b1/\x1, -0.15)-- (\b1/\x1, 0.2);
\draw (0,-0.15)-- (0,0.2);
\node[blue,above] at (3/\x1,0){$J_{-1}(2)$};
\end{scope}

\begin{scope}[shift={(\b1,-1.5)}]
\draw (\a1/\x1, -0.15)-- (\a1/\x1, 0.2);
\draw (0,-0.15)-- (0,0.2);
\node[red,above] at (-1.8/\x1,0){$J_{-1}(3)$};
\end{scope}

\end{scope}


\begin{scope}[shift={(0,-2)}]

\draw (\a1-0.3,-3)--(\b1+0.3,-3);
\node[left] at(\a1-0.3,-3){$\mathbb {J}_0$}; 
       
\begin{scope}[shift={(0,-3)}]
\draw (\b1/\x2, -0.15)-- (\b1/\x2, 0.2);
\draw (0,-0.15)-- (0,0.2);
\node[blue,above] at (3/\x2,0){$1$};
\draw (3/\x2,-0.1) edge[->,bend left=20] node[below]{$T_{-\vv_{-2}}$}  (3/\x2+\a1,-0.1);

\end{scope}

\begin{scope}[shift={(0,-3)}]
\draw (\a1/\x2, -0.15)-- (\a1/\x2, 0.2);
\draw (0,-0.15)-- (0,0.2);
\node[red,above] at (-1.8/\x2,0){$0$};
\draw (-1.8/\x2,0.6) edge[->,bend left=15] node[below]{$T_{\ww_{-2}}$}  (-1.8/\x2+\b1,0.6);
\end{scope}

\begin{scope}[shift={(\a1,-3)}]
\draw (\b1/\x2, -0.15)-- (\b1/\x2, 0.2);
\draw (0,-0.15)-- (0,0.2);
\node[blue,above] at (3/\x2,0){$-4$};
\end{scope}

\begin{scope}[shift={(\a1+\b1,-3)}]
\draw (\b1/\x2, -0.15)-- (\b1/\x2, 0.2);
\draw (0,-0.15)-- (0,0.2);
\node[blue,above] at (3/\x2,0){$4$};
\end{scope}

\begin{scope}[shift={(\a1+\b1+\a1,-3)}]
\draw (\b1/\x2, -0.15)-- (\b1/\x2, 0.2);
\draw (0,-0.15)-- (0,0.2);
\node[blue,above] at (3/\x2,0){$-1$};
\end{scope}

\begin{scope}[shift={(\a1+\b1+\a1+\b1,-3)}]
\draw (\b1/\x2, -0.15)-- (\b1/\x2, 0.2);
\draw (0,-0.15)-- (0,0.2);
\node[blue,above] at (3/\x2,0){$7$};
\end{scope}

\begin{scope}[shift={(\a1+\b1+\a1+\b1+\a1,-3)}]
\draw (\b1/\x2, -0.15)-- (\b1/\x2, 0.2);
\draw (0,-0.15)-- (0,0.2);
\node[blue,above] at (3/\x2,0){$2$};
\end{scope}

\begin{scope}[shift={(\a1+\b1+\a1+\b1+\a1+\a1,-3)}]
\draw (\b1/\x2, -0.15)-- (\b1/\x2, 0.2);
\draw (0,-0.15)-- (0,0.2);
\node[blue,above] at (3/\x2,0){$-3$};
\end{scope}

\begin{scope}[shift={(\a1+\b1+\a1+\b1+\a1+\a1+\b1,-3)}]
\draw (\b1/\x2, -0.15)-- (\b1/\x2, 0.2);
\draw (0,-0.15)-- (0,0.2);
\node[blue,above] at (3/\x2,0){$5$};
\end{scope}

\begin{scope}[shift={(\b1,-3)}]
\draw (\a1/\x2, -0.15)-- (\a1/\x2, 0.2);
\draw (0,-0.15)-- (0,0.2);
\node[red,above] at (-1.8/\x2,0){$8$};
\end{scope}

\begin{scope}[shift={(\b1+\a1,-3)}]
\draw (\a1/\x2, -0.15)-- (\a1/\x2, 0.2);
\draw (0,-0.15)-- (0,0.2);
\node[red,above] at (-1.8/\x2,0){$3$};
\end{scope}

\begin{scope}[shift={(\b1+\a1+\a1,-3)}]
\draw (\a1/\x2, -0.15)-- (\a1/\x2, 0.2);
\draw (0,-0.15)-- (0,0.2);
\node[red,above] at (-1.8/\x2,0){$-2$};
\end{scope}

\begin{scope}[shift={(\b1+\a1+\a1+\b1,-3)}]
\draw (\a1/\x2, -0.15)-- (\a1/\x2, 0.2);
\draw (0,-0.15)-- (0,0.2);
\node[red,above] at (-1.8/\x2,0){$6$};
\end{scope}

\end{scope}
\end{tikzpicture}

\caption{Renormalization tilings $\mathbb J_{0}, \mathbb J_{-1}, \mathbb J_{-2}$ for the golden mean combinatorics: $-\vv\approx -0.382$ and $\ww\approx 0.618$. The tiling $\mathbb J_{0}\cap [-\vv_2,\ww_2]$ is obtained by spreading around $J_0(0)$ and $J_0(1)$ using $T_{-\vv_2}$ and $T_{\ww_2}$. Similarly, $\mathbb J_{-1}\cap [-\vv_2,\ww_2]$ is obtained by spreading around $J_{-1}(0)$ and $J_{-1}(1)$ using $T_{-\vv_2}$ and $T_{\ww_2}$. Note that $\mathbb J_{n-1}$ is the rescaling of $\mathbb J_{n}$ by $\approx 2.618$.} 
\label{Fig:ren part of R}
\end{figure}
 
 The \emph{renormalization tiling} $\mathbb J_0$ of level $0$ (see Figure~\ref{Fig:ren part of R}) consists of closed intervals
\[\left\{T^P (J_0(0)) \mid P<(0,0,1)\right\}\cup \left\{T^P (J_0(1)) \mid P<(0,1,0)\right\} ,\] 
 i.e.~$\mathbb J_0$ consists of all the interval in the forward $T^{\ge0}$-orbits of $J_{0}(0) ,J_{0}(1)$ before they return back to $\widetilde J_0$. We say that 
 $\mathbb J_0$ is \emph{obtained by spreading} around $J_0(0)$ and $J_0(1)$ and we enumerate intervals in $\mathbb J_0$ by $\Z$ from left-to-right.

 Similarly, the \emph{renormalization tiling} $\mathbb J_n$ of level $n$ is defined: it is obtained by spreading around the intervals $J_{n}(0)\coloneqq [-\vv_n,0]$ and $ J_{n}(1)=[0,\ww_n]$. Note that $\mathbb J_n$ is the rescaling of $\mathbb J_0$ by $\tt^{-n}$.
 
\begin{lem}[Proper discontinuity]
\label{lem:transl:prop disc}
If $P\in \PT_{>0}\subset \R_{>0}$ is small, then $|T^P(0)|$ is large. 
\end{lem} 
\begin{proof}
For $n\ll0$, consider $\widetilde J_n\coloneqq J_{n}(0)\cup J_{n}(1)=[-\vv_n,\ww_n]$. By construction, if $P < \min \{(n,0,1),(n,1,0)\}$, then $T^P(0)\not\in \widetilde J_n$.
\end{proof}

\subsection{Combinatorics of close returns}
\label{ss:closed returns}
Consider the cascade $(T^P)_{P\in \PT}$ from \S\ref{sss:power-triples}. By Lemma~\ref{lem:PT:rot}, we view $\PT$ as a sub-semigroup of $\R_{\ge0}$.

For $P\in \PT_{>0}$, let $b_P=\T^{-P}(0)$ be the unique preimage of $0$ under $T^P$. Then $P$ is the \emph{generation} of $b_P$. We say that $b_P$ is \emph{dominant} if $[0,b_P]$ contains no $b_Q$ with $Q<P$. By definition, if $b_P$ and $b_Q$ are dominant such that $P<Q$, then $b_Q,0$ are on the same side of $b_P$.

Since every interval $[a,b]\subset \R$ contains at most finitely many $b_P$ with $P\le Q$ for every $Q$ (see Lemma~\ref{lem:transl:prop disc}), dominant points can accumulate only at $0$ and $\infty$. Moreover, if $b_P$ is close to $0$, then $P$ is big; while if $b_P$ is close to $\infty$, then $P$ is small. Therefore, we can enumerate all dominant points as $(b_{P_n})_{n\in \Z}$ such that $P_n>P_{n-1}$ for all $n\in \Z$. Suppressing indices, we write $b_n=b_{P_n}$.

By~\eqref{eq:T^P vs T^tP}, $b_P$ is dominant if and only if $b_{\tt P}=A_{1/\tt} (b_P)$ is dominant. Since $\{b_n\}$ are enumerated by their escaping time, there is a $k>0$ such that \[\tt P_i=P_{i+k}.\]
Let us also note that if $b_n,b_m$ are on the same side of $0$, then $b_n$ is closer to $0$ than $b_m$ if and only if $P_n>P_m$.

\begin{lem}
\label{lem:dynam of domin}
For every $[b_i,b_{i+1}]$ there is a $Q\in \PT_{>0}$ and $[b_n,b_m]$ with $i\ge m>n$ such that $T^Q$ maps $[b_i,b_{i+1}]$ to $[b_n,b_m]$.
\end{lem}
\begin{proof}
Suppose first that $0\in [b_i,b_{i-1}]$, see Figure~\ref{Fig:comb of b_n}. Then $b_{i+1}\in [b_i,b_{i-1}]$. Set $Q\coloneqq P_{i-1}$. Observe that $T^Q(b_i)$ is of smallest generation in \[[T^Q(b_i),0]=T^Q[b_{i},b_{i-1}]\] among the $b_P$, while $T^Q(b_{i+1})$ is of smallest generation in \[[T^Q(b_{i+1}),0]=T^Q[b_{i+1},b_{i-1}].\] Therefore, $T^Q(b_{i})$ and $T^Q(b_{i+1})$ are dominant; i.e.~$T^Q(b_{i})=b_n$ and $T^Q(b_{i+1})=b_m$ for some $n,m<i-1$.

\begin{figure}[tp!]  
\begin{tikzpicture}
\begin{scope}

\draw (-6,0)--(4,0);
\draw[line width=2pt,blue] (1.5,0) -- (-0.8,0);
\draw[line width=2pt,blue] (-3.1,0)  -- (-5.5,0);
\filldraw[red] (0.5,0) circle (2pt) node[below]{$0$}
(1.5,0) circle (2pt)node[below]{$b_{i+1}$}
(3.5,0) circle (2pt) node[below]{$b_{i-1}$}
(-0.8,0) circle (2pt) node[below]{$b_{i}$}
(-3.1,0) circle (2pt) node[below]{$b_{n}$}
(-5.5,0) circle (2pt) node[below]{$b_{m}$};
\draw (3.4,0.2)  edge[->,bend right] node[above]{$T^{Q}$} (0.6,0.2); 

\draw[shift={(0.6,0)}] (-0.9,0.2)  edge[->,bend right] node[above]{$T^{Q}$} (-5.4,0.2);

\end{scope}

\begin{scope}[shift={(0,-2.3)}]
\draw (-6,0)--(4,0);
\filldraw[red] (-2,0) circle (2pt) node[below]{$0$}
(0.5,0) circle (2pt)node[below]{$b_{i}$}
(3.5,0) circle (2pt) node[below]{$b_{i-1}$}
(-3.8,0) circle (2pt) node[below]{$b_{i+1}$};
\draw (0.6,0.2)  edge[->,bend left] node[above]{$T^{Q}$} (3.4,0.2); 
\draw (-3.7,0.2)  edge[->,bend left] node[above]{$T^{Q}$} (0.4,0.2); 
\end{scope}
\end{tikzpicture}
\caption{Illustration to the proof of Lemma~\ref{lem:dynam of domin}. Above: the configuration $b_i<0< b_{i-1}$, then $b_{i+1}\in [b_i,b_{i-1}]$. Set $Q\coloneqq P_{i-1}$. Since $b_i$ dominates on $[b_i,b_{i-1})$, we see that $b_m$ is dominant; similarly $b_n$ is dominant. Below: the configuration $0< b_i< b_{i-1}$, then $b_{i+1}<b_i$. For $Q\coloneqq P_{i}-P_{i-1}$ we obtain $T^Q(b_{i+1})=b_i$ because $T^{-Q}(b_i)$ dominates in $T^{-Q}[0,b_{i-1})$. Note that in both configurations the relative position of $0$ and $b_{i+1}$ is irrelevant.}
 \label{Fig:comb of b_n}
\end{figure}
Suppose now that $b_i\in [0,b_{i-1}]$ and assume that $0< b_i< b_{i-1}$; the opposite case is symmetric. Set $Q\coloneqq P_i-P_{i-1}$. We claim that $T^{-Q}(b_i)$ is dominant and, moreover, $T^Q(b_{i+1})=b_i$. Indeed, since $b_i$ dominates (i.e.~has the smallest generation) in $[0,b_{i-1})$, we obtain that $T^{-Q}(b_i)$ dominates in $T^{-Q}[0,b_{i-1})=[T^{-Q}(0),b_i)\ni 0$. This implies that $T^{-Q}(b_i)$ is dominant, thus $T^{-Q}(b_{i})=b_{i+a}$ for some $a\ge 1$. 

Suppose $a>1$. Since $P_{i+a}>P_{i+1}$, the image $T^Q(b_{i+1})$ has smaller generation than $b_i$. Since $T^Q(b_{i+1})\in (b_{i+1},b_{i-1})$, we obtain that $T^Q(b_{i+1})$ is dominant. This contradicts the enumeration of $(b_i)_{i\in \Z}$: the generation of $T^Q(b_{i+1})$ is between the generations of $b_i$ and $b_{i-1}$. 
\end{proof}

\subsection{Sector renormalization of homeomorphisms}
\label{s:sector renorm of homeo}
Consider a homeomorphism $f\colon \ovDisk\to\ovDisk$ with $f(0)=0$. We denote by $\theta$ the rotation angle of ${f\mid \Circle}$. If $\theta\not\in \Q$, then $f\mid \Circle $ is semi-conjugate to the rotation $\Lbb_{\theta}\colon \Circle\to \Circle$, see~\eqref{eq:defn:Lbb}. As we will show in this section, sector (anti- and pre-) renormalization can be defined for homeomorphisms of $\ovDisk$ in the same way as for rotations. We will specify certain conditions ensuring the transfer of curves between different dynamical planes in the renormalization tower of $f$. The results also hold with necessary adjustments for partial homeomorphisms $f\colon \ovDisk\dashrightarrow\ovDisk$.

\subsubsection{Dividing arcs}
\label{sss:div arcs} Let $\gamma_0$ be a simple arc connecting $0$ to a point in $\partial \Disk$. Then $\gamma_0$ is called \emph{dividing} if $\gamma_0\cap f(\gamma_0)=\{0\}$. Clearly, $\gamma_0$ is dividing if and only if $\gamma_i\coloneqq f^i(\gamma_0)$ is dividing for all $i\in \Z$. Note that $\gamma_i$ and $\gamma_{i+j}$ can intersect away from $0$ if $j\ge 2$.

 The curves $\gamma_0, \gamma_1$ split $\ovDisk$ into two closed sectors $\A$ and $\B$ denoted so that $\intr \A, \gamma_1,\intr \B,\gamma_0$ are clockwise oriented around $0$, see the left-hand of~Figure~\ref{Fig:f and its 1 3 anti ren}. We say that $\gamma_0=\ell(\A)=\rho(\B)$ is the \emph{left boundary of $\A$} and the \emph{right boundary of $\B$} and we say that $\gamma_1=\rho(\A)=\ell(\B)$ is the \emph{right boundary of $\A$} and \emph{the left boundary of $\B$}. 
 
 Let $X,Y$ be topological spaces and let $g\colon X\dashrightarrow Y$ be a partially defined continuous map. We define
\[X\sqcup_g Y\coloneqq  X\sqcup Y/ (\Dom g \ni x\sim g(x)\in \Im g).\]

Consider two sectors $S_0,S_1\in \{\A,\B\}$; each $S_i$ is a copy of either $\A$ or $\B$. We define the map $g\colon \rho(S_0)\dashrightarrow \ell(S_{1})$ by
\begin{equation}
\label{eq:deff:g_k}
    g\coloneqq \begin{cases}
    \id \colon \gamma_1\to \gamma_1& \text{ if } (S_0,S_{1})\cong (\A,\B),\\
        \id \colon \gamma_0\to \gamma_0& \text{ if } (S_0,S_{1})\cong (\B,\A),\\
        f^{-1} \colon \gamma_1 \to \gamma_0& \text{ if } (S_0,S_{1})\cong (\A,\A),\\
          f \colon \gamma_0 \to \gamma_1& \text{ if } (S_0,S_{1})\cong (\B,\B).
    \end{cases}
\end{equation} The \emph{dynamical gluing of  $S_0, S_1$} is $S_0\sqcup_g S_1$. Similarly, the dynamical gluing is defined for any finite or infinite sequence $\seq=(S_i)_i\in \{\A,\B\}^k$ with $k\le \infty$. If $\seq$ is a finite sequence, then the result of the dynamical gluing is a closed sector. If, in addition, we glue the right boundary of the last sector of $\seq$ with the left boundary of the first sector of $\seq$, then the result is a closed topological disk.

\subsubsection{Prime antirenormalizations} 
The (clockwise) \emph{$1/3$ antirenormalization of $f$} is a homeomorphism $f_{-1}\colon W\to W$ such that (see Figure~\ref{Fig:f and its 1 3 anti ren})
\begin{itemize}
\item $W$ is a closed topological disk that is the dynamical gluing of sectors $W[0]$, $W[1]$, $W[2]$, where  $W[0]$ and $W[1]$ are copies of $\A$, while $W[2]$ is a copy of $\B$;
\item the map $f_{-1}\colon W[0]\to W[1]$ is the canonical isomorphism of copies of $\A$;
\item the map $f_{-1}\colon W[1,2]\to W[0,1]$ is identified with $f\colon W\setminus \gamma_0\to W\setminus \gamma_1$.
\end{itemize}

Similarly, the $p/q$-antirenormalization is defined for any rational number $p/q$, see~\cite{DLS}*{\S~B.1.3}. The $1/3$ and $2/3$ antirenormalization are called \emph{prime}. Any other antirenormalization is an iteration of prime antirenormalizations, see~\cite{DLS}*{Lemma B.3}.

It follows from direct calculations that:

\begin{lem}[Inverse to~\eqref{eq:R_prm}]
Let $f\colon \ovDisk \to \ovDisk$ be a homeomorphism and let $\mu\in (0,1)$ be the rotation number of $f\mid \Circle$.
\begin{itemize}
\item If $f_{-1}\colon \ovDisk\to \ovDisk$ is the (clockwise) $1/3$ antirenormalization of $f$, then the rotation number of $f_{-1}\mid \Circle$ is 
 \[\theta=\frac{\mu}{1+\mu}.\]
\item If $f_{-1}\colon \ovDisk\to \ovDisk$ is the (clockwise) $2/3$ antirenormalization of $f$, then the rotation number of $f_{-1}\mid \Circle$ is 
\[\theta= \frac{1}{2-\mu}.\]
\end{itemize}\qed
\end{lem}

\begin{figure}
\centering{\begin{tikzpicture}
\draw[red] (0,0) circle (2.5cm);

 \draw  (0,0) -- (2.5,0);
 \node at (2.15,0.15) {$\gamma_0$};
\draw[rotate around={155:(0,0)}] (0,0) -- (2.5,0);
\draw (-1, 1)node {$\gamma_1=f(\gamma_0)$};

\draw (0.7,1.8) node {$\B$};
\draw (0,-1.25) node {$\A$};

\begin{scope}[shift={(6,0)}]

\draw[red] (0,0) circle (2.5cm);

\draw  (0,0) -- (2.5,0); 
\draw[rotate around={90:(0,0)}] (0,0) -- (2.5,0);
\draw[rotate around={225:(0,0)}] (0,0) -- (2.5,0);
 
\draw (1.1,1.1) node {$\B$};
\draw (0.8,-1.65) node {$\A$}; 

\draw (-1.75,0.7) node {$\A$}; 

\draw[blue] (0.6,0.8) edge[<-, bend right] node[above]{$f$} (-0.8,0.5); 

\draw[blue] (-0.9,0.1) edge[<-, bend right] node[left]{$\id$} (0.1,-1);

\draw[blue] (0.7,-1.) edge[<-, bend right] node[right]{$f$} (0.9,0.6); 
 
\node at (2.15,0.15) {$\gamma_0$};
\node at (2.15,-0.2) {$\gamma_0$};

\node at (0.2,2.15) {$\gamma_1$};
\node at (-0.15,2.15) {$\gamma_1$};

\node at (-1.65,-1.4) {$\gamma_0$};
\node at (-1.3,-1.65) {$\gamma_1$};

\end{scope}

\end{tikzpicture}}
\caption{Left: a homeomorphism $f\colon W\to W$ and a diving pair $\gamma_0,\gamma_1$. Right: the $1/3$ antirenormalization of $f$ (with respect to the clockwise orientation).}
\label{Fig:f and its 1 3 anti ren}
\end{figure}

\subsubsection{Pre-antirenormalizations}
\label{sss:pre antiren}

Recall from~\eqref{eq:sect ren:unv cover} that $\ee_-\colon\oH_-\to \overline {\Disk}^*$ denotes the universal covering map from the lower half plane to the punctured disk. We enumerate preimages of $\gamma_0,\gamma_1$ under $\ee_-$ as $\wgamma_i$ from left-to-right such that $\ee_-$ maps $\wgamma_i$ to $\gamma_{i \mod 2}$. We specify the lifts $f_-$ and $f_+$ of $f$ such that 
\begin{itemize}
\item $f_-$ maps $\wgamma_0$ to $\wgamma_{-1}$; and 
\item $f_+$ maps $\wgamma_0$ to $\wgamma_{1}$.
\end{itemize}
Then $\chi\coloneqq f_+\circ f_-^{-1}$ is a deck transformation of $\ee_-$. Note that $f_-$ and $f_+^{-1}$ move points to the left of $\oH_-$ while $f_+$ and $f_-^{-1}$ move points to the right of $\oH_-$.

The \emph{$1/3$ pre-antirenormalization} of $(f_-, f_+)$ is the commuting pair
\begin{equation}
\label{eq:anti ren:1/3}
 (f_{-, \ -1} \coloneqq f_-\circ f_+^{-1}=\chi^{-1},\sp f_{+,\ -1}\coloneqq f_+).
 \end{equation}
Setting $\chi_{-1}\coloneqq f_{+,-1}\circ f_{-,-1}^{-1}$ and identifying $\ovDisk^*\simeq \oH_-/\langle \chi_{-1} \rangle$, we recover the $1/3$ antirenormalization of $f$: \[f_{-1}\coloneqq f_{-1,\ -}/\langle \chi_{-1} \rangle\colon \ovDisk\to \ovDisk,\sp\sp f_{-1}(0)=0, \]
see \cite{DLS}*{Lemma~B.4}.

Similarly, the \emph{$2/3$ pre-antirenormalization} of $(f_-, f_+)$ is the commuting pair
\begin{equation}
\label{eq:anti ren:2/3} (f_{-, \ -1} \coloneqq f_-,\sp f_{+,\ -1}\coloneqq f_+\circ f_{-}^{-1}=\chi).
\end{equation}
Setting $\chi_{-1}\coloneqq f_{+,-1}\circ f_{-,-1}^{-1}$ and identifying $\ovDisk^*\simeq \oH_-/\langle \chi_{-1} \rangle$, we recover the $2/3$ \emph{antirenormalization}\[f_{-1}\coloneqq f_{-1,\ -}/\langle \chi_{-1} \rangle\colon \ovDisk\to \ovDisk,\sp\sp f_{-1}(0)=0, \]
see \cite{DLS}*{Lemma~B.6}.

\subsubsection{Tower of pre-antirenormalizations}
Let us fix an irrational periodic point $\theta=\cRRc^\mm(\theta)$ of~\eqref{eq:R_prm}. Then $\cRRc^\mm$ has an inverse branch $\cRR^-$ mapping $\Circle\setminus \{0\}$ into a neighborhood of $\theta$. Let us specify the antirenormalization operator associated with $\cRR^-$. For $i\in \{1,\dots, \mm\}$, set $\aRR_i$ to be 
\begin{itemize}
\item the $1/3$ pre-antirenormalization if $\cRRc^i(\theta)\in (1/2,1)$, 
\item the $2/3$ pre-antirenormalization if $\cRRc^i(\theta)\in (0,1/2)$;
\end{itemize}
and set
\begin{equation}
\label{eq:aRR}
\aRR\coloneqq \aRR_1\circ \aRR_2\circ \dots\circ \aRR_m.
\end{equation}
This operator is inverse to $\RR$ from~\eqref{eq:defn RR on transl}. Since $\theta\not\in \Q$, both $1/3$ and $2/3$ pre-antirenormalizations appear in~\eqref{eq:aRR}.

We have the \emph{pre-antirenormalization tower}:
\[F_n\coloneqq (f_{n,-},f_{n,+})\coloneqq (\aRR)^{-n}(f_-,f_+)\sp\sp \text{ for }n\le 0.\]
Writing $\chi_n\coloneqq f_{n,-}^{-1}\circ f_{n,+}$ and identifying $\oH_-/\langle \chi_n \rangle \simeq \ovDisk^*$, we obtain the projection
\begin{equation}
\label{eq:rho_n:proj}
\brho_n \colon \oH_-\to \oH_-/\langle \chi_n \rangle \subset \ovDisk
\end{equation} semi-conjugating the pair $F_n$ to a homeomorphism $f_n\colon \ovDisk\to \ovDisk$, where $f_n(0)\coloneqq 0$. If $\mu$ is the rotation number of $f_0\mid \Circle$, then $\big(\cRR^-\big)^{-n}(\mu)$ is the rotation number of $f_n\mid \Circle$.

\begin{lem}
\label{lem:iter of pairs}
For $m\le n\le0$, we have:
\begin{itemize}
\item $f_{m,-}$ and $f_{m,+}^{-1}$  are compositions of $f_{n,-}$ and $f_{n,+}^{-1}$; and
\item $f_{m,+}$ and $f_{m,-}^{-1}$ are compositions of $f_{n,+}$ and $f_{n,-}^{-1}$.
\end{itemize}
\end{lem}
\begin{proof}
The operator $\aRR$ is a composition of prime pre-antirenormalizations; the corresponding statement for a prime antirenormalization is immediate from the definition, see~\eqref{eq:anti ren:1/3} and \eqref{eq:anti ren:2/3}.
\end{proof}

\subsubsection{Cascade $F^{\ge0}=\big(F^P\big)_{P\in \PT}$}
Similar to~\eqref{eq:T n a b}, we define
\begin{equation} 
 \label{eq:F n a b}
 F^{(n,a,b)}\coloneqq f^a_{n,-}\circ f^b_{n,+},
\end{equation}
where $(n,a,b)\in \Z_{\le 0}\times \Z_{\ge0}^2$. 
As in~\S\ref{sss:power-triples}, $F^T$ depends only on the image of $(n,a,b)$ in the
semi-group of power-triples $\PT$, see~\eqref{eq:defn:PT}. 

Using Lemma~\ref{lem:PT:rot}, we view $\PT$ as a sub-semigroup of $\R_{\ge0}$.

\subsubsection{Renormalization triangulations}
\label{sss:RenTrian}

Let $\Delta_0(0)$ be the strip between $\wgamma_{-1}$ and $\wgamma_0$, and let 
$\Delta_0(1)$ be the strip between $\wgamma_0$ and $\wgamma_1$. We will refer to $\Delta_0(0)$ and $\Delta_0(1)$ as \emph{triangles}.

As in~\S\ref{sss:ren tiling:transl} we define the \emph{renormalization triangulation} $\bDelta_0$ of level $0$ to be
\[\left\{F^P (\Delta_0(0)) \mid P<(0,0,1)\right\}\cup \left\{F^P (\Delta_0(1)) \mid P<(0,1,0)\right\},\]
 i.e.~$\bDelta_0$ consists of all the triangles in the forward $F$-orbits of $\Delta_{0}(0) ,\Delta_{0}(1)$ before they return back to $\Delta_0(0,1)$. We say that $\bDelta_0$ is obtained by \emph{spreading around} $\Delta_{0}(0) $ and $\Delta_{0}(1)$. We enumerate triangles of $\bDelta_0$ from left-to-right as $\Delta_0(i)$ with $i\in\Z$.

Similarly, we define the triangulation $\bDelta_n$ for $n\le 0$. The triangle $\Delta_n(0)$ is bounded by $f_{n,-}(\wgamma_0)\cup \wgamma_0$, and the triangle $\Delta_n(0)$ is bounded by $\wgamma_0\cup f_{n,+}(\wgamma_0)$. The triangulation $\bDelta_n$ is obtained by spreading around  $\Delta_{n}(0) $ and $\Delta_{n}(1)$.

Note that $\bDelta_n(0,1)$ is a fundamental domain of $\chi_n$, and we have a projection $\brho_n\colon \bDelta_n(0,1)\to \ovDisk$ (see~\eqref{eq:rho_n:proj}) gluing the left and right boundaries of $\bDelta_n(0,1)$.

\subsubsection{Dynamics of triangles of $\bDelta_0$}
\label{sss:dyn of Delta}
 Consider the triangulation $\bDelta_0$. For every $i\in \Z$, we write $\seq[i]\coloneqq \B$ if $\Delta_0(i)$ is in the forward orbit of $\Delta_0(0)$ before it returns to $\Delta_0(0,1)$; and we write $\seq[i]\coloneqq \A$ if $\Delta_0(i)$ is in the forward orbit of $\Delta_0(1)$ before it returns to $\Delta_0(0,1)$.

We can view $\bDelta_0\simeq \oH_-$ as the dynamical gluing of $(\seq[i])_{i\in \Z}$ (see~\S\ref{sss:div arcs}), where the left boundary of $\seq[i+1]$ is glued with the right boundary of $\seq[i]$. The dynamics of triangles of $\bDelta_0$ is described as follows.    
  For every $i\in \Z$ there exists $P(i)$  such that 
\begin{itemize}
  \item if $\seq[i]=\B$, then $P(i)<(0,0,1)$ and 
\begin{equation}
\label{eq:isom B copies}
F^{P(i)}\colon \Delta_0(0)\to \Delta_0(i)
\end{equation}
 is an isomorphism of copies of $\B$; 
\item if $\seq[i]=\A$, then $P(i)<(0,1,0)$ and
\begin{equation}
\label{eq:isom A copies}
F^{P(i)}\colon \Delta_0(1)\to \Delta_0(i)
\end{equation}
is an isomorphism of copies of $\A$.
\end{itemize}   
Conversely, for every $P<(0,0,1)$ there exists $i$ such that $\seq[i]=\B$ and $P=P(i)$; and for every $P<(0,1,0)$ there exists $i$ such that $\seq[i]=\A$ and $P=P(i)$. On the other hand, for every $Q<\min \{(0,0,1),(0,1,0)\}$ there exists $j$ with 
\begin{equation}
\seq[j]=\B,\sp \seq[j+1]=\A,\sp P(j)= (0,0,1)-Q,\sp P(j+1)= (0,1,0)-Q
\end{equation}
such that 
\begin{equation}
\label{eq:bF^Q}
F^Q\colon \Delta_0(j,j+1)\to \Delta_0(0,1)\sp \text{ is identified with }\sp f\colon \ovDisk\setminus \gamma_0\to \ovDisk\setminus \gamma_1.
\end{equation}
In other words, $\bF^{-Q}\mid \Delta(0,1)$ is identified with $f^{-1}\mid \big(\ovDisk \setminus \gamma_1\big)$.

\subsubsection{Walls}
\label{sss:walls}
A \emph{wall around $0$ respecting $\gamma_0,\gamma_1$} is either a closed annulus or a simple closed curve $\Pi\subset\ovDisk$ such that
\begin{enumerate}
\item $\C\setminus \Pi$ has two connected components. Moreover, denoting by $\inn$ the bounded component of $\C\setminus \Pi$, we have $0\in \inn$.
\item $\gamma_0\cap \Pi$ and $\gamma_1\cap \Pi$ are connected.
\item if $x\in \inn$, then $f^{\pm 1}(x)\in \Pi\cup \inn$.
\end{enumerate}
In other words, points in $\ovDisk$ do not jump over $\Pi$ under one iteration of $f$. If $\Pi$ is a simple closed curve, then $f$ restricts to an actual homeomorphism $f\colon \overline \inn\to \overline \inn$.

We say that $\Pi$ is an \emph{$N$-wall} if it takes at least $N\ge 1$ iterates of $f^{\pm1}$ for points in $\inn$ to cross $\Pi$.  

By definition, $\Pi\cap \B$ and $\Pi\cap \A$ are connected closed rectangles (possibly degenerate if $\Pi$ is a curve). Let us denote by $\Pi(0)$ and $\Pi(1)$ the images of $\Pi\cap \B$ and $\Pi\cap \A$ under $\Delta_0(0)\simeq \B$ and $\Delta_0(1)\simeq \A$. The wall $\bPi$ is obtained by spreading around  $\Pi(0)$ and $\Pi(1)$
\[\bPi\coloneqq \left\{F^P (\Pi(0)) \mid P<(0,0,1)\right\}\cup \left\{F^P (\Pi(1)) \mid P<(0,1,0)\right\},\]
 i.e.~$\bPi$ consists of all the rectangles in the forward $F^{\ge0 }$-orbits of $\Pi(0) ,\Pi(1)$ before they return back to $\Delta_0(0,1)$.  We enumerate rectangles of $\bPi$ from left-to-right as $\Pi(i)$ with $i\in\Z$.

Every $\Pi(i)$ is a copy of either $\Pi\cap \B$ or $\Pi\cap \A$, and $\Pi(i)$ is glued to $\Pi(i+1)$ along $\id$ or $f^{\pm 1}$, see~\eqref{eq:deff:g_k}. Therefore, $\bPi$ is connected, and, moreover, if $\Pi$ is an $N$-wall, then $\bPi$ is an $(N-1)$-wall: it takes at least $N-1$ iterates of $f_{\pm}$ for a point to cross $\bPi$.

\subsubsection{The boundary point $\balpha$}
\label{sss:balpha}
Let us add the \emph{boundary point} $\balpha$ to $\oH_-$ which will be the preimage of $0$ under $\brho_n\colon \bDelta_n(0,1)\sqcup \{\balpha\}\to \ovDisk$. 

We now introduce the \emph{wall topology} $\Xi$ on $\oH_-\sqcup \{\balpha\}$. For a wall $\bPi\subset \oH_-$, let $Q'_{\bPi}\subset \oH_-$ be the connected component of $\C\setminus \bPi$ below $\bPi$, i.e.~$Q'_{\bPi}\not\ni 0$. We write $Q_{\bPi}\coloneqq Q'_{\bPi} \sqcup \{\balpha\}$. In the topology $\Xi$ of $\oH_-\sqcup \{\balpha\}$ open sets are generated by open sets of $\oH_-$ and $\{Q_{\bPi} \colon {\bPi\text{ is a }wall}\}$.

\begin{figure}
 \begin{tikzpicture}
 
\draw (-5,0) -- (5,0); 
\draw[red] (1,0)
 .. controls (2,-3) and (4,-3.8) ..
(5,-4);

\draw[red,shift={(-3,0)}] (1,0)
 .. controls (2,-3) and (4,-3.8) ..
(5,-4)
 .. controls (6,-4.3) and (7,-4.3) ..
(8,-4.3) ;

\draw[blue, dashed] (-2,0) -- (-2,-5);
\draw[blue, dashed] (1,0) -- (1,-5);

\draw (-1.9,-4.8) edge[->,bend left] node[above]{$\chi\colon z\mapsto z+\ww$}(0.9,-4.8); 

\node[red] at (-1,-1.3){$\beta$};
\node[red] at (2.1,-1.3){$\chi(\beta)$};

\node[red] at (2.1,-3.3){$S$};

\end{tikzpicture} 
\caption{The sector $S$ between $\beta$ and $\chi(\beta)$ is not a fundamental domain of $\chi\colon z\mapsto z+\ww$ because not every orbit passes trough $S$.}
\end{figure}

\begin{rem} It can be shown that $\Xi$ is independent of the choice of $\gamma_0$. 
\end{rem}

\subsubsection{Fundamental domains}
A simple arc $\beta\colon [0,1)\to \oH_-$ is \emph{dividing for $n$} if 
\begin{enumerate}
\item $\beta (0)\in \partial \H_-$ and $\beta(t)\in \H_-$ for $t>0$;
\item $\beta$ lands at $\balpha$ with respect to the wall topology $\Xi$; and\label{case:beta lands}
\item $\beta$ does not intersect $f_{n,-}(\beta)$ and $f_{n,+}(\beta)$.
\end{enumerate}For example, $\wgamma_i$ are dividing arcs.

If $\beta$ is dividing, then $\oH_-\setminus \beta$ has two connected components. (Indeed, ~\eqref{case:beta lands} implies that $\displaystyle\lim_{t\to 1} \beta(t)=\infty$.) We denote the closures of the connected components of  $\oH_-\setminus \beta$  by $\Left$ and $\Right$ enumerated so that $\Left$ is on the left of $\beta$ (i.e.~$\Left\ni x$ for $x\ll0$) while $\Right$ is on the right of $\beta$ (i.e.~$\Right\ni x$ for $x\gg 0$).

\begin{lem}
\label{lem:F:motion of pts}
For every $m\le n$, the maps $f_{m,-}, f_{m,+}^{-1}$ move points to the left:
\[f_{m,-} (\Left)\subset \Left \sp \text{ and } f_{m,+}^{-1} (\Left)\subset \Left,\]
while $f_{m,+}, f_{m,-}^{-1}$ move points to the right:
\[f_{m,-}^{-1} (\Right)\subset \Right \sp \text{ and } f_{m,+} (\Right)\subset \Right\]
\end{lem}
\begin{proof}
The case $m=n$ follows from the definition. The general case follows by induction because $f_{m,-}$ and $f_{m,+}^{-1}$  are compositions of $f_{n,-}$ and $f_{n,+}^{-1}$, while  $f_{m,+}$ and $f_{m,-}^{-1}$ are compositions of $f_{n,+}$ and $f_{n,-}^{-1}$, see Lemma~\ref{lem:iter of pairs}.
\end{proof}

\begin{prop}
\label{prop:beta:FunDom}
If $\beta$ is a dividing curve for $n$, then for all $m\le n$ the closed sector $S_m\subset \oH_-$ bounded by $f_{m,-}(\beta)\cup f_{m,+}(\beta)$ and containing $\beta$ is a fundamental domain for $\chi_m$.

Conversely, if $\beta_0\subset \ovDisk$ is a dividing curve of $f_n$ (see~\S\ref{sss:div arcs}), then a lift $\beta\subset \oH_-$ of $\beta_0$ is a dividing arc for $n$.
\end{prop}
\noindent As a corollary, 
\[\ovDisk\simeq \oH_-/\langle\chi_m \rangle\simeq S_m/_{f_{m,-}(\beta)\ni x\sim \chi_m(x) \in f_{m,+}(\beta)}.\]
To prove Proposition~\ref{prop:beta:FunDom} we need to verify that every $\chi_m$-orbit passes through $S_m$. Figure~\ref{case:beta lands} illustrates that Condition~\eqref{case:beta lands} can not be relaxed to the condition ``$\beta$ goes to infinity.'' 

\begin{proof}[Proof of Proposition~\ref{prop:beta:FunDom}]
Since $f_{m,-}(\beta)$ is on the left of $f_{n,-}(\beta)$ (by Lemmas~\ref{lem:iter of pairs} and~\ref{lem:F:motion of pts}), the arcs $f_{m,-}(\beta)$ and $\beta$ are disjoint. Similarly, $f_{m,+}(\beta)$ and $\beta$ are disjoint. We need to show that for every $z\in \oH_-$ there is a $k\in \Z$ such that $\chi^k_m(z)\in S_m$.

Suppose converse, there is a $z\in  \oH_-$ such that $\chi^k_m(z)\not\in S_m$ for all $k\in \Z$. Since point can not jump over $S_m$ under the iteration of $\chi_m$, the orbit of $z$ is either on the left of $S_m$ or on the right of $S_m$.

Let us assume that there is a $z\in \oH_-$ whose orbit is on the left of $S_m$; the opposite case is analogous. We will show that $F^P(z)$ is on the right from $S_m$ for some $P\in \PT$. Then for some $a,b\ge 0$ and $k\in \Z$, we would have \[f_{n,+}^a\circ f_{n,-}^{-b}\circ F^P(z)=\chi_m^k(z).\]
 This will be a contradiction, because $\chi_m^k(z)$ is still on the right of $S_m$ for all $k\in\Z$ by Lemma~\ref{lem:F:motion of pts}.

Let us choose a $1$-wall $\bA$ separating $z$ from $\balpha$. We denote by $Q$ the connected component of $\oH_-\setminus \bA$ attached to $\balpha$.
Since $f_{m,-}(\beta)$ and $ f_{m,+}(\beta)$ land at $\balpha$, the sector $S_m$ intersects at most finitely many truncated triangles $\Delta(i)\setminus Q$. This means that $\Delta(j)\setminus Q$ is on the right of $S_m$ for $j\gg0$. 

Write $z \in \Delta(i)$. There is a small $P\in \PT$ and a big $j\gg0$ such that $F^P\colon \Delta(i)\to \Delta(j)$ is a homeomorphism. Then $F^P(z)\subset \Delta(j)\setminus Q$ is on the right from $S_m$. This proves that $S_m$ is a fundamental domain.

 The converse statement is immediate.
\end{proof}

\subsubsection{Partial homeomorphisms}
\label{ss:PartiHomeo}
Let us show that Proposition~\ref{prop:beta:FunDom} with necessary adjustments holds for partial homeomorphisms. 

Consider  a partial homeomorphism $f\colon \ovDisk\dashrightarrow \ovDisk$ with $f(0)=0$ such that $\Dom f$ and $\Im f$ are closed topological disks containing $0$ in their interiors.

As in~\S\ref{sss:div arcs}, let $\gamma_0, \gamma_1$ be two simple arcs connecting $0$ to points in  $\partial \Disk$ such that $\gamma_0$ and $\gamma_1$ are disjoint except for $0$ and such that $\gamma_1$ is the image of $\gamma_0$ in the following sense: $\gamma'_0\coloneqq \gamma_0\cap \Dom f$ and  $\gamma'_1\coloneqq \gamma_1\cap \Im f$ are simple closed curves such that $f$ maps $\gamma'_0$ to $\gamma'_1$. We call $\gamma_0, \gamma_1$ a \emph{dividing pair}. Then $\gamma_0 \cup \gamma_1$ splits $\ovDisk$ into two closed sectors $\A$ and $\B$ denoted so that $\intr \A, \gamma_1,\intr \B,\gamma_0$ are clockwise oriented around $0$. 

Similar to~\eqref{eq:deff:g_k}, given two sectors $S_i,S_{i+1}\in \{\A,\B\}$, we naturally have a partial map $g\colon \rho(S_0)\dashrightarrow \ell(S_{1})$ defined by 
\begin{equation}
\label{eq:deff:g_k:part}
    g\coloneqq \begin{cases}
    \id \colon \gamma_1\to \gamma_1& \text{ if } (S_i,S_{i+1})\cong (\A,\B),\\
        \id \colon \gamma_0\to \gamma_0& \text{ if } (S_i,S_{i+1})\cong (\B,\A),\\
        f^{-1} \colon \gamma'_1 \to \gamma'_0& \text{ if } (S_i,S_{i+1})\cong (\A,\A),\\
          f \colon \gamma'_0 \to \gamma'_1& \text{ if } (S_i,S_{i+1})\cong (\B,\B),
    \end{cases}
\end{equation} The \emph{dynamical gluing} of a sequence of sectors $(S_i)_{i}\in \{\A,\B\}^k$ with $k\le \infty$ is defined in the same way as in~\S\ref{sss:div arcs}; i.e.~the left boundary of $S_i$ is glued with the right boundary of $S_{i-1}$ along~\eqref{eq:deff:g_k:part}.  

Let $\seq\in \{\A,\B\}^\Z$ be the sequence from \S\ref{sss:dyn of Delta} (where $f$ is assumed to be a homeomorphism), and let $\bDelta_0$ be corresponding dynamical gluing. Then $\bDelta_0$ is a triangulation \emph{associated} with $\gamma_0,\gamma_1$, we enumerate triangles of $\bDelta_0$ from left-to-right as $\Delta_0(i)\simeq \seq[i]$, with $i\in \Z$.

For every $Q<\min\{(0,1,0),(0,0,1)\}$, we define the partial homeomorphism $\bF^Q\colon  \bDelta_0\dashrightarrow \bDelta_0$ trianglewise using \eqref{eq:isom B copies},\eqref{eq:isom A copies}, and~\eqref{eq:bF^Q}. Taking iterates, we obtain the cascade of partial homeomorphisms \[F^{\ge0} \coloneqq \left\{ F^P\colon \bDelta_0\dashrightarrow \bDelta_0 \mid  P\in \PT\right\}.\] In particular, $f_-=F^{(0,1,0)}$ and $f_+=F^{(0,0,1)}$ are well defined. The commuting pair 
\begin{equation}
\label{eq:FRM:Delta 0 1}
f_+\colon \Delta_0(0)\to \Delta_0(0,1)\sp \text{ and }f_-\colon \Delta_0(1)\to \Delta_0(0,1)
\end{equation}
 realizes the first return of points in $\Delta_0(0,1)$ back to $\Delta_0(0,1)$. The pair~\eqref{eq:FRM:Delta 0 1} is obtained by cutting $f\colon \ovDisk\dashrightarrow \ovDisk$ along $\gamma_1$. Conversely, there is a projection
 \[\brho\colon \Delta_0(0,1)\to \ovDisk\] semi-conjugating $F=(f_-,f_+)$ to $f$ such that $\brho$ glues the left boundary $\lambda$ of $\Delta_0(0,1)$ with its right boundary $\rho$. The gluing map $\bchi\colon \lambda\to \rho$ coincides with $f_+\circ f_-^{-1}\colon \lambda\dashrightarrow \rho$ in $\Dom (f_+\circ f_-^{-1})$.
 
 A \emph{wall $\Pi\subset \Dom f\cap \Im f$ around $0$ respecting $\gamma_0,\gamma_1$} for $f\colon \ovDisk\dashrightarrow \ovDisk$ is defined in the same way as in the case of homeomorphisms, see~\S\ref{sss:walls}. 
 Lifting $\Pi$ under $\brho$ and spreading $\Pi$ around, we obtain the connected strip $\bPi\subset \bDelta$. As in \S\ref{sss:balpha}, we add the boundary point $\balpha$ to $\bDelta_0$, and we endow $\bDelta_0\sqcup \{\balpha\}$ with the \emph{wall topology}.

Fix a wall $\Pi$ and its full lift $\bPi\subset \bDelta$. Let us denote by $Q$ the connected component of $\bDelta\setminus \bPi$ attached to $\balpha$. Note that $Q\subset \Dom(f_-) \cap \Dom (f_+)$.

A \emph{fundamental domain} for $f$ is a sector $S^\new$ with distinguished sides $\lambda^\new$ and $\rho^\new$ such that
\begin{enumerate}
\item $S^\new\setminus Q=\Delta(0,1)\setminus Q$, and $\lambda^\new\setminus Q=\lambda \setminus Q$, and $\rho^\new\setminus Q=\rho\setminus Q$,\label{defn:FD:Cond:1}
\item $f_+\circ f_-^{-1}\big(\lambda^\new\cap (\bPi\cup Q)\big)\subset \rho^\new$ and $ f_-^{-1}\big(\lambda^\new\cap (\bPi\cup Q)\big)\subset \intr(S^\new)$;  
\item $\lambda^\new$ and $\rho^\new$ land at $\balpha$.
\end{enumerate}

We define the gluing map
\[\bchi^\new \colon \lambda \to \rho \] to be 
\begin{itemize}
\item $\bchi$ on $\lambda\setminus Q$; and 
\item $F^{(0,0,1)-(0,1,0)}=f_+\circ f_-^{-1}$ on $\lambda\cap Q$.
\end{itemize}

Gluing the left and right boundaries of $S^\new$ along $\bchi^\new$, we obtain a closed topological disk $W$. We realize $W$ as a subset of $\C$, and we denote by \[\brho^\new\colon S^\new\to W,\sp\sp \brho^\new(\balpha)=0\] the induced projection. Then $F=(f_-,f_+)$ projects to 
\[f^\new\colon W\dashrightarrow W.\]
Write $\inn =\brho (Q\cap S)\subset \ovDisk$ and $\inn^\new =\brho^\new (Q\cap S^\new)\subset W$.  

\begin{prop}
\label{prop:quot of fund domain}
Let $S^\new$ be a fundamental domain for $f$ as above. Then the quotient map
$f^\new\colon W\dashrightarrow W$ is canonically conjugate to $f\colon \ovDisk\dashrightarrow \ovDisk$.

Conversely, suppose $\gamma_0^\new,\gamma_1^\new$ is a dividing pair such that $\gamma_0^\new \setminus \inn =\gamma_0 \setminus \inn$ and $\gamma_1^\new \setminus \inn =\gamma_1 \setminus \inn$. Let $\bDelta_0^\new$ be the triangulation associated with $\gamma_0^\new,\gamma_1^\new$. Then the cascades $F\mid \bDelta_0$ and $F\mid \bDelta^\new_0$ are canonically conjugate.
\end{prop}
\noindent The canonical homeomorphism $h\colon \ovDisk\to W$ has the following characterization:
\begin{itemize}
\item $h\colon \ovDisk\setminus \inn\to W\setminus \inn^\new$ is the canonical identification using Condition~\eqref{defn:FD:Cond:1};
\item if $ h\circ \brho(x)=\brho^\new(y)\in \inn^\new$ for some $x,y\in \bDelta$, then $x$ and $y$ are related by the action of the deck transformation $\bchi\mid Q$: for some $n\in \Z$ we have $x\in \Dom (\bchi\mid Q)^n$ and $\bchi^n(x)=y.$
\end{itemize}
Equivalently, the canonical homeomorphism $h\colon \ovDisk\to W$ can be characterized using unique extension along curves as in~\cite{DLS}*{Theorem B.8}. A similar characterization has a canonical homeomorphism between $\bDelta_0$ and $\bDelta_0^\new$.
\begin{proof}[Proof of Proposition~\ref{prop:quot of fund domain}]
 Denote by $\inn$ and $\inn^\new$ the images of $Q\sqcup \{\balpha\}$ under $\brho$ and $\brho^\new$ respectively. Condition~\eqref{defn:FD:Cond:1} implies that $\ovDisk\setminus \inn$ is canonically identified with $W\setminus \inn ^\new$ by an equivariant homeomorphism.

We can now modify $f\colon \ovDisk\dashrightarrow \ovDisk$ away from $\inn \cup f(\inn)\cup f^{-1}(\inn)$ to obtain a self-homeomorphism $f\colon  \ovDisk \to \ovDisk$ mapping $\gamma_0$ to $\gamma_1$. Since the modification does not affect $f\mid \inn$, we obtain a reformulation of  Proposition~\ref{prop:beta:FunDom}.
\end{proof}

\section{Background on the pacman renormalization}
\label{s:PacmanRenorm}

\subsection{Quadratic-like renormalization}
\label{ss:QL renorm}
Recall that copies of the Mandelbrot set are canonically homeomorphic via the straightening map, see~\cite{DH}. See also~\cites{DH:Orsay,L-book} for the background on the Mandelbrot set. Given a small copy $\Mandel_s$ of the Mandelbrot set, we denote by $\mRR_s\colon  \Mandel_s \to \Mandel$ the canonical homeomorphism between $\Mandel_s$ and $\Mandel$.

A quadratic polynomial $p_c=z^2+c, \sp c\in \Mandel$, is \emph{renormalizable} if $c$ belongs to a small copy of the Mandelbrot set. Let $\Mandel_s$ be the biggest small copy of $\Mandel$ containing $c$. We set $\mRR_\QL(c)\coloneqq \mRR_s(c)$. This way we obtain the partial map $\mRR_\QL\colon \Mandel\dashrightarrow \Mandel$. 

A map $p_c$ is \emph{infinitely renormalizable} if $c$ is within the non-escaping set \[ \NE(\mRR_\QL) \coloneqq \bigcap_{i\ge0} \Dom \mRR^i_\QL\]
of $\mRR_\QL$. If the connected component of $\NE(\mRR_\QL)$ containing $c$ is a singleton, then $\Mandel$ is locally connected at $c$. The global MLC conjecture is equivalent to the assertion that every connected component of $\NE(\mRR_\QL)$ is a singleton. 

An infinitely renormalizable parameter $c$ has \emph{bounded type} if its orbit $\mRR_\QL^n(c)$ belongs to finitely many maximal small copies of $\Mandel$. If, moreover, the maximal copies of $\Mandel$ containing $\mRR_\QL^n(c)$ are satellite, then $c$ has  \emph{bounded satellite type}.

\subsubsection{Analytic renormalization operator}\label{ssss:QL anal ren} (See~\cite{L:FCT} for details.) Let us write a quadratic-like map as $f\colon U\to V$, and let us denote by $[f\colon U\to V]$ the quadratic-like germ of $f$ considered up to affine equivalence. The modulus of the fundamental annulus $V\setminus\overline U$ is the \emph{modulus of $f$}, denoted by $\mod f$. We denote by $\Kfilled(f)$ and $\Post(f)$ the non-escaping and postcritical sets of $f$.
 
In~\cite{L:FCT} the space $\QG$ of \emph{quadratic-like germs} is supplied with complex analytic structure. Let $\Conn\subset \QG$ be the connectedness locus;~i.e.~the set of germs with a non-escaping critical point.  

 The hybrid classes form a codimension-one lamination $\bFol_\QL$ of $\Conn$ with complex codimension-one analytic leaves. Every leaf of $\Fol_i\in \bFol_\QL$ contains a unique parameter $c_i $ in the actual Mandelbrot set. By collapsing every $\Fol_i\in \bFol_\QL$ to $c_i$, we obtain the projection $\strai\colon \Conn\to \Mandel$.

Consider a small copy $\Mandel_s \subsetneq \Mandel$ of period $n>1$. There is an analytic operator $\RR_s\colon \QG\dashrightarrow \QG$ associated with $\Mandel_s$ such that
\[\RR_s [f\colon U\to V] = [f^n\colon U_i\to V_i].\]
We assume that $V_i$ contains the critical value of $f$. The operator $\RR_s$ satisfies $\strai \circ \RR_s=\mRR_s\circ \strai$. If $\Mandel_s$ is satellite, then $\RR_s$ is defined on a neighborhood of $\strai ^{-1}(\Mandel_s\setminus \{cusp\})$; and if $\Mandel_s$ is primitive, then $\RR_s$ is defined on a neighborhood of $\strai ^{-1}(\Mandel_s)$.  

Combining all $\RR_s$ over all the maximal copies $\Mandel_s \subsetneq \Mandel$, we obtain an analytic operator $\RR_\QL\colon \QG\dashrightarrow \QG$ such that (with appropriate choices of branches)
\[\strai \circ \RR_\QL=\mRR_\QL\circ \strai .\]

\subsubsection{Satellite copies}
\label{sss:not:SatCopies}
Consider a rational number $\rr\in \Q$ between $0$ and $1$. We denote by $\Mandel_\rr\subset \Mandel$ the primary satellite copy of $\Mandel$ with rotation number $\rr$. In other words, $\Mandel_\rr$ is the unique copy of the Mandelbrot set attached to the parabolic parameter $p_{c(\rr)}\in \partial \HH$ with rotation number $\rr$. We have the canonical homeomorphism $\mRR_\rr\colon \Mandel_\rr\to \Mandel$. 

For $\rr,\ss \in \Q\cap (0,1)$, we write $\Mandel_{\rr,\ss}\coloneqq \RR_{\rr}^{-1} (\Mandel_\ss )$ and we denote by $\mRR_{\rr,\ss}\colon \Mandel_{\rr,\ss}\to \Mandel$ the canonical homeomorphism. The construction continues by induction. In particular, $\Mandel_{\rr,\ss,\tt}\coloneqq \mRR^{-1}_{\rr,\ss} (\Mandel_\tt)$.

\subsubsection{Local connectivity of the Julia set}
\label{sss:JLC}
An infinitely renormalizable quadratic-like map $f\colon U\to V$ is said to satisfy an \emph{unbranched a priori bounds} (see~\cite{McM2}) if there is an $\varepsilon>0$ such that for infinitely many $n$ the renormalization $f_n\coloneqq \RR^n_\QL (f)$ can be written as $f_n\colon U_n\to V_n$ so that 
\begin{equation}
\label{eq:defn:unbr cond}
V_n\cap \Post (f)\subset \Kfilled(f_n)
\end{equation}
 and the modulus of $V_n\setminus U_n$ is at least $\varepsilon$. We will refer to~\eqref{eq:defn:unbr cond} as the \emph{unbranched condition}. It is known that any infinitely renormalizable quadratic-like map with unbranched a \emph{priori} bounds has locally connected Julia set; see~\cite{J},\cite{McM2},~\cite{L:acta}*{Theorem~\RN{6}} for the reference.

\subsubsection{Renormalization periodic points and horseshoes} \label{sss:renorm horsh}A quadratic-like map $f\colon U\to V$ is a \emph{renormalization periodic point} if it is conformally conjugate to its proper quadratic-like renormalization: there is an iteration 
\[ f_\bullet =f^m\colon U_\bullet \to V_\bullet, \sp\sp\sp U_\bullet \subset U\]
and a conformal conjugacy (renormalization change of variables) $\phi\colon V_\bullet\to V$ between $f_\bullet \mid U_\bullet$ and $f\mid U$. The projection $\bchi(f)\in \Mandel$ is a periodic point of $\mRR_\QL$, see~\S\ref{ssss:QL anal ren}. We usually assume that $c_1(f)=c_1(f_\bullet)$ i.e., $c_1$ is a fixed point of $\phi$. Since $U_\bullet\subsetneqq U$, we have $|\phi'(c_1)|>1$ by the Shwarz lemma.

Linearizing $\phi$ one can assume that it is affine: if $L$ is the linearizer for $\phi$, then replacing $\phi,f,f_\bullet$ with their conjugacies by $L$ we obtain that the new $f_\bullet$ is affinely conjugate to$f$.

By a \emph{renormalization horseshoe} $\Hors$ we mean a precompact family of quadratic-like maps $f\colon U_f \to V_f$ such that every $f\in \Hors$ has a quadratic-like renormalization $f_1 =f ^{m(f)}\colon U_{f_1} \to V_{f_1}$ conformally conjugate to a map $\hat f_1$ in $\Hors$ and such that the renormalization $f\mapsto \hat f_1\colon \Hors\selfmap$ is injective. A horseshoe $\Hors$ is \emph{hyperbolic} if it is a hyperbolic set of a renormalization operator defined on a neighborhood of $\Hors$.

 Since the renormalization change of variables in a horseshoe $\Hors$ is conformal, unbrached a priori bounds descends from the top to all deep scales: if
 \[  \mod (V_f\setminus U_f)\ge \varepsilon\sp\sp\text{ and }\sp\sp V_{f_1} \cap \Post(f) \subset  \Kfilled (f_1) \sp\sp \text{ for all }\sp f\in \Hors\]
 then $\mod (V_{f_n}\setminus U_{f_n})\ge \varepsilon$ and $V_{f_n} \cap \Post(f) \subset  \Kfilled (f_n)$ for all $n$.

\subsection{Pacmen} In this subsection we collect the background information on the pacman renormalization from~\cite{DLS}. A \emph{full pacman} is a map \[f:\overline{U}
\to \overline{V}\] such that (see  Figure~\ref{Fg:Pacman} and~\cite[Definition 2.1]{DLS})
\begin{itemize}
\item $f(\alpha)=\alpha$;
\item   $\overline{U}$ is a closed topological disk with $\overline{U}\subset V$;

\item the critical arc $\gamma_1$ has exactly $3$ lifts $\gamma_0\subset U$ and $\gamma_-~,\gamma_+ \subset \partial U$ such that $\gamma_0$ starts at the fixed point $\alpha$ while $\gamma_-~,\gamma_+$ start at the pre-fixed point $\alpha'$; we assume that $\gamma_1$ does not intersect $\gamma_0,\gamma_-~,$  $\gamma_+$ away from $\alpha$;
\item $f:U\to V$ is analytic and $f:U\setminus \gamma_0\to V\setminus \gamma_1$ is a two-to-one branched covering;  
\item $f$ admits a locally conformal extension through $\partial U\setminus \{\alpha'\}$.
\end{itemize}

\begin{figure}[t!]
\[\begin{tikzpicture}[scale=1.3]

\coordinate  (w0) at (-3.9,1.5);

\draw (w0) circle (0.2cm);
\draw[right] (-3.7,1.5) node {$O\approx O_0$};

\node (v0) at (-2.9,1.4)  {};

\draw  (v0) ellipse (3.5 and 2);

\coordinate (v1) at (-2.2,1.3) {};
\coordinate (v1a) at (-2.02,1.39) {};
\coordinate (v1b) at (-2.02,1.21) {};

\coordinate (v2) at (-0.6,2.1) {} {};
\coordinate (v3) at (-0.6,0.6) {} {};

  \draw[ line width=0.5pt,orange] 
            (v3) 
            .. controls (v3) and (v1b) ..
            (v1b);
    \draw[orange, line width=0.5pt]           
            (v1a).. controls (v1a) and (v2)..(v2);
    \draw[line width=0.5pt,red]  
            (v2)
            .. controls (-7.3,4.4) and (-7.1,-1.5) ..
            (v3);

\draw [orange, shift={(-2.2,1.3)}, scale=0.2] plot[domain=-2.74468443448354:2.6325217651674095,variable=\t]({-1.*1.0089598604503545*cos(\t r)+0.*1.0089598604503545*sin(\t r)},{-0.*1.0089598604503545*cos(\t r)+1.*1.0089598604503545*sin(\t r)});

\coordinate (w1) at (-4.7,2.4) {} {};

\draw[line width=0.5pt,red]  (w0)--(w1);

\coordinate (w2) at (-6.1,0.6)  {};
\draw  [line width=0.5pt]  (w0)--(w2);

\draw[line  width=0.5pt,-latex] (-3.1,1.6) .. controls (-2.6,2.1) and (-3.1,3.2) .. (-3.8,2.8);

\node at (-5.5,0.5) {$\gamma_1$};
\node at (-2.8,2.7) {$f$};
\node[red] at (-4.5,1.9) {$\gamma_0$};

\node at (-1.9,4.6) {};
\node[red] at (-3.5,0.8) {$U$};
\node at (-1.6,0.1) {$V$};
\node[red] at (-3.6,-0.05) {$\partial ^\ext U$};
\node[orange] at (-1.1,1.3) {$\partial ^\frb U$};
\end{tikzpicture}\]
\caption{A pacman is a truncated version of a full pacman, see Figure~\ref{Fg:Pacman}; it is an almost $2:1$ map $f:(U,O_0)\to (V,O)$ with $f(\partial U)\subset \partial V\cup \gamma_1\cup \partial O$.}
\label{Fg:TruncPacman}
\end{figure}

A pacman is a obtained from a full pacman by removing a small neighborhood of $\alpha'$, see Figure~\ref{Fg:TruncPacman}.  More precisely, a \emph{pacman} is an analytic map
\begin{equation}
\label{eq:TrunPacm}
f\colon (U, O_0)\to (V,O)
\end{equation}
with $f(\partial U)\subset \partial V \cup \gamma_1\cup \partial O$ such that
\begin{itemize}
\item $O_0$ and $O$ are disk neighborhoods of $\alpha$ and $f$ maps $O_0$ conformally to $O$;
\item $f$ admits a locally conformal extension through $\partial U$;
\item every point in $V\setminus O$ has two preimages in $U$ as for a full pacman while every point in $O$ has a single preimage in $O_0$ (in other words, $f$ can be topologically extended to a full topological pacman by adding topologically the second preimage of $O\setminus \gamma_1$).
\end{itemize}

We recognize the following two subsets of the boundary of $U$:  the \emph{external boundary} $\partial^{\ext} U := f^{-1}(\partial V)$ and \emph{the forbidden part of the boundary} $\partial^{\frb} U:= \overline {\partial U\setminus \partial^{\ext} U}$.

Given a pacman $f\colon U\to V$, its \emph{non-escaping set} is \[\Kfilled_f\coloneqq \bigcap_{n\ge 0} f^{-n}\left(\overline U\right),\]
it is sensitive to a small deformation of $\partial U$. The \emph{escaping set} of $f$ is $V\setminus \Kfilled_f$. 

 Let us embed a topological rectangle $\Rect$ in $\overline{V} \setminus U$ so that the bottom horizontal side is equal to $\partial^{\ext} U $ and the top horizontal side is a subset of $\partial V$. The images of the vertical lines within $\Rect$ form a lamination of $\overline{V} \setminus U$. We pull back this lamination to all iterated preimages $f^{-n}(\Rect)$. Leaves of this lamination that start at $\partial V$ are called \emph{external ray segments} of  $f$; infinite external ray segments are called \emph{external rays} of $f$. Note that if $\gamma$ is an external ray, then $f(\gamma)\coloneqq f(\gamma\cap U)$ is also an external ray. Every external ray $\gamma$ has a well defined external angle $\phi$ such that the angle of $f(\gamma)$ is $2\phi$, see~\cite{DLS}*{\S 2.1}.

\begin{figure}[t!]

\centering{\begin{tikzpicture}[ scale=1.18]


\draw[ draw opacity=0 ,fill=red, fill opacity=0.2]
  (0,0) --(-1.34,3.54)--  (-1.34+ 1.09,3.54+0.72)-- (-3 + 1.09 ,2.4+0.72)--(-2.7 + 1.09 ,1.38+0.72)--(-4.22 + 1.09, 1.6+0.72)--(-4.9127182890505665,1.0896661385188597) --(0,0);

\draw[blue] (0,0) edge node[below] {$\beta_-$}  (-6.34,1.44)
(-6.34,1.44) edge (0.26,5.82)
(0.26,5.82) edge node[right] {$\beta_+$}  (0,0);

\draw[blue] (0,0) edge node[below] {$\beta$}   (-4.78,2.47);

\node[blue,below ] at (0,0) {$\alpha$};

\draw[red] (-2.7+ 1.09 ,0.72+1.38)--(-4.22+ 1.09 ,0.72+ 1.6);
\draw[red] (-4.9127182890505665,1.0896661385188597)-- (-4.22+ 1.09 ,0.72+ 1.6);

\draw[red,fill=red, fill opacity=0.1] (-2.7+ 1.09 ,0.72+1.38)--(-3+ 1.09 ,0.72+2.4)--(-3.74+ 1.09 ,0.72+1.92)--(-2.7+ 1.09 ,0.72+1.38);

\draw[draw opacity=0,fill=green, fill opacity=0.1] (0,0)--(-1.34,3.54)-- (0.21, 4.64)--(0,0);

\draw[red]  (0,0) --(-1.34,3.54)-- (-3+ 1.09 ,0.72+2.4);

\draw[red] (-1.34,3.54)-- (0.21, 4.64);

\draw[blue] (-4.48, 1.92) node{$S_+$};
\draw[blue] (-1.9, 3.8) node{$S_-$};

\draw[red] (-1.64, 1.28) node{$\Upsilon_-$};
\draw[red] (-3.15+ 1.09 ,0.72+ 1.90) node{$\Upsilon_+$};

\draw (-3.15+ 1.09 ,0.72+ 2.10) edge[->,bend right]  (-4.48, 2.12);

\draw[red] (-0.48, 2.8) node{$U_+$};

\draw(-0.38, 3.1)  edge[->,bend right]  (-4.58, 2.22);

\draw  (-2.34, 1.28)  edge[->,bend left]   node[left ]{$2:1$} (-1.9, 3.6);

\draw[dashed,red] (-1.71,4.51) edge node[right] {$\beta_0$} (-1.34, 3.54);

\draw (-1.74, 4.76) edge[->,bend right] node[above]{gluing} node[below]{ $\beta_\pm\simeq \gamma_1$} (-3.74, 5.2);


\begin{scope}[shift={(-4.5,6.2)} ,yscale=-1,scale=0.7]

\coordinate  (w0) at (-3.9,1.5);

\node (v0) at (-2.9,1.4)  {};

\draw[blue]  (v0) ellipse (3.5 and 2);

\coordinate (v1) at (-2.2,1.3) {};
\coordinate (v2) at (-0.6,2.1) {} {};
\coordinate (v3) at (-0.6,0.6) {} {};

  \draw[ line width=0.5pt,red] 
            (v3) 
            .. controls (v3) and (v1) ..
            (v1)
            .. controls (v1) and (v2) ..
            (v2)
            .. controls (-7.3,4.4) and (-7.1,-1.5) ..
            (v3);

\coordinate (w1) at (-4.7,2.4) {} {};

\draw[line width=0.5pt]  (w0)--(w1);

\coordinate (w2) at (-6.1,0.6)  {};
\draw  [line width=0.5pt]  (w0)--(w2);

\draw  [line width=0.5pt]  (w0)edge node[right]{$\gamma_2$}(-4.1,-0.5);

\node[right] at (w0) {$\alpha$};
\node at (-5.5,0.7) {$\gamma_1$};
\node[] at (-4.5,1.9) {$\gamma_0$};

\node at (-1.9,4.6) {};
\end{scope}

\end{tikzpicture}}

\caption{A (full) prepacman $(f_-\colon U_-\to S,\sp f_+\colon U_+\to S)$. We have $U_-=\Upsilon_-\cup \Upsilon_+$ and $f_-$ maps $\Upsilon_-$ two-to-one to $S_-$ and $\Upsilon_+$ to $S_+$. The map $f_+$ maps $U_+$ univalently onto $S_+$. After gluing dynamically $\beta_-$ and $\beta_+$ we obtain a full pacman: the arcs $\beta_-$ and $\beta_+$ project to $\gamma_1$, the arc $\beta_0$ projects to $\gamma_0$, and the arc $\beta$ projects to $\gamma_2$.
 }
\label{Fg:Prepacman}
\end{figure}

\subsubsection{Prepacmen}\label{sss:prepacman} (See Figure~\ref{Fg:Prepacman} and~\cite[Definition 2.2]{DLS}.) A prepacman is a pair of commuting maps that can be glued into a pacman. More precisely, consider a sector $S$ with boundary rays
$\beta_-,\beta_+\subset \partial S$ and with an interior ray $\beta_0$ that divides $S$ into two subsectors $T_-,T_+$. Let $f_-\colon U_-\to S, \sp f_+\colon U_+\to S$ be a pair of holomorphic maps, defined on $U_-\subset T_-$ and $U_+\subset T_+$. We say that $F=(f_\pm \colon U_{\pm}\to S)$ is a \emph{prepacman} if there exists a
gluing $\psi$ of $S$ which projects $(f_-,f_+)$ onto a pacman $f\colon U\to V$ where $\psi\mid \intr S$ is conformal, $\beta_-, \beta_+$ are mapped to the critical arc $\gamma_1=\psi(\beta_\pm)$, and $\beta_0$ is mapped to $\gamma_0$.

The definition implies that $f_-$ and $f_+$ commute in a neighborhood of $\beta_0$. Note that every pacman $f\colon U\to V$ has a prepacman obtained by cutting $V$ along the critical arc $\gamma_1$. Dynamical objects (such as the non-escaping set) of a prepacman $F$ are preimages of the corresponding dynamical objects of $f$ under $\psi$.

\begin{figure}[t!]

\centering{\begin{tikzpicture}[ scale=0.77]


\draw[red, fill=red, fill opacity=0.3]
  (0,0) --(-1.34,3.54)--  (-3+ 1.09 ,0.72+2.4)--(-3.74+ 1.09 ,0.72+1.92)--(-2.7+ 1.09 ,0.72+1.38)--(-4.22+ 1.09 ,0.72+ 1.6)--(-4.92, 1.12) --(0,0);

\draw[green,fill=green, fill opacity=0.3] (0,0)--(-1.2, 3.18)-- (0.18/1.2, 4.09/1.2)--(0,0);

\draw[blue] (0,0)-- (-9.96/1.1, 2.26/1.1) -- (-6.9, 3.6)--(0.3, 6.75) --(0,0);

\draw[blue] (0,0) edge   (-6.9, 3.6);

\draw[red] (-0.48, 2.8) node{$U_+$};


\draw[green ,fill=green, fill opacity=0.3] (0,0)--(-4.92, 1.12)--(-3.52*1.1,-1.32*1.1)--(0.,0.);
\draw[red, fill=red, fill opacity=0.3,scale=1.2] (0,0)--(-3.52,-1.32)--(-1.06 -0.50062, 0.42016-3.34)--(-0.6 -0.50062, 0.42016-2.4)
--(-0.32 -0.50062, 0.42016-3.78)
--(0.28, -4.32)
--(0.,0.);


\draw[green,fill=green, fill opacity=0.3] (0.,0.)-- (0.28*1.2, -4.32*1.2)-- (2.28*1.1, -3.26*1.1)--(0,0);

\draw[red,fill=red, fill opacity=0.3,scale=1.2] (0,0)--(2.28, -3.26)-- (3.82-0.25,-0.3-1.28)
-- (2.74-0.2,-0.25-1.36)
-- (4.3-0.25,-0.3-0.6)
-- (5.05, 0.32)
-- (0.,0.);

\draw[green,fill=green, fill opacity=0.3] (0,0)--(5.05*1.2, 0.32*1.2)
--(4.7675*1.1, 1.60875*1.1)-- (0.,0.);


\draw[red,fill=red, fill opacity=0.3,scale=1.2] (0,0)--(4.7675, 1.60875)-- (1.58+0.50125*2,-0.6775*2+4.64)-- (1.24+0.50125*2,-0.6775*2+3.88)--(1.2+0.50125*2,-0.6775*2+ 4.97)-- (1.2, 4.45)--(0,0);

\draw[red,scale=1.2] (1.24+0.50125*2,-0.6775*2+3.88)--(1.89+0.50125*2,-0.6775*2+ 4.37);

\draw[green,fill=green, fill opacity=0.3] (0,0)--(1.2*1.2, 4.45*1.2)--(0.25, 5.51)--(0,0);



\draw (0.705, 5.06812)edge[->, bend right]  (-4.68-1.2,+0.5+ 2.32);

\draw[blue] (-4.78-0.4,+0.3+ 1.92) node{$S_+$};

\draw[blue] (-1.7-0.4,-0.3+ 3.7) node{$S_-$};
\draw[red] (-1.58+0.3,+ 0.97) node{$U_-$};

\draw(1.58*1.2+0.50125*2.4,-0.6775*2.4 +4.31*1.2)  edge[->,bend right=35]  (-4.68-0.6,+0.3+ 2.32);

\draw  (-0.96, 0.45) edge[->,bend right]  node[left]{$f$} (-0.56-0.4,+0.3 -1.11); 
\draw  (-0.46, -1.37) edge[->,bend right]  node[below]{$f$} (2.25, -1.18);
\draw  (2.75, -0.9) edge[->,bend right] (2.41, 2.05);
\draw  (1.79*1.1, 2.81*1.1)edge[->,bend right] node[below]{$2:1$} (-1.5-0.4,-0.3+3.8);
\draw[scale=1.1] (1.33, 2.39) node{$c_0$};

\end{tikzpicture}}

\caption{ Pacman renormalization of $f$: the first return map of points in $U_-\cup U_+$ back to $S=S_-\cup S_+$ is a prepacman. Spreading around $U_\pm$: the orbits of $U_-$ and $U_+$ before returning back to $S$ triangulate a neighborhood $\bDelta$ of $\alpha$; we obtain $f\colon \bDelta\to \bDelta\cup S$, and we require that $\bDelta$ is compactly contained in $\Dom f$.}
\label{Fg:Prepacman:OrbitOfUi}
\end{figure}

\subsubsection{Pacman renormalization}\label{sss:prepacman} (See Figure~\ref{Fg:Prepacman:OrbitOfUi} and~\cite[Definition 2.3]{DLS}.) We say that a holomorphic map $f\colon (U,\alpha)\to   (V,\alpha)$ with a distinguished $\alpha$-fixed point is \emph{pacman renormalizable} if there exists a prepacman \[G=(g_-=f^\aa \colon U_-\to S,\sp g_+=f^\bb \colon U_+\to S)\] defined on a sector  $S \subset V$ with vertex at $\alpha$  such that $g_-~, g_+$ are iterates of $f$ realizing the first return map to $S$ and such that the $f$-orbits of $ U_-~, U_+$ before they return to $S$ cover a neighborhood of $\alpha$ compactly contained in $U$. We call $G$ the \emph{pre-renormalization} of $f$ and the pacman $g\colon \widehat{U}\to   \widehat{V} $ is the \emph{renormalization} of $f$. By default, we assume that $S$ contains the critical value of $f$.

The numbers $\aa,\bb$ are the \emph{renormalization return times}. The renormalization of $f$ is called \emph{prime} if $\aa+\bb=3$. Combinatorially, a pacman renormalization is an iteration of prime
renormalizations, see~\S\ref{s:SectRenorm}.

We define $\bDelta=\bDelta_G$ to be the union of points in the $f$-orbits of $\overline U_-~,\overline U_+$ before they return to $S$. Naturally, $\bDelta$ is a triangulated neighborhood of $\alpha$, see Figure~\ref{Fg:Prepacman:OrbitOfUi}. We call $\bDelta$ a \emph{renormalization triangulation} and we will often say that $\bDelta$ is obtained by \emph{spreading around} $U_-~,U_+$. We require $\bDelta_G\Subset \Dom f$ and $\bDelta_G\cup S\Subset \Im f$.

An \emph{indifferent pacman} is a pacman with indifferent $\alpha$-fixed point. The \emph{rotation number} of an indifferent pacman $f$ is $\theta \in \R/\Z$ so that $\ee(\theta)$ is the multiplier at $\alpha(f)$. If, in addition, $\theta\in \Q$, then $f$ is \emph{parabolic}. A pacman renormalization of an indifferent pacman is again an indifferent pacman.

\subsubsection{Banach neighborhoods} (See \cite[\S 2.4]{DLS}.)
Consider a pacman $f:(U_f,O_0,\gamma_0)\to (V,O,\gamma_1)$ with a non-empty truncation disk $O$. We assume that there is a topological disk $\widetilde U\Supset U_f$ with a piecewise smooth boundary such that $f$ extends analytically to $\widetilde U$ and continuously to its closure.
Choose a small $\varepsilon>0$ and define $N_{\widetilde U}(f, \varepsilon)$ to be the set of analytic maps $g:\widetilde U\to \C$ with continuous extensions to $\partial \widetilde U$ such that
\[\sup_{z\in \widetilde U} |f(z)-g(z)| < \varepsilon.\]
Then $N_{\widetilde U}(f,\varepsilon)$ is a Banach ball; it is the $\varepsilon$ neighborhood of $f$ in the Banach space of maps defined on $\widetilde U$.~\cite{DLS}*{Lemma 2.5} asserts that if $\gamma_0,\gamma_1$ land at $\alpha$ at distinct well-defined angles and $\varepsilon$ is sufficiently small, then every $g\in N_{\widetilde U}(f,\varepsilon)$ has a domain $U_g\subset \widetilde U$ such that $g\colon U_g\to V$ is a pacman with the same $V,\gamma_1,O$ (up to translation).

\subsubsection{Pacman analytic operator}\label{sss:RenOperat} (Summary of~\cite[\S 2.5]{DLS}.)
Suppose that a pacman $\hat f:\widehat U\to  \widehat V $ is a renormalization of a holomorphic map $f\colon (U,\alpha)\to   (V,\alpha)$ via a quotient map  $\psi_f:S_f\to \widehat V$. Assume that the curves $\beta_0, \beta_-,  \beta_+$ (see Figures~\ref{Fg:Prepacman} and~\ref{Fg:Prepacman:OrbitOfUi}) land at $\alpha$ at pairwise distinct well-defined angles. \cite{DLS}*{Theorem 2.7} asserts that for every sufficiently small neighborhood $N_{\widetilde U}(f,\varepsilon)$, 
there exists a compact analytic pacman renormalization operator $\mathcal{R}\colon g \mapsto \hat g$ defined on $  N_{\widetilde U}(f,\varepsilon) $ such that $\mathcal{R}(f)=\hat{f}$. Moreover, the gluing map $\psi_g$, used in this renormalization, also depends analytically on $g$. Note that the operator $\RR$ is non-dynamical: it goes from a small Banach neighborhood of $f$ (where $f$ needs not be a pacman) to a certain small Banach neighborhood of the pacman $\widehat f$. If $\widehat f=f$ as germs and $\widehat U\supset \widetilde U$, i.e., there is an ``improvement of the domain'', then $\RR\colon N_{\widetilde U}(f,\varepsilon) \to N_{\widetilde U}(f,\delta) , \sp \RR f=f$ is a dynamical operator for sufficiently small $\varepsilon$ and $\delta$.

\subsubsection{Siegel pacmen}
\label{sss:SiegPacmen}
(See~\cite[\S3]{DLS}) A holomorphic map $f\colon U\to V$ is  \emph{Siegel} if it has a fixed point $\alpha$, a Siegel quasidisk $\overline Z_f\ni \alpha$ compactly contained in $U$, and a unique critical point $c_0\in U$ that is on the boundary of $ Z_f$. Note that in~\cite{AL-posmeas} a Siegel map is assumed to satisfy additional technical requirements; these requirements are satisfied by restricting $f$ to an appropriate small neighborhood of $\overline Z_f$. 

It follows from \cite[Theorem 3.19, Proposition 4.3]{AL-posmeas} that any two Siegel maps with the same rotation number of bounded type are hybrid conjugate on neighborhoods of their closed Siegel disks.

A pacman $f:U\to V$ is \emph{Siegel} if 
\begin{itemize}
\item $f$ is a Siegel map with Siegel disk $Z_f$ centered at $\alpha$;
\item the critical arc $\gamma_1$ is the concatenation of an external ray $R_1$ followed by an inner ray $I_1$ of $Z_f$ such that the unique point in the intersection $\gamma_1\cap \partial Z_f$ is not precritical;  and
\item writing $f\colon (U, O_0)\to (V,O)$ as in~\eqref{eq:TrunPacm}, the disk  $O $ is a subset of $Z_f$ bounded by its equipotential.
\end{itemize}
The \emph{rotation number} of a Siegel pacman (or a Siegel map)
 is $\theta \in \R/\Z$ so that $\ee(\theta)$ is the multiplier at $\alpha$. It follows that the rotation number of Siegel map is in $ \Theta_\bnd$ -- the set of combinatorially bounded rotation numbers (i.e., rotation numbers with continued fraction expansion where all its coefficients are bounded).

 If $f$ is a Siegel pacman, then all external rays of $f$ land. Moreover, $\partial Z_f$ is in the closure of repelling periodic points, and every neighborhood of $\partial Z_f$ contains a repelling periodic cycle. Moreover, the non-escaping set of $f$ is locally connected.

A Siegel pacman is \emph{standard} if $\gamma_0$ passes through the critical value; equivalently if $\gamma_1$ passes through the image of the critical value. By~\cite[Corollary 3.7 and Lemma 3.4]{DLS} every Siegel map can be renormalized to a standard Siegel pacman.

\subsubsection{Hyperbolic pacman renormalization self-operator}\label{sss:HypSelfOper} 
Let us now fix a rotation number $\theta_\str$ periodic under $\cRRc$, see~\eqref{eq:R_prm}. By~\cite[Theorems 3.16 and 7.7]{DLS}, there is a pacman renormalization operator $\RR\colon N_{\widetilde U}(f_\str,\varepsilon)\to N_{\widetilde U}(f_\str,\delta)$ 
with a fixed standard Siegel pacman $\RR f_\str =f_\str$ such that the rotation angle of $f_\str$ is $\theta_\str$. Moreover, $\RR$ is compact, analytic, and hyperbolic. We write this operator as $\RR\colon \BB \dashrightarrow \BB$, where $\BB=N_{\widetilde U}(f_\str,\delta)$.

The renormalization operator $\RR\colon \BB \dashrightarrow \BB$ is hyperbolic at $f_\str$ with one-dimensional unstable manifold $\WW^u$ and codimension-one stable manifold $\WW^s$. In a small neighborhood of $f_\str$ the stable manifold $\WW^s$ coincide with the set of pacmen in $\BB$ that have the same multiplier at the $\alpha$-fixed point as $f_\str$. Moreover, every pacman in $\WW^s$ is Siegel. In a small neighborhood of $f_\str$ the unstable manifold $\WW^u$ is parametrized by the multipliers of the $\alpha$-fixed points of $f\in \WW^u$. 

Our convention is that the renormalization change of variables $\psi_f\colon S_f\to V$ associated with $\RR f$ is defined near the critical value; i.e.~$S_f\ni c_1(f)$ and $\psi_f(c_1(f))=c_1(\RR f)$. In other words, the renormalization zooms at the critical value. 

By~\cite[Lemma 3.18]{DLS}, $\RR$ acts on the rotation numbers of indifferent pacmen as $\cRRc^\mm$ for some $\mm\ge 2$. Namely, if $f\in \BB$ is an indifferent pacman with rotation number $\theta$, then $\RR f$ is again an indifferent pacman with rotation number $\cRRc^\mm (\theta)$. In particular, $\cRRc^\mm(\theta_\str)=\theta_\str$. We call $\mm$ the \emph{renormalization period} of $\RR$. We will show in Proposition~\ref{prop:Ren oper with min per} that $\RR\colon \BB\dashrightarrow \BB$ can be constructed so that $\mm$ is the minimal period of $\theta_\str$ under $\cRRc$.

\subsubsection{Combinatorics of $\RR\colon \BB\dashrightarrow \BB$}
\label{sss:AntMatrix}
Since the renormalization $\RR f_\str=f_\str$ restricted to the Siegel quasidisk $\overline Z_\str$ is the $\mm$-th iterate of the prime renormalization, we have (see~\S\ref{ss:ren of rotat}):
\begin{equation}
\label{eq:theta is a per pnt}
\theta_\str=\cRRc^\mm(\theta_\str).
\end{equation}

\begin{lem}[\cite{DLS}*{Lemma 3.17}]
\label{lem:RR:rot numb act}
If $f\in \BB$ is an indifferent pacman with a rotation number $\theta$, then $\RR f$ is again an indifferent pacman with rotation number $\cRRc^\mm(\theta)$.
\end{lem}

 Since the unstable manifold $\WW^u$ is parametrized by the multipliers of the $\alpha$-fixed points, the unstable eigenvalue $\lambda_\str$ is equal to the derivative $(\cRRc^\mm)'(\theta_\str)$.

Let $\M$ be the antirenormalization matrix associated with~\eqref{eq:theta is a per pnt}, see~\eqref{eq:sect to quarter}. Recall that $\M$ has positive entries. As in \S\ref{sss:renorm:pair of rotat}, let $\tt$ be the leading eigenvalue of $\M$. By Lemma~\ref{lem:lambda:t*t}, 
\begin{equation}
\label{eq:lambda str is t^2}
\lambda_\str=\tt^2.
\end{equation}

\subsubsection{Operators on near Siegel maps}
\label{sss:from SM to SP}
Consider a Siegel map $g$ with rotation angle $\theta_g$. Suppose $\cRRc^k(\theta_g)=\theta_\str$ for some $k\ge0$. Then  $g$ can be renormalized to a pacman on the stable manifold of $f_\str$, see~\cite[Corollary~3.7 and Lemma~7.8]{DLS}. This allows us to define a compact analytic renormalization operator $\RR_\Sieg\colon\AA\to \BB$ with $\RR_\Sieg (g)\in \WW^s$, where $\AA$ is a small Banach neighborhood of $g$.

\begin{figure}[tp!]

{\begin{tikzpicture}

\begin{scope}[shift={(0,5)},scale =0.8]
\draw[blue, shift ={(0,0.6)},yscale =-1] (-1.1,4.3) edge[<-,bend left=8] node [above right] {$\phi_{-2}=\psi_{-2}^{-1}$} (0.7,6.7);
\end{scope}

\begin{scope}[shift={(0,0)},scale =0.8]
\draw [rotate around={0.:(0.,0.)}] (0.,0.) ellipse (2.246099127742333cm and 1.0222334819623482cm);
\draw (-0.7,-0.04)-- (-1.4589248333679499,-0.7772346338620598);

\draw[red] (-0.7,-0.04)-- (-1.46,0.28);
\draw[red] (-1.46,0.28)-- (-1.06,0.64);
\draw[red] (-1.06,0.64)-- (-0.7,-0.04);
\draw (-1.5,-0.5) node {$\gamma_1$};
\draw[red] (-.5,0.4) node {$S_{-1}$};
\draw[blue, shift ={(0,0.6)},yscale =-1] (-1.1,0.3) edge[<-,bend left] node [ right] {$\phi_{-1}=\psi_{-1}^{-1}$} (0.7,6.7);
\draw[shift= {(0,-1)}] (0.8,0.5) node{$f_{-2}$};

\draw[shift={(1,-3)},blue]  (0,3.2) edge[->,bend left] node [above ] {$h_{-2}$} (6,3);

{\begin{scope}[shift={(8,-1.)},scale =1.5,red]
\draw (-1.,-1.)-- (-1.,2.)-- (1.,2.)-- (1.,-1.);

\draw[scale=0.4] (-1.,-1.)-- (-1.,2.)-- (1.,2.)-- (1.,-1.);
\draw[scale=0.16] (-1.,-1.)-- (-1.,2.)-- (1.,2.)-- (1.,-1.);

\draw (0,1.4) node{$\bS_{-2}$};

\draw[blue] (-0.6,1.2) edge[->] node[above right] {$A_\str$}(-0.47,2.6);

\end{scope}} 

\end{scope}


\begin{scope}[shift={(0,-5)},scale =0.8]
\draw [rotate around={0.:(0.,0.)}] (0.,0.) ellipse (2.246099127742333cm and 1.0222334819623482cm);
\draw (-0.7,-0.04)-- (-1.4589248333679499,-0.7772346338620598);

\draw[red] (-0.7,-0.04)-- (-1.46,0.28);
\draw[red] (-1.46,0.28)-- (-1.06,0.64);
\draw[red] (-1.06,0.64)-- (-0.7,-0.04);
\draw (-1.5,-0.5) node {$\gamma_1$};
\draw[red] (-.5,0.4) node {$S_{0}$};
\draw[blue, shift ={(0,0.6)},yscale =-1] (-1.1,0.3) edge[<-,bend left] node [ right] {$\phi_{0}=\psi_0^{-1}$} (0.7,6.7);
\draw[shift= {(0,-1)}] (0.8,0.5) node{$f_{-1}$};

\draw[shift={(1,-3)},blue]  (0,3.2) edge[->,bend left] node [above ] {$h_{-1}$} (6,3);

{\begin{scope}[shift={(8,-1.)},scale =1.5,red]
\draw (-1.,-1.)-- (-1.,2.)-- (1.,2.)-- (1.,-1.);

\draw[scale=0.4] (-1.,-1.)-- (-1.,2.)-- (1.,2.)-- (1.,-1.);

\draw (0,1.4) node{$\bS_{-1}$};
\draw [blue](-0.6,1.2) edge[->] node[right] {$A_\str$}(-0.25,4.6);
\end{scope}}

\end{scope}


\begin{scope}[shift={(0,-10)},scale =0.8]
\draw [rotate around={0.:(0.,0.)}] (0.,0.) ellipse (2.246099127742333cm and 1.0222334819623482cm);
\draw (-0.7,-0.04)-- (-1.4589248333679499,-0.7772346338620598);

\draw (-1.5,-0.5) node {$\gamma_1$};
\draw[shift= {(0,-1)}] (0.8,0.5) node{$f_{0}$};

\draw[shift={(1,-3)},blue]  (0,3.2) edge[->,bend left] node [above ] {$h_{0}$} (6,3);

{\begin{scope}[shift={(8,-1.)},scale =1.5,red]
\draw (-1.,-1.)-- (-1.,2.)-- (1.,2.)-- (1.,-1.);


\draw (0,1.4) node{$\bS_{0}$};

\draw [blue](-0.6,1.2) edge[->] node[right] {$A_\str $}(-0.25,4.6);

\end{scope}}

\end{scope}

\end{tikzpicture}}

\caption{Tower of antirenormalizations.} \label{Fig:Sf1dash}
\vspace{128in}
\clearpage
\end{figure}

\subsubsection{Maximal prepacmen}(Summary of~\cite[\S5]{DLS}.)
\label{sss:max prepacmen}
Every pacman $f\in \WW^u$, can be anti-renormalized infinitely many times. For $n\le 0$, we write $f_n\coloneqq \RR^{n} f$ and we denote by $F_n$ the associated prepacman (obtained by cutting $f_n$ along its critical arc $\gamma_1$). Let $\psi_{n}\colon S_n\to V$ be the renormalization change of variables realizing the renormalization of $f_{n-1}$. This means that there is a prepacman 
\begin{align*}
 F^{(n-1)}_n= &\left(f^{(n-1)}_{n,\pm }\colon U^{(n-1)}_{n,\pm}\to S^{(n-1)}\right)\\=& \left(f^{\aa}_{n-1}\colon U^{(n-1)}_{n,-}\to S^{(n)},\sp f^{\bb}_{n-1}\colon U^{(n-1)}_{n,+}\to S^{(n)}\right),
\end{align*}
in the dynamical plane of $f_n$ such that $\psi_{n}$ projects $F^{(n-1)}_n$ to $f_n$. We also say that $\psi_{n}^{-1}$ \emph{embeds} $F_n$ to the dynamical plane of $f_{n}$, and we call $F_{n}^{(n-1)}$ the \emph{embedding}.

Write 
\begin{align*}
\psi_{\str}&=\psi_{f_\str},\\
\mu_{\str}&\coloneqq \left(\psi^{-1}_{\str}\right)' (c_1),\\
A_\str&\colon  z\mapsto \mu_{\str} z,\\
T_n&\colon z\mapsto z-c_1(f_n).
 \end{align*}
Then the limit 
\begin{equation}
\label{eq:h_0:defn}
h_0(z)\coloneqq \lim_{n\to -\infty} A_{\str}^{n}\circ T_{n+1} \circ \left(\psi_{n+1}^{-1}\circ \dots \circ \psi_{-1}^{-1}\circ\psi _0^{-1} (z)\right)
\end{equation}
exists for all $z\in V\setminus \gamma_1$ with an extension through $\gamma_1\setminus \{\alpha\}$. The $\alpha$-fixed point is not in the domain of $h_0$: as $z$ approaches $\alpha$, its image $h_0(z)$ approaches $\infty$. Similarly $h_n$ are defined for $n\le 0$. The maps $h_n$ linearize $\psi$-coordinates (see Figure~\ref{Fig:Sf1dash}):
let 
\begin{equation}
\label{eq:Max Prep U pm}
\bF= (\bbf_{\pm}\colon \bU_{\pm }\to \bS)
\end{equation}
 be the image of $F=F_0=(f_\pm \colon U_\pm \to S)$ via $h_0$ (i.e.~$\bS $ is the closure of $h_0 (V\setminus \gamma_1)$, and $\bbf_{\pm}\coloneqq h_0 \circ f_{\pm} \circ h_0^{-1}$), then $\bbf_{0,\pm}$ are iterations of $\bbf_{n,\pm}$ rescaled by $A^{n}_\str$ (see~\eqref{eq:iter of max prep} below). The map~\eqref{eq:Max Prep U pm} is a prepacman (Figure~\ref{Fg:Prepacman}) in $\wC$ with $\alpha=\infty$.

\begin{figure}[tp!]

{\begin{tikzpicture}

\begin{scope}[shift={(0,1)},scale =0.8]
\draw[blue, shift ={(0,0.6)},yscale =-1] (-1.1,4.3) edge[<-,bend left=8] node [above right] {$\phi_{n}$} (0.7,6.7);
\end{scope}

\begin{scope}[shift={(0,-4)},scale =0.8]
\draw [rotate around={0.:(0.,0.)}] (0.,0.) ellipse (2.246099127742333cm and 1.0222334819623482cm);
\draw (-0.7,-0.04)-- (-1.4589248333679499,-0.7772346338620598);

\draw[red] (-0.7,-0.04)-- (-1.35,0.35);
\draw[red] (-1.35,0.35)-- (-1.2,0.5);
\draw[red] (-1.2,0.5)-- (-0.7,-0.04);
\draw (-1.5,-0.5) node {$\gamma_1$};
\draw[red] (-.5,0.4) node {$S_{-n}$};
\draw[blue, shift ={(0,0.6)},yscale =-1] (-1.1,0.3) edge[<-,bend left] node [ right] {$\phi_{n+1}\circ\dots \circ\phi_{0}$} (0,7.7);
\draw[shift= {(0,-1)}] (0.8,0.5) node{$f_{n}$};

\draw[shift={(1,-3)},blue]  (0,3.2) edge[->,bend left] node [above ] {$A_\str^{n}\circ h_{n}$} (7.5,2.5);

{\begin{scope}[shift={(9.1,-4.3)},scale =2.5,red]
\draw (-1.,-1.)-- (-1.,2.)-- (1.,2.)-- (1.,-1.);

\draw[scale=0.2] (-1.,-3.5)-- (-1.,2.)-- (1.,2.)-- (1.,-3.5);

\draw (0.05,0.15) node{$\bS_{0}$};

\draw (0.7,1.4) node{$\bS^\#_{n}$};
\end{scope}} 

\end{scope}


\begin{scope}[shift={(0,-10)},scale =0.8]
\draw [rotate around={0.:(0.,0.)}] (0.,0.) ellipse (2.246099127742333cm and 1.0222334819623482cm);
\draw (-0.7,-0.04)-- (-1.4589248333679499,-0.7772346338620598);

\draw (-1.5,-0.5) node {$\gamma_1$};
\draw[shift= {(0,-1)}] (0.8,0.5) node{$f_{0}$};

\draw[shift={(1,-3)},blue]  (0,3.2) edge[->,bend left] node [above ] {$h_{0}$} (8,6);

\end{scope}

\end{tikzpicture}}

\caption{$\bS^\#_n=A_\str^n \bS_n$ and $\displaystyle\bigcup_{n\le0} \bS^\#_n=\C$, compare with Figure~\ref{Fig:Sf1dash}.} \label{Fig:Sf1dash:2}
\end{figure}

Let $g : X \to Y$ be a holomorphic map between Riemann surfaces.  
Recall that $g$ is:
\begin{itemize}
\item proper, if $g^{-1}(K)$ is compact for each compact $K \subset Y$;
\item $\sigma$-proper (see \cite[\S 8]{McM3}) if each component of $g^{-1}(K)$ is compact for each compact $K \subset Y$; or equivalently if $X$ and $Y$ can be expressed as increasing unions of subsurfaces $X_i$, $Y_i$ such that $g : X_i \to Y_i$ is proper.
\end{itemize}
\noindent A proper map is clearly $\sigma$-proper.

\cite[Theorem 5.5]{DLS} asserts that $\bbf_{\pm}$ admit maximal analytic extensions to $\sigma$-proper maps of the complex plane; we call the pair of extensions
\begin{equation}
\label{eq:MaxPrep}
\bF= (\bbf_-\colon  \bX_-\to \C,\sp \bbf_+\colon  \bX_+\to \C),
\end{equation}
a \emph{maximal prepacman}. The maximal extension~\eqref{eq:MaxPrep} is obtained by iterating a certain number of times in the dynamical plane of $f_k$, $k\le0$.  Namely, a big open topological disk $\bD^k$ around $0$ in the dynamical plane of $\bF$ is identified via $(A_\str^k \circ h_k)^{-1}$ with a fixed disk $D$ in the dynamical plane of $f_k$ around the critical value $c_1(f_k)$, see Figure~\ref{Fig:Sf1dash:2}. Moreover, $D$ also contains $f_k^{\aa_k}(c_1)$ and $f_k^{\bb_k}(c_1)$ for certain $\aa_k,\bb_k$. Let $W_-^{(k)}$ and $W_+^{(k)}$ be the pullbacks of $D$ along the orbits $c_1,f_k(c_1),\dots, f_k^{\aa_k}(c_1)$ and  $c_1,f_k(c_1),\dots, f_k^{\bb_k}(c_1)$ respectively. This way we obtain a pair of branched coverings (see~\cite{DLS}*{(5.9)}) 
\begin{equation}
\label{eq:pair:W_pm to D}
(f_k^{\aa_k}\colon  W_-^{(k)} \to D,\sp  f_k^{\bb_k}\colon  W_+^{(k)}\to D).
\end{equation}
(The main step is to show that the backward orbits of $D$ in these pullbacks do not hit $\partial^\frb U_k$, see Figure~\ref{Fg:TruncPacman}; this is {\cite[Key Lemma 4.8]{DLS}} stated as Lemma~\ref{lem:KeyLemma} later in the paper.) Then $A_\str^k \circ h_k$ conjugates \eqref{eq:pair:W_pm to D} to the pair
\begin{equation}
\label{eq:pair:bW_pm to bD^k}
(\bbf_- \colon  \bW_-^{(k)} \to \bD^k,\sp  \bbf_+\colon \bW_+^{(k)}\to \bD^k)
\end{equation}
so that $\bigcup_{k\ll 0} \bD^k=\C$ and
\[
\Dom \bbf_- = \bigcup_{k \ll 0} \bW_-^{(k)},\sp\sp  \Dom \bbf_+ = \bigcup_{k \ll 0} \bW_+^{(k)}.
\]
It follows from the construction that $\Dom \bbf_-$ and $\Dom \bbf_+$ are simply connected. 

 Let 
\[ \bF^{\#}_n=\left(\bbf_{n,-}^\#, \bbf_{n,+}^\# \right)\coloneqq A_\str^{n}\circ \bF_n\circ A_\str^{-n}\]
be the rescaled version of $\bF_n$, see Figure~\ref{Fig:Sf1dash:2}. Then $\bF^{\#}_{n}$ is an iteration of $\bF^{\#}_{n-1}$: writing the antirenormalization matrix as
\begin{equation}
\label{eq:def:M}
\M\coloneqq \left(\begin{matrix}
         m_{1,1} &m_{1,2} \\ m_{2,1} &m_{2,2}
        \end{matrix}\right)
\end{equation}
we obtain 
\begin{equation}
\label{eq:iter of max prep}
\left\{\begin{array}{c}
\bbf^\#_{n,-}=\left( \bbf^\#_{n-1,-}\right)^{m_{1,1}}\circ \left( \bbf^\#_{n-1,+}\right)^{m_{1,2}}\\
\bbf^\#_{n,+}=\left( \bbf^\#_{n-1,-}\right)^{m_{2,1}}\circ \left( \bbf^\#_{n-1,+}\right)^{m_{2,2}}
\end{array}\right..
\end{equation}
In particular, $\bF=\bF_0^\#$ is an iteration of $\bF_n^\#$.

\subsubsection{Renormalization triangulation}
\label{sss:ren trian:pacmen}
(See~\cite[\S4.1]{DLS} and Figure~\ref{Fig:ren triang:f0}.)
For a pacman $f_0\in \BB$ its \emph{$0$th renormalization triangulation} $\bDelta_0(f_0)$ consists of two closed triangles $\Delta_0(0,f_0)$ and $\Delta_0(1,f_0)$ that are the closures of the connected components of $U_0\setminus (\gamma_0\cup \gamma_1)$. For $n> 0$, the \emph{$n$th renormalization triangulation} $\bDelta_n(f_0)$ consists of all the triangles obtained by spreading around $\Delta_n(0,f_0)$ and $\Delta_n(1,f_0)$ (compare with Figure~\ref{Fg:Prepacman:OrbitOfUi}), where the latter triangles are the embeddings of $\Delta_0(0,f_n)$ and $\Delta_0(1,f_n)$ to the dynamical plane of $f_0$. We also say that $\bDelta_n(f_0)$ is the \emph{full lift} of $\bDelta_0(f_n)$. {\cite[Theorem 4.6]{DLS}} asserts that $ \bDelta_m(f)$ approximates $ \overline Z_\str$ dynamically and geometrically.

More precisely, suppose $f_0\in \BB$ is renormalizable $m\ge 1$ times. We write
\[\phi_k\coloneqq \psi_k^{-1}.\]
The map 
\[ \Phi_{m} \coloneqq \phi_{1}\circ \phi_{2}\circ \dots\circ \phi_{m}\] 
admits a conformal extension from a neighborhood of $c_1(f_m)$ (where $\Phi_{m}$ is defined canonically) to $V\setminus \gamma_1$. The map $\Phi_{m}\colon V\setminus \gamma_1 \to V$ embeds the prepacman $F_m$ to the dynamical plane of $f_0$; we denote the embedding by 
\begin{align*}
 F^{(0)}_m= &\left(f^{(0)}_{m,\pm }\colon U^{(0)}_{m,\pm}\to S_{m}^{(0)}\right)\\=& \left(f^{\aa_{m}}_{0}\colon U^{(0)}_{m,-}\to S_{m}^{(0)}, \sp f^{\bb_{m}}_{0}\colon U^{(0)}_{m,+}\to S_{m}^{(0)}\right),
\end{align*}
where the numbers $\aa_{m}, \bb_{m}$ are the renormalization return times satisfying
\begin{equation}
\left(
         \aa_m, \bb_m
        \right) = \left(
         \aa , \bb
       \right) \mathbb M^{m-1} =\left(
         1 ,1
        \right)\mathbb M^{m}, \sp\sp\sp \aa=\aa_1,\sp\bb=\bb_1.
\end{equation}

\begin{figure}[tp!]
\centering{\begin{tikzpicture}

\draw  (1,-0.1) ellipse (4 and 2.5);

\begin{scope}[shift={(2,0)},rotate=45,scale =0.3]

\draw (0,0) arc (0:270:1);
\coordinate (A1) at (0,0);
\coordinate (B1) at (-1,-1);
\end{scope}
\draw (A1) -- (3,1);
\draw (B1) -- (3,-1);
\draw plot[smooth,tension=0.5] 
coordinates{ (3,1) (2,1.3) (1,1.35) (0,1.2)(-1,0.7) (-1.5, 0.2) (-1.6,-0.2)
 (-1.5, -0.5) (-1,-1) (0,-1.5) (1,-1.65) (2,-1.6) (3,-1)};

 \draw (0,-2.5)-- (0.1,-0.2) --(-1,0.7);
 \node[right] at(0,-2.2) {$\gamma_1$};
 \node[right] at(-0.6,0.4) {$\gamma_0$};

 \draw[blue] (0.9,-0.2) node {$\Delta_0(0)$};
  \draw[red] (-0.9,-0.2) node {$\Delta_0(1)$};

\begin{scope}[shift={(13,0)},scale =0.85]  
  
  \begin{scope}[shift={(-6.5,0)}, scale =1.25]
\begin{scope}[shift={(2,0)},rotate=45,scale =0.3]
\draw (0,0) arc (0:270:1);
\coordinate (A2) at (0,0);
\coordinate (B2) at (-1,-1);
\end{scope}
\draw (A2) -- (3,1);
\draw (B2) -- (3,-1);
\draw plot[smooth,tension=0.5] 
coordinates{ (3,1) (2,1.3) (1,1.35) (0,1.2)(-1,0.7) (-1.5, 0.2) (-1.6,-0.2)
 (-1.5, -0.5) (-1,-1) (0,-1.5) (1,-1.65) (2,-1.6) (3,-1)};
 \coordinate (C) at (0,-1.5); 
\coordinate (D) at (-1,0.7);
  \end{scope}

 \begin{scope}[shift={(-6.5,0.3)},rotate=120,scale =0.08]
\draw (0,0) arc (0:270:1);
\coordinate (A3) at (3,-1);
\draw (0,0)--(A3);
\coordinate (B3) at (1,-3.8);
\draw (B3) --(-1,-1);
\end{scope} 
\draw (A3)-- (-6.8,0.5) -- (-6.4,-0.4)--(-6,0.5)--(B3); 

\begin{scope}[ scale =0.7]
\draw[red] (-9.87, -0) node {$1$}
(-9., -1.3)  node {$3$}; 
\draw[blue] (-9.14, -0.1) node {$0$}
(-9.7, -1.2) node {$2$}
(-8.47, -0.6) node {$4$};
\end{scope}

 \begin{scope}[shift={(-6.95,-0.65)},rotate=250,scale =0.08]
\draw (0,0) arc (0:270:1);
\coordinate (A4) at (3,-1);
\draw (0,0)--(A4);
\coordinate (B4) at (1,-3.8);
\draw (B4) --(-1,-1);
\end{scope} 
\draw (B4)-- (-7.5,-0.3)-- (-6.4,-0.4)--(-6.8,-1.2)--(A4);
\draw(-7.5,-0.3) ..controls (-7.2,0.2)..(-6.8,0.5);

\draw (-6.8,-1.2) .. controls (-6,-1.3) .. (-5.7,-1.2);

\begin{scope}[shift={(-5.6,-0.2)},rotate=55,scale =0.1]
\draw (0,0) arc (0:270:1);
\coordinate (A4) at (3,-1);
\draw (0,0)--(A4);
\coordinate (B4) at (1,-3.8);
\draw (B4) --(-1,-1);
\end{scope} 

\draw (B4) --(-5.7,-1.2)-- (-6.4,-0.4);
\draw (A4) -- (-5.9,0.7) --(-6,0.5);

\draw (-7.8,-0.3) node {$\bDelta_1$};

\end{scope}
\end{tikzpicture}}
\caption{Left: the triangulation $\bDelta_0=\Delta_0(0)\cup \Delta_0(1)$. The closed triangles $\Delta_0(0)$ and $\Delta_0(1)$ are the closures of the connected components of $U\setminus (\gamma_0\cup \gamma_1)$. Right: the triangulation $\bDelta_1(f_0)$ is obtained by spreading around $\Delta_1(0,f_0)$ and $\Delta_1(1,f_0)$ -- the embeddings of $\Delta_0(0,f_1)$ and $ \Delta_0(1,f_1)$ into the dynamical plane of $f_0$.}
\label{Fig:ren triang:f0}
\end{figure}

Write $\Delta_m(0,f_0)\coloneqq \overline U^{(0)}_{m,+}$ and $\Delta_m(1,f_0)\coloneqq \overline U^{(0)}_{m,-}$. Then $\bDelta_m$ consists of triangles 
\[\big\{f_0^i\big(\Delta_m(0,f_0) \big)\mid i\in\{0,1,\dots, \aa_m-1\}\big\}\cup \big\{f_0^i\big(\Delta_m(1,f_0) \big)\mid i\in\{0,1,\dots, \bb_m-1\}\big\}.\] 
We enumerate counterclockwise these triangles as $\Delta_m(i)$ with $i\in \{0,1,\dots, \qq_m-1\}$. By construction, $\Delta_m(0,1)\coloneqq \Delta_m(0)\cup \Delta_m(1)$ contains the critical value $c_1$, while $\Delta_m(-\pp_m,-\pp_m+1)$ contains the critical point, where $\pp_m/\qq_m$ is the \emph{combinatorial rotation number} of $f\mid \bDelta_m $.  

We have:
\begin{equation} 
\label{eq:ImprOfDomain}
 \bDelta_n(f)\Subset \bDelta_{n-1}(f)\Subset \dots\Subset \bDelta_1(f) \Subset \bDelta_1(f)\cup S_f \Subset \bDelta_0(f),
\end{equation}
and renormalization sector $S_f$ is disjoint from $\gamma_1$ away from a small neighborhood of $\alpha$.

\subsubsection{Walls $\bPi_n$ of $\bDelta_n$}
\label{sss:walls bPi} Consider the dynamical plane of a pacman $f$. By a \emph{univalent $N$-wall} we mean a closed annulus $A$ surrounding an open disk $O$ containing $\alpha$, such that $f\mid O\cup A$ is univalent, and for every $z\in O$, we have $(f\mid A\cup O)^{\pm k}(z)\in A\cup O$ for $k\in\{0,1,\dots,  N\}$. In other words, it takes at least $N$ iterates for a point in $O$ to cross $A$. An \emph{$N$-wall} is an annulus $A$ surrounding an open disk $O$ containing $\alpha$ such that $A$ contains a univalent $N$-wall. A wall $A$ \emph{respects} $\gamma_0,\gamma_1$ if $A\cap \gamma_0$ and $A\cap \gamma_1$ are two arcs. 

Let us define the wall $\bPi_n(f)$ of $\bDelta_n$; the wall approximates $\partial Z_\str$ just like $\bDelta_n(f)$ approximates $\overline Z_\str$. In the dynamical plane of $f_\str$ consider its Siegel disk $Z_\str$. It is foliated by equipotentials parametrized by their heights
ranging from 0 (the height of $\alpha$) to $1$ (the height of $\partial Z_\str$). Fix an $r\in (0,1)$ and consider the open subdisk $Z^r$ of $Z_\str$ bounded by the equipotential at height $r$.  Consider next $f\in \WW^u$ close to $f_\str$. Then $\gamma_0(f)$ and $\gamma_1(f)$ still intersect $\partial Z^r$ at single points. The \emph{wall $\bPi_0(f)$ of $\bDelta_0(f)$} is the closed annulus $\bDelta_0(f)\setminus Z^r$ consisting of rectangles \[\Pi_0(0)= \bPi_0 \cap \Delta_0(0)\text{ and  }\Pi_0(1)= \bPi_0 \cap \Delta_0(1).\]

Suppose that $f_n$ with $n>0$ is sufficiently close to $f_\str$ so that $\bPi_0(f_n)$ is defined. Let $\Pi_n(0,f_0)$ and $\Pi_n(1,f_0)$ be the embedding of the rectangles  $\Pi_0(0,f_n)$ and $\Pi_0(1,f_n)$ into the dynamical plane of $f_0$, see~\cite[Lemma 4.2, (2)]{DLS}. The \emph{wall $\bPi_n=\bPi_n(f)$ of $\bDelta_n(f)$} is obtained by spreading around $\Pi_n(0,f_0)$ and $\Pi_n(1,f_0)$. The wall $\bPi_n$ consists of rectangles $\Pi_n(i)=\bPi_n\cap \Delta_n(i)$. We have $\bPi_n(f_0)\Subset \bPi_{n-1}(f_0)$ is an annulus approximating  $\partial Z_\str$ and, moreover, the dynamics $f\mid (\bDelta_n\setminus \bPi_n)$ is univalent.

\subsubsection{Siegel triangulation}
\label{sss:SiegTriang} Consider a pacman $f$ close to $f_\str$ and let $\gamma_1^\new, \gamma^\new_0$ be a new ``dividing'' pair (similar to~\S\ref{ss:PartiHomeo}) that is coincide with $\gamma_0,\gamma_1$ on $\bPi_0$; i.e.:
\[\gamma^\new_0=f(\gamma^\new_1), \sp \gamma_0^\new\cap  \gamma_1^\new=\{\alpha\},\sp \gamma_i^\new\cap \bPi_0=\gamma_i.\]
Let $\bDelta^\new_0$ be the associated new triangulation consisting of the closures of $U_{f}\setminus {(\gamma_0^\new\cup \gamma_1^\new)}$. Such $\bDelta^\new_0$ will appear in Theorem~\ref{thm:quot of fund domain} as the projection of a fundamental domain and will be used in Theorem~\ref{thm:proj of val flow} to select $\gamma_0^\new\setminus \{\alpha\} , \gamma^\new _1\setminus \{\alpha\}$ away from a valuable flower. Lemma 4.3 from \cite{DLS} (see also Proposition~\ref{prop:quot of fund domain}) assets that if $f=\RR^n(f_{-n})$, then $\bDelta^\new_0(f)$ has the full lift $\bDelta_n^\new(f_{-n})$ with $\bPi_n^\new(f_{-n})=\bPi_n(f_{-n})$ just like $\bDelta_n$ is the full lift of $\bDelta_0$. Moreover, the assumption ``$\gamma_i^\new\cap \bPi_0=\gamma_i$'' can be relaxed into ``$\gamma_i^\new\cap \bPi_0$ is sufficiently close to $\gamma_i$.''

We will also consider triangulations that are small perturbations of $\bDelta^\new_n$. A \emph{Siegel triangulation} $\bDelta$ is a triangulated neighborhood of $\alpha$ consisting of closed triangles, each has a vertex at $\alpha$, such that
\begin{itemize}
\item triangles of $\bDelta$ are $\{\Delta(i)\}_{i\in \{0,\dots \qq-1\}}$ enumerated counterclockwise around $\alpha$ so that $\Delta(i)$ intersects only $\Delta(i-1)$ (on the right) and $\Delta(i+1)$ (on the left); $\Delta(i)$ and $\Delta(i+j)$ are disjoint away from $\alpha$ for $j\not\in \{-1,0,1\}$;
\item there is a $\pp>0$ such that $f$ maps $\Delta(i)$ to $\Delta(i+\pp)$ for all $i\not\in \{-\pp, -\pp+1\}$;
\item $\bDelta$ has a distinguished $2$-wall $\bPi$ enclosing $\alpha$ and containing $\partial \bDelta$ such that each $\Pi(i)\coloneqq \bPi\cap \Delta(i)$ is connected and $f$ maps $\Pi(i)$ to $\Pi(i+\pp)$ for all $i\not\in \{-\pp, -\pp+1\}$; and 
\item  $\bPi$ contains a univalent $2$-wall $\wall$ such that each $\wall(i)\coloneqq \wall\cap \Pi(i)$ is connected and $f$ maps $\wall(i)$ to $\wall(i+\pp)$ for all $i\not\in \{-\pp, -\pp+1\}$.
 \end{itemize}

We say that $\bPi$ \emph{approximates} $\partial Z_\str$ if $\partial Z_\str$ is a concatenation of short arcs $J_0J_1\dots J_{\qq-1}$ such that $\Pi(i)$ and $J_i$ are close in the Hausdorff topology.

\begin{lem}[\cite{DLS}*{Lemma 4.4}]
\label{lem:SiegTriangLifting}
Let $f\in \BB$ be a pacman such that all $f,\RR f,\dots, \RR^n f$ are in  a small neighborhood of $f_\str$. Let $\bDelta(\RR^n f)$ be a Siegel triangulation in the dynamical plane of $\RR^n f$ such that $\bPi(\RR^n f)$ approximates $\partial Z_\str$. Then $\bDelta(\RR^n f)$ has a full lift $\bDelta(f)$ which is again a Siegel triangulation. Moreover, $\bPi(f)$ also approximates $\partial Z_\str$.
\end{lem}
\noindent More generally, $\bDelta(f)$ can be lifted under a renormalization $\RR_\Sieg\colon\AA\to \BB$ defined on near-Siegel maps (see \S\ref{sss:from SM to SP}) assuming that $\bPi(f)$ sufficiently approximates $\partial Z_\str$. In the proof, we first lift $\bPi$ (it has a lift because  $\bPi$  approximates $\partial Z_{f_\str}$), and then extend the lift into $\bDelta\setminus \bPi$ (see Proposition~\ref{prop:quot of fund domain}). If $f_0\in \WW^u$ has a Siegel triangulation $\bDelta(f_0)$, then $\bDelta( f_0)$ has a full lift $\bDelta(f_n)$ converging to $\overline Z_\str$ as $n\to -\infty$.

\subsubsection{Renormalization change of variables near $c_0$}
\label{sss:CofV near c_0}
We will consider in~\S\ref{ss:PositArea} the renormalization change of variables $\psi_0$ defined on a neighborhood of the critical point $c_0$; i.e.~$\psi_0(c_0(f))=c_0(f_1)$ and $\psi_0$ projects the first return of $f$ to $f_1=\RR f$. If $\psi\colon S\to V$ is the renormalization change of variables near the critical value as above, then $\psi_0$ is uniquely characterized by
 \[ \psi \circ f = f_1 \circ \psi_0.\] 
We have $\Dom \psi_0=f^{-1} (S)$ and $\Im \psi_0= \Dom f_1$.

\subsubsection{Parabolic pacmen}(See~\cite{DLS}*{\S 6})
\label{sss:par pacmen}
In a small neighborhood of $f_\str$ consider a parabolic pacman $f_\rr\in \WW^u$ with rotation number $\rr=\pp/\qq$ close to $\theta_\str$. We denote by $H_0$ a small attracting parabolic flower around $\alpha$. Petals in $H_0$ are enumerated counterclockwise as $H^i_0$ with $i\in \{0,1,\dots, \qq-1\}$

We assume that $H_0$ is small
enough so that $H_0 \subset V\setminus \gamma_1$, possibly up to a slight rotation of $\gamma_1$. Therefore, the flower $H_0$ lifts to the dynamical plane of $\bF_\rr$ via the identification $V\setminus \gamma_1\simeq \intr \bS$; we denote by $\bH_0$ the lift. The \emph{global attracting basin} $\bH$ of $\bF_\rr$ is the full orbit of $\bH_0$. There are Fatou coordinates in $\bH_0$; globalizing the Fatou coordinates we obtain that $0\in \bH$.
 \cite[Proposition~6.5]{DLS} parametrizes periodic components of $\bH$ as $\bH^i$ from left-to-right with $0\in \bH^0$, see Figure~\ref{Fig:ss:geom pict}. Every petal $\bH^i$ is an open topological disk in $\widehat \C$. By re-enumerating, we assume that the lift $\bH^0_0$ of $H_0^0$ is contained in $\bH^0$. Note that  $\bH^0_0$  is disjoint from $\partial \bH^0$. The actions of $\bbf^{\#}_{n,\pm}$ on $(\bH^i)_{i\in \Z}$ are given by (see~\cite{DLS}*{(6.7)})
\begin{equation}
\label{eq:ActionOnHp}
 \bbf^{\#}_{n,-}(\bH^i) =  \bH^{i-\pp_n} \text{ and } \bbf^{\#}_{n,+}(\bH^i) = \bH^{i+\qq_n-\pp_n},
\end{equation}
 where $\pp_n/\qq_n$ is the rotation number of $f_n.$

\subsubsection{The molecule map} (See~\cite{DLS}*{Appendix C}.)
\label{ss:MolecMap} Consider a primary $\pp/ \qq$-limb $\LL_{ \pp/\qq}$ of the Mandelbrot set. In the dynamical plane of $p\in \LL_{ \pp/\qq}$ there
are exactly $\qq$ external rays landing at the $\alpha$-fixed point; these rays are permuted as $\pp/\qq$. We can apply the Branner--Douady surgery \cite{BD} and delete the smallest sector between external rays $\gamma,p(\gamma)$ landing at $\alpha$ (as with the prime renormalization of a rotation); the result is a map $\mRRc(p)\in \LL_{\cRRc(\pp/\qq)}$, where $\cRRc(\pp/\qq)$ is defined in~\eqref{eq:R_prm}. This defines a partial continuous map $\mRRc\colon \LL_{ \pp/\qq}\dashrightarrow \LL_{ \cRRc(\pp/\qq)}$ whose inverse is an embedding of $\LL_{ \cRRc(\pp/\qq)}$ into $\LL_{ \pp/\qq}$. If $\pp/\qq=1/2$, then $\mRRc\colon \LL_{ 1/2} \dashrightarrow \LL_{0/1}=\Mandel$ is the canonical homeomorphism between $\Mandel_{1/2}$ and $\Mandel$.

\begin{wrapfigure}[12]{r}{6cm}
  \vspace{-5mm}\begin{tikzpicture}
  \node at (0,0){\includegraphics[scale=0.5]{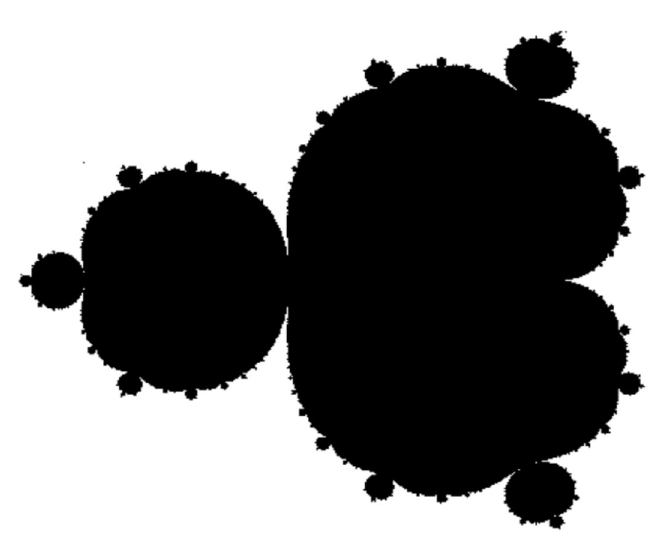}};
 \draw[white] node at (1,0) {$z(z+1)^2$}; 
      
 \end{tikzpicture}
\end{wrapfigure}

The \emph{molecule map} is obtained by taking all $\mRRc\colon \LL_{ \pp/\qq}\dashrightarrow \LL_{ \cRRc(\pp/\qq)}$ and extending continuously $\mRRc$ to the boundary of the main hyperbolic component $\partial \HH$. The molecule map is a $3$-to-$1$ partial map $\mRRc\colon \Mandel \dashrightarrow \Mandel$; its domain depends on the choices of the rays $\gamma$. However, on the boundary of main molecule, $\mRR_c$ is defined canonically and is semiconjugate to $q(z)=z(z+1)^2$ restricted to its Julia set, see the figure. The \emph{Molecule Conjecture} asserts that this is in fact a conjugacy and, moreover, there is a hyperbolic renormalization operator associated with the molecule map.

\section{Dynamics of maximal prepacmen}
\label{s:max prep}
Recall from \S\ref{sss:max prepacmen} that every pacman $f\in \WW^u$ has the associated maximal prepacmen $\bF=(\bbf_-,\bbf_+)$ consisting of two $\sigma$-proper maps. We define $\UnstLoc\simeq \WW^u$ to be the space of maximal prepacmen arising this way.

The renormalization operator on $\UnstLoc$ is an iteration and rescaling by $A_\str$: there are $m_{1,1},m_{1,2},$ $m_{2,1},m_{2,2}\ge 1$ such that (see~\eqref{eq:iter of max prep})
\begin{equation}
\label{eq:ff_n:ff_n-1}
\left\{\begin{array}{c}
\bbf^\#_{n,-}=\left( \bbf^\#_{n-1,-}\right)^{m_{1,1}}\circ \left( \bbf^\#_{n-1,+}\right)^{m_{1,2}}\\
\bbf^\#_{n,+}=\left( \bbf^\#_{n-1,-}\right)^{m_{2,1}}\circ \left( \bbf^\#_{n-1,+}\right)^{m_{2,2}}
\end{array}\right..
\end{equation}
In particular, $\bF=\bF_0^\#$ is an iteration of $\bF_n^\#$. Recall from~\eqref{eq:def:M} that the antirenormalization matrix is defined by
\begin{equation}
\label{eq:def:M:2}
\M\coloneqq \left(\begin{matrix}
         m_{1,1} &m_{1,2} \\ m_{2,1} &m_{2,2}
        \end{matrix}\right).
\end{equation}
Clearly, if
\begin{equation}
\label{eq:relat on triples}
\left(\begin{matrix}
         c,\ d
        \end{matrix}\right) = \left(\begin{matrix}
          a,\ b
        \end{matrix}\right)\M, 
\end{equation} 
then 
\begin{equation}
\label{eq:max prepacm ren}
\left( \bbf^\#_{n,-}\right)^a\circ \left( \bbf^\#_{n,+}\right)^b= \left( \bbf^\#_{n-1,-}\right)^{c}\circ \left( \bbf^\#_{n-1,+}\right)^{d}.
\end{equation}
We denote by $\tF$ the tower $(\bF_n)_{n\le 0}$ and we denote by $\tF^{\#}$ the rescaled tower $(\bF^{\#}_n)_{n\le 0}$.

We can now globalize $\UnstLoc$ as follows. The operator $\RR\colon \UnstLoc \dashrightarrow \UnstLoc$ is conjugate, say via $\bbh$, to $v\mapsto \lambda_\str v$ in a neighborhood of $\bF_\str$ so that $\bbh(\bF_\str)=0$. Using~\eqref{eq:max prepacm ren}, we inductively define $\bF_n^\#$ for all $n\ge 0$ and all $\bF_0$ in a neighborhood of $\bF_\str$. Note that the domain of $\bF_n^\#$  needs not be connected. Define $\bF_n=\RR^n \bF$ to be $A_\str^{-n}\circ \bF_n^\# \circ A_\str^{n} $ and set $\bbh(\bF_n)\coloneqq \lambda_\str^n \bbh(\bF_0)$. We enlarge $\UnstLoc$ by adding all new maximal prepacmen $\{\bF_n\}$. Then $\bbh$ parameterizes $\Unst$ by $\C$; i.e.~the new operator $\RR\colon \Unst\to\Unst $ is globally conjugate to $v\mapsto \lambda_\str v\colon \C\to \C$.

\subsection{$\Unst$ as a geometric limit of the quadratic slice}
\label{ss:geom pict}
The space $\Unst$ naturally arises as the set of limits of rescaled iterations of quadratic polynomials. Let us write $c_\str\coloneqq c(\theta_\str)$ and let us change the normalization of $p_c(z)\coloneqq z^2+c$ by putting the critical value and the parameter $c_\str$ at $0$: 
\[g_c\coloneqq T^{-1}_{c+c_\str} \circ p_{c+c_\str}\circ T_{c+c_\str}.\]

\begin{figure} 
\centering{\begin{tikzpicture}
\node at (0,0){\includegraphics[scale=0.6]{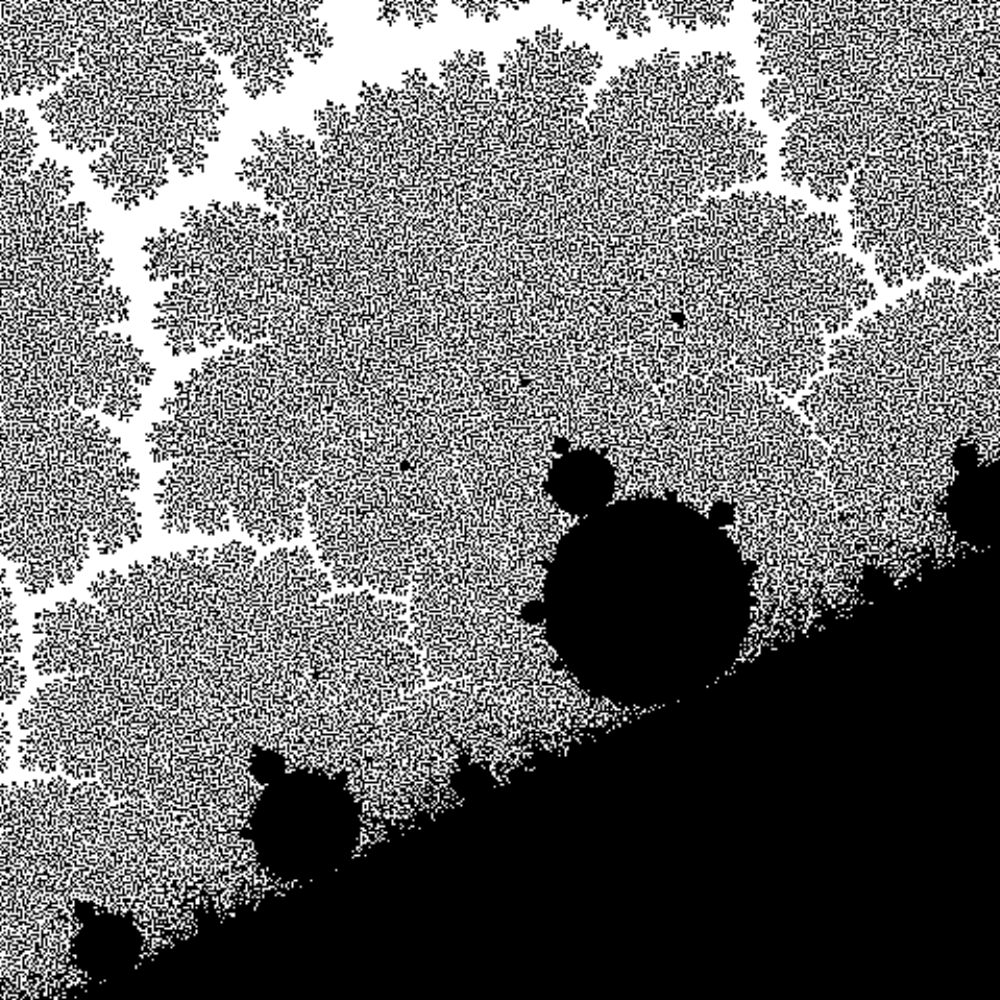}};
  \node at (0,-10.5){\includegraphics[scale=0.6]{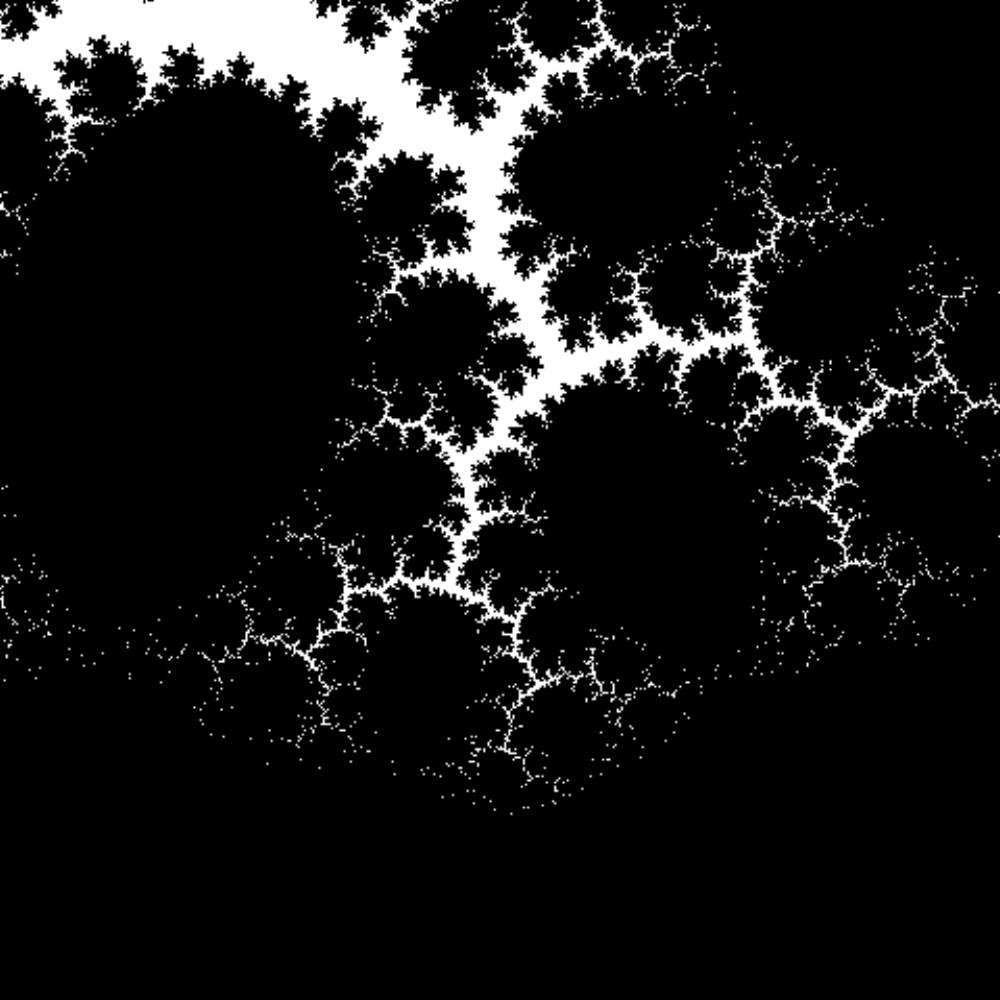}};
 
 \end{tikzpicture}}
\caption{Approximations of $\Unst$ (top) and of the maximal prepacmen $\bF_\str$ (bottom), see~\S\ref{ss:geom pict}. Note that the white set between limbs disappears in the limit.}
\vspace{128in} \label{Fig:unst man+max sieg}
\end{figure}

Using the hyperbolicity of $\RR$, it is possible to show that the  limit
\[\bF_c =\lim_{n\to +\infty} A^{-n}_{\str}\circ \left(g_{c \lambda^{-n}_\str}^{\aa_n},\sp g_{c \lambda^{-n}_\str}^{\bb_n}\right)\circ A^n_{\str}\]
exists and defines the parameterization of $\Unst$ by $c$ such that $\RR \bF_{c}=\bF_{c\lambda_\str}$. The existence of the limit for $c=0$ is shown in~\cite[Theorem 8.1, Claim 7]{McM2}.

Therefore, a zoomed picture of the Mandelbrot set near $f_{c_\str}$ gives a good approximation of $\Unst$, see Figure~\ref{Fig:unst man+max sieg}. Similarly, a zoomed picture of $f_{c_\str+ \lambda_\str^{-n} c}$ near the critical value gives a good approximation of $\bF_c$. This will help us to illustrate by pictures different constructions in the parameter and dynamical planes. 

Let us denote by $\HH$ the \emph{main hyperbolic component} of $\Unst$: the set of maximal prepacmen $\bF\in \Unst$ such that the $\alpha$-fixed point of $f_n$ is attracting for $n\ll0$. Then $\HH$ is the rescaled limit of the main hyperbolic component of the Mandelbrot set and $\partial \HH$ is a straight line passing through $0$.

\subsection{Power-triples} (Compare with~\S\ref{sss:power-triples}.)
\label{ss:PowerTriples}
A \emph{power-triple} is a triple $(n,a,b)\in \Z\times \Z^2_{\ge 0}$. Given a power-triple $P=(n,a,b)$, we write
\[\bF^{P}\coloneqq \left(\bbf^\#_{n,-}\right)^a\circ \left(\bbf^\#_{n,+}\right)^b;\]
$\bF^{P}$ is a $\sigma$-proper map; however its domain needs not be connected.

If $a,b,c,d$ satisfy~\eqref{eq:relat on triples}, then we say that $(n,a,b)$ and $(n-1,c,d)$ are \emph{equivalent power-triples}. This generates the equivalence relation ``$\simeq$'' on the set of power-triples; we will usually consider power-triples up to this natural equivalence relation. By construction, $\bF^P$ depends only on the equivalence class of $P$.

Let $P,Q$ be two power-triples. For every $n\ll0$, there are $a,b,c,d$ such that $P\simeq (n,a,b) $ and $Q\simeq (n,c,d) $. We set 
\[P+Q\simeq (n,a+c,b+d).\]
Then \[\bF^{P+Q}= \bF^{P}\circ \bF^{Q}.\]
We denote by $\PT$ the commutative semigroup consisting of the equivalence classes of power-triples with the operation ``$+$.'' We denote by $0\simeq (n,0,0)$ the \emph{zero power-triple}:  $\bF^0=\id$.

 For $P,Q\in \PT$ we say that $P\ge Q$ if for every sufficiently big $n\ll 0$ the following holds. Write $P\simeq (n,a,b)$ and $Q\simeq (n,c,d)$. Then $a\ge c$ and $b\ge d$. Clearly, $\ge$ is a well defined order on $\PT$. The next lemma is a consequence of Lemma~\ref{lem:PT:rot}:

\begin{lem}
\label{lem:power-triples:order}
For every $P,Q\in \PT$ either $P\ge Q$ or $P\le Q$ holds. There is an order-preserving embedding 
\[\iota\colon (\PT,+,\ge)\hookrightarrow (\R_{\ge 0},+,\ge)\]
and there is a $\tt>1$ such that 
\[
\iota(n-1,a,b)= \iota(n,a,b)/\tt\]
 for all $a,b\ge0$.\qed
\end{lem}

From now on, we fix the order-preserving embedding $\PT\subset \R_{\ge0}$ as in Lemma~\ref{lem:power-triples:order} and we write $(n,a,b)/\tt \coloneqq (n-1,a,b).$

 We denote by $\bF^{\ge 0}$ the \emph{cascade} $(\bF^P)_{P\in \PT}$.  It follows from ~\eqref{eq:ff_n:ff_n-1} that
\begin{equation}
\label{eq:ren:F:F_n}
 \bF_0^P=\left(\bF^{\#}_{-n} \right)^{\tt^n P}.
\end{equation}
If $\bF=\bF_\str$, then~\eqref{eq:ren:F:F_n} takes form
 \begin{equation}
\label{eq:ren:F:F_n^P}
 \bF_\str^P=A_\str^{-n} \circ \left(\bF_{\str}^{\tt^n P} \right) \circ A_\str^{n}.
\end{equation}

\subsection{Renormalization triangulations}
\label{ss:RenTriang}
Recall from \S\ref{sss:ren trian:pacmen} that the triangulation $\bDelta_{-n}(f_n)$ is the full lift of $\bDelta_0(f_0)$. For $\bF$ close to $\bF_\str$ we define the renormalization triangulation $\bDelta_0(\bF_0)$ to be the full lift of $\bDelta_0(f_0)$ to the dynamical plane of $\bF_\str$. More precisely, consider
\begin{equation}
\label{eq:Max Prep U pm:2}
\bbf_-\colon \bU_-\to \bS,\sp \bbf_+\colon \bU_+\to \bS,
\end{equation}
see~\eqref{eq:Max Prep U pm}, and note that this realizes the first return map of points in $\overline \bU_\pm$ back to $\bS$ under the cascade $\bF^{\ge0}$ because~\eqref{eq:Max Prep U pm:2} is the first return map under $\bbf^{\#}_{n,\pm }\colon \bU^\#_{n,\pm}\to \bS^\#_n$ for all $n\le 0$. Then $\bDelta=\bDelta_0(\bF)$ is obtained by spreading around $\Delta_0(0)\coloneqq \overline \bU_+$ and $\Delta_0(1)\coloneqq \overline \bU_-$;~i.e.~$\bDelta_0$ consists of triangles \[ 
\left\{\bF^P(\Delta_0(0))\mid P< (0,0,1) \right\} \cup
\left\{\bF^P(\Delta_0(1))\mid P< (0,1,0) \right\}.
\] We enumerate triangles in $\bDelta_0(\bF)$ as $(\Delta_0(i,\bF))_{i\in \Z}$ from left-to-right so that 
\[\Delta_0(0,\bF)=\overline \bU_+\sp\text{ and }\sp \Delta_0(1,\bF)= \overline\bU_-,\]
see Figure~\ref{Fig:ren triang:bF}. The triangulation $\bDelta_0(\bF)$ depends holomorphically on $\bF$.

\begin{lem}
Every $\Delta_0(i)$ is a triangle in $\hC$ with a vertex at $\infty$. For every compact subset $X\subset \C$, there are at most finitely many triangles in $\bDelta_0$ intersecting $X$.  
\end{lem}
\begin{proof}
Since $\bDelta_0(\bF)$ depends holomorphically on $\bF$ in a small neighborhood of $\bF_\str$, it is sufficient to prove the lemma for $\bF_\str$.

For $f_\str$, the map $h_\str\coloneqq h_0$ (see~\eqref{eq:h_0:defn}) is the linearizer of $\psi_\str$ (the renormalization change of variables associated with $f_\str=\RR f_\str$). This implies that $\bS_\str,\bU_{\str,\pm}, $ and all $\Delta(i,\bF_\str)$ are triangles of $\wC$ with a vertex at $\infty$, see Figure~\ref{Fig:Sf1dash}.

Since triangles of $\bDelta_n(f_\str)$ intersect $Z_\str$ along its internal rays, we can slightly rotate $\gamma_1$ so that the new $\gamma_1^\new$ intersects $\bDelta_n(f_\str)$ along the boundary of a certain triangle in $\bDelta_n(f_\str)$. Let us cut $\bDelta_n(f_\str)$ along $\gamma_1^\new$ and embed using $A_\str^{-n} \circ h_{\str}$ the obtained triangulation to the dynamical plane of $\bF_\str$; we denote by $\bDelta^{(n)}_0(\bF_\str)$ the embedding. Let us also denote by $\bS^\#_{\new, -n}$ the closure of $A_\str^{-n} \circ h_{\str}(V\setminus \gamma_1^\new)$, compare with Figure~\ref{Fig:Sf1dash:2}.

Then $\bDelta^{(n)}_0(\bF_\str)\subset \bDelta^{(n+1)}_0(\bF_\str)$, and the union of $\bDelta^{(n)}_0(\bF_\str)$ is $\bDelta_0(\bF_\str)$. Moreover, \[\bDelta_0^{(n)}= \bS^\#_{\new,-n}\cap \bDelta_0.\] Since $\bS^\#_{\new,-n}$ contains a disk around $0$ with a big radius for $n\ll 0$, we have $X\subset \bS^\#_{\new,-n}$ for $n\ll 0$.
\end{proof}

The \emph{$n$th renormalization triangulation $\bDelta_n(\bF)$} is $\bDelta_0(\bF^\#_n)=A_\str^{n}(\bDelta_0(\bF_n))$.

\emph{From now one we assume} that $\WW^u$ is chosen in a sufficiently small neighborhood of $f_\str$ so that every $f\in \RR^{-n}(\WW^u)$ has a renormalization triangulation $\bDelta_n(f)$ and every $\bF\in \RR^{-n}(\UnstLoc)$ has a renormalization triangulation $\bDelta_n(\bF)$ for $n\ge0$.

 For $\bF\in \Unst$, the triangulation $\bDelta_n(\bF)$ is defined for all sufficiently big $n\ll0$. We have
\begin{equation}
\label{eq:Delta_m appr C}
\bigcup _{m\ll 0}\Delta_m(j,\bF)=\C\sp \sp \text{if }\Delta_0(j,\bF_\str)\ni 0. 
\end{equation}
Moreover,
 \begin{equation}
\label{eq:S in Delta-1}
  \bS\setminus \Delta_0(0,1)\subset \Delta_{-1}(0,1).
 \end{equation}
because $S_f\setminus \bDelta_1(f)\Subset V\setminus \gamma_1$, see~\eqref{eq:ImprOfDomain}. 

Given a set $K\subset U_f$ in the dynamical plane of $f\in \WW^u$, the \emph{full lift} $\bK$ of $K$ to the dynamical plane of $\bF$ is defined as follows. Write \[K_0\coloneqq K\cap \Delta_0(0,f),\sp\sp \text{ and }
K_1\coloneqq K\cap \Delta_0(1,f).\]
(Recall that the triangles of $\bDelta_0(f)$ are closed thus $K=K_-\cup K_+$). Let $\bK_0$ and $\bK_1$ be the embeddings of $K_0$ and $K_1$ to the dynamical plane of $\bF$ via $\bS\simeq V\setminus \gamma_1$. Then
\[\bK\coloneqq  \bigcup_{0\le P<(0,0,1)} \bF^P(\bK_0)\bigcup_{0\le P< (0,1,0)} \bF^P(\bK_1).\]

In the dynamical plane of $\bF_\str$, we define its \emph{Siegel} disk $\bZ_\str$ to be the full lift of $Z_\str$  (the Siegel disk of $f_\str$). Then $\bZ_\str$ is a forward  invariant unbounded open disk.

For every $N>1$, if $f\in \WW^u$ is sufficiently close $f_\str$, then the wall $\bPi_0(f)$ contains a univalent $N$-wall $A$ (see~\S\ref{sss:walls bPi}) respecting $\gamma_0,\gamma_1$. The full lift of a univalent $N$-wall of $\bPi_0(f_n)$ is a univalent $N-1$ wall in $\bPi_0(f_0)$.

\subsection{Rational and critical points}
\label{ss:rat crit pts}
A point $x$ is \emph{periodic} if $\bF^P(x)=x$ for a power-triple $P>0$. It will follow from Lemma~\ref{lem:discr of dyn} that every periodic point has a unique minimal \emph{period} $P\in \PT$. If $P$ is the minimal period of $x$, then the \emph{multiplier} of $x$ is $(\bF^P)'(x)$. Periodic and preperiodic points are called \emph{rational}. 

 A \emph{critical point} of a cascade $\bF^{\ge0}$ is a critical point of some $\bF^P$ for $P\in \PT$. The following lemma describes basic properties of critical points. 
\begin{lem}
\label{lem:crit pts}
Consider $\bF\in \Unst$. Then $x$ is a critical point of $\bF^{\ge 0}$ if and only if
 there is a $P>0$ such that $\bF^P(x)=0$. The set of critical points  $\CP\big(\bF^P\big)$ of $\bF^P$ is $\bigcup_{0<S\le P } \bF^{-S}(0)$. The set of critical values $\CV\big(\bF^P\big)$ of $\bF^P$ is $\{  \bF^{S}(0)\mid S< P\}$. The postcritical set $\Post(\bF)$ of $\bF^P$ is the forward orbit of $0$.

Write $K\coloneqq \min\{(0,1,0),\sp (0,0,1)\}.$ Then
 \begin{equation}
\label{eq:CV are away from 0}
\CV\big(\bG^P\big)\setminus \{0\}\subset \bDelta_0(\bG)\setminus \bS\sp \sp \sp \text{ for } P< K \sp \text{ and }\sp \bG\in \UnstLoc.
\end{equation}
For every $P<K$, the set $\CV(\bG^P)$ moves holomorphically with $\bG\in \UnstLoc$, and every critical point of $\bG^P$ has degree $2$ for  $\bG\in \UnstLoc$.

For every $\bF\in \Unst$, there is a $K_{\bF}>0$ such that every critical point of $\bG^P$ has degree $2$ for $P < K_{\bF}$.  

For every $\bF\in \Unst$ and every $P\in \PT$, there exists $k\in  \N$ such that the degree of every critical point of $\bF^P$ is at most $k$. If $0$ is not periodic, then the degree of every critical point of $\bF^P$ is $2$.
\end{lem} 
\noindent The first claim is essentially  \cite[Lemma~6.1]{DLS}.
\begin{proof}
 Let $\bW\Subset \C$ be an open topological disk, and let $\bW_1$ be a connected component of $\bF^{-P}(\bW)$. For $n\ll 0$, the map $\bF^P\colon \bW_1\to \bW$ is identified via $A_\str^k\circ h_k\colon V\setminus \gamma_1\to  \bS^{\#}_k$ with $f_n^{\ss(k)}:W_1\to W$ for some $\ss(k)\ge 0$ because  $A_\str^k\circ h_k$ conjugates \eqref{eq:pair:W_pm to D} and  \eqref{eq:pair:bW_pm to bD^k}, see \S\ref{sss:max prepacmen}. Therefore, $z\in \bW_1$ is a critical point of $\bF^P$ if and only if the $f_k$-orbit of $\big(A_\str^k\circ h_k\big)^{-1} (z)$ passes through $0$ during the first $\ss(k)$ iterates. This is equivalent to $\bF^S(z)=0$ for some positive $S\le P$. The degree of $\bF^P$ at $z$ is $2^t$, where $t$ is the number of positive $S\le P$ with  $\bF^S(z)=0$.  

 The claims about $\CP\big(\bF^P\big)$, $\CV\big(\bF^P\big)$ and $\Post(\bF)$ are immediate.

For $\bG\in \UnstLoc$ and $P\le K$, the point $\bG^S(0)$ belongs to a certain triangle $\Delta_0(i,\bG)$ that is disjoint from $\bS(\bG)$. This implies that $\bG^S(0)$ is well defined, depends holomorphically on $\bG$, and does not collide with $0$. Since points $\bG^S(0),\bG^Q(0)$ with $0<S<Q<T$ belong to different triangles of $\bDelta_0(\bG)\setminus \bS$, we obtain a holomorphic motion of $\CV(\bG^T)=\{\bG^S(0)\mid S<T\}$ with $\bG\in \UnstLoc$. We also proved~\eqref{eq:CV are away from 0}. Since every critical point of $\bG^P$ with $P\le K$ passes exactly once through $0$, the degree of every such critical point is $2$.

For $\bF\in \Unst$, choose $n\in \Z$ so that $\RR^n(\bF)\in \UnstLoc$. Then $K_{\bF}$ can be taken to be $\tt^{n} K$.

The last claim follows from the observation that every critical point of $\bF^P$ passes through $0$ at most $(P/K_{\bF})+1$ times.
\end{proof}

\subsection{Proper discontinuity of $\bF^{\ge 0}$}
\label{ss:prop disk of F>0}
 The action of the cascade $\bF^{\ge 0}$ is proper discontinuous in the following sense
\begin{lem}
\label{lem:discr of dyn}
For every bounded open set $W\subset \C$ there is a $Q>0$ in $\PT$ such that for all $\bG$ close to $\bF$ the following holds:
\begin{itemize}
\item $W\Subset \Dom \bG^P$ for all $P\le Q$;
\item $\bG^P\mid W$ is univalent for all $P\le Q$; and 
\item $\bG^P(W)\cap \bG^T(W)=\emptyset$ for all $P<T\le Q.$
\end{itemize}

For every $x\in \C$ and $T\in\PT$, the set 
\[\bigcup_{P\le T} \bF^P\{x\} \]
is discrete in $\C$. In particular,  the set of critical values $\CV(\bF^T)$ of $\bF^T$ is discrete in $\C$ for all  $T\in \PT$.
\end{lem}
\begin{proof}
By~\eqref{eq:Delta_m appr C} there is an $m\le 0$ such that $W\subset \Delta_m(j,\bG)$ for all $\bG$ close to $\bF$. Let us take $Q=\min\{(0,0,1),(0,1,0)\}\tt^m$. For $P\le Q$ the map $\bG^P$ maps 
$\Delta_m(j,\bG)$ to a different triangle of $\bDelta_m(\bG)$; this show the first claim.

\begin{figure}
\centering{\begin{tikzpicture}

\coordinate (A) at (0.15,0);
\coordinate (B) at (0.85,0);
\draw (-0.5,0) -- (A); 
\draw (1.5,0)--(B);
\begin{scope}[shift={(0.4,-0.5)},rotate=135,scale =0.2]
\draw (0,0) arc (0:270:1);
\draw (0,0)--(A);
\draw (-1,-1) -- (B); 
\end{scope}

\draw (-0.5,0)--(-0.5,-3);
\draw (-2,0)--(-0.5,0);

\draw[red] (1.5,-3)--(1.5,1)--(-2,1)--(-2,-3);

\node at (0.5,-2) {$\Delta_0(1)$}; 
\node at (-1.2,-2) {$\Delta_0(0)$}; 

\node at (-2.9,-2) {$\Delta_0(-1)$}; 
\node at (-4.65,-2) {$\Delta_0(-2)$}; 
\node at (2.6,-2) {$\Delta_0(2)$}; 
\node at (4.2,-2) {$\Delta_0(3)$};

\begin{scope}[shift={(2,1)}]
\coordinate (A) at (0.15,0);
\coordinate (B) at (0.85,0);
\draw (-0.5,0) -- (A); 
\draw (1.5,0)--(B);
\begin{scope}[shift={(0.4,-0.5)},rotate=135,scale =0.2]
\draw (0,0) arc (0:270:1);
\draw (0,0)--(A);
\draw (-1,-1) -- (B); 
\end{scope}
\end{scope}
\draw (3.5,1)--(3.5,-3);
\draw (3.5,1)--(4.5,1);

\begin{scope}[shift={(-3.5,0)}]
\coordinate (A) at (0.15,0);
\coordinate (B) at (0.85,0);
\draw (-0.5,0) -- (A); 
\draw (1.5,0)--(B);
\begin{scope}[shift={(0.4,-0.5)},rotate=135,scale =0.2]
\draw (0,0) arc (0:270:1);
\draw (0,0)--(A);
\draw (-1,-1) -- (B); 
\end{scope}
\end{scope}

\draw (-4,0)--(-4,-3);
\draw (-4,0)--(-6,0);
\draw (-5.5,0)--(-5.5,-3);

\draw (-1.2,-1.1) edge[->, bend left ]  (0.8,0.5);
\node[left]at (-1,-0.8) {$\bbf_+$};
\draw (0.4,-1.2) edge[->, bend right ]  (-1.2,0.5);
\node[left]at (0.35,-0.8){$\bbf_-$};
\node [red,above] at(-0.4,0.3) {$\bS$};

\end{tikzpicture}}
\caption{The triangulation $\bDelta_0(\bF)$ is obtained by spreading around the triangles $\Delta_0(0,\bF)$ and  $\Delta_0(1,\bF)$ -- the embeddings of $\Delta_0(1,f)$ and $\Delta_0(0,f)$, see Figure~\ref{Fig:ren triang:f0}.
The triangulation has the same combinatorics as the renormalization tiling of $\R$, see Figure~\ref{Fig:ren part of R}}.
\label{Fig:ren triang:bF}
\end{figure}

Suppose that $\displaystyle\bigcup_{P\le T} \bF^P\{x\}$ accumulates on $y$. Choose a small neighborhood $W$ of $y$. Then there is a $Q>0$ such that $\bF^P(W)$ is disjoint from $W$ for all $P<Q$.  Since $T< kQ$ for some $k\gg 1$, the intersection $W\cap \displaystyle \bigcup_{P\le T} \bF^P\{x\}$ consists of at most $k$ points.

The set of critical values of $\bF^T$ is discrete because it is equal to $\displaystyle\bigcup_{P<T}\bF^P\{0\}$, see Lemma~\ref{lem:crit pts}.
\end{proof}

\begin{cor}
\label{cor:orbit of bY}
Let $\bY$ be a compact set such that $\bY\Subset \Dom \bF^P$. Then for every $\bX\Subset \C$ there are at most finitely many $T\le P$ such that $\bF^T(\bY)$ intersects $\bX$.
\end{cor}
\begin{proof}
For every $y\in \bY$, the orbit $\orb^P_z\coloneqq \{\bF^S(z)\mid S\le P\}$ is discrete and depends continuously on $z$ in a small neighborhood of $y$ because $y\in \Dom \bF^P$. (In fact, if $x\not \in \Dom \bF^P$, then $\orb^P_z$ does not depends continuously on $z$ in a small neighborhood of $x$). The corollary now follows from a compactness argument.
\end{proof}

\begin{cor}
\label{cor:min period:per pts}
Every periodic point has a minimal period.

For every critical point $x$ of $\bF^{\ge 0}$, there is a minimal $P>0$, called the \emph{generation} of $x$, such that $\bF^P(x)=0$. 
\end{cor}
\begin{proof}
Let $x$ be a periodic point of $\bF$, and let \[\PT_x\coloneqq \big\{P\in \PT \mid \bF^P(x)=x\big\} \]
be the semigroup of all the periods of $x$. By Lemma~\ref{lem:discr of dyn}, there is a neighborhood $W$ of $x$ and a small $Q>0$ such that $\bF^T(W)\cap W=\emptyset $ for all $T\le Q$; in particular, $\bF^T(x)\not= x$. Therefore, $\PT_x$ is of the form $\{nS\mid n\ge 1\}$, where $S>0$ is the minimal period.

By Lemma~\ref{lem:crit pts}, if $x$ is a critical point of $\bF^{\ge0}$, then $\bF^P(x)=0$ for some $P\in \PT$. Since $\{\bF^S(x)\mid S\le P\}$ does not accumulate on $0$, there is a minimal $S>0$ such that $\bF^S(x)=0$.
\end{proof}

\subsection{Walls $\bPi_0(\bF)$}
\label{ss:MP:walls}
Recall from~\S\ref{sss:walls bPi} that a triangulation $\bDelta_n(f)$ has a wall $\bPi_n(f)$ for $f\in \WW^u$. Let $\Pi_0(0,\bF)$ and $\Pi_0(1,\bF)$ be the embeddings of the rectangles  $\Pi_0(0,f)$ and $\Pi_0(1,f)$ to the dynamical plane of $\bF$ via $\bS\simeq V\setminus \gamma_1$. The \emph{wall $\bPi_0(\bF)$ of $\bDelta_0(\bF)$} is obtained by spreading around $\Pi_n(0,\bF)$ and $\Pi_n(1,\bF)$.

The \emph{wall $\bPi_n(\bF)$} is $A_\str^{-n} \big(\bPi_0(\RR^n f)\big)$. We define $\bQ_n(\bF)\coloneqq \bDelta_n\setminus \bPi_n(\bF)$. This is the interior of the full lift of $Q_n(f)\coloneqq \bDelta_n( f)\setminus \bPi_n(f)$.

\subsection{The boundary point $\balpha$}
\label{ss:balpha:wall topoloy}
As in~\S\ref{sss:balpha}, we add a boundary point $\balpha$ at ``$-i \infty$'' to the dynamical space of a prepacman. The boundary point $\balpha(\bF)$ corresponds to $\alpha(f)$. Let us introduce the \emph{wall topology} for $\C\sqcup \balpha(\bF)$, compare with~\S\ref{sss:balpha}.

Consider a univalent wall $A$ respecting $\gamma_0,\gamma_1$ in a small neighborhood of $\alpha(f)$. Then the full lift $\bA(\bF)$ of $A(f)$ is a closed strip in $\C$ such that $\C\setminus \bA(\bF)$ has two connected components. We denote by $\binn=\binn(A)$ the component of $\C\setminus \bA(\bF)$ not containing $0$. Equivalently, $\binn$ is the interior of the full lift of the component of $\C\setminus A$ containing $\alpha$. We say that $\binn$ is \emph{below} $\bA$. The open sets of $\C\sqcup \{\balpha(\bF)\}$ are generated by open sets in $\C$ and by $\binn(A) \sqcup \{\balpha(\bF)\}$ for all univalent walls as above. 
 
 A curve $\ell\colon [0,1)\to \C$ \emph{lands} at $\balpha$ if $\lim _{t\to 1} \ell(t)=\balpha$.
\subsection{A fundamental domain for $\bF$}
\label{ss:fund domain for bF}Recall that for $\bF\in \UnstLoc$, the sector $\bS$ has distinguished sides $\lambda (\bS)$ and $\rho(\bS)$. Moreover $V$ is a quotient of $\bS$ under $\bF^{(0,0,1)-(0,1,0)}=\bbf_-^{-1}\circ \bbf_+\colon \lambda(\bS)\to \rho(\bS)$, and $\lambda (\bS)$, $\rho(\bS)$ project to $\gamma_1$.

Suppose $\bF\in \RR^{-n}\big( \UnstLoc\big)$ for $n\ge 0$. Recall \S\ref{ss:MP:walls} that ${\bQ_n(\bF)=\bDelta_n(\bF)\setminus \bPi_n(\bF)}$. 
A \emph{fundamental domain} in the dynamical plane of $\bF$ is a sector $\bS^\new$ with distinguished sides $\lambda(\bS^\new)$ and $\rho(\bS^\new)$ such that 
\begin{enumerate}
\item $\bS\setminus \bQ_n=\bS^\new\setminus \bQ_n$ and $\lambda(\bS)\setminus \bQ_n=\lambda(\bS^\new)\setminus \bQ_n$ and\newline ${\rho(\bS)\setminus \bQ_n=\rho(\bS^\new)\setminus \bQ_n}$, \label{cond:1:fund domain}
\item $\lambda(\bS^\new)$ and $\rho(\bS^\new)$ land at $\balpha$; and\label{cond:2:fund domain}
\item $\intr(\bS^\new)$ contains an arc $\ell$ such that $\bbf_-(\ell)=\lambda( \bS^\new)$ and $\bbf_+(\ell)=\rho( \bS^\new)$.\label{cond:3:fund domain}
\end{enumerate}

Similar to~\S\ref{ss:PartiHomeo}, we say that $\beta_0,\beta_1=f(\beta_0)$ is a diving pair of arcs in the dynamical plane of a pacman $f\colon U\to V$ if
\begin{itemize}
\item  $\beta_0$ is a simple arc connecting $\alpha$ and point on $\partial U$;
\item  $\beta_1$ is a simple arc connecting $\alpha$ and a point on $\partial V$; 
\item $\beta_0$, $\beta_1$ are disjoint away from $\alpha$.
\end{itemize}

\begin{thm}
\label{thm:quot of fund domain}
Suppose $\bS^\new$ is a fundamental domain in the dynamical plane of $\bF\in  \RR^{-n}\big( \UnstLoc\big)$ as above. Then the quotient of $\bS^\new$ under 
 \[\bF^{(n,0,1)-(n,1,0)}\colon \lambda( \bS^\new)\to \rho( \bS^\new)\] is canonically conformally homeomorphic to $V$. Under this homeomorphism $\bF$ projects to $f$.
 
 Conversely, if $\gamma_0^\new,\gamma_1^\new$ is a dividing pair of arcs such that $\gamma_0^\new\setminus Q_n=\gamma_0$ and $\gamma_1^\new\setminus Q_n=\gamma_1$ (see~\S\ref{ss:MP:walls}), then $h$ extends from a neighborhood of $c_1$ to a conformal map defined on $V\setminus \gamma_1^\new$ so that the closure of $h\big(V\setminus \gamma_1^\new \big)$ is a fundamental domain.
\end{thm}
\begin{proof}
Choose in $\bPi_n(f)$ a univalent wall $A$ respecting $\gamma_0,\gamma_1$ and surrounding $\inn\ni\alpha$. Let $\bA\subset \bPi_n(\bF)$ and $\binn$ be the full lifts of $A$ and $\inn$ to the dynamical plane of $\bF$. The theorem now follows from Proposition~\ref{prop:quot of fund domain} applied to $f\mid \inn\cup A$ and $\bF\mid \binn\cup \bA$.
\end{proof}

The image of $\bS^\new$ is of the form $\bDelta_0^\new(f)$ as in~\S\ref{sss:SiegTriang}; it has the full lift to $\bDelta^\new_{n-m}(f_m)$ for all $m\le 0$. Condition~\eqref{cond:1:fund domain} can be related into ``$\bS$ and $\bS^\new$ are sufficiently close in $\bPi_n(\bF)$.''

\begin{figure}[t!]\begin{tikzpicture}[scale=0.85]
\draw[red] (-7,0.5)--(7,0.5)
 (-7,0)--(7,0)
 (-6,0.5)--(-6,0)
 (-5,0.5)--(-5,0)
 (-4,0.5)--(-4,0)
 (-3,0.5)--(-3,0) 
 (-2,0.5)--(-2,0)
 (-1,0.5)--(-1,0)
 (-0,0.5)--(-0,0) 
 ( 6,0.5)--( 6,0)
 ( 5,0.5)--( 5,0)
 ( 4,0.5)--( 4,0)
 ( 3,0.5)--( 3,0) 
 ( 2,0.5)--( 2,0)
 ( 1,0.5)--( 1,0)
 ;
\draw (2.5,-5) -- (2.5,2)--(-3.5,2)--(-3.5,-5);

\draw[blue,dashed] (2.5,0) 
     .. controls (2.8,-2) and (4,-4) ..
 (5.5,-5)
 (-3.5,0) 
     .. controls (-3.3,-2) and (-2,-4) ..
 (-0.5,-5)
;

\node[blue] at(3.3,-4.5){$\bS^\new$};

\node at(2,1.2){$\bS$};
\node[red] at (5.3,0.8){$\bPi_n$};

\node at (5.5,-1){$\bQ_n$};

\draw (-2.1,-3) edge[->,bend left] node[above]{$\bbf_-\circ \bbf_+$}(3.4,-3);

\node[blue] at(-1.3,-3.5){$\lambda^\new$};
\node[blue] at(4.7,-3.5){$\rho^\new$};

\node at(-3.8,-3.5){$\lambda$};
\node at(2.2,-3.5){$\rho$};
\end{tikzpicture}

\caption{A fundamental domain $\bS^\new$ is a sector in $\C\sqcup \{\alpha\}$ such that $\bS^\new \setminus \bQ_n$ coincide (or close) to $\bS\setminus \bQ_n$ and such that the ``deck transformation'' $\bbf_-^{-1}\circ \bbf_+$ maps the left boundary $\lambda^\new$ of $\bS^\new$ to its right boundary  $\rho^\new$. }\label{Fig:FundDomain}
\end{figure}

\subsection{Fatou, Julia, and escaping sets}
\label{ss:Fat Jul  Esc}
 
Consider $x\in \C$. If there is an open set $U$ such that $U\subset \Dom \bF^P $ for all $P\ge 0$ and, moreover, $\{\bF^P\mid U\}_{P\ge 0}$ forms a normal family, then $x$ is a \emph{regular point of $\bF$}. The \emph{Fatou set} $\Fat(\bF)$ of $\bF$ is the set of regular points of $\bF$. By construction, all $\bF^{\#}_{n}$ have the same Fatou sets.

The \emph{Julia set}  $\Jul(\bF)$ of $\bF$ is $\C\setminus \Fat(\bF)$. 
 Clearly, all repelling periodic points are within the Julia set of $\bF$. 

The cascade $\bF^{\ge 0}$ acts on the set of components of $\Fat(\bF)$.  A component $X$ of $\Fat(\bF)$ is \emph{periodic} if there is a power-triple $P$ such that $\bF^P(X)= X$. We call $P$ a \emph{period} of $X$. 

A Fatou component $X$ is \emph{invariant} if $\bF^P(X)= X$ for every $P\in \PT$. By Corollary~\ref{cor:min period:per pts}, if a Fatou component $X$ has an attracting point in $\C$, then $X$ has a minimal period; in particular, $X$ is not invariant. (Note that $\balpha(\bF)$ is an attracting point of an invariant Fatou component if $\alpha(f)$ is attracting.)

Two Fatou components $X,Y$ are \emph{dynamically related} if they are in the same grand orbit: there are $P,Q\in \PT$ such that a certain branch of $\bF^{-P}\circ \bF^Q$ maps $X$ to $Y$. Dynamically related periodic components have the same periods.

A Fatou component $X$ is \emph{preperiodic} if there is a power-triple $Q$ such that $\bF^Q(X)$ is a periodic Fatou component. In this case, $Q$ is the \emph{preperiod}.

 Given $P\in \PT$, we define the \emph{$P$-th escaping set} as 
\[\Esc_P (\bF) \coloneqq \C\setminus \Dom(\bF^P).\]
The \emph{escaping set} is
\[\Esc(\bF)\coloneqq \bigcup_{P\ge0} \Esc_P(\bF).\]

Since the domains of $\bbf_\pm$ are simply connected (see~\S\ref{sss:max prepacmen}) for $\bF\in \UnstLoc$, every connected component of $\Dom \bF^P$ is simply connected. Therefore, every connected component of $\Esc_P(\bF)$ is unbounded. Since $\Esc_P(\bF)$ is a rescaling of $\Esc_{\tt ^nP} (\bF_{-n})$, every connected component of $\Esc_P(\bG)$ is unbounded for every $\bG\in \Unst$.

By definition, $\overline {\Esc(\bF) }\subset \Jul(\bF)$. We will show in Corollary~\ref{cor:Jul is ovl Esc} that $\overline {\Esc(\bF) }\not=\emptyset$, hence $\overline {\Esc(\bF) }= \Jul(\bF)$.

\subsection{QC deformation of maximal prepacmen}
\label{ss:QC deform}
Suppose that $\C$ has a Beltrami form $\mu$ such that $\mu$ is invariant under the cascade $\bF^{\ge 0}$, where $\bF\in \Unst$. Integrating $\mu$, we obtain a path $\bF_t^{\ge 0}$ with $t\ge 0$ of cascades emerging from $\bF^{\ge0}_0=\bF^{\ge0}$. We claim that $\bF_t$ is, up to scaling, a path on the unstable manifold $\Unst$. We will use the argument from \cite{DLS}*{\S 8.1.2} to find a correct scaling. 

Applying antirenormalization, we can assume that $\bF$ is in a small neighborhood of $\bF_\str$. In particular, $\bF\in \UnstLoc$. Projecting $\mu$ to the dynamical plane of $f_n$ for $n\le 0$, we obtain the Beltrami form $\mu_n$ invariant under $f_n$. Integrating $\mu_n$, we obtain a path $f_{n,t}\in \BB$ emerging from $f_{n,0}=f_n$. Set $f_{t}^{(n)}\coloneqq \RR^{-n} f_{n,t}$, and observe that for small $t$ all $f^{(n)}_t$ are qc conjugate with bounded dilatation uniformly in $n$. Therefore, we can take a limit and construct a path $f^{(\infty)}_t$ in $\BB$ of infinitely anti-renormalizable pacmen. Therefore,  $f_t\coloneqq f^{(\infty)}_t$ is a path in $\WW^u$.

The qc deformation of maximal prepacmen allows us to modify multipliers of attracting periodic cycles.

\section{External structure of $\bF_\str$}
\label{s:Dyn F_str}
In this section we set $\bF=\bF_\str$ and we let $\bZ=\bZ_\str$ to be its Siegel disk which is defined to be the full lift of $Z_\str$, see~\S\ref{ss:RenTriang}.  We also write the fixed pacman $f_\str \colon U_\str\to V$ as $f \colon U\to V$.

\subsection{Chess-board rule} Let us say that a simply connected open set $U\subset \Disk$ has a \emph{single access to infinity} if $ \partial \Disk\cap \partial U\not= \emptyset$ and $\Disk\setminus U$ is connected. Similarly, a simply connected open set $U\subset \C$ has a \emph{single access to infinity} if $U$ is unbounded and $\C\setminus U$ is connected.

We need the following fact.

\begin{lem}
\label{lem:chess board}
Let $g\colon \Dom g \to \C$ be a $\sigma$-proper map, where $\Dom g$ is either $\Disk$ or $\C$. Suppose that the set of critical values $\CV(g)$ of $g$ is discrete and assume that $\ell\colon \R\to \C $ is a simple properly embedded arc such that
\begin{itemize}
\item[1)] $\ell(\R)\supset \CV(g)$, and 
\item[2)] $\ell$ splits $\C$ into two open half-planes $V$ and $W$.
\end{itemize} 

Then 
\begin{itemize}
\item $g^{-1}(\ell)$ is a tree in $\Dom g$; in particular, if $U$ is a connected component of $\Dom g \setminus g^{-1}(\ell)$, then $U$ has a single access to infinity; and
\item there is a ``chess-board rule'': if $U_1$ and $U_2$ are two different components of $g^{-1}(V)$, then $\partial U_1\cap \partial U_2\cap \Dom g$ is either empty or a single critical point of $g$. 
\end{itemize}
\end{lem}

\noindent Note that since $g$ is $\sigma$-proper, $g^{-1}(\CV(g))$ is in discrete $\Dom g$.
\begin{proof}
We will verify the case when $\Dom g=\Disk$; the case $\Dom g=\C$ is completely analogous.

Consider a connected component $U$ of $\Disk\setminus g^{-1}(\ell)$. Recall that a $\sigma$-proper map has no asymptotic values. Since $\CV(g)\subset \ell$, the map \[g\colon U\to g(U)\in \{V,W\}\] is a covering. Since $V$ and $W$  are simply connected, so is $U$; i.e.~$g\colon U\to g(U)$ is univalent. And since $V$ and $W$  are unbounded, $U$ is not properly contained in $\Disk$. This implies that $g^{-1}(\ell)$ is a forest.

Since $\ell$ is a properly embedded arc in $\C$, we can express $\C$ as an increasing union of open topological disks $Y_i$ such that $Y_i\cap \ell $ is an arc. Defining the $X_i$ to be lifts of the $Y_i$ with $X_{i+1}\supset X_i$, we express $\Disk $ as an increasing union of open disks $X_i$ such that $g\colon X_i\to Y_i$ is proper. Then ${(g\mid X_i)^{-1}(Y_i\cap \ell )}=g^{-1}(\ell)\cap X_i$ is a finite tree. Therefore, $g^{-1}(\ell)$ is a tree in $\Disk$.

Let $U_1$ and $U_2$ be two different components of $g^{-1}(V)$. Then $\partial U_1\cap \partial U_2\cap \Dom g$ is a discrete set of critical points. If $\partial U_1\cap \partial U_2\cap \Dom g$ has at least two points, then $\overline U_1\cup \overline U_2$ surround a component of  $g^{-1}(W)$; this is a contradiction.
\end{proof}

\subsection{Bubbles of $\bF$}
\label{ss:FatComp of bFstr}
Recall from~\S\ref{ss:RenTriang} that $\bZ=\bZ_\str$ is the full lift of $Z=Z_\str$. Alternatively, $\overline \bZ_\str$ can be viewed as a rescaled limit of $\overline Z_\str$,~\cite{McM2}. Then $\bbf_\pm \mid \overline \bZ$ is a pair of homeomorphisms, and $ \bbf_+\circ \bbf_-^{-1}$ is a deck transformation of $\overline \bZ$: the quotient $\overline \bZ/\langle\bbf_-^{-1}\circ \bbf_+ \rangle$ is identified with $\overline Z_\str$; i.e.~$ \overline \bZ$ is the universal cover of $ \overline Z_\str\setminus \{\alpha\}$ and $\bbf_\pm\mid \overline\bZ$ are lifts of $f\mid Z_\str$.

\begin{lem}[Compare with~\cite{McM2}*{Theorem 8.1}]
\label{lem:irr rot of bZ}
The disk $\bZ$ is an invariant Fatou component of $\bF$. Let $\bbh\colon  \bZ\to \{ \Im z< 0\}$ be a conformal map. Then $\bbh$ extends to a quasisymmetric map $\bbh\colon \overline \bZ\to \{ \Im z \le 0 \}$. 

Let us normalize $\bbh\colon  \bZ\to \{ \Im z< 0\}$ to be the unique conformal map so that $\bbh(0)=0$ and $|\bbh'(0)|=1$. Then $\bbh$ conjugates $\bbf_{-},\bbf_{+}$ to a pair of translations $z\mapsto z -\vv$ and $z\mapsto z+\ww$ with $\vv,\ww\ge0$ such that $\theta_\str=\vv/(\vv+\ww)$. Moreover, $\bbh$ conjugates the cascade $\big(\bF^P\mid \overline \bZ\big)_{P\in \PT}$ with the cascade of translations $\big(T^P\big)_{P\in \PT}$ from \S\ref{sss:renorm:pair of rotat}. 
\end{lem}
\begin{proof}
Recall from \S\ref{ss:RenTriang} that the support of the renormalization triangulation $\bDelta_0(\bF)$ is a neighborhood of $\overline \bZ$ and $\bDelta_0(\bF)$ is the full lift of $\bDelta_0(f)$. As a consequence, $\partial \bZ$ is locally a quasiarc; $\partial \bZ$ is globally a quasiarc  because it is invariant under the scaling $A_\str$.

In the dynamical plane of $f$, the boundary $\partial Z_\str$ is contained in the closure of repelling periodic points (see~\S\ref{sss:SiegPacmen}). Lifting these periodic points to the dynamical plane of $\bF$ we obtain that $\partial \bZ$ is also in the closure of repelling periodic points; thus $\partial \bZ\subset \Jul(\bF)$. 

Since $\bbf_\pm\mid \overline\bZ$ are lifts of $f\mid \overline Z_\str$, the pair $\bbf_\pm\mid \overline \bZ$ is conjugate to a pair of translations $z\mapsto z -\vv$ and $z\mapsto z+\ww$ as required (see also~\S\ref{sss:renorm:pair of rotat}); thus $\bZ$ is a Fatou component of $\bF$.   
\end{proof}

 Observe that $\partial \bZ\supset \Post(\bF) \supset\CV(\bF)$ because $0\in \partial \bZ$ and $\partial \bZ$ is invariant. We have the following corollary of Lemma~\ref{lem:chess board}: 
\begin{cor}
\label{cor:preimage:bZ:tree}
For every $P\in \PT_{>0}$ the preimage $\bF^{-P}(\partial \bZ)$ is a tree in $\Dom \bF^P$. \qed
\end{cor}

 Since $\bF^P\mid \partial \bZ$ is a homeomorphism, for every $P>0$, there is a unique critical point $c_P\in \partial \bZ$ of $\bF^{\ge 0}$ of generation $P$, see Lemma~\ref{lem:crit pts}. Since $\bF^P$ is two-to-one around $c_P$, there is a preperiodic Fatou component $\bZ_P$ attached to $c_P$ such that $\bF^P(\bZ_P)=\bZ$. 
\begin{lem}
\label{lem:Z_p is att to Z}
For $c_P$ and $\bZ_P$ as above, $\overline \bZ\cap \overline \bZ_P=\{c_P\}$.  Set \[\widehat \bZ_P\coloneqq  \overline \bZ_P\cap \Dom \bF^P.\] Then $\bF^P\colon \widehat \bZ_P\to \overline \bZ$ is a homeomorphism.
\end{lem}
\noindent The point $c_P$ is called the \emph{root} of $\bZ_P$.  We will show in Lemma~\ref{lem:alpha' is a pnt} that \[\widetilde \alpha_P\coloneqq \overline \bZ_P\setminus \widehat \bZ_P\] consists of a single point, called the \emph{top} of $\bZ_P$. We set 
\[\partial^c \bZ_P\coloneqq \partial \bZ_P \cap \Dom (\bF^P)=\widehat \bZ_P\setminus\bZ_P. \]
\begin{proof}
By Corollary~\ref{cor:preimage:bZ:tree}, $\bF^{-P}(\partial \bZ)$ is a tree in $\Dom(\bF^P)$. Therefore, $\bZ_P$ is a connected component of $\Dom(\bF^P)\setminus \bF^{-P}(\partial \bZ)$ specified so that $\overline \bZ\cap \overline \bZ_P=\{c_P\}$. Moreover, $\partial^c \bZ_P=\partial \bZ_P\cap \Dom(\bF^P)$ is a simple arc. 

Since $\bF^Q(\partial^c \bZ_P)$ is disjoint from $0$ for all $Q< P$ (because $\bF^Q(\partial^c \bZ_P)\subset \partial^c \bZ_{P-Q}$), the curve $\partial^c \bZ_P$ contains a single critical point $c_P$ of $\bF^P$. Therefore, $\bF^P\colon \partial^c \bZ_P \to \partial \bZ$ is a homeomorphism; this proves the second claim. 
\end{proof}

For every $Q\in \PT$, there is a unique critical point $c_{P,Q}\in \partial \bZ_P$ of $\bF^{\ge0}$ of generation $P+Q$. If $Q>0$, then there is a unique Fatou component $\bZ_{P,Q}$ attached to $\bZ_P$ at $c_{(P,Q)}$ such that $\bF^{P+Q}(\bZ_P)=\bZ$ . As above, we have a homeomorphism $\bF^{P+Q}\colon \widehat \bZ_{P,Q} \to \overline \bZ$. 

Continuing this process, we define $c_{P_1,\dots , P_m}$ and the \emph{bubble} $\bZ_{P_1,\dots , P_m}$ for every finite sequence in $s=(P_1,\dots, P_n)\in \PT^n_{>0}$. We call $|s|=P_1+\dots+P_n$ the \emph{generation} of $\bZ_s=\bZ_{P_1,\dots , P_m}$. Lemma~\ref{lem:Z_p is att to Z} implies:
\begin{lem}
\label{lem:Z_s is att to Z_v}
For every bubble $\bZ_s$ with $|s|>0$, there is a unique bubble $\bZ_v$ with $|v|\le |s|$ (possibly $\bZ_v=\bZ$) and $v\not=s$ such that $\partial^c \bZ_s\cap \partial^c \bZ_v  \not=\emptyset$.  Moreover, $\partial^c \bZ_s\cap \partial^c \bZ_v=\{c_s\}$.\qed
\end{lem}

 We define the corresponding \emph{finite bubble chain} as \[B_s=(\widehat \bZ_{P_1},\widehat \bZ_{P_1,P_2},\dots\widehat \bZ_{P_1,P_2,\dots , P_n}).\]

The \emph{primary limb rooted at $c_{P_1}$} is
\[\bL_{P_1}\coloneqq \bigcup _{n\ge 1}\bigcup_{(P_1,\dots , P_n)}\widehat \bZ_{P_1,\dots , P_n},\]
where the union is taken over all finite sequences in $\PT_{>0}$ starting with $P_1$\footnote{This is not a standard definition of a limb because $\bL_{P_1}$ is not closed.}. Similarly, the \emph{secondary limbs} $\bL_{P_1,P_2}$ of $\bZ_{P_1}$ are defined.

We also consider \emph{infinite bubble chains}: given \[s=(P_1,P_2,\dots)\in \PT^{\N}_{>0},\] we set $B_s=(\widehat \bZ_{P_1},\widehat \bZ_{P_1,P_2},\dots)$. The \emph{generation} of $B_s$ is
 \[|s|=P_1+P_2+\dots \le \infty\]  
 (recall that we view $\PT$ as a sub-semigroup of $\R_{>0}$); it is the supremum of generations of all the bubbles in the chain.
 
 Given a finite or infinite bubble chain $B_s$ with $s=(P_1,P_2,\dots )$, we write its geometric realization as \[\bB_s=\widehat \bZ_{P_1}\cup\widehat\bZ_{P_1,P_2}\cup \dots  \]
 and call it a \emph{bubble chain} as well.

Suppose $s$ is an infinite sequence. The \emph{accumulating set} of $B_s$ (and of $\bB_s$) is the accumulating set of $\widehat \bZ_{P_1}, \widehat \bZ_{P_1,P_2},\dots$. If the accumulating set is a singleton $\{x\}$, then we say that $B_s$ \emph{lands} at $x$. (It will follow from Lemma~\ref{lem:nest wakes shrink} that every infinite bubble chain lands.)

If $s$ is a finite sequence, then the \emph{accumulating set} of $B_s$ is 
\begin{equation}
\label{eq:defn:accum set} 
\overline \bZ_{s}\setminus \widehat \bZ_s\subset \Esc_{P}.
\end{equation}

\begin{prop}
Every strictly preperiodic preimage of $\bZ$ is contained in some limb $\bL_x$.
\end{prop}
\begin{proof}
Let $\bZ'\not=\bZ$ be a component of $\bF^{-P}(\bZ)$. By Lemma~\ref{lem:chess board}, $\bZ'$ is a component of $\Dom \bF^P\setminus \bF^{-P}(\partial \bZ)$, where is $\bF^{-P}(\partial \bZ)$ is a tree in $\Dom \bF^P$. Let $\gamma\subset \bF^{-P}(\partial \bZ)$ be the tree-geodesic connecting $\partial \bZ$ to $\partial^c \bZ'$. There are finitely many  critical points \[c_{P_1},c_{P_1,P_2},\dots , c_{P_1,\dots, P_n}\] of $\bF^P$ in $\gamma$. We have $\bZ'=\bZ_{P_1,\dots, P_n}$.
\end{proof}

For a limb $\bL_s$ we write
\[\partial^c \bL_s = \bigcup_{\bZ_v\subset \bL_s}\partial ^c \bZ_v \] and, similarly, for a bubble chain $\bB_s$:
\[\partial^c \bB_s = \bigcup_{\bZ_v\subset \bB_s}\partial ^c \bZ_v.\]

\begin{figure}[tp!]
\begin{tikzpicture}
\draw(-5,0) -- (5,0);
\node at (-4,-0.2) {$\partial \bZ$};
\node at (0,-0.2) {$z=0$};
\node at (0.2,1){$\bO'$};
\begin{scope}[shift={(-1.2,0)}]
\draw 
(-1,1)--(-2,2)--(-3,1)--(-2,0);
\draw  (0.2,1.5)--(-0.6, 1.5)--(-1,1);
\draw (0.7,2)--(0.35, 1.85)--(0.2,1.5);

\filldraw[opacity=0.5] (-1,1)--(-2,2)--(-3,1)--(-2,0)--(-1,1);
\filldraw[opacity=0.5] (0.2,1.5)--(-0.6, 1.5)--(-1,1) --(-0.2,1)--(0.2,1.5);
\filldraw[opacity=0.5] (0.7,2)--(0.35, 1.85)--(0.2,1.5)--(0.55, 1.65)--(0.7,2);

\draw[line width=0.5mm, red] (1.2,0)--(-2,0)--(-1,1) --(-0.2,1) --(0.2,1.5)--(0.55, 1.65) --(0.7,2)-- (1,2.1);
\node at (-0.4,1.8) {$\bB_\lambda$};
\node[red] at(-1,0.6){$\bI_\lambda$};
\end{scope}

\begin{scope}[shift={(1.2,0)},xscale=-1]
\draw 
(-1,1)--(-2,2)--(-3,1)--(-2,0);
\draw  (0.2,1.5)--(-0.6, 1.5)--(-1,1);
\draw (0.7,2)--(0.35, 1.85)--(0.2,1.5);
\draw[line width=0.5mm, blue] (1.2,0)--(-2,0)--(-1,1) --(-0.2,1) --(0.2,1.5)--(0.55, 1.65) --(0.7,2)-- (1.4,2.1);

\node at (-0.4,1.8){$\bB_\rho$};
\node[blue] at(-1,0.6){$\bI_\rho$};

\filldraw[opacity=0.5] (-1,1)--(-2,2)--(-3,1)--(-2,0)--(-1,1);
\filldraw[opacity=0.5] (0.2,1.5)--(-0.6, 1.5)--(-1,1) --(-0.2,1)--(0.2,1.5);
\filldraw[opacity=0.5] (0.7,2)--(0.35, 1.85)--(0.2,1.5)--(0.55, 1.65)--(0.7,2);

\end{scope}
\end{tikzpicture}
\caption{The case $z=0$: the bubble chains $\bB_\lambda$ and $\bB_\rho$ of the lake $\bO'$ contain $\bI_\lambda\setminus \partial \bZ$ and  $\bI_\rho\setminus \partial \bZ$ -- the left and right sides of $\partial^c \bO'$.} 
\label{Fig:I lambda B Lambda}
\end{figure}

\subsection{Lakes}
\label{ss:Lakes}
We call $\bO\coloneqq \C\setminus \overline \bZ$ the \emph{lake of generation $0$} or the \emph{ocean}. A \emph{lake of generation $P\in \PT_{>0}$} is a connected component of $\bF^{-P}(\bO)$. In particular, lakes of generation $P$ are pairwise disjoint. If $\bO_1$ is a lake of generation $P\in \PT$, then its \emph{coast}  is \[\partial^c \bO_1\coloneqq  \partial \bO_1\cap \Dom \bF^P.\] By Lemma~\ref{lem:chess board}:
\begin{itemize}
\item $\bO_1$ has a single access to $\Esc_P (\bF)$ (i.e. after identifying $\C\setminus \Esc_P(\bF)\simeq \Disk$);
\item $\bF^P\colon \bO_1\to \bO$ is conformal;
\item $\partial^c \bO_1$ is a properly embedded simple arc in the tree $\bF^{-P}(\partial \bZ)$;
\item $\bF^P\colon \partial^c \bO_1\to \partial^c\bO=\partial \bZ$ is a homeomorphism.
\end{itemize} 
 We will call lakes of positive generation \emph{proper}. 

Lakes form a ``puzzle partition'' (by open sets): since $\bZ=\bF^S(\bZ)$ is invariant, lakes of generation $R$ are inside lakes of generations $P$ for $R>P\ge 0$. The map $\bF^{R-P}$ maps every lake of generation $R$ univalently to a lake of generation $P$. Informally speaking, when $P$ grows there are less water and more land, and the ``drought'' occurs in a self-similar fashion. We will use lakes to study $\Esc(\bF)$; lakes will not be used beyond the current section.

Consider a lake $\bO'$ of generation $P$. Let $z\in \partial^c \bO'$ be the closest point to $0$ in the tree $\bF^{-P}(\partial \bZ)$. Then $z$ splits $ \partial^c \bO'$ into two arcs $I_\lambda$ and $I_\rho$. We assume that $I_\lambda,[z,0],I_\rho$ has a counterclockwise orientation at $z$; if $z=0$, then we assume that $I_\lambda$ is on the left of $0$ relative $\partial \bZ$ while $I_\rho$ is on the right of $0$, see Figure~\ref{Fig:I lambda B Lambda}.

There are unique bubble chains $\bB_\lambda$ and $\bB_\rho$ containing $I_\lambda\setminus \partial \bZ$ and $I_\rho\setminus \partial \bZ $ respectively. (If $I_\lambda\subset \bZ$,\footnote{this is in fact impossible by Lemma~\ref{lem:lakes lambda and rho}} then set $\bB_\lambda=\emptyset$, and similarly with $I_\rho$.) If $\partial \bZ\cap \partial \bO'=\emptyset$, then
 $\partial^c \bO'\subset \partial^c\bB_\lambda \cup  \partial^c\bB_\rho.$
 We now view $\lambda=\lambda(\bO')$ and $\rho=\rho(\bO')$ as sequences of positive power-triples parameterizing bubble chains as above, and we say that $\bB_\lambda$ and $\bB_\rho$ are the \emph{left and the right} bubble chains of $\bO'$.

\begin{lem}
\label{lem:lakes lambda and rho}
Consider a lake $\bO'$ of generation $P$ and let $\bB_\lambda$ and $\bB_\rho$ be the left and right bubble chains of $\bO'$. Then $\bB_\lambda$ and $\bB_\rho$ have generation $P$ as well. Moreover, one of the $\lambda$, $\rho$ is a finite sequence while another is an infinite sequence.

\begin{figure}[tp!]
\begin{tikzpicture}

\draw (-5,0)--(5,0);
\filldraw (-4,0) circle (1pt);
\node[above] at (-4,0){$z$};

\filldraw (4.5,0) circle (1pt);  
\node[above] at (4.5,0){$c_P$};

\draw (0,0)--(-1.5,2)--(0,4)--(1.5,2)--(0,0);

\node at (0,2){$\bZ_R$};

\filldraw (0,0) circle (1pt);  
\node[below] at (0,0){$c_R$};

\node[above] at (0,4){$\widetilde \alpha_P$};

\filldraw (-0.75,3) circle (1pt);  
\node[above left] at (-0.75,3) {$c_{(R,P-R)}$};

\node[ right] at (-1.5,2) {$c_{(R,Q)}$};

\filldraw[opacity=0.3]  (-1.5,2) -- (-2,2.3) -- (-2.5,2)--(-2,1.7) --(-1.5,2);
\node[above] at (-2,2.3){$\bZ_{(R,Q)}$};

\node[left] at(-2.5,2) {$\widetilde \alpha_{(R,Q)}$};

\filldraw[opacity=0.7] (-2,1.7) -- (-1.75,1.35)--(-2,1)--(-2.25,1.35)-- (-2,1.7);
\node[below ] at (-2,1) {$\bZ_{(R,Q,P)}$}; 

\end{tikzpicture}
\caption{Suppose a lake $\bO'$ touches $\partial \bZ$; say $z\in \overline \bO'\cap \partial \bZ$. If $c_P$ is on the right from $z$, then $\bB_{\rho}(\bO')$ is finite. Indeed, let $c_R$ be the critical point of generation $\le P$ such that $c_R$ is the closest to $z$ on the right. Assume $R<P$. Then, as shown on the figure, $c_{(R,P-R)}\subset \partial ^c \bZ_R$ is on the left of $c_R$. Let $c_{(P,Q)}$ be the critical point of generation $\le P-R$ closest to $c_R\in  \partial^c \bZ_R$ on the left.  Assume $Q<P-R$. Then $c_{(R,Q,P-R-Q)}\in \partial ^c \bZ_{(R,Q)}$ is on the left of $c_{(R,Q)}$. Assume now that $c_{(R,Q,P-R-Q)}$ is the closest to $c_{(R,Q)}\in \partial^c \bZ_{(R,Q)}$ on the left. Then $\bB_\rho(\bO')=\bB_{(R,Q,P-R-Q)}=\bZ_R\cup \bZ_{(R,Q)}\cup \bZ_{(R,Q,P-R-Q)}$.} 
\label{Fig:finite bubble chain}
\end{figure}

\end{lem}
\noindent Lemma~\ref{lem:lakes lambda and rho} is illustrated in Figure~\ref{Fig:finite bubble chain}.
\begin{proof}
\emph{Claim}. If $J\coloneqq \bO'\cap \bZ\not=\emptyset$, then $\bB_\lambda$ and $\bB_\rho$ are non-empty. 
\begin{proof}
By Corollary~\ref{cor:preimage:bZ:tree}, $J$ is an arc; choose $x\in J$. Both connected components of $\partial \bZ\setminus\{x\}$ contain infinitely many critical points of generation $\le P$ and these points are branch points of the tree $\bF^{-P}(\partial \bZ)$. Since $J$ does not cross such branch points, $\bB_\lambda$ and $\bB_\rho$ are non-empty.
\end{proof}

Suppose now that the generation $R$ of $\bB_\lambda$ is less than $P$. Choose $S$ with $R< S<P$ such that $\overline{\bF^S(\bO')}$ intersects $\partial \bZ$. The left bubble chain of $\bF^S(\bO')$ is $\bF^S(\bB_{\lambda})$ which is empty; this is impossible by the above \emph{Claim}. Since $\bO'$ is bounded by $\bF^{-P}(\partial \bZ)$, the generation of $\bB_\lambda$ can not be bigger than $P$. This shows that $\bB_\lambda$ (and similarly $\bB_\rho$) has generation $P$.

Suppose now that $\lambda$ and $\rho$ are finite; write $\lambda=(P_1,P_2,\dots, P_n)$ and $\rho=(Q_1,Q_2,\dots, Q_m)$. Then $P_1+\dots +P_n=Q_1+\dots + Q_m =P$. Choose $H\in \PT$ so that \[P>H>\max\{P_1+\dots + P_{n-1},Q_1+\dots + Q_{m-1}\}.\]
On the one hand,  $\bF^H\big(\widehat \bZ_\lambda\big)=\widehat \bZ_{P-H}$ (the unique bubble of generation $P-H$ attached to $\partial \bZ$) is the left bubble chain of the lake $\bF^H(\bO')$; and on the other hand $\bF^H\big(\widehat\bZ_\rho\big)=\widehat \bZ_{P-H}$ is the right bubble chain of $\bF^H(\bO')$ -- this is impossible.

Let us now prove that either $\lambda$ or $\rho$ is a finite sequence. 

\emph{Case 1:} assume $J=\overline \bO'\cap \overline\bZ\not=\emptyset$ and there is $y\in J$ such that the subarc $[y,c_P]$ of $\partial\bZ$ contains no critical points of generation $<P$. Since $[y,c_P)$ and $\partial^c\bZ_P\setminus \{c_P\} $ contain no branch points of the tree $\bF^{-P}(\partial \bZ)$, we obtain that $\widehat \bZ_P$ is ether $\bB_\lambda$ or $\bB_{\rho}$.

\emph{Case 2:} assume, more generally, $J=\overline \bO'\cap \overline\bZ\not=\emptyset$. Choose $y\in J$. There are at most finitely many critical points in $[y,c_P]$ of generation $<P$. Then the generation of these critical points is less than some $R<P$. Then $[\bF^R(y),c_{P-R}]$ contains no critical points of generation $<P-R$;~i.e.~$\bF^R(\bO')$ satisfies Case 1. Therefore, either $\bF^R(\bB_\lambda)$ or $\bF^R(\bB_\rho)$ is a finite chain; this implies that either $\bB_\lambda$ or $\bB_\rho$ is a finite chain.

For a general lake $\bO'$ consider $R<P$ such that $\overline{\bF^R(\bO')}$ intersects $\partial \bZ$. By Case 2, either $\bF^R(\bB_\lambda)$ or $\bF^R(\bB_\rho)$ is a finite chain, implying the desired.
\end{proof}

Consider a bubble $\bZ_s$ where $s=(P_1,P_2,\dots , P_n)\in  \PT^n_{>0}$, Write $P=P_1+\dots +P_n$. Let $\gamma\subset \bF^{-P}(\partial \bZ)$ be the unique tree-geodesic connecting $0$ to $c_{s}$ which is the root of $\bZ_s$. Denote by $\partial ^c_- \bZ_s$ and $\partial ^c_+ \bZ_s$ two connected components of $\partial ^c \bZ_s\setminus \{c_s\}$ enumerated so that the triple $\partial ^c_- \bZ_s,\gamma, \partial ^c_+ \bZ_s$ has counterclockwise orientation at $c_s$.

There are unique lakes $\bO_-(s)$ and $\bO_+(s)$ of generation $P$ such that \[\partial^c\bO_{-}(s)\supset \partial ^c_- \bZ_s\sp \text{ and }\sp \partial^c\bO_{+}(s)\supset \partial ^c_+ \bZ_s\] because $\partial ^c_- \bZ_s$ and $\partial ^c_+ \bZ_s$ have no critical points of generation $\le P$. We define $\bB_{\lambda(s)}$ to be the left bubble chain of $\bO_-(s)$ and we define $\bB_{\rho(s)}$ to be the right bubble chain of $\bO_+(s)$, see Figure~\ref{Fig:Blambda Brho}.

We call $\bO_{-}(s)$ the \emph{left lake} of $\bZ_s$ and we call $\bO_{-}(s)$ the \emph{right lake} of $\bZ_s$. By construction, $\bB_{\lambda(s)}$ and $\bB_{\rho(s)}$ are the closest to $\bB_s$ bubble chains of generation at most $P$. The next lemma implies that every lake is of the form $\bO_\pm(s)$ for a unique $s$.

\begin{lem}
\label{lem:closest BChains}
For every $\bZ_s$ the bubble chains $\bB_{\lambda(s)}$ and $\bB_{\rho(s)}$ are infinite chains of generation $|s|=P_1+P_2+\dots+P_n$.

For every proper lake $\bO'$ there is a unique $\bZ_s,\sp s=(P_1,\dots ,P_n),\sp n\ge1$ such that either $\bO'=\bO_-(s)$ or $\bO'=\bO_+(s)$.
\end{lem}
\begin{proof}
The first claim follows immediately from Lemma~\ref{lem:lakes lambda and rho}. 

By Lemma~\ref{lem:lakes lambda and rho}, either the left bubble chain $\bB_\lambda$ of $\bO'$ or its right bubble chain $\bB_\rho$ is a finite chain. In the former case, $\bO'=\bO_+(\lambda)$; in the latter case, $\bO'=\bO_-(\rho)$.
\end{proof}

\begin{figure}[t!]\begin{tikzpicture}[scale=0.85]

\draw[blue] (-4.68,-3.4)-- (1.66,3.58);
\draw[blue] (1.66,3.58)-- (5.7,0.64);
\draw[red] (-1.09977812767106,0.5416322821539432)-- (-2.54,0.86);
\draw[red] (-2.54,0.86)-- (-2.39,2.06);
\draw[red] (-0.88,1.73)-- (-2.39,2.06);
\draw[red] (-1.09977812767106,0.5416322821539432)-- (-0.88,1.73);
\draw (-3.9454909352647434,-2.5913449098025096)-- (-6.38,-2.4);
\draw (-6.38,-2.4)-- (-6.35,0.29);
\draw (-6.35,0.29)-- (-3.94,0.18);
\draw (-3.94,0.18)-- (-3.9454909352647434,-2.5913449098025096);
\draw (-4.347926861123204,0.19861906835002177)-- (-4.5,0.92);
\draw (-3.63,1.3)-- (-4.5,0.92);
\draw (-3.63,1.3)-- (-3.38,0.62);

\draw (-3.38,0.62)-- (-4.347926861123204,0.19861906835002177);
\draw (0.6993740609114214,2.5224023572810297)-- (0.58,4.52);
\draw (0.6993740609114214,2.5224023572810297)-- (-0.76,2.46);
\draw (-0.76,2.46)-- (-0.92,4.36);
\draw (-0.92,4.36)-- (0.58,4.52);
\draw (-0.84947959071405,3.522570139729343)-- (-1.75,3.87);
\draw (-1.75,3.87)-- (-2.27,3.07);
\draw (-2.27,3.07)-- (-1.43,2.66);
\draw (-0.84947959071405,3.522570139729343)-- (-1.43,2.66);
\draw (-3.5552183406113538,1.0965938864628821)-- (-3.37,1.5);
\draw (-3.06,1.39)-- (-3.37,1.5);
\draw (-3.06,1.39)-- (-3.19,1.02);
\draw (-3.19,1.02)-- (-3.5552183406113538,1.0965938864628821);

\draw (-2.0031441048034937,2.9397489082969432)-- (-2.28,2.73);
\draw (-2.28,2.73)-- (-2.02,2.52);
\draw (-2.02,2.52)-- (-1.76,2.75);
\draw (-1.76,2.75)-- (-2.0031441048034937,2.9397489082969432);

\draw (-3.06,1.39)-- (-2.8,1.49)--(-2.75,1.75)--(-2.99,1.75)--(-3.06,1.39);

\draw[dashed] (-2.02,2.52) -- (-2.39,2.06);
\draw[dashed] (-2.75,1.75) -- (-2.39,2.06);

\node[blue] at(1.12, -0.85){$\bZ_{P_1,\dots,P_{n-1} }$};
\node[red] at(-1.76, 1.25){$\bZ_s$};
\node at(.-0.4, 2.05){$\bO_+(s)$};
\node at (-5.24, -1.15){$\bZ_{\dots,Q_n}$};
\node at (-4.04, -0.){$c_{\dots,Q_{n+1}}$};
\node at (-3., -0.5){$\bO_-(s)$};
\node at(-0.05, 3.38){$\bZ_{\dots, R_n}$};
\node at (-2.66, 3.47){$\bB_{\rho(s)}$};
\node at (-4.43, 1.58) {$\bB_{\lambda(s)}$};
\node[red] at (-0.71, 0.26){$c_s$};
\node at (-3.4, -2.89) {$c_{\dots,Q_n}$};
\end{tikzpicture}
\caption{The bubble chains $\bB_{\lambda}$ and $\bB_{\rho}$ are the closest to $\bZ_{s}$ of generation $\le P$ (see also Figure~\ref{Fig:GeomScal}). By Lemma~\ref{lem:alpha' is a pnt}, $\bB_{\lambda}$ and $\bB_{\rho}$ actually land at the top of $\bZ_s$.}\label{Fig:Blambda Brho}  \end{figure}

For a finite or infinite sequence $s=(P_1,P_2,\dots)$ let us write $\tt s\coloneqq (\tt P_1,\tt P_2,\dots )$
\begin{lem}
\label{lem:prop of Astr}
The self-similarity $A_\str$ preserves the Fatou and Julia sets of the fixed maximal prepacman $\bF$; moreover for every finite sequence $s$ we have:
\begin{itemize}
\item $A_\str(c_s)=c_{\tt s}$;
\item $A_\str(\bZ_s)=\bZ_{\tt s}$;
\item $A_\str(\bB_s)=\bB_{\tt s}$, here $s$ is either a finite or an infinite sequence;
\item $A_\str(\bL_s)=\bL_{\tt s}$; 
\item $A_\str(\widetilde \alpha_s)=\widetilde \alpha_{\tt s}$.
\end{itemize}
\end{lem}
\begin{proof}
Recall that $A_\str$ conjugates $\bF^P$ to $\bF^{\tt P}$; hence $A_\str$ preserves the Fatou and Julia sets. Since $A_\str (\partial \bZ)=\partial \bZ$, we have $A_\str(c_P)=c_{\tt P}$. As a consequence, $A_\str (\bZ_{P})=\bZ_{\tt P}$ and $A_\str(\widetilde \alpha_P)=\widetilde \alpha_{\tt P}$. Since $A_\str (\partial^c \bZ_{P})=\partial^c \bZ_{\tt P}$, we also have $A_{\str}(c_{(P,Q)})=c_{(\tt P, \tt Q)}$ and we can proceed by induction on $|s|$.
\end{proof}

\subsection{Boundedness of limbs}

\begin{lem}
\label{lem:limbs are bounded}
The closure of every limb is compact.
\end{lem}
\begin{proof}
The idea of the proof is illustrated in Figure~\ref{Fg:lem:limbs are bounded}. We will show that there is an $R\in \PT_{>0}$ such that $\bF^R\left(\bS^{\#}_{0}\setminus \bZ\right)\supset \bS^{\#}_{-1}\setminus \bZ$. By self-similarity, $\bF^{R/\tt}\left(\bS^{\#}_{-1}\setminus \bZ\right)\supset \bS^{\#}_{-2}\setminus \bZ$. Therefore, if $Q\ge R+R/\tt+R/\tt^2+\dots$, then $\bF^Q\left(\bS^{\#}_{0}\setminus \bZ\right)=\C$. This implies that some limbs are bounded. Therefore, all the limbs are bounded because they are dynamically related. Let us provide more details.

\begin{figure}[tp!]
\[\begin{tikzpicture}[scale=1.3]
\draw[red] (-3.,1.)-- (5.64,1.);
\draw (1.2,0.7)-- (1.18,1.5)-- (2.,1.5)--(2,0.7);


\draw [
   red,  thick,
    decoration={
        brace,
        mirror,
        raise=0.04cm
    },
    decorate
](1.28,1.)-- (1.9,1.) node[pos=0.5,anchor=north,yshift=-0.08cm] {$\bJ_0$};

\draw [
   red, shift ={(1.6,1)}, scale=3,shift ={(-1.6,-1)}, thick,
    decoration={
        brace,amplitude=16pt,mirror,        
        raise=0.3cm
    },
    decorate
](1.28,1.)-- (1.9,1.) node[pos=0.5,anchor=north,yshift=-0.75cm] {$\bJ_{-1}$};

\draw [
   red, shift ={(1.6,1)}, scale=9,shift ={(-1.6,-1)}, thick,
    decoration={
        brace,amplitude=40pt,mirror,        
        raise=0.3cm
    },
    decorate
](1.28,1.)-- (1.9,1.) node[pos=0.5,anchor=north,yshift=-1.65cm] {$\bJ_{-2}$};


\node at (1.5,1.25) {$\bS^{\#}_0$};

 \draw (1.8,1.4) edge[->,bend left] node[above]{$\bF^R$}(2.5,1.8);

\node at (0.9,2) {$\bS^{\#}_{-1}$};

 \draw[shift ={(1.6,1)}, scale=3,shift ={(-1.6,-1)}]  (1.6,1.42) edge[->,bend left] node[above]{$\bF^{R/\tt}$}(2.5,1.8);

 \draw[shift ={(1.6,1)}, scale=9,shift ={(-1.6,-1)}]  (1.6,1.39) edge[->] node[above]{$\bF^{R/\tt^2}$}(1.9,1.52);
 
\node at (-0.9,4) {$\bS^{\#}_{-2}$};

\node[red] at (-0.9,1.2) {$\partial \bZ$};

\draw[shift ={(1.6,1)}, scale=3,shift ={(-1.6,-1)}] (1.2,0.8)-- (1.18,1.5)-- (2.,1.5)--(2,0.8);

\draw[shift ={(1.6,1)}, scale=9,shift ={(-1.6,-1)}] (1.2,0.9)-- (1.18,1.5)-- (2.,1.5)--(2,0.9);

\end{tikzpicture}\]
\caption{Illustration to the proof of Lemma~\ref{lem:limbs are bounded}: it takes $R/\tt^n$ iterates to cover $\bS_{-n-1}^{\#}\setminus \bZ$ from a domain in $\bS_{-n}^{\#}\setminus \bZ$.}
\label{Fg:lem:limbs are bounded}
\end{figure}

Consider the dynamical plane of $f\colon U \to V$. For every $n$, there is a gluing map $\rho_n \colon \bS_{n}^\#\to V\setminus \{\alpha\}$ projecting $\bF^\#_{n}$ to $f$. The map $\rho_n$ glues two distinguished sides of $\bS_{n}^\#$ to $\gamma_1$; the preimage of $\alpha$ is at infinity. Since $\bF$ is a renormalization fixed point, $\bS^\#_{n-1}$ is a rescaling of $\bS^\#_{n}$; we also recall that: 
\begin{itemize}
\item $\bS^\#_{n}\subset \bS^\#_{n-1}$; and
\item $\displaystyle\bigcup_n \bS^\#_{n}=\C$.
\end{itemize}

Choose a big $n\ll0$ and let $X$ and $Y$ be the open sectors in the dynamical plane of $f$ obtained by projecting $\bS^\#_0$ and $\bS_{-1}^{\#}$ via $\rho_n\colon \bS_n^{\#}\to V$, see Figure~\ref{Fg:cl1:lem:limbs are bounded}. Write $W\coloneqq Y\setminus \overline Z_\str$ and $I\coloneqq X\cap \partial Z_\str$ (depicted in blue bold in Figure~\ref{Fg:cl1:lem:limbs are bounded}); and let $J$ be a slightly shrunk version of $I$.

\begin{figure}[t!]
\[\begin{tikzpicture}[scale=0.6]
\draw (6.5, 6.62)--(5.57, 7.83)-- (-7.65, -4.19)-- (-6.5, -5.34);

\draw[blue,line width=0.5mm] (-1.8,1.13) -- (0.44,3.16);

\draw (4.42, -2.38)-- (-3.4343346716619574,2.0593418593933244);
\draw (4.42, -2.38)-- (-0.5944502811735275,4.596035745559144);
\draw (-0.5944502811735275,4.596035745559144)-- (-3.4343346716619574,2.0593418593933244);


\draw (-0.42, 10.58)--(-9.22, 0.98)-- (4.42, -2.38)  --(-0.42, 10.58);

\draw[] plot[shift={(-0.5641312307412661,2.8375308204880154)},domain=3.35268598681254:5.385892693271996,variable=\t]({1.*0.43410273889756196*cos(\t r)+0.*0.43410273889756196*sin(\t r)},{0.*0.43410273889756196*cos(\t r)+1.*0.43410273889756196*sin(\t r)})
(-0.9062147696931351,2.570278055681868)-- (-1.402958292700482,2.1705117109782783)
-- (-0.8907152326350126,1.9557134248772614);
\draw[draw opacity=0, fill=red, fill opacity=0.2]
(-0.29, 2.5)--(-0.9062147696931351,2.570278055681868)-- (-1.402958292700482,2.1705117109782783)
-- (-0.8907152326350126,1.9557134248772614)--(-0.29, 2.5)--(-0.4277, 2.41559)--(-0.57424, 2.3878)--(-0.76626, 2.4257)--(-0.91, 2.57);

\draw [
   red,  thick,
    decoration={
        brace,
        mirror,
        raise=0.0cm
    },
    decorate
](-1.41, 1.26) -- (0.24, 2.74)node[pos=0.5,anchor=north,yshift=-0.cm,xshift=0.2cm] {$J$};

\draw (-0.93807, 2.28168) edge[-> ,bend left] node[left]{$f^m$} (-1.7, 5.86);

\node[red] at (-0.62, 3.19) {$W_{-m}$};

\node at (2.32628, 1.4) {$Z_\str$};
\node at (-2.6, 2.2) {$X$};

\node at(-1.62, 6.8) {$W=Y\setminus \overline Z_\str$};
\end{tikzpicture}\]
\caption{Illustration to Claim~\ref{cl1:1}: the disk $W=Y\setminus \overline Z_\str$ pullbacks along the orbit of $x$ to the disk $W_{-m}\subset X$.}
\label{Fg:cl1:lem:limbs are bounded}
\end{figure}

\begin{claim}
\label{cl1:1}
There is an $M>0$ such that the following property holds. If $m\ge M$, $x\in J$, and $f^m(x)\in \partial W$, then $W$ has a conformal pullback $W_{-m}$ along the orbit $x,f(x),\dots, f^m(x)\in W$ such that $W_{-m}\subset X$.
\end{claim} 
\noindent We remark that if $x$ is a critical point, then there are two choices for $W_{-m}$: on the left and on the right of $c_0$.
\begin{proof}
Let $W_{-k}$ be the pullback of $W$ along the orbit of $f^{-k}(x),\dots, f^{m}(x)$. Let us show that $W_{-k}$ does not intersect the forbidden boundary $\partial^\frb U$ for all ${k\in  \{m-1,\dots ,0\}}$; this will imply that $W_{-m}$ is a conformal pullback.

Recall from~\S\ref{sss:SiegPacmen} that $f$ has a lamination by external rays. Choose two rays $R_-$ and $R_+$ landing at $\overline Z_\str$ such that $R_-$ is slightly on the left of $W$ and $R_+$ is slightly on the right of $W$. Since $n\ll 0$, the difference $\delta$ between the external angles of $R_-$ and $R_+$ is small. 

Let $R_{-k,-}$ and $R_{-k,+}$ be the preimages of $R_-$ and $R_+$ under $f^{m-k}$ such that $R_{-k,-}$ is slightly on the left of $W_{-k}$ and $R_{-k,+}$ is slightly on the right of $W_{-k}$. Then the difference between the external angles of $R_{-k,-}$ and $R_{-k,+}$ is $\delta/2^k$; i.e.~$W_{-k}$ has a small angular size. Recall that $\gamma_-\cup \gamma_+ \setminus \overline Z'_\str$ are external rays that are disjoint from $\overline Z_\str$, see~\S\ref{sss:SiegPacmen}. Since $\partial W_{-k}$ intersects $\partial Z_\str$, we obtain that $\partial W_{-k}$ is disjoint from  $\partial^\frb U$. This shows that $f^m\colon W_{-m}\to W$ is conformal.

If $m$ is big, then $W_{-m}$ is contained in a small neighborhood of the non-escaping set $K$ between $R_{-m,-}$ and $R_{-m,+}$. Since $K$ is locally connected and the difference between the external angles of $R_{-m,-}$ and $R_{-m,+}$ is small, the set $W_{-m}$ is contained in a small neighborhood of $J$; hence $W_{-m}\subset X$.
\end{proof}

Suppose $\bJ_n$ corresponds to $J$ under the identification $\bS_n^{\#} \simeq X$, see Figure~\ref{Fg:lem:limbs are bounded}.  As a corollary of Claim~\ref{cl1:1} we have:

\begin{claim}
\label{cl1:2}
There is power-triple $R> 0$ with $\bF^R(\bJ_0)\subset \bJ_{-1}$ such that the following property holds. If $x\in \bJ_0$ and $\bF^P(x)\in \bS_{-1}^\#$ for some $P\ge R$, then there is an open set $W_P\subset \bS_{0}$ with $x\in \partial W_P$ such that $\bF^P$ maps $ W_P$ conformally onto $\intr\left(\bS_{-1}^\#\setminus \bZ\right)$. 
\end{claim} 
\begin{proof}
By Lemma~\ref{lem:irr rot of bZ}, the cascade  $\big(\bF^P\mid  \partial \bZ\big)_{P\in \PT}$ is conjugate to the cascade of translations $\big(T^P \mid \R \big)_{P\in \PT}$. Therefore, we can choose a sufficiently big $R$ with $\bF^R(\bJ_0)\subset \bJ_{-1}$.

Since $\bF^P(x)\in \bS^\#_{-1}$, we can write 
\[\bF^P = \bbf^\#_{-1,\iota(1)}\circ\bbf^\#_{-1,\iota(2)}\circ\dots \circ \bbf^\#_{-1,\iota(m)},\sp\sp \iota(j)\in\{-,+\}\]
because the prepacman $\bbf^\#_{-1,\pm}\colon \bU^\#_\pm \to \bS^\#_{-1}$ realizes the first return of the cascade $\bF^{\ge0}$. If $R$ is sufficiently big, then $m\ge M$, where $M$ is the constant from Claim~\ref{cl1:1}. The statement now follows from Claim~\ref{cl1:1}.
\end{proof}

By self-similarity, we can shift indices in Claim~\ref{cl1:2}: we can replace $\bJ_0, \bJ_{-1},\bS^\#_{-1}$ by $\bJ_n, \bJ_{n-1},\bS^\#_{n-1}$ and replace $R$ by $R/\tt^{-n}$. Inductively applying Claim~\ref{cl1:2}, we obtain:
\begin{claim}
\label{cl1:3}
There is power-triple $Q> 0$ such that for every $n<0$ the following property holds. Let $x\in \bJ_0$ be a point such that $\bF^P(x)\in \bS_{n}^\#$ with $P\ge Q$. Then there is an open set $W\subset \bS_{0}$ with $x\in \partial W$ such that $\bF^P$ maps $ W$ conformally to $\intr\left(\bS_{n}^\#\setminus \bZ\right)$.
\end{claim} 
\begin{proof}
Let $R$ be a power-triple from Claim~\ref{cl1:2}. Choose a power-triple $Q$ with
\[Q> R+ R /t+ R /t^2+\dots, \]
see Lemma~\ref{lem:power-triples:order}. Write $x_0=x$ and for $j\in \{0,-1,\dots, n+2\}$ inductively set $P_j\coloneqq R/\tt^{-j}$ and $x_{j-1}\coloneqq \bF^{P_j}(x_j)$. Finally set 
\[P_{n+1}\coloneqq P-P_0-P_1-\dots -P_{n+2}\ge R/\tt^{n+1}.\]
and $x_n=\bF^P_{n+1}(x_{n+1})$. Since $P_{n+1}\ge R/t^{-n-1}$, we can apply Claim~\ref{cl1:2} and construct $W_{n+1}\subset \bS_{n+1}^{\#}\setminus \bZ$ so that $x_{n+1} \in \partial W_{n+1}$ and $\bF^{P_{n+1}}$ maps $W_{n+1}$ 
conformally to $\bS_{n}^{\#}\setminus \bZ$. Applying induction from upper levels to deeper ones, we construct $W_j\subset \bS_{j}^{\#}\setminus \bZ$ such that $x_{j} \in \partial W_{j}$ and 
$\bF^{P_{j}}$ maps $W_{j}$ conformally to $W_{j-1}$. Therefore, $\bF^P$ maps $W_0$ conformally to $\bS_{n}^{\#}\setminus \bZ$.
\end{proof}

Consider a limb $\bL_M$; recall that its root is denoted by $c_M$. Choose a big $T\ge Q+M $ such that the critical point $c_T$ is in $\bJ_0$. Then $c_T$ is the root of $\bL_T$ and $\bF^{T-M}$ maps $\bL_T$ to $\bL_M$.

By Claim~\ref{cl1:3}, for all $n\ll 0$ the connected component of $\bS_n^{\#}\cap \bL_M$ containing $c_M$ can be pulled back along ${\bF^{T-M}\colon c_T \mapsto c_M}$ and, moreover, the pullback is within $\bS_0$. Since $n\ll 0$ is arbitrary, the pullbacks of $\bS_n^{\#}\cap \bL_M$ exhaust $\bL_T$, and we obtain that $\bL_T\subset \bS_0$.

By self-similarity, the limb $\bL_{T/\tt^{m}}$ is within $\bS_{-m}^\#$ for all $m\ge 0$. For every $H> 0$ choose an $m\ge 0$ with $H\ge T/\tt^{m}$. Then  $\bF^{H-T/\tt^m}$ maps $\bL_H$ conformally to $\bL_{T/\tt^{m}}$. Since $\bL_{T/\tt^{m}}$ is bounded, so is $\bL_H$ by $\sigma$-properness.
\end{proof}

\begin{figure}[tp!]
\begin{tikzpicture}
\draw (-5.5,0)--(5.5,0);
\node[below] at(2.8,0) {$\partial \bZ$};

\draw (-5,0)--(-4.5,2)--(-3.5, 3)--(-2,3.8)--(0,4);

\draw ( 5,0)--( 4.5,2)--( 3.5, 3)--( 2,3.8)--(0,4);

\node[above] at (-1, 3) {$\bO_1$};

\draw (-0.7, 3.3) edge[ line width=0.8pt,bend left, ->] node[above]{$\bF^{(\tt-1)P}$} (1.7,4.7);

\node[above] at(2,4.55) {$\bO$};

\begin{scope}[shift={(-0.7,0)},scale=0.8,xscale=0.8]
\filldraw[opacity=0.5] (-2,0)--(-2.3,0.5) --(-2,1)--(-1.7,0.5) --(-2,0);
\node[below] at(-2,0) {$\bZ_{\tt P}$};

\draw[ line width=1pt,blue] (-3.2,0)--(-3.1,0.8)--(-2.6,1.1)--(-2.2,1.1)--(-2,1);
\node[blue,left] at (-3.1,0.8){$\bB_{\tt \lambda}$}; 

\draw[ line width=1pt,blue] (-0.8,0)--(-0.9,0.8)--(-1.4,1.1)--(-1.8,1.1)--(-2,1);
\node[blue,right] at (-0.9,0.3){$\bB_{\tt \rho}$};
\end{scope}

\begin{scope}[yscale=0.8,scale=3,shift={(2,0)}]
\filldraw[opacity=0.2] (-2,0)--(-2.2,0.5) --(-2,1)--(-1.8,0.5) --(-2,0);
\node at(-2,0.2) {$\bZ_P$};

\node[above] at(-2,1) {$\alpha_P$};

\draw[ line width=2pt,blue] (-3.2,0)--(-3.1,0.8)--(-2.6,1.1)--(-2.2,1.1)--(-2,1);
\node[blue,left] at (-3.1,0.8){$\bB_\lambda$}; 

\draw[ line width=2pt,blue] (-0.8,0)--(-0.9,0.8)--(-1.4,1.1)--(-1.8,1.1)--(-2,1);
\node[blue,right] at (-0.9,0.8){$\bB_\rho$};
\end{scope}

\draw (-1.98,0.7) edge[ line width=0.8pt, ->,bend left] node[above]{$\bF^{(\tt-1)P}$}(0,2);

\draw (-1.98,0.3) edge[ line width=0.8pt, <-,red] node[above]{$A_\str$} (0,1);

\begin{scope}[shift={(0,-5)}]

\draw(-4,0)--(4,0);
\filldraw[opacity=0.5,blue,scale=2,shift={(1,0)}]  (-2,0)--(-2.25,0.5) --(-2,1)--(-1.75,0.5) --(-2,0);
\filldraw[opacity=0.5,blue] (-1.75, 1.5) -- (-1.5,1.4)-- (-1.3,1.8)--(-1.55,1.9)-- (-1.75, 1.5);
\filldraw[opacity=0.5,blue](-1.3,1.8)-- (-1.1,1.7)--(-0.9,1.8)--(-1.1,1.9)--(-1.3,1.8);
\filldraw[opacity=0.5,blue](-0.9,1.8)-- (-0.8,1.75)--(-0.7,1.8)--(-0.8,1.85)--(-0.9,1.8);
\node[blue,above]at (-1.1,1.9){$\bB_\lambda$};

\draw[shift={(0,2)}]  (-2,0)--(-2.25,0.5) --(-2,1)--(-1.75,0.5) --(-2,0);

\node[left]at(-2.25,2.5) {$\bZ_{Q_1,R_-}$};

\draw[shift={(-0.2,-0.8)}](-1.3,1.8)-- (-1.1,1.7)--(-0.9,1.8)--(-1.1,1.9)--(-1.3,1.8);

\node[below right,shift={(-0.1,-0.8)} ]at(-1.1,1.7) {$\bZ_{Q_1,R_+}$};

\draw[red,line width =0.9pt]  (-1.1,1)--(-0.8,1.4)--(-0.7,1.8)--(-0.9,2.4)--(-1.5,2.8)--(-2,3);

\node[ above right,red]at (-1.5,2.8){$\Esc_S(\bF)$};
\node[below] at (2.5,0) {$\partial \bZ$};

\node[left,blue] at (-2.5,1) {$ \bZ_{Q_1}$};
\end{scope}

\end{tikzpicture}
 \caption{Illustration to the proof of Lemma~\ref{lem:alpha' is a pnt} (see also Figure~\ref{Fig:GeomScal}): the accumulating set of $\bZ_P, \bB_{\lambda}, \bB_{\rho}$ is invariant under the expanding map $\bF^{(\tt-1)P}\circ A_\str \colon \bO_1\to \bO$; thus $\bZ_P, \bB_{\lambda}$, and $\bB_{\rho}$  land at $\alpha_P$. Bottom: $\Esc_S(\bF)\cup \bZ_{Q_1,R_-}\cup \bZ_{Q_1}\cup \bZ_{Q_1,R_+}$ separates the accumulating set of $\bB_\lambda$ from $\partial \bZ$.
   }
 \label{Fig:B B Z land together}
\end{figure}

\subsection{Alpha-points}
\label{ss:alpha-points}
Consider a finite sequence $s=(P_1, P_2,\dots ,P_n)$ in $\PT_{>0}$ and the corresponding bubble $\bZ_s$. Write $P=|s|=P_1+\dots +P_n$.
Recall from~\eqref{eq:defn:accum set} that $\widetilde \alpha_s$ denotes the accumulating set of $\bB_s$.
By Lemma~\ref{lem:limbs are bounded}, $\widetilde \alpha_s$ is a compact subset of $\Esc_P(\bF)$, which is disjoint from $\partial \bZ$.

\begin{lem}
\label{lem:alpha' is a pnt} The set $\widetilde \alpha_s =\{\alpha_s\}$ is a singleton. Moreover, $\alpha_s$ is the landing point of $\bB_s,\bB_{\lambda(s)},\bB_{\rho(s)}$. We have: 
\[\partial \bO_-(s)=\partial ^c \bO_-(s)\cup \widetilde \alpha_s\sp \text { and }\sp  \partial \bO_+(s)=\partial ^c \bO_+(s)\cup \widetilde \alpha_s,\]
 $ \bO_-(s)$ and $ \bO_+(s)$ are bounded sets, and  $\overline \bZ_s$, $\overline \bO_-(s)$, $\overline \bO_+(s)$ are closed topological disks. 
\end{lem}
\begin{proof}

Consider a critical point $c_P\in \partial \bZ$ of generation $P$. Since all $(\bO_-(s),\bZ_s,(\bO_+(s))$ are dynamically related it is sufficient to verify the statement for $(\bO_-(P),\bZ_P,(\bO_+(P))$.

By Lemma~\ref{lem:prop of Astr}, $c_{\tt P}=A_{\str} (c_{ P})$,  and 
\begin{equation}
\label{eq:1:lem:alpha' is a pnt}
A_{\str}\sp  \text{ maps }\sp  (\bO_-(P),\bZ_P,\bO_+(P))  \sp\text{ to }\sp  (\bO_-(\tt P),\bZ_{\tt P},\bO_+(\tt P)).
\end{equation}
 On the other hand (see Figure~\ref{Fig:B B Z land together}), by the classification of bubbles attached to $\bZ$:
 \begin{equation}
\label{eq:2:lem:alpha' is a pnt}
 \bF^{(\tt-1)P}\sp \text{ maps }\sp(\bO_-(\tt P),\bZ_{\tt P},\bO_+(\tt P))\sp\text{ to }\sp(\bO_-(P),\bZ_P,\bO_+(P)).
 \end{equation}

Let $\bO_1$ be the lake of generation $(\tt-1)P/\tt$ containing $\bO_-( P)\cup \bZ_{P}\cup \bO_+(P)$. Then the lake $A_{\str}(\bO_1)\supset \bO_-(\tt P)\cup \bZ_{\tt P}\cup \bO_+(\tt P)$ has generation $(\tt-1)P$. By the Shwarz lemma, \begin{equation}
\label{eq:scal+map}
\bF^{(\tt-1)P}\circ A_{\str}\colon \bO_1\to \bO\sp\sp \text{ (where $\bO=\C\setminus \bZ$ is an ocean)}
\end{equation}
expands the hyperbolic metric of $\bO$. Since $\widetilde \alpha_P$ is a set  compactly contained in $\bO$ (because $\widetilde \alpha_P$ is a compact subset of $\Esc_{P}(\bF)$, which is disjoint from $\partial \bZ$) and since  $\widetilde \alpha_P$ is invariant  under~\eqref{eq:scal+map}, we see that $\widetilde \alpha_P =\{\alpha_P\}$ is a singleton and $\alpha_P$ is a repelling fixed point of~\eqref{eq:scal+map}.

\begin{figure}[tp!] 
\centering{\begin{tikzpicture}
  \node at (0,0){\includegraphics[scale=0.6]{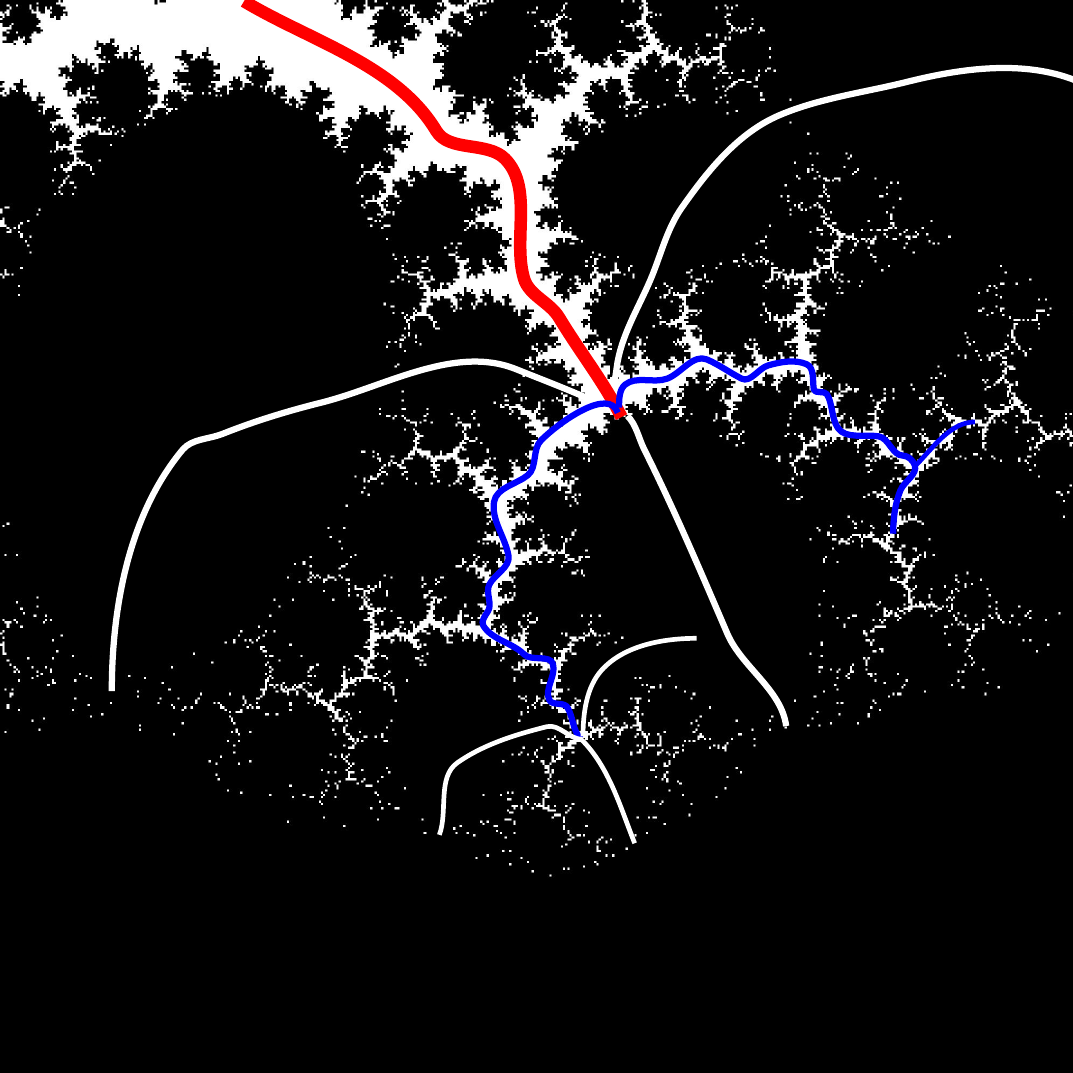}};
 \draw[white] node at (1.7,0.7) {$\bZ_P$};   
 \draw[white] node at (-2.7,1.7) {$B_\lambda$};   
 \draw[white] node at (1.1,3.5) {$B_\rho$};

 \end{tikzpicture}}
 \caption{Illustration to the proof of Lemma~\ref{lem:alpha' is a pnt} (compare with Figure~\ref{Fig:unst man+max sieg}): $\bZ_P, B_{\lambda}, B_{\rho}$ land at $\alpha_P$ that belongs to $\Esc_P(\bF)$ (partially marked red). There are two branches of $\Esc_{\tt P}(\bF)\setminus \Esc_{P}(\bF)$ (partially marked blue) at $\alpha_P$. One of the ends of the left branch is $\alpha_{\tt P}$; this point is the landing point of $\bZ_{\tt P}, B_{\tt \lambda }, B_{\tt \rho}$.}
 \label{Fig:GeomScal}
\end{figure}
Let $\widetilde \alpha^{-}_P$ be the accumulating set of $\bB_{\lambda(s)}$. Let us argue that $\widetilde \alpha^{-}_P\Subset \bO$. By Lemma~\ref{lem:limbs are bounded}, $\widetilde \alpha^{-}_P$ is a compact subset of $\C$. Write ${\lambda(s)}=(Q_1,Q_2,\dots)$ and choose two $R_-$ and $R_+$ such that
$\bZ_{Q_1,R_-}$ is on the left of $\bZ_{Q_1,Q_2}$ while $\bZ_{Q_1,R_+}$ is on the right of $\bZ_{Q_1,Q_2}$, see Figure~\ref{Fig:B B Z land together} (bottom). Set $S=\max\{Q_1+R_-,Q_2+R_+\}$. Since every connected component of $\Esc_P(\bF)$ is unbounded (see~\S\ref{ss:Fat Jul  Esc}),
 \[\Esc_{S}(\bF)\cup \widehat\bZ_{Q_1} \cup \widehat \bZ_{Q_1,R_-} \cup \widehat\bZ_{Q_1,R_+}\] separates $\widetilde \alpha^{-}_P\setminus \Esc_{S}(\bF)$ from $\bZ$.

Since $\widetilde \alpha^{-}_P$ is also invariant under~\eqref{eq:scal+map} (because of~\eqref{eq:1:lem:alpha' is a pnt} and~\eqref{eq:2:lem:alpha' is a pnt}), we obtain that $\widetilde \alpha^{-}_P=\widetilde \alpha_P=\{\alpha_P\}$; i.e.~$\bB_{\lambda(s)}$ lands at $\alpha_P$. Similarly, $\bB_{\rho(s)}$ lands at $\alpha_P$. As a consequence, $\overline \bZ_s$, $\overline \bO_-(s)$, $\overline \bO_+(s)$ are closed topological disks. 
\end{proof}

We say that $\{\alpha_s\}$ are \emph{alpha-points}. They are viewed as preimages of $\alpha$: 
\begin{lem}
\label{lem:preim of alpha}
Suppose $\gamma\colon [0,1)\to \bZ$ is a curve that goes to $\binfty$.  Let $\Gamma=\{\gamma_i\}$ be the set of lifts of $\gamma$ under $\bF^T$, where $T\in \PT_{>0}$. There is a unique lift $\gamma_0\in \Gamma$ such that $\gamma_0\subset \bZ$. Every remaining $\gamma_i\in \Gamma\setminus \{\gamma_0\}$ is within $\bZ_s$ with $0<|s|\le T$ and, moreover, $\gamma_i$ lands at $\alpha_s$. Conversely, for every $\alpha_s$ with $|s|\le T$ there is a unique $\gamma_i\in \Gamma$ such that $\gamma_i$ lands at $\alpha_i$.
\end{lem}
\begin{proof}
Since $\bZ_s$ with $|s|\le T$ are univalent preimages of $\bZ$, every $\gamma_i$ is contained in a certain $\bZ_s$. Moreover, $\gamma_i$ accumulates at $\widetilde \alpha_s=\{\alpha_s\}$ which is a singleton by Lemma~\ref{lem:alpha' is a pnt};~i.e.~$\gamma_i$ lands at $\alpha_s$. 
\end{proof}

\begin{lem}
\label{lem:alpha_s= alpha_v}
If $\alpha_s=\alpha_v$, then $s=v$. 
\end{lem}
\begin{proof}
If $\alpha_s$ and $\alpha_v$ have different generations, then $\alpha_v\not=\alpha_s$. Suppose that $|s|=|v|$ and write $s=(P_1,P_2,\dots, P_n)$ and $v=(Q_1,Q_2,\dots, Q_n)$. Choose $R<|s|$ such that $R>\max\{P_1+\dots + P_{n-1}, Q_1+\dots + Q_{m-1},\}$. Then $\bF^R(\bZ_s) = \bF^R(\bZ_v)=\bZ_{|s|-R}$. Since $\overline \bZ_{|s|-R}$ does not contain a critical value of $\bF^{R}$, we obtain that $\bZ_s$ and $\bZ_v$ are different degree one preimages of $\bZ_{|s|-R}$.
\end{proof}

\begin{cor}[The tree structure of $\{\overline \bZ_s\}$]
\label{cor:tree str of closed bubbles}
Suppose that $\overline \bZ_v\cap \overline \bZ_w\not=\emptyset$ for $|w|\ge |v|$ and $w\not=v$. Then $c_w\in \partial \bZ_v$ and $\overline \bZ_v \cap \overline \bZ_w=\{c_w\}$.

For every two closed bubbles $\overline \bZ_v\not= \overline \bZ_w$, there is a unique sequence of pairwise different closed bubbles $\overline \bZ_{s(1)},\overline \bZ_{s(2)},\dots,\overline \bZ_{s(n)} $ such that
\begin{itemize}
\item $\overline \bZ_{s(i)}$ intersects $\overline \bZ_{s(i+1)}$;  
\item $s(1)=v$ and $s(n)=w$. 
\end{itemize}
\end{cor}
\begin{proof}
By Lemma~\ref{lem:alpha_s= alpha_v}, $\overline \bZ_v $ and $\overline \bZ_w$ do not intersect at their alpha-points.  By Lemma~\ref{lem:Z_s is att to Z_v},  $\overline \bZ_v $ and $\overline \bZ_w$ can only intersect when $\overline \bZ_v $ contains the root of $\overline \bZ_w$.

The second claim follows from Lemma~\ref{lem:Z_s is att to Z_v}.
\end{proof}

Since every lake $\bO'$ is either $\bO_-(s)$ or $\bO_+(s)$ for a certain sequence $s$ (Lemma~\ref{lem:closest BChains}), we have:
\begin{cor}
\label{cor:lake:closure is bnd}
The closure of every proper lake is a compact subset of $\C$. For a proper lake $\bO'$, we have $\partial \bO'=\partial^c \bO' \cup \{\alpha'\}$, where $\alpha'$ is an alpha-point of the same generation as $\bO'$.
\qed 
\end{cor}

\subsection{Lakes exhaust the ocean}
Choose $P\in\PT_{>0}$ and let $\bO(P)$ be the lake of generation $P$ such that $\partial \bO(P)\ni 0$. Then $\partial \bO(P)$ also contains an arc $J\ni 0$ of $\partial \bZ$ such that $0$ is not an endpoint of $J$. It follows that $\bO(\tt^n P) =A_\str^{n}(\bO(P))$ also contains $0$ on its boundary, and we have (see Figure~\ref{Fig:Lakes:O(t)})
\begin{equation}
\label{eq:lakes exhaust ocean}
\bigcup_{n<0}\bO(\tt^n P) =\bO. 
\end{equation}
Let us denote by $\alpha(\tt^n P)$ the unique alpha-point in $\partial \bO(\tt^n P)$, see Corollary~\ref{cor:lake:closure is bnd}. 

\begin{lem}
\label{lem:I contains alpha n}
Let $I$ be a connected component of $\Esc_P(\bF)$. Then $I$ contains $\alpha(\tt^{n} P)$ for all sufficiently big $n\ll 0$.

If $J$ is a connected subset of $\Esc_P(\bF)$ such that $J\cap \bO(\tt^nP)\not=\emptyset$ but $J\not\ni \alpha(\tt^n P)$, then $J\subset \bO(\tt^n P)$; in particular, $J$ is bounded.
\end{lem}
\begin{proof}
Recall from~\S\ref{ss:Fat Jul  Esc} that every connected component $I$ of $\Esc_Q(\bF)$ is unbounded. Thus if $I$ intersects $\overline{\bO(\tt^n P) }$, then $I\ni \alpha_{\tt^n P}$ (otherwise $\partial \bO(\tt^n s)$ encloses $I$ by Corollary~\ref{cor:lake:closure is bnd}). By~\eqref{eq:lakes exhaust ocean} $I$ intersects $\bO(\tt^nP)$ for $n\ll 0$. Therefore, $I$ contains all $\alpha(\tt^m P)$ for $m\ge n$. 

In the second claim, $J$ is surrounded by $\partial \bO(\tt^n P)$; thus $J\subset \bO(\tt^n P)$.
\end{proof}

\begin{cor}
\label{cor:Q-P comp are bound}
The escaping set $\Esc_Q(\bF)$ is connected. For every $R>Q$, every connected component of $\Esc_R(\bF)\setminus \Esc_Q(\bF)$ is bounded.\qed
\end{cor}
\begin{proof}
By Lemma~\ref{lem:I contains alpha n}, every two connected components contain $\alpha(\tt^{n} P)$ for $n\ll 0$; thus $\Esc_Q(\bF)$ has a single connected component.

If $J$ is a connected component of $\Esc_R(\bF)\setminus \Esc_Q(\bF)$, then $J$ intersects $\bO(\tt^n P)$ for $n\ll 0$ but does not contain $\alpha(\tt^n P)$ for $n\ll 0$; thus $J$ is bounded.
\end{proof}

\begin{figure}[tp!]
\begin{tikzpicture}
\draw (-5,0)--(5,0); 

\begin{scope}[shift={(2,0)}]
\draw[ line width=1pt,blue] (-3.2,0)--(-3.1,0.8)--(-2.6,1.1)--(-2.2,1.1)--(-2,1);

\begin{scope}[shift={(-0.2,0)}]
\filldraw[opacity=0.5] (-1.5,0) -- (-1.7, 0.4) --(-1.5,0.8) -- (-1.3, 0.4)-- (-1.5,0);

\filldraw[opacity=0.5] (-1.5,0.8) -- (-1.6, 1) --(-1.5,1.2) -- (-1.4, 1)-- (-1.5,0.8);

\filldraw[opacity=0.5] (-1.6, 1) -- (-1.7, 1.05)--(-1.8, 1)--(-1.7, 0.95); 
\draw (-1.5,0) --(-1.5,0.8);
\draw (-1.5,0.8) edge[bend right=40] (-1.6, 1);
\draw (-1.6, 1)--(-1.8, 1);
\end{scope}

\coordinate (a1) at  (-2,1){};

\draw[red] (-2,1)--(-2.5,0.5);

\node at (-2.4,0.2){$\bO(P)$};

\end{scope}

\begin{scope}[xscale =2.5,yscale= 2, shift={(2.5,0)}]
\draw[ line width=1pt,blue] (-3.2,0)--(-3.1,0.8)--(-2.6,1.1)--(-2.2,1.1)--(-2,1);

\begin{scope}[shift={(-0.2,0)}]
\filldraw[opacity=0.5] (-1.5,0) -- (-1.7, 0.4) --(-1.5,0.8) -- (-1.3, 0.4)-- (-1.5,0);

\filldraw[opacity=0.5] (-1.5,0.8) -- (-1.6, 1) --(-1.5,1.2) -- (-1.4, 1)-- (-1.5,0.8);

\filldraw[opacity=0.5] (-1.6, 1) -- (-1.7, 1.05)--(-1.8, 1)--(-1.7, 0.95); 
\draw (-1.5,0) --(-1.5,0.8);
\draw (-1.5,0.8) edge[bend right=40] (-1.6, 1);
\draw (-1.6, 1)--(-1.8, 1);
\end{scope}

\coordinate (a2) at  (-2,1);

\draw[red] (a2)--(-2.3,0.8) --(a1);

\node at (-2.8,0.7){$\bO(P/\tt)$};

\node[red,below] at (-2,0.95) {$\alpha(P/\tt)$}; 

\draw[red](-2.3,0.8)-- (-2.5,0.75);

\end{scope}

\begin{scope}[xscale =5,yscale= 4, shift={(2.5,0)}]
\draw[ line width=1pt,blue] (-3.2,0)--(-3.1,0.8)--(-2.6,1.1)--(-2.2,1.1)--(-2,1);

\begin{scope}[shift={(-0.2,0)}]
\filldraw[opacity=0.5] (-1.5,0) -- (-1.7, 0.4) --(-1.5,0.8) -- (-1.3, 0.4)-- (-1.5,0);

\filldraw[opacity=0.5] (-1.5,0.8) -- (-1.6, 1) --(-1.5,1.2) -- (-1.4, 1)-- (-1.5,0.8);

\filldraw[opacity=0.5] (-1.6, 1) -- (-1.7, 1.05)--(-1.8, 1)--(-1.7, 0.95); 

\draw (-1.5,0) --(-1.5,0.8);
\draw (-1.5,0.8) edge[bend right=40] (-1.6, 1);
\draw (-1.6, 1)--(-1.8, 1);
\end{scope}

\coordinate (a3) at  (-2,1); 

\node[red,below] at (-2,0.95) {$\alpha(P/\tt^2)$};

\draw[red] (-2,1.2)--(a3)--(-2.3,0.8) --(a2);

\draw[red](-2.3,0.8)-- (-2.5,0.75); 

\node[left,red] at (-2,1.2) {$\Esc(\bF)$};
\node at (-2.8,0.7){$\bO(P/\tt^2)$};

\end{scope}

\filldraw[red] (a1) circle (2pt);
\filldraw[red] (a2) circle (2pt);
\filldraw[red] (a3) circle (2pt);
\node[below] at (-4.7,0) {$\partial \bZ$};

\end{tikzpicture}
\caption{Lakes $\bO(\tt^n P)$ cover $\C$. The boundary $\partial \bO(\tt^nP)$ contains a unique point $\alpha(\tt^n P)$ in $\Esc(\bF)$. Therefore, every connected component of $\Esc_P(\bF)$ intersecting $\bO(\tt^n P)$ also contains $\alpha(\tt^n P)$. Each bubble chain contains a skeleton (black), see~\S\ref{ss:TreeLike str}}
\label{Fig:Lakes:O(t)}
\end{figure}

Similarly to Lemma~\ref{lem:preim of alpha}, we have:

\begin{lem}
\label{lem:comp of Esc R - Esc T}
Let $\gamma\colon [0,1)\to \bO$ is a curve that goes to $\binfty$.  Let $\{\gamma_i\}$ be the set of lifts of $\gamma$ under $\bF^T$ with $T\in \PT_{>0}$. Then every $\gamma_i$ is contained in a unique $\bO_\iota(s)$ with $\iota\in \{-,+\}$ and $|s|=T$. Moreover, $\gamma_i$ lands at $\alpha_s$. Conversely, every $\bO_\iota(s)$ with $|s|=T$ contains a unique $\gamma_i$ which lands at $\alpha_s$.

For $R>T$ every connected component component $\bL$ of $\Esc_R(\bF)\setminus \Esc_T(\bF)$ is contained in a unique $\bO_\iota(s)$ with $\iota\in \{-,+\}$ and $|s|=T$.  Moreover, $\bL$ is a lift of $\Esc_{R-T}(\bF)$ under $\bF^{T}$ and $\bL$ is attached to $\alpha_s$: $\bL\cup\{\alpha_s\} $ is compact and connected. Conversely, for every $\alpha_s$ with $|s|=T$ there are two connected components of $\Esc_R(\bF)\setminus \Esc_T(\bF)$ attached to $\alpha_s$.
\end{lem}
\begin{proof}
Every $\gamma_i$ is contained in a unique lake which is of the form $\bO_\iota(s)$  by Lemma~\ref{lem:closest BChains}. By Lemma~\ref{lem:alpha' is a pnt}, $\gamma_i$ lands at $\alpha_s$. 

Similarly, connected components of $\Esc_R(\bF)\setminus \Esc_T(\bF)$ are in bijection with lifts of $\Esc_{R-T}(\bF)$ under $\bF^{T}$ and are in certain $\bO_\iota(s)$ with $|s|=T$. By 
Lemma~\ref{lem:alpha' is a pnt}, every component of $\Esc_R(\bF)\setminus \Esc_T(\bF)$ is attached to a certain $\alpha_s$ with $|s|=T$.
\end{proof}

\subsection{The tree-like structure of $\Esc(\bF)$}
\label{ss:TreeLike str} In this subsection, we will first discuss the combinatorics of alpha-points and then use these cut points to define external chains in $\Esc(\bF)$. We will introduce two orders on alpha-points: the tree order ``$\preceq$'' and the ambient order ``$\le$''. Informally, $\alpha_v\prec \alpha_w$ if $\alpha_v$ is closer to $\infty$ in $\Esc(\bF)$, and $\alpha_v< \alpha_w$ if $\alpha_v$ is on the ``left'' of $\alpha_w$ with respect to $\C\setminus \Esc(\bF)$. Proposition~\ref{prop:order of alphas} relates these two orders: the separation in ``$\prec$'' is equivalent to the separation in ``$<$''. 

Consider two sequences $v=(P_1,P_2,\dots )$ and $w=(Q_1,Q_2,\dots)$; each sequence can be finite or infinite. Let 
\[v\triangle w= (P_1,\dots, P_m)\] be the largest common prefix of $v$ and $w$; i.e.~$P_i=Q_i$ for $i\le m$ but $P_{m+1}\not=Q_{m+1}$. If $v=w$, then $v\triangle w= v=w$.

 We write $\alpha_v\succeq\alpha_w$ if
\[\alpha_v \in \overline{\bO_-(w)\cup \bO_+(w)}.\]
We call ``$\succeq$'' the \emph{tree order}. If $\alpha_s\in \bO_-(w)$ and $\alpha_v\in \bO_+(w)$, then we say that $\alpha_s$ and $\alpha_v$ are \emph{$\succ$-separated} by $\alpha_w$. In other words, $a_w$ is the biggest element with respect to the tree order such that $a_v\succneqq a_w$ and $a_s\succneqq a_w$.

Given a bubble $\bZ_s$ and two different points $x,y\in \partial \bZ_s$, let us write by 
$(x,y)_{\mF}\subset \bZ_s$ the unique hyperbolic geodesic joining $x$ and $y$ within the Fatou component $\bZ_s$ (recall that $\overline \bZ_s$ is a closed topological disk by Lemma~\ref{lem:alpha' is a pnt}). Similarly, for $x,y\in \partial \bZ$, the arc $(x,y)_{\mF}$ is the hyperbolic geodesic of $\bZ$ connecting $x$ and $y$. We also denote by $(0,\binfty)_{\mF}$ the hyperbolic geodesic of $\bZ$ connecting $0$ and $\binfty$. 

 Given a finite chain $s=(P_1,P_2,\dots, P_n)$, we define the \emph{skeleton} of $\bB_s$ as
\[\skel_s \coloneqq [0,c_{P_1})_{\mF} \cup [c_{P_1},c_{P_2})_{\mF}\cup \dots\cup  [c_{P_{n-1}},c_{P_n})_{\mF}\cup [c_{P_n}, \alpha_{P_n})_{\mF},\]
see Figure~\ref{Fig:Lakes:O(t)}. If $s=(P_1,P_2,\dots )$ is an infinite sequence, then the \emph{skeleton} of $\bB_s$ is
\[\skel_s \coloneqq [0,c_{P_1})_{\mF} \cup [c_{P_1},c_{P_2})_{\mF}\cup  [c_{P_{2}},c_{P_3})_{\mF}\cup \dots\]

We view each skeleton as an arc starting at $0$. For $v\not=w$, we denote by $c_{v\triangle w}$ the last common point of $\skel_v\cap \skel_w$. By construction, $c_{v\triangle w}$ is either $0$ or a critical point. Note also that the union of any number is skeletons is uniquely geodesic. Let us say that $\skel_v$ is on the \emph{left} of $\skel_w$ and \emph{write} $\skel_v< \skel_w$, if $[0,\binfty)_{\mF},\skel_w, \skel_v$ have a counterclockwise orientation around $\skel_w\cap \skel_v$. We say that $\skel_w$ and $ \skel_v$ are \emph{$<$-separated} by $\skel_s$ if either
\[\skel_v<\skel_s <\skel_w \sp \text{ or }\sp \skel_v>\skel_s >\skel_w\]
holds.

For an infinite sequence $s=(P_1,P_2,\dots)$, we say that $s(n)\coloneqq (P_1,P_2,\dots, P_n)$ is the $n$th \emph{truncation} of $s$. Clearly,  $\skel_v$, $\skel_w$ are $<$-separated by $\skel_{s}$ if and only if  $\skel_v$, $\skel_w$ are $<$-separated by $\skel_{s(n)}$ for all sufficiently big $n$. 

We also define the order ``$<$'' on alpha-points: $\alpha_v<\alpha_w$ if and only if $\skel_v<\skel_w$. We call ``$<$'' the \emph{ambient order}.

\begin{lem}
\label{lem:bubble chains:sep comb}
Consider finite distinct skeletons $\skel_v$ and $\skel_w$. 
\begin{itemize}
\item If $|v|=|w|$, then there is a finite skeleton $\skel_s$ with $|s|<|v|$ such that $\skel_v$ and $\skel_w$ are $<$-separated by $\skel_s$.
\item If $\skel_v$ and $\skel_w$ are separated by a finite bubble skeleton of generation $\leq\min\{|v|,|w|\}$, then there is a unique finite skeleton $\skel_s$ with the following properties. The skeletons $\skel_v$ and $\skel_w$ are $<$-separated by $\skel_s$ and $|s|<\min\{|v|,|w|\}$. And if $\skel_v$ and $\skel_w$ are $<$-separated by $\skel_{s'}$, then either $|s'|>|s|$ or $s'=s$.
\end{itemize}
\end{lem}
\begin{proof}
Suppose $\skel_v$ is on the left of $\skel_w$. Recall (see~\S\ref{ss:Lakes}) that the infinite bubble chain $\bB_{\rho(v)}$ is the closest bubble chain on the right of $\bB_v$ of generation $\le |v|$. Therefore, $\skel_v$ and $\skel_w$ are $<$-separated by $\skel_{\rho(v)}$. For $m\gg 0$, let $s$ be the $m$th truncation of $\rho(v)$. Then $\skel_s$ separates $\skel_v$ and $\skel_w$ and $|s|<|\rho(v)| = |v|$.

For the second claim, set $\tau\in \R_{\ge 0}$ be the infimum over all $|x|$ such that $\skel_x$ separates $\skel_v$ and $\skel_w$. 

Observe first that $\tau<\min\{|v|,|w||\}$. Indeed, let $\skel_y$ be a skeleton separating $\skel_v$ and $\skel_w$ with $|y|\le \min\{|v|,|w||\}$. Then $|y|=\tau$ and $y$ is finite (otherwise we can truncate $y$ and construct a skeleton of smaller generation that still separates $\skel_v$ and $\skel_w$). This contradicts to the first claim: there is a finite skeleton of smaller generation between $\skel_y$ and one of $\skel_v$, $\skel_w$.

We claim that every realization of $\tau$ is a finite sequence $x$. Then it would follow from the first claim that $x$ is unique.
\begin{proof}[Proof of the statement]
Let $\overline \bZ_{s(1)},\overline \bZ_{s(2)}, \dots , \overline \bZ_{s(n)} $ be a finite sequence of neighboring closed bubbles from Lemma~\ref{cor:tree str of closed bubbles} connecting $\bZ_v=\bZ_{s(1)}$ and $\bZ_w=\bZ_{s(n)}$. For every $s(i)$ let $\tau_i$ be the infimum over all $|x|$ such that 
\begin{itemize}
\item $\skel_x$ separates $\skel_v$ and $\skel_w$,
\item either $x\triangle v= s(i)$ or $x\triangle w= s(i)$ holds, and
\item if $v\triangle w= s(i)$, then  $x\triangle v= x\triangle w= s(i)$.
\end{itemize}
 In other words, the infimum is taken over all $\skel_x$ that coincide with one of $\skel_v,\skel_w$ up to $c_{s(i)}$ (if $\bZ_{s(i)}=\bZ_{\emptyset}=\bZ$, then $c_{s(i)}=0$). Observe that $\tau_1 \ge |v|$ and $\tau_k\ge |w|$ because every skeleton associated with either $\tau_1$ or $\tau_2$ travels trough either $c_v$ or $c_w$. Therefore, $\tau=\tau_k$ for some $k\in \{2,\dots, n-1\}$.

If $\skel_v$ and $\skel_w$ are $<$-separated by $\skel_{s(k)}$, then $|s(k)|=\tau_k$ and the claim follows. Otherwise, let $J\subset \partial^c \bZ_{s(k)}$ be an open arc between $\overline \bZ_{s(k-1)}$ and  $\overline \bZ_{s(k+1)}$. By construction, every skeleton associated with $\tau_k$ travels through $J$. Since $\overline J$ is disjoint from $\alpha_{s(k)}$, there is a unique point $c_y$ of the smallest generation in $J$. Therefore, $\skel_y$ is a required skeleton. 
\end{proof}
\end{proof}

\begin{prop}
\label{prop:order of alphas}
Consider $\alpha_w\not=\alpha_v$ with $|v|\ge |w|$. Then the following are equivalent:
\begin{enumerate}
\item ${\alpha_v\succ\alpha_w}$;\label{cond:1:prop:order of alphas}
\item $\alpha_v$ and $\alpha_w$ are not $<$-separated by $\alpha_s$ with $|s|<|w|$;\label{cond:2:prop:order of alphas}
\item $\alpha_v$ and $\alpha_w$ are not $\prec$-separated;\label{cond:3:prop:order of alphas}
\end{enumerate}
\end{prop}
\begin{proof}
Suppose that $\alpha_v< \alpha_w$; the opposite case is symmetric. 

We have ${\alpha_v\succ\alpha_w}$ if and only if $\alpha_v \in \bO_-(w)$. The latter is equivalent to the property that $\skel_v$ and $\skel_w$ are not $<$-separated by $\skel_{\lambda(w)}$ (where $\bB_{\lambda(w)}$ is defined in~\S\ref{ss:Lakes}).
For $m\gg 1$, let $s$ be the $m$th truncation of $\lambda(w)$. Note that $|s|<|\lambda(w)|=|w|$. Then $\skel_v$ and $\skel_w$ are not $<$-separated by $\skel_{\lambda(w)}$ if and only if $\skel_v$ and $\skel_w$ are not separated by $\skel_{s}$. This proves the equivalence between~\eqref{cond:1:prop:order of alphas} and~\eqref{cond:2:prop:order of alphas}.

Let us prove that \eqref{cond:3:prop:order of alphas} is also equivalent to \eqref{cond:1:prop:order of alphas} and \eqref{cond:2:prop:order of alphas}. If $\alpha_v\succ\alpha_w$, then clearly $\alpha_v$ and $\alpha_w$ are not $\prec$-separated. Suppose $\alpha_v\not \succ\alpha_w$. By the first claim, $\skel_v$ and $\skel_w$ are $<$-separated by a finite skeleton $\skel_s$ with $|s|<|w|$. Using Lemma~\ref{lem:bubble chains:sep comb} we can assume that $\skel_s$ has the smallest possible generation. Therefore, $\skel_v$ is between $\skel_{\lambda(s)}$ and $\skel_s$ while $\skel_w$ is between $\skel_{\rho(s)}$ and $\skel_s$. By definition,  $\alpha_v$ and $\alpha_w$ are $\prec$-separated by $\alpha_s$.
\end{proof}

Suppose $\alpha_v\succ\alpha_w$. The \emph{external chain} $[\alpha_v,\alpha_w]$ is \[ \Esc_{|v|}(\bF)\cap \overline{ \bO_-(w)\cup \bO_+(w)} \setminus \bigcup_{\bO_\iota(s)\not\ni \alpha_v}  \bO_\iota(s),\]
where the union is taken over all $\iota\in\{-,+\}$ and all the sequences $s$ satisfying $\bO_\iota(s)\not\ni \alpha_v$. In other words, $[\alpha_v,\alpha_w]$ is obtained from $\Esc (\bF)$ by chopping off all the lateral decorations at alpha-points.

\begin{lem}
\label{lem:conc of chains}
We have 
\[ [\alpha_v,\alpha_s]=[\alpha_v,\alpha_w] \cup [\alpha_w,\alpha_s]\sp \text{ and }\sp [\alpha_v,\alpha_w] \cap [\alpha_w,\alpha_s]=\{\alpha_w\}\]
for all $\alpha_v\succ \alpha_w\succ \alpha_s$, and
\[A_\str[\alpha_w,\alpha_v]= [\alpha_{\tt w }, \alpha_{\tt v}]\]
for all $\alpha_v\succ \alpha_w$.
\end{lem}
\begin{proof}
Since $\alpha_w$ is a cut point between $\alpha_v$ and $\alpha_s$ with respect to the tree order ``$\prec$,'' we have $[\alpha_v,\alpha_w] \cap [\alpha_w,\alpha_s]=\{\alpha_w\}$. Recall from  Lemma~\ref{lem:comp of Esc R - Esc T} that components of $\Esc_P\setminus \Esc_{|v|}$ are attached to alpha-points of generation $|s|$. We also have 
\[ [\alpha_v,\alpha_w]=\Esc_{P}(\bF)\cap \overline{ \bO_-(w)\cup \bO_+(w)} \setminus \bigcup_{\bO_\iota(s)\not\ni \alpha_v}  \bO_\iota(s)\]
for all $P\ge |s|$ (because every component $\Esc_P\setminus \Esc_{|v|}$ is deleted). As a consequence,
$[\alpha_v,\alpha_s]=[\alpha_v,\alpha_w] \cup [\alpha_w,\alpha_s]$.

The second claim follows $A_\str(\alpha_w)=\alpha_{\tt w}$ and $A_\str(\alpha_v)=\alpha_{\tt v}$, see Lemma~\ref{lem:prop of Astr}.
\end{proof}

\subsection{External chains are arcs}
\label{ss:ExtChainsAreArcs}
Recall that $\bF^P\mid \overline \bZ$ is conjugate by $\bbh$ to the cascade of translations $(T^P)_{P\in \PT}$, see Lemma~\ref{lem:irr rot of bZ}. By construction, $\bbh(c_P)= b_P$, where $b_P$ are defined in~\S\ref{ss:closed returns}. As in \S\ref{ss:closed returns}, we say that a critical point $c_{P}$ is \emph{dominant} if the arc $[0,c_P]$ contains no critical point of generation less than $P$; in this case $\bZ_P$ and $\alpha_P$ are also called  \emph{dominant}.  

Let $c_P$ and $c_Q$ be two dominant critical points and assume that $P<Q$. Then $c_Q,0\in \partial \bZ$ are on the same side of $c_P$. 

As in~\S\ref{ss:closed returns}, we enumerate dominant critical points as $(c_{P_n})_{n\in \Z}$ with $P_{n+1}>P_{n}$. Suppressing indices, we write $\alpha_i=\alpha_{P_i}$, $c_i=c_{P_i}$, and $\bZ_i=\bZ_{P_i}$. 
\begin{lem}
We have
\[\dots \succ\alpha_{1} \succ\alpha_{0} \succ\alpha_{-1}\succ\dots , \]
\end{lem}
\begin{proof}
We claim that there is no $\alpha_s$ separating $\alpha_{i+1}$ and $ \alpha_{i}$ with respect to the ambient order ``$<$'' such that the generation of $\alpha_s$ is less than $P_i$. Write $s=(R,\dots )$. Since $R<P_i$, the chain $\bB_s$ does not go through $\bZ_{i}$ and $\bZ_{i+1}$; hence $\bZ_R$ is attached to $ (c_{i+1},c_i)$. This is impossible, because $c_R$ is not counted as dominant. By Proposition~\ref{prop:order of alphas},  $\alpha_{i+1}\succ\alpha_i$.
\end{proof}

The \emph{zero chain} (see Figure~\ref{fig:FibonComb}) is 
\begin{equation}
\label{eq:ZeroChain:defn}
\dots \cup[\alpha_1,\alpha_0]\cup [\alpha_0,\alpha_{-1}]\cup \dots
\end{equation}

It follows from the definition that $c_R$ is dominant if and only if $A_{\str} (c_R)$ is dominant. Therefore, there is a $k>0$ such that $\tt P_i=P_{i+k}$ and (equivalently) $A_{\str} (\alpha_{i})= \alpha_{i+k}$ for all $i\in \Z$. As a consequence,
\begin{equation}
\label{eq:zero chain:A str}
 A_{\str}[\alpha_{tk},\alpha_{(t-1)k}]=[\alpha_{(t+1)k},\alpha_{tk}],
 \end{equation} thus the $[\alpha_{(t+1)k},\alpha_{tk}]$ shrink to $0$; i.e.~$0$ is the landing point of the zero chain.

\begin{figure}
\begin{tikzpicture}
\draw (-5,0)--(5,0);

\begin{scope}[yscale=2/3, xscale=-2/3]
\draw[shift={(-0.5,0)}] (1,0) .. controls (1.1,0.3) and (1.1,0.6).. (0.7,0.8)
.. controls (0.6,0.6) and (0.7,0.2).. (1,0); 
\node[below] at(0.5,0) {$c_5$};

\coordinate (w5) at (0.2,0.8);
\end{scope}

\begin{scope}
\draw[shift={(-0.5,0)}] (1,0) .. controls (1.1,0.3) and (1.1,0.6).. (0.7,0.8)
.. controls (0.6,0.6) and (0.7,0.2).. (1,0); 
\node[below] at(0.5,0) {$c_4$};

\coordinate (w4) at (0.2,0.8);
\end{scope}

\begin{scope}[yscale=1.5, xscale=-1.5]
\draw[shift={(-0.5,0)}] (1,0) .. controls (1.1,0.3) and (1.1,0.6).. (0.7,0.8)
.. controls (0.6,0.6) and (0.7,0.2).. (1,0); 
\node[below] at(0.5,0) {$c_3$};

\coordinate (w3) at (0.2,0.8);
\node[red,above] at (w3) {$\alpha_{3}$};
\end{scope}

\begin{scope}[yscale=2.25, xscale=2.25]
\draw[shift={(-0.5,0)}] (1,0) .. controls (1.1,0.3) and (1.1,0.6).. (0.7,0.8)
.. controls (0.6,0.6) and (0.7,0.2).. (1,0); 
\node[below] at(0.5,0) {$c_2$};

\coordinate (w2) at (0.2,0.8);
\node[red,above] at (w2) {$\alpha_{2}$};
\end{scope}

\begin{scope}[yscale=1.5*2.25, xscale=-1.5*2.25]
\draw[shift={(-0.5,0)}] (1,0) .. controls (1.1,0.3) and (1.1,0.6).. (0.7,0.8)
.. controls (0.6,0.6) and (0.7,0.2).. (1,0); 
\node[below] at(0.5,0) {$c_1$};

\coordinate (w1) at (0.2,0.8);
\node[red,above] at (w1) {$\alpha_{1}$};
\end{scope}

\begin{scope}[yscale=2.25*2.25, xscale=2.25*2.25]
\draw[shift={(-0.5,0)}] (1,0) .. controls (1.1,0.3) and (1.1,0.6).. (0.7,0.8)
.. controls (0.6,0.6) and (0.7,0.2).. (1,0); 
\node[below] at(0.5,0) {$c_0$};

\coordinate (w0) at (0.2,0.8);
\node[red,above] at (w0) {$\alpha_{0}$};
\end{scope}

\begin{scope}[yscale=1.5*2.25*2.25, xscale=-1.5*2.25*2.25]
\draw[shift={(-0.5,0)}] (1,0) .. controls (1.1,0.3) and (1.1,0.6).. (0.7,0.8)
.. controls (0.6,0.6) and (0.7,0.2).. (1,0); 
\node[below] at(0.5,0) {$c_{-1}$};

\coordinate (ww) at (0.2,0.8);
\node[red,above] at (ww) {$\alpha_{-1}$};
\end{scope}

\draw[red,thick] (ww)--(w0)--(w1)--(w2)--(w3)--(w4)--(w5);

\end{tikzpicture}
\caption{Fibonacci (golden mean rotation number) combinatorics: for every $i$ the map $\bF^{P_{i+1}}$ maps $[\alpha_{i},\alpha_{i-1}]$ to  $[\alpha_{i-1},\alpha_{i-3}]$.}
\label{fig:FibonComb}
\end{figure}

\begin{lem}
\label{lem:dynam of domin:alpha}
For every $[\alpha_i,\alpha_{i+1}]$ there is a $Q\in \PT_{>0}$ and $[\alpha_n,\alpha_m]$ with $i\ge m>n$ such that $\bF^Q$ maps $[\alpha_i,\alpha_{i+1}]$ homeomorphically to $[\alpha_n,\alpha_m]$.
\end{lem}
\begin{proof}
Follows from Lemma~\ref{lem:dynam of domin}, which provides the corresponding property for $b_P=\bbh(c_P)$.
\end{proof}

Given $M\in \R_{>0}$, we write $\Esc_M(\bF)=\displaystyle\bigcap_{T>M}\Esc_T(\bF)$. Given $x\in \Esc(\bF)$, its \emph{escaping time} is the minimal $M$ such that $\Esc_M(\bF)\ni x$.

\begin{cor}
\label{cor:zero chian is arc}
The zero chain is an arc landing at $0$. The points on this arc are parametrized by their escaping time ranging continuously from $+\infty$ (for points close to $0$) to $0$. Alpha-points are dense on the zero chain. 
\end{cor}
\begin{proof}
Consider $I\coloneqq [\alpha_{tk},\alpha_{(t+1)k}]$. Let us construct Markov partitions $\mI_r$ of $I$ for $r\ge0$. For $i\in  \{tk,\dots, (t+1)k-1\}$, the chain $I_i\coloneqq [\alpha_{i},\alpha_{i+1}]$ is an element of the partition $\mI_0$ of level $0$. By Lemma~\ref{lem:dynam of domin:alpha}, $\bF^Q(I_i)=[\alpha_n,\alpha_m]$ for some $Q>0$ depending on $i$. For $j\in  \{n,n+1,m-1\}$ we say that the preimage of $[\alpha_j,\alpha_{j+1}]$ under $\bF^Q\colon I_i\to [\alpha_n,\alpha_m]$ is an element of the partition $\mI_1$ of level $1$.

By construction, for every chain $J$ of $\mI_1$ there is a chain $\widetilde J$ of $\mI_0$ and a homeomorphism
\begin{equation}
\label{eq:chi:cor:zero chian is arc}
\chi _J=A_{\str}^{m}\circ \bF^Q\colon J\to \widetilde J\sp \sp \text{ with }Q>0, m\ge0.
\end{equation}
The map~\eqref{eq:chi:cor:zero chian is arc} expands the hyperbolic metric of $\bO$: if $\bO'\supset J$ is the lake of generation $Q$, then $\chi_J\colon \bO'\to \bO$ is expanding. 

 Elements of the partition $\mI_{r+1}$ are the preimages of the elements in $\mI_{r}$ under all possible $\chi_J$. Since $\chi_J$ are expanding, the diameters of elements in $\mI_r$ are bounded by $C \lambda^{-r}$ for some $C>0$ and $\lambda >1$.

We can enumerate chains in $\mI_r$ as $I^r_i$ with $i\in \{1,\dots, a(r)\}$ such that $I^r_i$ and $I^r_j$ are disjoint if $i>j+1$ (see Lemma~\ref{lem:conc of chains}). Since the diameters of chains in $\mI_r$ tend to $0$, we obtain that alpha-points are dense in $I$ and $I$ is an arc by a well-known characterization of arcs. More precisely, $\prec$ is a total order on $I$ compatible with topology: the sets $I_{ \succeq v}\coloneqq \{x \in I : x\succeq v\}$ and $I_{ \preceq v}\coloneqq \{x \in I : x\preceq v\}$ are closed in $\C$ for all $v\in I$ (they are closed if $v$ is an alpha-point (by Lemma~\ref{lem:alpha' is a pnt}) and we just showed that alpha-points are dense). By~\cite{Na}*{Theorems 6.16 and 6.17}, $I$ is an arc. The assertion that $I$ is an arc also follows from the next paragraph justifying a continuous parametrization of $I$ by an interval (a continuous bijection between compact sets is a homeomorphism).

By construction, $I^r_i$ are arcs connecting two alpha-points. For every $I^r_i$ write $\overline \chi =\chi_1\circ \dots \chi_{r}$ 
a composition of maps~\eqref{eq:chi:cor:zero chian is arc} mapping  $I^r_i$ to an element of $\mI_0$.  Using~\eqref{eq:ren:F:F_n^P} we write $\overline \chi =A^{m(r,i)}_\str\circ \bF^{Q(r,i)}$. Observe that $m(r,i)\to +\infty$ as $r\to +\infty$ uniformly in $i$. (Indeed, since $Q>0$ in ~\eqref{eq:chi:cor:zero chian is arc}, there is a constant $M\ge 0$ such that $\overline \chi$ contains at most $M$ consecutive $\chi_j$ that does not contain the scaling $A_\str$.) Then \[\bF^{Q(r,i)}(I^r_i)\in  A_\str^{-m(r,i)} [\alpha_{tk},\alpha_{(t+1)k}],\] and the difference in the escaping times between the endpoints of $I^r_{i}$ is less than $\big(|\alpha_{(t+1)k}|-|\alpha_{tk}|\big)/\tt^{m(r,i)}$. This proves that the escaping time parametrizes continuously points on $I$.
\end{proof}

\begin{rem}
In fact, the zero chain is a quasi-arc. In the proof of Corollary~\ref{cor:zero chian is arc}, we can extend $\chi_J$ to a conformal map defined on a neighborhood of $J$. The Koebe distortion theorem implies that the distance between any pair of points $x,y$ in $I$ is comparable to the diameter of the subarc $[x,y]\subset I$. This is one of the characterizations of a quasi-arc.
\end{rem}

For every $\alpha_s$ there is a sequence $\alpha_s \succ \alpha_{v_1}\succ \alpha_{v_2}\succ \dots$ with $|v_i|$ tending to $0$. We define:
\[ [\alpha_s,\binfty)\coloneqq [\alpha_s,\alpha_{v_1}]\cup [\alpha_{v_1},\alpha_{v_2}]\cup [\alpha_{v_2},\alpha_{v_3}]\cup \dots  \]

\begin{prop}
\label{prop:alpha v alpha s is arc}
  The chain $[\alpha_{s},\binfty)$ is a simple arc for every alpha-point $\alpha_s$ with $|s|>0$. Points in $[\alpha_{s},\binfty)$ are parametrized by their escaping time ranging continuously from $|s|$, the escaping time of $\alpha_s$, to $0$. Alpha-points are dense in $[\alpha_{s},\binfty)$.
\end{prop}
\noindent As a consequence, $[\alpha_{s},\alpha_v]$ is a simple arc for every $\alpha_s \succ\alpha_v$. We call both $[\alpha_{s},\alpha_v]$ and $[\alpha_s,\binfty)$ \emph{external ray segments.}
\begin{proof}
Suppose first $\alpha_s=\alpha_Q$. Choose a dominant $\alpha_P$ with $P\ge Q$. By Corollary~\ref{cor:zero chian is arc}, $[\alpha_P,\binfty)$ is a simple arc. Let $x\in [\alpha_P,\binfty)$ be the unique point of generation $P-Q$. Then $[\alpha_P,x)$ is a lift of $[\alpha_Q,\binfty)$ under $\bF^{P-Q}$; therefore, $[\alpha_Q,\binfty)$ is a simple arc.

Suppose now $s=(Q_1,Q_2,\dots, Q_n)$. Then $(\binfty,\alpha_s]$ is $(\binfty,\alpha_{Q_1}]$ followed by the lift of $(\binfty,\alpha_{Q_2}]$ under $\bF^{Q_1}$ connecting $\alpha_{Q_1}$ and $\alpha_{(Q_1,Q_2)}$, followed by the lift of $(\binfty,\alpha_{Q_3}]$ under $\bF^{Q_1+Q_2}$ connecting $\alpha_{(Q_1,Q_2)}$ and $\alpha_{(Q_1,Q_2, Q_3)}$, and so on. By Lemma~\ref{lem:comp of Esc R - Esc T}, $(\binfty,\alpha_s]$ is a simple arc. The claims about the parameterization and alpha-points follow from Corollary~\ref{cor:zero chian is arc}. 
 \end{proof}

\subsection{External rays}
\label{ss:ExtRayFstr}
 An external ray $\bR$ is a simple arc
\[\dots \cup [\alpha_{i+1},\alpha_i]\cup [\alpha_{i},\alpha_{i-1}]\cup \dots\] 
subject to the condition \[\dots \succ\alpha_{i+1}\succ\alpha_i\succ\alpha_{i-1}\succ\dots\] such that
\begin{itemize}
\item the generation of $\alpha_i$ tends to $0$ as $i$ tends $+\infty$, and
\item there is no alpha-point $\alpha'$ such that $\alpha'\succ \alpha_i$ for all $i\in \Z$.
\end{itemize}
In other words, an external ray is a simple arc between $\infty$ and an end of the escaping set. The \emph{generation} of $\bR$ is $\lim_{i\to +\infty } |\alpha_i|\in \R_{>0}\sqcup \{+\infty \}$. We say that
\begin{itemize}
\item $\bR$ has \emph{type \RN{1}} if the generation of $\bR$ is $\infty$; and
\item $\bR$ has \emph{type \RN{2}} if the generation of $\bR$ is $<\infty$.
\end{itemize}

Given a ray $\bR$, its \emph{image} is \[\bF^P(\bR)\coloneqq \bF^P(\bR\cap \Dom \bF^P).\]
Then $\bR\cap \Dom \bF^P$ is a subarc (possibly empty) of $\bR$ consisting of all the points with the escaping time in $(P,+\infty)$, see Proposition~\ref{prop:alpha v alpha s is arc}.  If $Q$ is the generation of $\bR$, then the generation of $\bF^P(\bR)$ is $Q-P$. Note that $\bF^P(\bR)$ is empty if and only if $P\ge Q$.

The following is a corollary of Proposition~\ref{prop:order of alphas}:
\begin{cor}
\label{cor:rays meet}
Any two external rays $\bR_1\neq \bR_2$ meet at a unique $\alpha_s$; i.e.~$\bR_1\cap \bR_2=[\alpha_s,\binfty)$.\qed
\end{cor}

A ray $\bR$ is \emph{periodic} if $\bF^P(\bR)=\bR$ for some $P>0$. In this case $P$ is a \emph{period} of $\bR$. We will show in Corollary~\ref{cor:ext rays land} that every periodic ray has a minimal period. Preperiodic rays are defined accordingly.

The zero chain~\eqref{eq:ZeroChain:defn} is a ray  landing at $0$; we will denote this ray as $\bR^\str=\bR^0$. Writing $\bI_t=[\alpha_{t k}, \alpha_{(t-1)k}]$
 in~\eqref{eq:zero chain:A str}, we obtain the decomposition
\begin{equation}
\label{eq:bRstr:decomp}
\bR^\str= \dots \cup \bI_1 \cup \bI_0\cup \bI_{-1}\cup \dots
\end{equation}
where 
\begin{equation}
\bI_i\subset \overline{\Esc_{P\tt^i}\setminus \Esc_{P\tt^{i-1}}(\bF)},\sp\sp \sp P\coloneqq |\balpha_0|
\end{equation}
 is an external ray segment satisfying 
\begin{equation}
\label{eq:I:lem:ray decomp as fund segm}
\bI_{i}=A_\str\big( \bI_{i-1}\big).
\end{equation}

\subsection{Wakes}
\label{ss:wakes}
Since the zero ray $\bR^0$ lands at $0$, for every critical point $c_s$, there are two preimages $\bR_{s,-}$ and $\bR_{s,+}$ of $\bR^0$ landing at $c_s$. We assume that $\bR_{s,-}$ is on the left of $c_s$ and $\bR_{s,+}$ is on the right of $c_s$ (relative the boundary of the bubble containing $c_s$). Let $\alpha_{s,-}\in \bR_{s,-}$ and $\alpha_{s,+}\in \bR_{s,+}$ be two alpha-points on the rays close to $c_s$. Observe that $\alpha_{s,+}$ and  $\alpha_{s,-}$ are $\prec$-separated by $\alpha_s$. We denote by $\bR'_{s,-}$ and $\bR'_{s,+}$ the closed subarcs of $\bR_{s,-}$ and $\bR_{s,+}$ between $\alpha_s$ and $c_s$, see Figure~\ref{fig:wake:defn}. By Corollary~\ref{cor:rays meet},  
\begin{equation}
\label{eq:part Ws}
\{c_s\} \cup \bR'_{s,-}\cup \bR'_{s,+} \eqqcolon \partial \bW_s
\end{equation}
 encloses the closed topological disk $\bW_s$ containing $\bL_s$. We call $\bW_s$ the (closed) \emph{wake at $c_s$}, and we say that $c_s$ is the \emph{root} of $\bW_s$. We will show in Corollary~\ref{cor:wake:clos:limb} that $\bW_s=\overline \bL_s$. If $s=(P_1,P_2\dots, P_m)$, then $m$ is called the \emph{level} of $\bW_s$. We say that $\alpha_s$ is the \emph{top} point of $\bW_s$. Wakes $\bW_P$ are called \emph{primary}.

\begin{figure}[tp!]
\begin{tikzpicture}
\draw[red](0,0)--(7,0);


\begin{scope}[scale=0.7]
\draw[yscale=2] (5,0) -- (4.2,1)--(5,2)--(5.8,1)--(5,0);

\draw (5,0) .. controls (3,0.8) and (3,3.7).. (5,4)
  .. controls (7,3.7) and (7,0.8)..(5,0);
\node at (5,1.7) {$\bZ_s$};

\draw (5,4) edge node[above left ]{$(\alpha_s,\binfty)$} (7,7);

\node[left] at (3.5,2) {$\bR'_{s,-}$};
\node[right] at (6.5,2) {$\bR'_{s,+}$};
\node[below] at (5,0) {$c_s$};
\node[above left] at (5,4) {$\alpha_s$};

\end{scope}

\end{tikzpicture}
\caption{The wake $\bW_s$ is the closed topological disk containing $\bZ_s$ and enclosed by $\overline {\bR_{s,-}\cup \bR_{s,+}}$.}
\label{fig:wake:defn}
\end{figure}

The dynamics of wakes follows the dynamics of their roots:
\begin{lem} If $\bF^Q(c_v)=c_s$, then $\bF^Q\colon \bW_v\to \bW_s$ is a homeomorphism. For every wake $\bW_s$, we have a conformal map \[ \bF^{|s|}\colon \intr \bW_s\to \C\setminus \overline \bR^0.\]
\end{lem}
\begin{proof}
By construction, $\bF^Q$ maps $\overline{\bR_{v,-}\cup \bR_{v,+}}$ homeomorphically to $\overline{\bR_{s,-}\cup \bR_{s,+}}$. Therefore, $\bF^Q\colon \bW_v\to \bW_s$ is a homeomorphism. 

Since $\bF^{|s|}$ maps each curve $\bR_{v,-}, \bR_{v,+}$ homeomorphically to $\bR^0$, we see that  $\bF^{|s|}\colon \intr \bW_s\to \C\setminus \overline \bR^0$ is conformal.
\end{proof}

As before, we write $\bJ_n\coloneqq \partial \bZ\cap\bS^\#_n$ and $\bJ=\bJ_0$.
\begin{lem}[Primary wakes shrink]
\label{lem:wake shrinks}
For every $n\in \Z$ and every $\varepsilon>0$ there are at most finitely many primary wakes $\bW_P$ with $c_P\in \bJ_n$ such that the diameter of $\bW_P$ is greater than $\varepsilon$.
\end{lem}
\begin{proof}
It is sufficient to prove the statement for $n=0$. Write \[\bJ_-\coloneqq \bJ\cap \bU_- \sp \text{ and }\sp    \bJ_+\coloneqq \bJ\cap \bU_+  \]
(see~\eqref{eq:Max Prep U pm}); then $\bbf_{\pm}\colon \bJ_\pm \to \bJ$ realizes the first return of points in $\bJ$ back to $\bJ$. Let $\bW_{0}$ be the primary wake of smallest generation touching $\bJ$. Then all other primary wakes touching $\bJ$ are iterated lifts of $\bW_{0}$ under $\bbf_{\pm}\colon \bJ_\pm \to \bJ$; we enumerate these wakes as $(\bW_i)_{i\le 0}$ so that $\bW_{i-1}$ is a preimage of $\bW_i$ under $\bbf_\pm$. We will show that the diameter of $\bW_i$ tends to $0$.

 Let $\bO_\lambda$ be the union of all the lakes whose generation is $(0,0,1)$ and whose closure intersects $\bJ_-$. Similarly, $\bO_\rho$ is the union of all the lakes whose generation is $(0,1,0)$ and whose closure intersect $\bJ_+$. If $\bW_0=\bW_P$ intersects $\bJ_-$, then $\bO_\lambda=\bO_+(P)\cup \bO_i(P)$ because $\bJ_-$ contains a unique critical point $c_P$ of generation $\le P=(0,0,1)$; and $\bO_\rho$ consists of a single lake because $\bJ_+$ does not contain a critical point of generation $\le P$. If $\bW_0$ intersects $\bJ_+$, then $\bO_\rho$ consists of two lakes and $\bO_\lambda$ consists of a single lake.

The maps
\begin{equation}
\label{eq:prf:lem:wake shrinks}
\bbf_-\colon \bO_\lambda\to \bO=\C\setminus \overline \bZ \sp \text{ and }\sp \bbf_+\colon \bO_\rho \to \bO=\C\setminus \overline \bZ
\end{equation}
expand the hyperbolic metric of $\bO$.

Let us denote by $c_n$ the root of $\bW_n$. Let $y_0$ be an arbitrary point in $\bO$ and let $\ell_0$ be a curve in $\overline \bO$ connecting $c_0$ to $y_0$ such that $\ell_0\setminus\{c_0\}\subset \bO$. For $n\le 0$, denote by $\ell_n$ the unique lift of $\ell$ (by an appropriate $\bF^{P_n}$) starting at $c_n$ and denote by $y_n$ the endpoint of $\ell_n$.

\emph{Claim.} There is a sequence $\varepsilon_n>0$ converging to $0$ such that the following holds. If the diameter of $\ell_0$ is less than $\varepsilon_0$, then the Euclidean diameter of $\ell_n$ is less than $\varepsilon_n$.

Since the maps in~\eqref{eq:prf:lem:wake shrinks} expand the hyperbolic metric, it follows from the Claim that the $\bW_n$ shrink in the Euclidean metric. Indeed, $\bW_0$ minus a small neighborhood of $c_0$ is a compact subset of $\bO$. Lifts of this compact subset either shrink in the hyperbolic metric or converge to $\partial \bZ$; in the latter case, they shrink in the Euclidean metric.
\begin{proof}[Proof of the Claim]
It is sufficient to verify the claim in the dynamical plane of the pacman $f$. Since Siegel maps with the same rotation number are conjugate in small neighborhoods of  their Siegel disks (see~\S\ref{sss:SiegPacmen}), it is sufficient to prove the claim in the dynamical plane of the quadratic polynomial $p$ that has a Siegel fixed point with the same rotation number as $f$.

Choose two points $a,b\in \partial Z_p$ such that $a$ is slightly on the left of $c_0$ while $b$ is slightly on the right of $c_0$. Let $R_a$ and $R_b$ be two external rays landing at $a$ and $b$. Let $X$ be an open topological disk bounded by $\overline Z_p\cup R_a\cup R_b$ and truncated by some equipotential such that the boundary of $X$ contains $c_0$. Let $X_n$ be the unique lift of $X$ under $p^n$ such that $\partial X_n$ contains $c_n$ -- the unique preimage of $c_0$ under $p^n\colon \partial Z_p\to \partial Z_p$. Then $X_n$ is bounded by $p^{-n}(\overline Z_p) $, two preimages $R_{a,n},R_{b,n}$ of $R_a,R_b$, and an equipotential. If $n$ is big, then the difference between the external angles of $R_{a,n},R_{b,n}$ is small and the equipotential is close to the Julia set. Since the Julia set of $p$ is locally connected \cite{Pe}, the diameter of $X_n$ tends to $0$. 
\end{proof}
\end{proof}

\begin{figure}[tp!]
\begin{tikzpicture}[scale=0.7]
\draw[red] (-3.,0.)-- (10.,0.);
\node[red] at(0,-0.4){$\partial \bZ$};
\node at (3.1302, 1.50314) {$\bW_P$};
\node at (1.23603, 2.48491){$\bW_{P^-}$};
\node at (5.155, 3.55858){$\bW_{P^+}$};

\draw (3.08,0.)-- (2.06,1.34);
\draw (2.06,1.34)-- (2.98,2.64);
\draw (2.98,2.64)-- (4.04,1.54);
\draw (4.04,1.54)-- (3.08,0.);
\draw (0.7,0.)-- (2.06,1.34);
\draw (2.98,2.64)-- (3.06,3.66);
\draw (-0.32,2.28)-- (0.7,0.);
\draw (3.06,3.66)-- (2.34,5.08);
\draw (2.34,5.08)-- (0.82,4.5);
\draw (0.82,4.5)-- (-0.32,2.28);
\draw (5.7,0.)-- (4.04,1.54);
\draw (2.34,5.08)-- (4.878944786440976,7.453490331102311);
\draw (4.878944786440976,7.453490331102311)-- (7.034904302297577,5.596969636892458);
\draw (7.034904302297577,5.596969636892458)-- (7.76,3.24);
\draw (7.76,3.24)-- (5.7,0.);
\draw (4.220179378818126,0.)-- (3.5664083132199376,0.7802800024569837);
\draw (4.220179378818126,0.)-- (4.824697068361477,0.8120280209176659);

\draw[dashed ,blue] (5.7,0.)  .. controls (4.2,3.5) and (4.3,3.9).. (4.878944786440976,7.453490331102311) .. controls (6.3,3.5) and (6.2,3.7)  .. (5.7,0.) ;

\draw[dashed,blue ](3.08,0.) .. controls (2.7, 1) and (2.7,1.5).. (2.98,2.64)..
controls (3.3, 1) and (3.3,1.5).. (3.08,0.) ;
\draw[dashed,blue ] (0.7,0.)  .. controls (0, 2) and (1.5,4).. (2.34,5.08)
.. controls(2.4,4) and (2, 2) ..(0.7,0.)  ;
  .. controls (7,3.7) and (7,0.8)..(5,0);

 \draw[dashed,blue ]  (2.02,2.8) .. controls (2.3, 2.9) and (2.5,2.8).. (2.8,2.7)
..controls  (2.5,2.5) and  (2.3, 2.6) ..(2.02,2.8) ;
 \draw[blue] (2.8,2.7) -- (2.98,2.64);
 
  \draw[dashed,blue ]  (4.65,2.9)  .. controls (4.1, 3.1) and (3.8,3).. (3.3,2.7) 
  .. controls (3.8,2.5) and (4.1, 2.6) ..(4.65,2.9) ;
  
\draw[blue] (3.3,2.7) --  (2.98,2.64);

\draw (4.220179378818126,0.)-- (3.566408313219938,0.7802800024569837);
\draw (4.220179378818126,0.)-- (4.824697068361477,0.8120280209176661);
\draw (1.8891895158449743,0.)-- (1.3989135792855971,0.6886354384137502);
\draw (1.8891895158449743,0.)-- (2.5224136436502516,0.7325154093222184);
\draw (3.620070127306448,0.)-- (3.3443937688971115,0.4241316709391161);
\draw (3.620070127306448,0.)-- (3.9230063248666567,0.3546779652220272);
\draw (2.4607066988747057,0.)-- (2.1929615109123337,0.35140427778491135);
\draw (2.4607066988747057,0.)-- (2.8070748024118894,0.35854878898830234);
\draw (1.2768496768845468,0.027385750816400132)-- (1.049644019784973,0.34450219596460596);
\draw (1.2768496768845468,0.027385750816400132)-- (1.6467514256059657,0.34052550433446105);
\draw (4.524360199007429,0.40859573456984705)-- (4.9182305859025535,0.);
\draw (4.9182305859025535,0.)-- (5.267295908150453,0.4014242779809045);
\end{tikzpicture}
\caption{Combinatorics of primary wakes: $\partial W_{P}\setminus \{c_P\}$ is covered by neighboring wakes. The bubble chains $\bB_{\lambda(P)}, \bB_{P}=\bZ_P, \bB_{\rho(P)}$ landing at $\alpha_P$ are marked blue dashed.}
\label{FIg:comb:prim wakes}
\end{figure}

\begin{lem}[Combinatorics of primary wakes, Figure~\ref{FIg:comb:prim wakes}]
\label{lem:prim wakes:comb}
Consider a wake $\bW_P$. Write 
\[\lambda(P)=(P^-,\dots )\sp \text{ and }\sp \rho(P)=(P^+,\dots).\]
Then $\bW_P,\bW_{P^-},\bW_{P^+}$ are all the primary wakes containing $\alpha_P$ on the boundary.

As in~\eqref{eq:part Ws}, write $\partial \bW_P=\{c_P\} \cup \bR'_{P,-}\cup \bR'_{P,+}$. Then
$\bR'_{P,-}$ splits as a concatenation 
\[\bR'_{P,-}= [\alpha_P,\alpha_{Q_1}]\cup [\alpha_{Q_1},\alpha_{Q_2}]\cup \dots\] such that
\begin{itemize}
\item $[\alpha_P,\alpha_{Q_1}]=\bW_{P^-}\cap \bW_P $;
\item $[\alpha_{Q_i}, \alpha_{Q_{i+1}}]=\bW_{Q_{i}}\cap \bW_P$ for $i\ge 1$;
\item $Q_{i+1}^-=Q_{i},~ Q_1^-=P^-$, and $Q_{i}^+=P$;
\item  $\alpha_{Q_i}$ tends to $c_P$.
\end{itemize} 

Similarly, $\bR'_{P,+}$ splits as a concatenation 
\[\bR'_{P,+}= [\alpha_P,\alpha_{S_1}]\cup [\alpha_{S_1},\alpha_{S_2}]\cup \dots\] 
such that
\begin{itemize}
\item $[\alpha_P,\alpha_{S_1}]=\bW_{P^+}\cap \bW_P $;
\item $[\alpha_{S_i}, \alpha_{S_{i+1}}]=\bW_{S_{i}}\cap \bW_P$ for $i\ge 1$;
\item $S_{i+1}^+=S_{i},~ S_1^+=P^+$, and $S_{i}^-=P$;
\item  $\alpha_{S_i}$ tends to $c_P$.
\end{itemize} 
\end{lem}
\begin{proof}
Recall that $\alpha_P$ is the landing point of $\bB_P=\widehat \bZ_P,\bB_{\lambda(P)}, \bB_{\rho(P)}$, see Lemma~\ref{lem:alpha' is a pnt}. Since $\bB_{\lambda(P)}, \bB_{\rho(P)}$ are in wakes $\bW_{P^-},\bW_{P^+}$, the first claim follows.

We will verify the decomposition of $\bR'_{P,-}$; the decomposition of $\bR'_{P,+}$ can be verified similarly.  It follows from Corollary~\ref{cor:rays meet} that the intersection $ \bW_{P^-}\cap \bW_{P}$ is an arc of the form $[\alpha_v,\alpha_P]$ with $\alpha_v\succ \alpha_P$. Observe that all three bubble chains landing at $\alpha_v$ (see Lemma~\ref{lem:alpha' is a pnt}) belong to different prime wakes because $\alpha_v$ is a point where $\partial \bW_P$ and  $\partial \bW_{P^-}$ split. This implies that $\alpha_v=\alpha_{Q_1}$ for some $Q_1\in \PT_{>0}$; by the first claim, $Q_1^-=P^-$, and $Q_1^+=P$.
 
Similarly, the intersection $ \bW_{Q_1}\cap \bW_{P}$ is of the form $[\alpha_{Q_1},\alpha_{Q_2}]$, where $Q_2^-=Q_1$, and $Q_2^+=P$. Applying induction, we construct $Q_i$ for all $i\ge 1$. It remains to show that $\alpha_{Q_i}$ converges to $c_P$.

By Proposition~\ref{prop:alpha v alpha s is arc}, there is an alpha-point $\alpha_s\in \bR'_{P,-}$ close to $c_P$. At least one of the rays $\bB_s, \bB_{\lambda(s)},\bB_{\rho(s)}$ is not in $\bW_P$; thus $\alpha_s$ is on the boundary of a primary wake $\bW_T\not=\bW_P$. As above, the intersection $\bW_T\cap \bW_P$ is of the form $[\alpha_T,\alpha_{T_2}]$, where $T_2^-=T$.

Since there are at most finitely many critical points in $[c_T,c_{P^-}]$ (a subarc of $\partial \bZ$) of generation less than $T$, the arc $[\alpha_{T}, \alpha_{P}]$ intersects only finitely many primary wakes. This means that $T=Q_i$ for some $i\ge 1$. Since $c_s$ can be chosen arbitrary close to $c_P$, $\alpha_{Q_i}$ converges to $c_P$.
 \end{proof}

\begin{cor}[Tiling]
\label{cor:no ghost limb}
The union of primary wakes $\displaystyle\bigcup_{P>0} \bW_P$ contains $\bO$. Similarly, for every wake $\bW_s$
 with $s=(P_1,\dots, P_n)$ we have
\begin{equation}
\label{eq:cor:no ghost limb}
\bW_s=\overline \bZ_s\cup \bigcup_{P_{n+1}>0}\bW_{(P_1,\dots, P_n,P_{n+1})}. 
\end{equation} 

For every $z\in \Esc(\bF)$ and every $m\ge 1$ there are at most three wakes with disjoint interiors of level $\ge m$ containing $z$. The union of these wakes is a neighborhood of $z$.
\end{cor} 

\begin{proof}
There is a pair of primary wakes $\bW_P$ and $\bW_Q$ such that $\bW_P\cup \bW_Q\cup \overline \bZ$ surrounds an open topological disk $X$ with $0\in \partial X$. Then for every $y\in \bO$, there is an $n\ll 0$ such that $A^n_\str\left(\bW_P\cup \bW_Q\cup \overline \bZ\right)$ encloses $y$. 

By Lemma~\ref{lem:prim wakes:comb}, if $y\not\in \overline \bZ$ is not contained in any prime wake, then there is connected set $Y\ni y$ (a ``ghost limb'') such that $Y\subset \bO\setminus \displaystyle\bigcup_P \bW_P$ and $\overline Y$ intersects $\partial \bZ$, say at $x$. We can choose sequences $\bW_{P_i}$ and $\bW_{Q_i}$ such that $\bW_{P_i}\cup \bW_{Q_i}\cup \overline \bZ$ encloses $Y$ and $c_{P_i},c_{Q_i}$ tend to $x$. By Lemma~\ref{lem:wake shrinks}, the diameters of $\bW_{P_i}$ and $\bW_{Q_i}$ tend to $0$. This is a contradiction.

By construction, 
\begin{equation}
\label{eq:2:prf:cor:no ghost limb}
\bF^{|s|}\colon \intr(\bW_s)\to \C\setminus \overline \bR^0
\end{equation}
is conformal. Then~\eqref{eq:cor:no ghost limb} follows from the first claim by applying the inverse of~\eqref{eq:2:prf:cor:no ghost limb}.
\end{proof}

\begin{cor}
\label{cor:wakes are close to bJ}
For every $\varepsilon>0$ there is a $P\in \PT_{>0}$ such that every connected component of 
\begin{equation}
\label{eq:gaps between bW:bFstr}
\C\setminus\left(\overline \bZ_\str \cup \bigcup_{S\le P} \bW_S\right) 
\end{equation}
is less than $\varepsilon$ in the spherical metric.
\end{cor}
\begin{proof}
Follows from Lemmas~\ref{lem:wake shrinks} and~\ref{lem:prim wakes:comb}. 
\end{proof}

\subsection{Rigidity of the escaping set} 
\label{ss:RigEscSet}
 For an open topological disk $U$, we denote by $\diam_U$ and $\dist_U$ the diameter and distance with respect to the hyperbolic metric of $U$. 
\begin{lem}
\label{lem:F^P exp}
For every primary wake $\bW_P$ the following holds. The map 
\begin{equation}
\label{eq:lem:F^P exp}
\bF^P\colon\bW_P\setminus \overline \bZ_P\to \bO
\end{equation}
 is uniformly expanding with respect to the hyperbolic metric of $\bO$. There is a $C>0$ such that $\diam_\bO(\bW_{(P,Q)})\le C=C_P$ for every secondary subwake $\bW_{(P,Q)}$ of $\bW_P$.
\end{lem}
\noindent Since all wakes are dynamically related, one can show that~\eqref{eq:lem:F^P exp} is uniformly expansion over all $P$. 
\begin{proof}
Let us denote by $\mu$ the hyperbolic metric of $\bO$ and by $\mu'$ the hyperbolic metric of $\bO \setminus \overline \bZ_P$. We will show that \[\iota\coloneqq \id\colon \sp (\bW_P\setminus \overline \bZ_P, \mu)\to (\bW_P\setminus \overline \bZ_P, \mu') \]
is uniformly contracting; this will imply that~\eqref{lem:F^P exp} is uniformly expanding as it factors through $\iota^{-1}$ followed by a (non-uniformly) expanding map.

\begin{figure}
\begin{tikzpicture}

\draw[white,fill=red, fill opacity=0.1] (-5,-2)-- (-5,0)--(0,0)--(0,-2);

\draw[red] (-5,0)--(0,0);

\draw (-3.5,2)--(-2.5,0);
\node[left]at (-3.5,2) {$\bR^0$};

\draw[red,fill=red, fill opacity=0.1] (1,-2)--(3.5,0)--(6,-2);

\node at (3.5,-1) {$\bZ^\circ$};

\node at (-2.5,-1) {$\bZ$};

\node at (3.5,1) {$\bZ_P^\circ$};


\node at (1.7,0) {$\partial \bW_P^\circ$};

\draw[orange] (3.3,0.1) edge[bend right,->] node[above]{$\bF^P\circ L^{-1}$}(-2.4,0.1);

\draw[blue,fill=blue, fill opacity=0.1] (1,2)--(3.5,0)--(6,2);

\draw (1,-0.5)--(3.5,0)--(6,0.5);

\end{tikzpicture}

\caption{$\bZ^\circ,\bZ_P^\circ$ are the preimages of $\bZ$ under $\bF^P\circ L^{-1}$. }
\label{Fig:Z circ}
\end{figure}

Since $\partial \bW_P\cap \partial \bO=\{c_P\}$, the map~$\iota$ is uniformly contracting away from a small neighborhood of $c_P$ (by compactness). In a small neighborhood of $c_P$ the uniform contraction of $\iota$ follows from the self-similarity of $\bW_P,\bZ_P, \overline \bZ$ at $c_P$; it implies that points on $\partial \bW$ have comparable distances to $\overline \bZ$ and $ \bZ_P$. Let us provide details.

 Let us lift the global self-similarity $A_\str$ to a local self-similarity of $\bW_P$ near $c_P$. Let $U\Supset A_\str(U)$ be a small open disk around $0$ such that $0$ is the only critical value of $\bF^P$ in $U$. Let $V\ni c_P$ be the lift of $U$ along $\bF^P\colon c_p\mapsto 0$. Then $A_\str\mid U$ lifts to a conformal map $\psi$ defined on $V$. By construction, $c_P$ is an attracting fixed point of $\psi$. Let $L\colon (V,c_P)\to (\C,0)$ with $L'(c_p)=1$ be the linearizer conjugating  $\psi$ to the scaling $A_\nu$. (We note that $\nu^2=\mu_\str$ because $\bF^P$ is $2$-to-$1$ near $c_P$.) Consider 
\begin{equation}
\label{eq:L:lem:F^P exp}
L(\bW_P), \sp L(\bZ),\sp  L(\bZ_P),\sp L(\bO);
\end{equation}
 these objects are forward invariant under $A_\nu$. Using backward iterates of $A_\nu$, we globalize~\eqref{eq:L:lem:F^P exp}; we denote the results by $\bW_P^\circ,\sp \bZ^\circ,\sp \bZ_P^\circ, \sp \bO^\circ;$ these new objects are completely invariant under $A_\nu$, see Figure~\ref{Fig:Z circ}. Let $\mu^\circ$ and $\mu'^\circ$ be the hyperbolic metrics of $\bO^\circ \cup \bZ^\circ_P$ and $\bO^\circ $ respectively. The invariance under $A_\nu$ implies that
\[ \iota^\circ \colon (\bW^\circ_P\setminus \bZ^\circ_P,\mu^\circ)\to (\bW^\circ_P\setminus \bZ^\circ_P,\mu'^\circ)\]
is uniformly contracting. Since $L'(c_p)=1$, the uniform contraction of $\iota^\circ$ implies the uniform contraction of $\iota$ near $c_P$. 
 
Consider a secondary wake $\bW_{(P,S)}$ and observe that if $\bW_{(P,S)}$ is close to $c_P$, then $\bW_{(P,S)}$ is contained in a small neighborhood of $c_P$. Indeed, if $\bF^P(\bW_{(P,S)})$ is close to $0$, then $\bF^P(\bW_{(P,S)})$ is small by Lemma~\ref{lem:wake shrinks} and Corollary~\ref{cor:no ghost limb}. Therefore, the second claim of the lemma is obvious unless $\bW_{(P,S)}$ is contained in a small neighborhood of $c_P$. At a small neighborhood of $c_P$, the second claim follows from the self-similarity at $c_P$ by applying $L$.
\end{proof}

\begin{lem}[Nested wakes shrink]
\label{lem:nest wakes shrink}
For an infinite sequence $(P_1,P_2,\dots)$ write $s(n)\coloneqq (P_1,\dots ,P_n)$. Then 
\[\bigcap _{n\ge 1} \bW_{s(n)}\]
is a singleton.
\end{lem}
\begin{figure}
\begin{tikzpicture}
\draw[red](-4,0)--(7,0);
\draw
(0,0)--(-2,4);
\node[left] at(-2,4){$\bR^0$};
\node[below] at(0,0){0};
\draw[yscale=2] (5,0) -- (4.2,1)--(5,2)--(5.8,1)--(5,0);
\filldraw[opacity=0.3,shift={(0.3,1)}] (5.5,1) .. controls (5.7,0.8) and (5.9,0.8).. (6,1)
  .. controls (5.9,1.2) and (5.6,1.2)..(5.5,1);
  
\filldraw[shift={(0.3,1)}] (5.75,1) circle (1.5pt);  
\draw[shift={(0.3,1)}] (5.75,1) -- (6.3,0.7);
\node[shift={(0.3,1)},right] at  (6.3,0.7){$\bW_{v}$};
 
\draw  (5,0) .. controls (3,0.8) and (3,3.7).. (5,4)
  .. controls (7,3.7) and (7,0.8)..(5,0);
\node at (5.8,3) {$\bW_P$};

\filldraw[opacity=0.5]  (-3,0) .. controls (-3.3,0.1) and (-3.3,0.75).. (-3,0.8)
  .. controls (-2.7,0.75) and (-2.7,0.1)..(-3,0);

\node[above]at  (-3,0.8) {$\bF^P(\bW_{v})$};
\draw[blue](5.9,1.85) edge[->,bend left] node[below]{$\bF^P$}(-2.9,-0.1);

\draw[scale=0.3]  (5,0) .. controls (3,0.8) and (3,3.7).. (5,4)
  .. controls (7,3.7) and (7,0.8)..(5,0);
\node at (1.5,0.6) {$\bW_{\tt^n P}$};

\draw[blue] (-2.7,0.3) edge[->,bend left] node[below]{$\bF^Q$}(0.95,0.7);

\draw[red] (1.5,1.1) edge[->] node[above left]{$A^{-n}_\str$}(4.8,3.8);

\end{tikzpicture}
\caption{Construction of the first scaled return $\chi\coloneqq A_\str^{-n }\circ \bF^{P+Q}\colon \bW_{v}\to \bW_P$.}
\label{fig:lem:nest wakes shrink}
\end{figure}

\begin{proof}

For a secondary wake $\bW_{v}\subset \bW_P$, we define its \emph{first scaled return} \[\chi\coloneqq A_\str^{-n}\circ \bF^{P+Q}\colon \bW_v\to \bW_P\] as follows, see Figure~\ref{fig:lem:nest wakes shrink}. 
The map $\bF^P\colon \intr (\bW_P)\to \C\setminus \overline{\bR^0}$ is univalent and $\bF^P(\bW_{v})$ is a primary wake. Let $Q\in \PT$ be the minimal such that $\bF^{P+Q}(\bW_{v})=\bW_{\tt^{n} P}$ for some $n\in \Z$. Then $A_\str^{-n}\circ \bF^{P+Q}(\bW_v) =\bW_P$. By Lemma~\ref{lem:F^P exp}, $\chi=\bF^{P+Q}$ is expanding uniformly over all the secondary subwakes of $\bW_P$.

The first scaled return $\chi$ is defined for all $x\in\bW_P\setminus \bZ_P$ (with a natural interpretation of the ambiguity at common boundary points of secondary wakes). By~\eqref{lem:F^P exp}, 
\[\diam_{\bO}\chi^{k-2}(\bW_{s(k)}) \le C.\]
Since $\chi$ is uniformly expanding, we obtain that $\diam_{\bO} (\bW_{s(k)})$ tends to $0$. 
\end{proof}

\begin{cor}
\label{cor:ext rays land}
Every external ray lands.  Every periodic external ray has a minimal period.  
\end{cor}
\begin{proof}
Observe that the accumulating set $Y$ of a ray does not transversally intersect the boundaries of wakes: either $Y\subset \bW_s$ or $Y\subset \overline{\C\setminus \bW_s }$ for every wake $\bW_s$. If $Y$ intersects $\partial \bZ_s$ for some $s$, then $Y\subset \partial \bZ_s$ by Corollary~\ref{cor:no ghost limb}. Since the roots of wakes (i.e.~critical points) are dense on $\partial \bZ_s$, we obtain that $Y$ is a singleton. If $Y$ does not intersect any  $\partial \bZ_s$, then $Y$ belongs to a nested sequence of closed wakes (by Corollary~\ref{cor:no ghost limb}). By Lemma~\ref{lem:nest wakes shrink}, $Y$ is a singleton.  

If $\bR$ is a periodic ray, then $\bR$ lands, say at $x$. By Corollary~\ref{cor:min period:per pts}, $x$ has a minimal period $P\in \PT_{>0}$. Therefore, $\bR$ has a minimal period $m P$ for some $m\ge 1$.
 \end{proof}

Lemma~\ref{lem:nest wakes shrink} implies that a limb is dense in the corresponding wake:
\begin{cor}
\label{cor:wake:clos:limb}
We have $\bW_s=\overline \bL_s$ for every $s$.\qed
\end{cor}

Let us declare  $x,y\in \Esc_P(\bF)$ to be \emph{combinatorially equivalent} if there is no alpha-point $\alpha_s$ such that $x,y$ are in different connected components of $\Esc_P(\bF)\setminus \{\alpha_s\}$. This generates a combinatorial equivalence. A point $x\in  \Esc(\bF)$ is an \emph{endpoint} of $\Esc(\bF)$ if $x\not\in [\alpha_1,\alpha_2]$ for alpha-points $\alpha_1\succ \alpha_2$.

\begin{prop}
\label{prop:Esc M}
Every combinatorial equivalence class of $\Esc(\bF)$ is a singleton. For every $M\in \R_{>0}$ we have
\begin{equation}
\label{eq:prop:Esc M}
\Esc_M(\bF)=\overline{\bigcup_{P<M} \Esc_P(\bF)}.
\end{equation}
For every $P\in \PT_{>0}$ the escaping set $\Esc_P(\bF)$ is uniquely geodesic. 
\end{prop}
\begin{proof}
Consider $x\in \Esc_P(\bF)$. If $x$ is not an endpoint of $\Esc(\bF)$, then Proposition~\ref{prop:alpha v alpha s is arc}  implies that $x$ can be separated by infinitely many alpha-points from any other point in the escaping set.

 Suppose $x$ is an endpoint of $\Esc(\bF)$. By Corollary~\ref{cor:no ghost limb}, $x$ belongs to a nested sequence of wakes. By Lemma~\ref{lem:nest wakes shrink}, the combinatorial class of $x$ is a singleton. Moreover, the external ray $(x,\binfty)$ lands at $x$; i.e.~$\Esc_P$ is uniquely geodesic. 
 
 Equation~\eqref{eq:prop:Esc M} is immediate. 
\end{proof}

\begin{cor}
\label{cor:Esc is Tree}
Suppose $X\subset \Esc_P(\bF)$ is a discrete subset of $\C$. Then the \emph{connected hull} of $X$
\[\bigcup_{x,y\in X} [x,y]\]
is a tree.
\end{cor}
\begin{proof}
Since $\Esc_P(\bF)$ is uniquely geodesic, the intersection of $\displaystyle\bigcup_{x,y\in X} [x,y]$ with any proper lake is a finite tree. Since lakes exhaust $\bO$ (see  for example~\eqref{eq:lakes exhaust ocean}), the claim follows. 
\end{proof}

\begin{lem}
\label{lem:Esc has empty int}
The escaping set $\Esc(\bF)$ has empty interior and supports no invariant line field. Every Fatou component of $\bF$ is either $\bZ$ or its iterated preimage. 
\end{lem}
\noindent Such a statement is referred to as the Hairiness Theorem.
\begin{proof}
We give the sketch of the proof because the argument is standard.

Since every combinatorial class of $\Esc(\bF)$ is trivial (Proposition~\ref{prop:Esc M}), we have $\intr(\Esc(\bF))=\emptyset$.

Suppose $X$ is a Fatou component of $\bF$ such that $X$ is not an iterated preimage of $\bZ$. By Corollary~\ref{cor:no ghost limb}, $X$ belongs to a nested sequence of wakes; the wakes shrink to a point by Lemma~\ref{lem:nest wakes shrink}. This is a contradiction.

Suppose $\Esc(\bF)$ supports an invariant line field $L$. There are two cases. If the integration along $L$ as in~\S\ref{ss:QC deform} gives a non-trivial continuous path $\bG_t$ emerging from $\bF=\bF_{\str}$, then $\bG_t$ is a path in $\Unst$ contradicting the rigidity of $\bF$. 

In the second case, we would obtain a non-trivial qc map $h\colon \C\to \C$ commuting with $\bF$. Since $h$ is identity on $\overline \bZ$, we obtain that $h$ is identity on all $\bZ_s$ which are dense. This implies that $h$ is identity everywhere.
\end{proof}

\begin{lem}
The closure of the escaping set $\Esc_P(\bF)\cup \{\binfty\}\subset \widehat \C$ is locally connected for all $P\in \PT_{>0}$.
\end{lem}
\begin{proof}
Local connectivity of $\Esc_P(\bF)$ at any its point $z$ follows from Lemma~\ref{lem:nest wakes shrink}. Indeed, if $\displaystyle\bigcup_i \bW_i$ is a neighborhood of $z$ consisting of at most $3$ pairwise non-nested wakes of level $\ge m$ (see Corollary~\ref{cor:no ghost limb}), then $\left(\displaystyle\bigcup_i \bW_i\right)\cap \Esc_P(\bF)$ is a neighborhood of $z$ in $\Esc_P(\bF)$. By Lemma~\ref{lem:nest wakes shrink},  $\displaystyle\bigcup_i \bW_i$ shrinks as $m$ increases.

Recall that there are nested lakes  $\bO(P)\subset \bO(\tt^{-1}P)\subset \dots$ such that
$\C=\displaystyle\bigcup_{n\le 0} \bO(\tt^n P)$, see~\eqref{eq:lakes exhaust ocean}. Then $\big(\Esc_{P}(\bF)\cup \{\binfty\}\big)\setminus\bO(\tt^{n}P) $ is a small neighborhood of $\binfty$.
\end{proof}

We can now replace $\RR\colon \BB\dashrightarrow \BB$ with a new operator so that its renormalization period $\mm$ is optimal: 
\begin{prop}
\label{prop:Ren oper with min per}
The pacman renormalization operator $\RR\colon \BB\dashrightarrow \BB$ from ~\S\ref{sss:HypSelfOper} can be constructed so that $\mm$ is the minimal period of $\theta_\str$ under $\cRRc$.
\end{prop}
\begin{proof} 
Let $\mm_{\min}$ be the minimal period of $\theta_\str$, and let $\tt_{\min}$ be the eigenvalue of the associated antirenormalization matrix, see Lemma~\ref{lem:lambda:t*t}. The cascades $\{ \bF_\str^P\}_{P\in \PT}$ and $\{\bF_\str^{\tt_{\min} P}\}_{P\in \PT}$ are combinatorially equivalent. By the McMullen Rigidity Theorem (see~\cite{McM2}*{Theorem 8.1}), these cascades are conformally conjugate. Let us denote by \[A_{\str,\min}\colon z\mapsto \mu_{\str,\min} z,\sp\sp |\mu_{\str,\min}|<1\] the scaling map conjugating $\bF_\str^P$ to $\bF_\str^{\tt_{\min} P}$. Consider the renormalization sector $\bS$ from \eqref{eq:Max Prep U pm}. A priori, $A_{\str,\min}(\bS)$ does not have to be in $\bS$. (In~\cite{DLS} we passed to an iteration to obtain such a property.) Below we will use external rays of $\bF_\str$ to construct a new $\bS^\new$ satisfying necessary inclusion properties.

Set $J=[x,y]\coloneqq \partial \bZ\cap \bS$. Let $\bR_x,\bR_y$ be the external rays landing at $x,y$, and let $\bI_x,\bI_y$ the internal rays of $\bZ_\str$ landing at $x,y$. Define $\bS^\new$ with $\bS^\new\cap \bZ_\str=\bS\cap \bZ_\str$ to be the sector bounded by $\bI_x\cup \bR_x\cup \bR_y\cup \bI_y$. Set \[\bR_2\coloneqq \bbf_-(\bR_y)=\bbf_+(\bR_x)\sp\sp\text{ and }\sp\sp \bI_2\coloneqq \bbf_-(\bI_y)=\bbf_+(\bI_x).\]
Cutting $\bS^\new$ along $\bI_2\cup \bR_2$ and and pulling back the resulting two sectors under $\bbf_{\str,\pm}=\bbf_{\pm}$, we obtain a full prepacman $\bF^\new= (\bbf_{\pm}\colon \bU^\new_{\pm }\to \bS^\new)$; see Figure~\ref{Fg:Prepacman} for illustration (where $\beta_-,\sp \beta_+$ are $\bI_x\cup \bR_x,\sp \bI_Y\cup \bR_y$ and $\beta$ is $\bI_2\cup \bR_2$). Conjugating $\bF^\new$ by $A_{\str,\min}$, we obtain a full prepacman $\bF_1^\new=(\bbf_{1,\min,\pm}\colon \bU^\new_{1,\pm }\to \bS_1^\new)$. Then $\bS_1^\new\subset \bS^\new$ and, moreover, the forward orbit of $\bU^\new_{1,\pm }$ under $\bbf_{\pm}\mid \bU^\new_{\pm }$ is within $\bU^\new_{-}\cup \bU^\new_{+}$.

Let $\alpha'$ be the preimage of $\balpha$ under $(\bbf_{\pm}\colon \bU^\new_{\pm }\to \bS^\new)$; the $\alpha'$ is the point where $\bR_x$ and $\bR_y$ meet. Choose an equipotential $\bE\subset \bZ_\str$ close to $\balpha$, and let $O$ be the connected component of $\bS^\new\setminus \bE$ attached to $\balpha$. Then $O$ has an $\bF^\new$ lift $O'$ attached to $\alpha'$. We remove $O'$ from $\bU^\new_{-}\cup \bU^\new_{+}$ and we glue dynamically the sides of $\bS^\new$ to obtain a pacman $f_{\str,\new}\colon U\to V$ out of $\bF^\new$. As with $\bF^\new$, we remove a lift of $A_{\str,\new}(O)$ from $\bU^\new_{1,- }\cup \bU^\new_{1, +}$; then we project the new prepacman into the dynamical plane of $f_{\str,\new}$; we denote by $F_1$ the resulting prepacman. By construction, $F_1$ is a prepacman for $f_{\str,\new}$. Moreover, the associated renormalization triangulation $\bDelta_{F_1}$ is compactly contained in $U$.  By~\cite{DLS}*{Theorem 2.7}, see~\S\ref{sss:RenOperat}, there is a pacman renormalization operator defined in a small Banach neighborhood of $f_{\str,\new}$ realizing $\mm_{\min}$. By \cite[Theorem 7.7]{DLS} (see~\S\ref{sss:HypSelfOper}), $\RR_\new$ is hyperbolic with one-dimensional unstable manifold.
\end{proof}

\section{Holomorphic motion of the escaping set}
\label{s:HolMot Esc}
Consider a maximal prepacman $\bF\in\Unst$. Recall that the escaping set $\Esc_S(\bF)=\C\setminus \Dom (\bF^S)$ is defined for $S\in \PT\subset \R_{\ge0}$. For $M\in \R_{\ge0}$ we define \[\displaystyle{\Esc_M(\bF)\coloneqq \bigcap_{S\ge M}\Esc_S(\bF)},\] where the intersection is taken over all $S\in \PT$ with $S\ge M$.  For $x\in \Esc(\bF)$ its \emph{escaping time} is the minimal $M$ such that $\Esc_M(\bF)\ni x$.

\subsection{Stability of $\boldsymbol\sigma$-branched structure}
\label{ss:impl finct thm}
We need the following approximation of $\bF$ by finite degree branched coverings.
\begin{lem}
\label{lem:fin cov appr}
There is a neighborhood $\bUU \subset \Unst$ of $\bF_\str$ such that for every ${\bF\in \bUU}$ there are sequences of disks $\bD^{(-1)}\subset \bD^{(-2)}\subset \dots $ and $ \bW^{(-1)}_{\pm}\subset  \bW^{(-2)}_{\pm}\subset \dots $ with $\displaystyle{\bigcup _{k<0}\bD^{(k)}=\C}$ and $\displaystyle{\bigcup_{k<0} \bW_\pm ^{(k)}=\Dom \bbf_\pm}$ such that
\begin{equation}
\label{eq:lem:fin cov appr:defn}
(\bbf_{0,-},\bbf_{0,+}) \colon  \bW^{(k)}_{-} ,  \sp \bW^{(k)}_{+} \to \bD^{(k)}
\end{equation}
are proper branched coverings of finite degree depending continuously on $\bF$ when $k$ is fixed. Moreover,  all maps~\eqref{eq:lem:fin cov appr:defn} are Hurwitz equivalent for a fixed $k$. 
\end{lem}
\noindent It follows, in particular, that the degree of $\bbf_{0,\pm }\mid  \bW^{(k)}_{\pm}$ is constant and critical points of $\bbf_{0,\pm }$ do not collide.
\begin{proof}
The proof follows the lines of~\cite[Theorem 5.5]{DLS} asserting that a pacman on the unstable manifold extends to a maximal prepacman. Let us first review the proof of \cite[Theorem 5.5]{DLS}; see also~\S\ref{sss:max prepacmen}.

For a pacman $f_0\in \WW^u$, let 
\[\mathfrak F_0 =\left(f_{0, \pm } \colon \mathfrak U_{0,\pm} \to \mathfrak S \coloneqq V\setminus (\gamma_1\cup O)\right)\] 
be a commuting pair obtained from $F_{0} = \left(f_{0, \pm } \colon U_{0,\pm}\to V\setminus \gamma_1\right)$ by removing a small neighborhood $O$ of $\alpha$ from $ V\setminus \gamma_1$ and by removing $f_{0,\pm}^{-1}(O)$ from $  U_{0,\pm}$. The map
$\phi_{k}\circ   \dots\circ \phi_{-1}$ embeds $\mathfrak F_0$ to the dynamical plane of $f_k$ as a commuting pair denoted by
\begin{equation}
\label{eq:CommPaif:n}
\mathfrak F^{(k)}_0 =\left( f_k^{\aa_{k}},f^{\bb_{k}}_k \right)\colon \mathfrak U^{(k)}_{0,-} ,\sp \mathfrak U^{(k)}_{0,+} \to  \mathfrak S_0^{(k)}.
\end{equation}

Since $\phi_{k}$ is contracting at the critical value the diameter of $U^{(k)}_{0,-} \cup U^{(k)}_{0,+}\cup  \mathfrak S^{(k)}_{0}\ni c_1(f_n)$ tends to $0$. Let us set $c_{1+k}(f_n)\coloneqq f^k(c_1)$. The key step in verifying the maximal $\sigma$-proper extension is the following lemma:

\begin{lem}[{\cite[Key Lemma 4.8]{DLS}}]
\label{lem:KeyLemma}
There is a small open topological disk $D$ around $c_1(f_\str)$ and there is a small neighborhood $\UU\subset \WW^u$ of $f_\str$ such that the following property holds. For every sufficiently big $n\ge 1$, for each $\tt\in \{\aa_n,\bb_n\}$, and for all $f\in \RR^{-n} (\UU)$, we have ${c_{1+\tt}\in D
}$ and $D$ pullbacks along the orbit $c_1(f), c_2(f),\dots , c_{1+\tt}(f)\in D$ to a disk $D_{0}$ such that $f^{\tt}\colon D_{0} \to D$ is a branched covering; moreover, $D_0\subset U_{f}\setminus \gamma_1$.
\end{lem}

Lemma~\ref{lem:KeyLemma} implies that for a sufficiently big $k<0$ the pair~\eqref{eq:CommPaif:n} extends into a pair of commuting branched coverings
\begin{equation}
\label{eq:CommPaif:n:ext}
F^{(k)}_0 =\left( f_{k}^{\aa_k},f^{\bb_k}_{k} \right)\colon  W_{-}^{(k)} ,\sp   W^{(k)}_{+} \to D,
\end{equation} with $W_{-}^{(k)} \cup   W^{(k)}_{+} \cup D\subset V\setminus \gamma_1$. Conjugating~\eqref{eq:CommPaif:n:ext} by $A_\str^k\circ h_k$ we obtain the commuting pair  
\[(\bbf_{0,-},\bbf_{0,+}) \colon\sp   \bW^{(k)}_{-} ,\sp   \bW^{(k)}_{+} \to \bD^{(k)}\]
such that $\displaystyle{\bigcup_{k\ll 0} \bD^{(k)} =\C}$ (because the modulus of $\bD^{(tm-t)} \setminus \bD^{(tm)}$ is uniformly bounded away from $0$ for $t\gg 0$ and all $m\le 0$). This implies that
\[(\bbf_{0,-},\bbf_{0,+}) \colon  \bigcup_{k\ll 0}\bW^{(k)}_{-} ,\sp   \bigcup_{k\ll 0}   \bW^{(k)}_{+} \to \C\] is a pair of $\sigma$-proper maps.

Let us argue that~\eqref{eq:CommPaif:n:ext} can be slightly adjusted such that the new commuting pair depends continuously on $f_0$ as required. Recall that $c_0$ and $c_1$ denote the critical point and the critical value of a pacman $f$. For $i\ge 0 $ we write $c_i=f^i(c_0)$.

By~\cite{DLS}*{Theorem 4.6}, if $f_0$ is sufficiently close to $f_\str$, then for  every $k\ll0$ the disk $D$ can be slightly shrunk to the disk $D(f_0,k)$ (uniformly in $k$ and $f_0$) such that the following property holds: for $i\le \max\{\aa(k),\bb(k)\}$ we have 
\begin{equation}
\label{eq:1:prf:lem:fin cov appr}
D(f_\str,k)\ni c_i(f_\str)\sp \text{ if and only if }\sp D(f_0,k)\ni c_i(f_k)
\end{equation}
(because points $c_i(f_k)$ are sufficiently close to $c_i(f_\str)$). Pulling back $D(f_0,k)$ along the orbit of~\eqref{eq:CommPaif:n:ext}, we obtain the extension of~\eqref{eq:CommPaif:n} to the commuting pair   
\begin{equation}
\label{eq:CommPaif:n:ext:new}
F^{(k)}_0 =\left( f_{k}^{\aa_k},f^{\bb_k}_{k} \right)\colon  W_{-}^{\new, (k)} ,  \sp W^{\new,(k)}_{+} \to D(f_0,k).
\end{equation}
Note that $f_k\colon f_k^i \big( W_{\iota}^{\new, (k)}  \big)\to  f_k^{i+1} \big( W_{\iota}^{\new, (k)}  \big)$ for $\iota\in\{-,+\}$ has degree $2$ if and only if $ f_k^i \big( W_{\iota}^{\new, (k)}  \big)$ contains the critical point of $f_k$. By~\eqref{eq:1:prf:lem:fin cov appr}, the pair~\eqref{eq:CommPaif:n:ext:new} depends continuously on $f_0$ as required. Conjugating~\eqref{eq:CommPaif:n:ext:new} by $A_\str^k\circ h_k$ we obtain the required commuting pair~\eqref{eq:lem:fin cov appr:defn}.
\end{proof}

Combining with the Implicit Function Theorem, we obtain:
\begin{cor}
\label{cor:hol mot of preimages}
Set \[T\coloneqq \min\{(0,1,0),\sp (0,0,1)\}.\]
There is a point $x\in \C$ and an open disk $\bUU$ containing $\bF_\str$ such that $\displaystyle{ \bigcup_{S\le T}\bF^{-S}(x)}$ moves holomorphically with $\bF\in \bUU$.
\end{cor}
\begin{proof}
By Lemma~\ref{lem:crit pts}, $\CV(\bF^T)$ moves holomorphically within a small neighborhood of $\bF_\str$. Therefore, we can shrink $\bUU$ from Lemma~\ref{lem:fin cov appr} and choose a point $x\in \bU_-(\bF_\str)$ such that $x$ belongs to the interior of $\bU_-(\bF)$ and $x$ does not hit $\CV(\bF)$ for all $\bF\in \bUU$.  By Lemma~\ref{lem:fin cov appr}, $\bF^{-S}(x)$ moves holomorphically with $\bF\in \bUU$ for all $S\le T$. Since the points $\bF^S(x)$ with $S\le T$ belong to different triangles of $\bDelta_0(\bF)$, the set $\bF^{-S}(x)$ is disjoint from $\bF^{-Q}(x)$ for $S<Q\le T$. We obtain a holomorphic motion of $\displaystyle{ \bigcup_{S\le T}\bF^{-S}(x)}$.
\end{proof}

\subsection{Holomorphic motion of the escaping set}
\label{ss:HolMotEsc}
We need the following facts.

\begin{lem}
\label{lem:unique crt pnt}
Let $g\colon U\to V$ be a finite branched covering between open topological disks. If $g$ has a unique critical value, then $g$ also has a unique critical point.
\end{lem}
\begin{proof}
By assumption, $g\colon U\setminus \CP(g)\to V\setminus \CV(g)$ is a covering. Since $\pi_1(V\setminus \CV(g))$ is an abelian group, so is $\pi_1(U\setminus \CP(g))$. Therefore, $\CP(g)$ is a singleton. 
\end{proof}

\begin{lem}
\label{lem:Acc Of preim}
Let $g\colon \Disk \to \C$ be a $\sigma$-proper map. Fix an open disk $W\subset \C$. Let $U$ be an open set intersecting $\partial \Disk$. Then $g^{-1}(W)$ intersects $U$.
\end{lem}
\begin{proof}
Suppose $g^{-1}(W)\cap U=\emptyset$. Choose an open topological disk $V\subset U\cap \Disk$ such that $ \partial V$ is a simple closed curve containing an arc $I\subset \partial \Disk$. Choosing a point $w$ in $W$ and postcomposing $g$ with a  M\"obius transformation $h$ moving $w$ to $\infty$, we obtain that the new function $s\coloneqq h\circ g\colon \Disk \to \widehat \C$ is bounded on $V$.  By the Fatou Theorem, $s\mid V$ has a radial limit at almost every point in $\partial V$. By the Riesz Theorem, almost all of the radial limits are different from any fixed number, in particular from $h(\infty)$. Therefore, there is an $x\in I$ such that $g$ has a finite radial limit $y$ at $x$. This contradicts to the assertion that $g$ is $\sigma$-proper: if $\Sigma$ is a small open neighborhood of $y$, then every connected component of $g^{-1}(\Sigma)$ is disjoint from $\partial \Disk\ni x$ and the radial limit at $x$ is not $y\in \Sigma$. (In fact, $\sigma$-proper maps have no asymptotic values.) Therefore, $g^{-1}(W)\cap U\not=\emptyset$.
\end{proof}

\begin{rem}
\label{rem:lem:Acc Of preim}
Lemma~\ref{lem:Acc Of preim} also holds for every $\sigma$-proper map $g\colon U \to \C$ defined on a simply connected set $U\subsetneq \C$. In the proof, the angular metric of $\partial \Disk$ is replaced with the harmonic measure of $\partial U\subset \hC$. 
\end{rem}

\begin{cor}
\label{cor:preim accum at the boundary}
Consider $\bF\in \Unst$ and $P\in \PT_{>0}$. Then for every $x\in \C$, the boundary $\partial \Dom(\bF^P)$ is the set of accumulating points of $\bF^{-P}(x)$.
\end{cor}
\begin{proof}
Let $U$ be a connected component of $\Dom(\bF^P)$. We claim that the preimages of $x$ under $\bF^P\mid U$ accumulate at $\partial U$.

 Since the set of critical values of $\bF^P$ is discrete (Lemma~\ref{lem:discr of dyn}), we can choose small neighborhoods $W\Supset \Omega \ni x$ so that $W\cap \CV(\bF^P)\subset  \{x\}$. Thus $W$ contains at most one critical value. By Lemma~\ref{lem:unique crt pnt}, every connected component $W'$ of $\bF^{-P}(W)$ contains at most one critical point. By Lemma~\ref{lem:crit pts}, the degree of $\bF^P\colon W'\to W$ is at most $k$, where $k$ is independent of $W'$. Therefore, if $\Omega'$ is the connected component of $\bF^{-P}(\Omega)$ within $W'$, then the modulus of $W'\setminus \Omega'$ is uniformly bounded from $0$.

  Let us precompose $\bF^P\colon U\to \C$ with a Riemann map $\Disk\to  \Dom(\bF^P)$; we obtain a $\sigma$-proper map $g\colon \Disk\to \C$. By Lemma~\ref{lem:Acc Of preim} and Remark~\ref{rem:lem:Acc Of preim}, for every $y\in \partial \Disk$, there is a connected component $\Omega'$ of $g^{-1}(\Omega)$ close to $y$. Denote by $W'$ the connected component of $g^{-1}(W)$ containing $\Omega'$. Since $W'\setminus \Omega'\subset \Disk$ and the modulus of $W'\setminus \Omega'$ is bounded from $0$, the component $\Omega'$ is small. Since $\Omega'$ contains a preimage of $x$, the corollary is verified.
\end{proof}

A \emph{conformal motion} of a set $E\subset \C$ parametrized by a complex manifold is a holomorphic motion whose dilatation on $E$ is $0$. The following lemma follows from Corollaries~\ref{cor:preim accum at the boundary} and~\ref{cor:hol mot of preimages} and is reminiscent to \cite{Re}*{Theorem 1.1}.

\begin{lem}
\label{lem:esc set moves hol}
The set 
\begin{equation}
\label{eq:lem:esc set moves hol}
\DEsc_P=\C\setminus \ParEsc_P\coloneqq\{ \bF\in \Unst\mid 0\not \in \Esc_P(\bF)\}
\end{equation}
is open. There is a unique equivariant conformal motion $\tau$ of $\Esc_P(\bF)$ along any curve in $\DEsc_P$.

The escaping set $\Esc(\bF)$ has empty interior and supports no invariant line field for every $\bF\in \Unst$.
\end{lem}
\noindent We do not claim in this lemma that $\DEsc_P$ is connected. As in~\cite{Re}, we will show that the motion of $\Esc_S(\bF)$ has small dilatation for small $S$ (here~\eqref{eq:ren:F:F_n} is used), and then dynamically extend the motion of $\Esc_S(\bF)$ to the motion of $\Esc_P(\bF)$ with the same dilatation. 
\begin{proof}

Consider $T,\bUU$, and $x$ from Corollary~\ref{cor:hol mot of preimages}; i.e.~$\displaystyle{ {\bigcup_{S\le T}\bF^{-S}(x)}}$ moves holomorphically with $\bF\in \bUU$, where $\bUU$ is a neighborhood of $\bF_\str$. Applying the $\lambda$-lemma, we obtain the holomorphic motion $\tau$ of $\displaystyle{ \overline{\bigcup_{S\le T}\bF^{-S}(x)}}$. By Corollary~\ref{cor:preim accum at the boundary}, 
\begin{equation}
\label{eq:prf:eq:lem:esc set moves hol}
\partial \Esc_T(\bF)\subset { \overline{\bigcup_{S\le T}\bF^{-S}(x)}}
\end{equation}
 moves holomorphically with $\bF\in \bUU$. Clearly, $\tau$ is an equivariant motion. 
 
 In particular, $\tau$ induces an equivariant qc map between $\Esc_T(\bF)$ and $\Esc_{T}(\bF_\str)$ for every $\bF\in \bUU$. By Lemma~\ref{lem:Esc has empty int}, $\Esc_T(\bF_\str)$ has empty interior and supports no invariant line field. Therefore, $\Esc_T(\bF)$ has empty interior and supports no invariant line field and $\tau$ is a holomorphic motion of $\Esc_T(\bF)=\partial \Esc_T(\bF)$. Observe that the dilatation of $\tau$ is small if $\bF$ is close to $\bF_\str$. Moreover, by Proposition~\ref{prop:Esc M}, we have:
\begin{equation}
\label{eq:Esc M is: cup Esc S}
 \Esc_M(\bF)=\overline{\bigcup_{S<M} \Esc_S(\bF)}.
\end{equation}
for all $M\le T$.

Let us now show that the definition of $\tau$ is independent of $x$. Suppose that $y_{\bG}\in \C\setminus \CV(\bG)$ is a point that depends holomorphically on $\bG$ in a small neighborhood of $\bF\in \bUU$. In a smaller neighborhood of $\bF$, we can connect $x,y_\bG$ by a simple arc $\ell_\bG\subset \C\setminus \CV(\bG)$ and surround $\ell_\bG$ by an annulus $A_\bG\subset \C\setminus \CV(\bG)$ (we recall that $\CV(\bG^T)$ depends continuously on $\bG$ by Lemma~\ref{lem:crit pts}). If $\ell_n$ is a sequence of preimages of $\ell_\bG$ under $\bG^{T}$ accumulating at a point in $\Esc_T(\bG)=\partial \Esc_T(\bG)$, then the diameter of $\ell_n$ tends to $0$ because every $\ell_n$ is separated from $\Esc_T(\bG)$ by a conformal preimage of $A_\bG$. Therefore, the motion $\tau$ coincides with the motion of 
\[ \Esc_T(\bG)\subset \overline {\bG^{-T}(y_{\bF})}\] in a small neighborhood of $\bF$.

 We claim that $\tau$ is the unique equivariant holomorphic motion of $\Esc_T(\bF)$. Suppose  $\tau'$ is another equivariant holomorphic motion of $\Esc_T(\bG)$ defined in some neighborhood of $\bF$. Choose a small $S\in \PT$ and a point $y\in \Esc_S(\bG)$ and let $y_\bF$ be the motion of $y$ induced by $\tau'$. Since $\tau'$ is equivariant, the motion of $\bF^{S-T}(y_{\bF})$ is $\tau'$. Since
\[ \Esc_{T-S}(\bF)\subset \overline {\bF^{S-T}(y_{\bF})}\] we see that $\tau'$ and $\tau$ coincide on $ \Esc_{T-S}(\bF)$ for all $S\in \PT_{>0}$ such that $S<T$. Therefore, the motions $\tau$ and $\tau'$ coincide on $\Esc_{T-S}(\bF)$ and, by~\eqref{eq:Esc M is: cup Esc S}, on $\Esc_{T}(\bF)$.

 Let us show that the set
\begin{equation}
\label{eq:lem:esc set moves hol:2}
\DEsc_P\cap \bUU=\{\bF\in \bUU\mid 0\not \in \Esc_P(\bF)\}
\end{equation} is open for every $P\in \PT$ and that the motion of $\Esc_T(\bF)$ can be dynamically extended (with the same dilatation) to a motion of $\Esc_P(\bF)$ with $\bF\in \DEsc_P$. If $P\le T$, then $\DEsc_P\supset  \bUU$ and the claim is immediate. Assume that $T<P\le 2T$. We have: $\bF\in \DEsc_P\cap \bUU$ if and only if $\bF^{T}(0)\not \in \Esc_{P-T}(\bF)$ because $0\in \Dom(\bF^T)$; this is an open condition because $\Esc_{P-T}(\bF)\subset \Esc_{T}(\bF)$ moves holomorphically with $\bF\in\bUU$. Moreover, for every $\bF\in  \DEsc_P\cap \bUU$, we  can pull back the holomorphic motion of $ \Esc_T(\bF)$ to the holomorphic motion of $ \Esc_P(\bF)\setminus  \Esc_{P-T}(\bF)$ via a covering \[\bF^T\colon \Esc_P(\bF)\setminus  \Esc_{P-T}(\bF)\to \Esc_T(\bF);\] combining with the motion of $ \Esc_{P-T}(\bF)\subset \Esc_T(\bF)$, we obtain the motion of $\Esc_P(\bF)$ without changing dilatation. For $2^{n-1}T<P\le2^nT$, we apply an induction.

For every $\bF\in \Unst$, there is a sufficiently big $n\ll0$ such that $\bF_n$ is close to $\bF_\str$; in particular, $\bF_n\in \bUU$. Setting $L=P \tt^{-n}$, we obtain from~\eqref{eq:ren:F:F_n} that $\bF\in \DEsc_P$ if and only if $\bF_n\in \DEsc_L$. This shows that~\eqref{eq:lem:esc set moves hol} is open as a union of open sets. It also follows that $ \Esc_P(\bF)$ moves holomorphically with $\bF\in \DEsc_P$. The dilatation of the motion of $ \Esc_P(\bF)$ is equal to the dilatation of $\Esc_L(\bF_n)$ and can be made arbitrary small by choosing $\bF_n$ close to $\bF_\str$.  This shows that the holomorphic motion is conformal.

It was already proven that $\Esc_{L}(\bF)$ has empty interior and supports no invariant line field for sufficiently small $L$. Therefore, $\Esc_{P}(\bF)$ supports no invariant line field and has empty interior.
\end{proof}

\begin{cor}
\label{cor:Jul is ovl Esc}
For every $\bF\in \Unst$ we have $\Esc(\bF)\not=\emptyset$ and $\Jul (\bF)=\overline {\Esc(\bF)}$.  
\end{cor}
\begin{proof}
For a small $P\in \PT_{>0}$, the escaping set $\Esc_P(\bF)$ is homeomorphic to $\Esc_P(\bF_\str)\not=\emptyset$. Since $\Esc_P(\bF)$ contains at least two points, applying the Montel theorem we obtain that every neighborhood of $z\in \Jul (\bF)$ contains a preimage of  a point in  $\Esc_P(\bF)$.  
\end{proof}

\subsection{External rays and alpha-points of $\bF$} Choose $S\in \PT_{>0}$ and a sufficiently small $\varepsilon>0$ such that the open $\varepsilon$-neighborhood $\bUU$ of $\bF_\str\in\Unst\simeq \C$ is contained in $ \DEsc_S$, see Lemma~\ref{lem:esc set moves hol}. For every pair of prepacmen $\bF,\bG$ there is a sufficiently big $n\ll 0$ such that $\bF_n, \bG_n\in \bUU$. Since $\bUU$ is a topological disk, there is a unique up to homotopy path $\ell$ in $\bUU$ connecting $\bF_n$ and $\bG_n$. The holomorphic motion  $\tau$ (from Lemma~\ref{lem:esc set moves hol}) along $\ell$ produces an equivariant homeomorphism between $\Esc_S(\bF_n)$ and $\Esc_S(\bG_n)$; after rescaling we obtain an equivariant homeomorphism
 \[h\colon \Esc_{\tt^n S}(\bF)\to \Esc_{\tt^n S}(\bG). \]
Taking the restriction, we obtain an equivariant homeomorphism 
\[h\colon \Esc_{Q}(\bF)\to \Esc_{Q}(\bG). \]
for all $Q<\tt^n S$; clearly $h$ is independent of $n$ (because $\RR^{-1}(\bUU)\subset \bUU$). We say that $h$ is the \emph{canonical identification} of $\Esc_{Q}(\bF)$ and $\Esc_Q(\bG)$.

Consider $\bG\in \Unst$ and choose a sufficiently small $T\in \PT_{>0}$ so that $\Esc_T(\bG)$ and $\Esc_T(\bF_\str)$ are canonically identified.  \emph{Alpha-points of $\bG$ of generation $S\le T$} are the images of the corresponding alpha-points of $\bF_\str$ under the canonical identification. Similarly, \emph{ray segments in $\Esc_T(\bG)$}  are the images of ray segments in $\Esc_T(\bF_\str)$ under the canonical identification. 

If $\alpha_i\in \Esc_T(\bG)$ is an alpha-point of generation $S$, then $\bF^{-P}(\alpha_i)$ are \emph{alpha-points of generation $S+P$.} Similarly, if $\gamma \subset \Esc_T(\bF)$ is a ray segment connecting two alpha-points and $\ell$ is a connected component of $ \bF^{-P}(\gamma)$ such that $\bF^P\colon \gamma\to \ell $ is a homeomorphism, then $\ell$ is also a \emph{ray segment}. Note that $\ell$ also connects two alpha-points.

External rays for $\bG$ are defined in the same way as for $\bF_\str$, see~\S\ref{ss:ExtRayFstr}. Namely, an external ray $\bR$ is a maximal concatenation of external ray segments provided that it does not hit an iterated preimage of $0$. (An external ray ``breaks'' at a pre-critical point.)

\begin{lem}
\label{lem:Esc:Lifting}
Suppose $0\not \in \Esc_P(\bF)$ and $0\not \in \Esc_P(\bG)$. Then there is a unique equivariant bijection $h\colon \Esc_P(\bF)\to\Esc_P(\bG)$ such that $h\colon \Esc_Q(\bF)\to\Esc_Q(\bG)$ coincides with the homeomorphism induced by $\tau$ for all sufficiently small $Q\in\PT_{>0}$.

Let $\bR$ be an external ray of $\bG$. Then $\bR$ has a unique counterpart $\bR(\bF_\str)$ in the dynamical plane of $\bF_\str$ such that for all sufficiently small $Q$ the natural homeomorphism $h\colon \Esc_Q (\bG)\to \Esc_Q(\bF_\str)$ induced by $\tau$ extends to a homeomorphism \[h\colon \Esc_Q(\bG)\cup \bR (\bG)\to \Esc_Q(\bF_\str)\cup \bR (\bF_\str)\]
that is equivariant in the following sense. For every $x\in \bR(\bG)$ with $\bG^T(x)\in \Esc_Q(\bG)$ we have $h\circ \bG^T(x)=\bF_\str^T \circ h(x)$.
\end{lem}

\begin{proof}
Let $P'$ be the minimal escaping time of $0$ for $\bF$ and $\bG$; we have $P<P'$. Since $\PT$ is dense in $\R_{\ge0}$, we can slightly increase $P$ so that the new $P$ is still less than $P'$, and $R\coloneqq P/m \in \PT$, and $R<Q/3$ for some $m\in \Z_{>0}$.

Let us say that a \emph{decoration} is a connected component of $\Esc_{iR}(\bF)\setminus \Esc_{(i-1)R}(\bF)$ for $i\in \{1,2,\dots ,m\}$; we say that $i$ is the \emph{generation} of the decoration. Note that $\Esc_R(\bF)$ is a decoration of generation $1$. We will show that decorations for $\bG$ and $\bF_\str$ are arranged in the same way.

\emph{Claim.} In the dynamical planes of $\bF\in \{\bF_\str, \bG\}$ consider  a decoration $X_j$ of generation $i>1$. Then $X_j$ is precompact and there is a unique alpha-point $\alpha_j$ of generation $(i-1)R$ such that $X_j\cup \alpha_j$ is compact. Moreover, $X_j$ and $X_j\cup \alpha_j$ are filled-in (i.e., their complements are connected). The image $\bF^R(X_j)$ is a decoration $X_{\bF_\str(j)}$ of generation $i-1$. If $i<m$, then for every $\alpha_s\in \bF^{-R}(\alpha_i)$, there is a unique decoration $X_s$ of generation $i+1$ such that $\overline X_s=X_s\cup \{\alpha_s\}$ and $\bF^R (X_s) =X_j$.
\begin{proof}[Proof of the Claim]
The case $\bF=\bF_\str$ follows from the second part of Lemma~\ref{lem:comp of Esc R - Esc T}. Since there is an equivariant homeomorphism between $\Esc_{3R}(\bG)$ and $\Esc_{3R}(\bF_\str)$ (recall that ), the claim is also true if $\bF=\bG$ and $i\le 3$.

Suppose that  $X_j$ is a decoration of generation $i>3$ in the dynamical plane of $\bG$. Choose $x\in X_j$ and let $X_k$ be the decoration of generation $2$ containing $\bF^{(i-2)R}(x)$. Since the claim is already verified for $X_k$, we have $\overline X_k=X_k\cup \{\alpha_k\}$ is compact and filled-in. Since $\overline X_k$  does not contain a critical value of $\bF^{(i-2)R}$ and since the set of the critical values is discrete (see Lemma~\ref{lem:discr of dyn}), there is an open topological disk $U_k\supset \overline X_k$ such that $U_k$ does contain any critical point of $\bF^{(i-2)R}$. Pulling back  $U_k$ along the orbit of $x$ and using $\sigma$-properness, we construct a univalent preimage $U_j\ni x$ of $U_k$ under $\bF^{(i-2)R}$. It now follows that $\overline X_j=X_j\cup \{\alpha_j\}$ is the preimage of $X_k\cup \{\alpha_k\}$ under $\bF^{(i-2)R}\colon U_j\to U_k$.

If $\alpha_s\in \bG^{-R}(\alpha_j)$, then pulling back $U_i$ along $\bF^R\colon \alpha_s\mapsto \alpha_i$  (and possibly shrinking $U_i$ so that $U_i$ does not contain a critical value of $\bF^R$) we construct a univalent preimage $U_s\ni \alpha_s$ of $U_i$. Then $X_s$ is the preimage of $X_j$ under $\bF^{R}\colon U_s\to U_k$.
\end{proof}

Observe that all decorations of $\bF_\str$ and $\bG$ are canonically homeomorphic by an equivariant homeomorphism: all decorations are univalent preimages of $\Esc_R(\bF_\str)$ and  $\Esc_R(\bG)$ which are canonically identified. We proceed by induction: suppose that $h$ has been already extended to an equivariant bijection
\[h\colon \Esc_{t R}(\bF_\str)\to  \Esc_{t R}(\bG)\]
where $t<m$. Then $h$ induces a bijection between alpha-points of generation $tR$. There is a decoration of generation $t+1$ attached to every alpha-point of generation $tR$; since these decorations are canonically homeomorphic, we can uniquely extend $h$ to the bijection 
\[h\colon \Esc_{(t+1) R}(\bF_\str)\to  \Esc_{(t+1) R}(\bG).\]

To prove the second claim, fix $T<Q$ and decompose $\bR$ as a concatenation of arcs $\bR_1\cup \bR_2\cup \dots$ such that $\bR_i\subset  \Esc_{iT}(\bG)\setminus \Esc_{(i-1)T}(\bG)$. Inductively define $\bR_i(\bF_\str)$ to be the unique lift of $h \circ \bG^{T(i-1)} \big( \bR_i(\bG)\big) $ under $\bF_\str^{T(i-1)}$ starting where $\bR_{i-1}(\bF_\str)$ ends.
 \end{proof}
 
By Lemmas~\ref{lem:esc set moves hol} and~\ref{lem:Esc:Lifting}, the combinatorics of external rays for $\bG\in \Unst$ is the same as for $\bF_\str$.

\begin{cor}[$\tau$ has no holonomy]
If $\bF,\bG$ are in the same connected component of $\DEsc_P$, then the equivariant homeomorphism $h\colon \Esc_P(\bF)\to \Esc_{P}(\bG)$ induced by $\tau$ (see Lemma~\ref{lem:esc set moves hol}) is independent of the curve connecting $\bF$ and $\bG$.\qed
\end{cor}

\begin{cor}
For every $K>1$ and $\bF,\bG\in \Unst$, there exists a sufficiently small $T\in \PT_{>0}$ such that the equivariant homeomorphism 
\[h\colon \Esc_T(\bF)\to \Esc_T(\bG)\]
induced by the holomorphic motion from Lemma~\ref{lem:esc set moves hol}
extends to a qc map $\widetilde h\colon \C\to \C$ with dilatation less than $K$. 
\end{cor}
\begin{proof}
For a sufficiently small $T$, the hyperbolic distance between $\bF$ and $\bG$ in $\DEsc_T$ is small. The $\lambda$-lemma extends $h$ to a qc map with a small dilatation. 
\end{proof}

\subsection{Puzzle pieces}
\label{ss:puzzles} 
Consider a dynamical plane of $\bG$. Let us say a ray $\bR$ \emph{lands} if either $\bR$ lands at a point $x\in \C$ in the classical sense, or $\bR$ lands at $\balpha$, see~\S\ref{ss:balpha:wall topoloy}. A \emph{rational ray} is either a periodic or preperiodic ray.

Let $\widetilde \bR= \{\bR^1,\bR^2,\dots, \bR^n\}$ be a finite set of rational rays. We assume that every $\bR^i$ lands. We define 
\[\overline \bR\coloneqq  \overline{\bigcup_{i} \bigcup_{P\ge 0} \bG^P(\bR^i)}\]
to be the forward orbit of rays in $\widetilde \bR$.
\begin{lem}
\label{lem:over R is graph}
The set $\overline \bR$ is a forward invariant connected graph.
\end{lem} 
\begin{proof}
Recall from Corollary~\ref{cor:rays meet} that every two external rays eventually meet. Therefore, $\overline \bR$ is connected. 

Fix a compact subset $\bX\Subset \C$. Let $T_1$ be a common period of rays in $\widetilde \bR$ and let $T_2$ be a common preperiod of rays in $\widetilde \bR$; write $T\coloneqq T_1+T_2$.  We need to prove that there are at most finitely many $P<T$ such that $\bG^P(\bR^i)$ intersects $\bX$ for some $i$; this would imply that $\overline \bR$ is an increasing union of finite graphs as required, see List of Notations.

Choose a sufficiently small $Q>0$ such that $\Esc_Q(\bG)$ is disjoint from $\bX$. Let $\bxx'$ be the set of the landing points of rays in $\widetilde \bR$, and set $\bxx=\bigcup_{P\in  \PT} \bF^{P}(\bxx')=\bigcup_{P\le T} \bF^{P}(\bxx')$. By Lemma~\ref{lem:discr of dyn}, $\bX\cap \bxx$ is a finite set. Note that $\bxx$ may contain $\balpha$. Choose a small neighborhood $\bO$ of $\bxx$. 

 We can choose finitely many external ray segments $\bR_1,\dots , \bR_m$ in $\overline \bR$ such that 
 \begin{itemize}
 \item $\overline \bR\cap \big(\bX\setminus \bO\big)$ is covered by the forward orbits of the $\bR_j$; and
 \item for every $\bR_j$ there exists $P_j$ with $\bR_j\Subset \Dom( \bG^{P_j})$ such that $\bG^{P_j}\big(\bR^j\big)\subset \Esc_Q(\bG)$.
 \end{itemize}
 By Corollary~\ref{cor:orbit of bY}, there are at most finitely many $S<P_j$ such that $\bG^S(\bR_j)$ intersects $\bX\setminus \bO$. Therefore, at most finitely many rays in $\overline \bR$ intersect $\bX\setminus \bO$. Since $\bO$ is a sufficiently small neighborhood of $\bxx$, at most finitely many rays in $\overline \bR$ can enter $\bO$.
\end{proof}

A \emph{puzzle piece} $\bX$ is the closure of a connected component of $\C\setminus \overline \bR$. Since the forward orbit of $\partial \bX$ is disjoint from $\intr \bX$, we have the following classical property. If $\bY$ is the closure of a connected component of $\bF^{-S}(\intr \bX)$, then either $\bX$ and $\bY$ have disjoint interiors, or $\bY\subset \bX$.

Let us note that a puzzle piece can be surrounded by a single external ray, see example in Lemma~\ref{lem:decomp:bW(bF)}. This happens when a preperiodic ray lands at itself; a certain iterate of such ray eventually lands at $\balpha$.

We say that a puzzle piece $\bX(\bG)$ from the dynamical plane of $\bG$ \emph{exists} in the dynamical plane of $\bF$ if  $\bF$ has a puzzle piece $\bX(\bF)$ such that $\partial \bX(\bG)$ and $\partial \bX(\bF)$ are combinatorially equivalent: there is a homeomorphism $h\colon \partial \bX(\bG)\to \partial \bX(\bF)$ induced by the natural identification of external rays, see Lemma~\ref{lem:Esc:Lifting}.

\subsection{Parameter rays} \label{ss:param rays}Similar to many parameter spaces in complex dynamics, it is natural to expect:
\begin{conj}
\label{conj:P-P relation}
For every $x\in \Esc(\bF_\str)$ there is a unique parameter $\bG\in \Unst$ such that $0=x(\bG)$; i.e.~there is a path in $\Unst$ connecting $\bF_\str$ and $\bG$ such that the geodesic ray segment $[x,\infty]\in \Esc(\bF_\str)$ moves holomorphically along the path and such that $0$ collides with $x$ in the dynamical plane of $\bG$.
\end{conj}

Conjecture~\ref{conj:P-P relation} would imply that the \emph{phase parameter relation} 
\begin{equation}
\label{eq:P-P relation}
 \Esc(\bF_\str) \longrightarrow \Unst, \sp\sp\sp\ x(\bF_\str)\mapsto \bG \sp \text{ such that } \sp 0=x(\bG)
\end{equation}
is a well-defined homeomorphism onto the image. A \emph{parameter ray} is the image of a dynamic ray under~\eqref{eq:P-P relation}. Conjecture~\ref{conj:P-P relation} would immediately imply that the parameter limbs are bounded (as pictures predict, see Figure~\ref{Fig:unst man+max sieg}) and it should significantly improve the results of our paper. Combined with the conjectural Full Hyperbolicity of neutral renormalization, this may give a complete understanding of geometry near the boundaries of hyperbolic components with a single bifurcating critical point, see Remark~\ref{rem: pacman str map}.

\section{Parabolic bifurcation and small $\boldsymbol\Mandel$-copies}
\label{s:par pacm}

\subsection{Parabolic prepacman $\bF_\rr$}
\label{ss:par pacm:bF rr}
In a small neighborhood of $f_\str$ consider a parabolic pacman $f_\rr\in \WW^u$  with rotation number $\rr=\pp/\qq$. Recall from \S\ref{sss:par pacmen} that $\bH$ denotes the global attracting basin; its periodic components are parametrized as $\big(\bH^i\big)_{i\in \Z}$. We write the first return map to $\bH^0$ as 
 \begin{equation}
\label{eq:FRM:H^0}
 \bF_\rr^{Q(\rr)}\colon \bH^0\to\bH^0.
 \end{equation}
Then~\eqref{eq:FRM:H^0} is a two-to-one map with a unique critical value at $0$.  Moreover, $\bF_\rr^T\colon  \bH^0\to  \bF_\rr^T\big(\bH^0\big)=\bH^{i(T)}$ is univalent for all $T<Q(\rr)$, where $i(T)\in \Z\setminus\{0\}$. For standard reasons, $\bH^i$ is a Fatou component.

Let $U\subsetneq \C$ be an open simply connected set. 

We say that a univalent map $f\colon U\to U$ is a \emph{local attracting parabolic petal} if it admits a \emph{local Fatou coordinate}: a univalent map $h\colon U\to \C$ conjugating $f\colon U\to U$ to the translation $T_1\colon z\to z+1$ such that $\Im f \supset \{z\in \C \mid \Re z\ge M\}$ for some $M\in \R$. 

 We say that a branched covering of finite degree $f\colon U\to U$ is a \emph{full attracting parabolic petal} it $f$ restricts to a local attracting parabolic petal $f:U'\to U'$ for some $U'\subset U$. For standard reasons, every point in $U$ is eventually attracted by $U'$. Hence, a local Fatou coordinate system $h\colon U'\to \C$ extends to a branched covering $h\colon U\to \C$ semi-conjugating $f$ to $T_1$. We call $h\colon U\to \C$ a \emph{global Fatou coordinate}.

We say that a full attracting parabolic petal $f\colon U\to U$ is \emph{unicritical} if $f$ has a unique critical value. By Lemma~\ref{lem:unique crt pnt}, $f$ also has a unique critical point.

\begin{lem}
\label{lem:unique conj}
Let $f\colon U\to U,g\colon V\to V$ be two unicritical full attracting parabolic petals of the same degree. Then $f$ and $g$ are conformally conjugate.
\end{lem}
\begin{proof}
We normalize full Fatou coordinate systems $h_f\colon U\to \C$ and $h_g\colon V\to \C$ so that \[0=h_f(\text{critical value of }f)=h_g(\text{critical value of }g).\]
Choose next local attracting parabolic petals $f:U'\to U'$ and $g\colon V'\to V'$ so that $h\coloneqq h_f\circ h_g^{-1}$ is a conjugacy between $f:U'\to U'$ and $g\colon V'\to V'$. We can also assume that $U'$ and $V'$ contain the critical values of $f$ and $g$. Since $h$ respects the postcritical sets, we can apply the pullback argument and extend $h$ to a global conjugacy $h\colon U\to V$ between $f$ and $g$.
\end{proof}

Recall that the quadratic polynomial $p_{1/4}=z^2+1/4$ has a parabolic fixed point at $\alpha(p_{1/4})=1/2$. We denote by $\Parab$ the attracting basin of $\alpha(p_{1/4})$. Note that $p_{1/4}\colon \partial \Parab \to \partial \Parab$ is topologically conjugate to $z\mapsto z^2\colon \mathbb S^1\to \mathbb S^1$. By Lemma~\ref{lem:unique conj}, there is a conformal conjugacy 
\begin{equation}
\label{eq:bHi:conj}
\bbh_i\colon \bH^i \to \Parab
\end{equation}
between $\bF_\rr^{Q(\rr)}\colon \bH^i\to\bH^i$ and $p_{1/4}\colon\Parab\to \Parab$.

Recall that $\HH$ denotes the main hyperbolic component of $\Unst$, see~\S\ref{ss:geom pict}. Choose a curve $\ell\subset \overline \HH$ connecting $\bF_\str$ to $\bF_\rr$. For every $T\in \PT$ there is a small neighborhood of $\ell$ where the holomorphic motion $\tau$ of $\Esc_T(\bG)$ is defined, see Lemma~\ref{lem:esc set moves hol}. Then $\tau$ produces an equivariant homeomorphism \[h\colon \Esc_T(\bF_\str)\to \Esc_T(\bF_\rr).\]
Therefore, $\Esc(\bF_\rr)$ has the same properties as $\Esc(\bF_\str)$; in particular, $\Esc(\bF_\rr)$ is uniquely geodesic.

\begin{thm}
\label{thm:parab prepacm}
If $i\not=j$, then $\overline \bH^i \cap \overline \bH^j=\emptyset$. The conjugacy $\bbh_i \colon \bH^i\to \Parab$ extends uniquely to a topological conjugacy
\begin{equation}
\label{eq:over:bHi:conj}
\bbh_i\colon \overline \bH^i \cup \{\balpha(\bF_\rr)\}\to \overline \Parab
\end{equation}

For every $i$ there is a unique external ray $\bR^i$ landing at $\balpha(\bF_\rr)$ such that $\bR^i$ is between $\bH^i$ and $\bH^{i+1}$. The ray $\bR^i$ is $Q(\rr)$-periodic. Moreover, $\bR^i$ and $\bR^{i-1}$ meet at $\alpha(1/2,\bH^i)$ -- the unique preimage of $\balpha(\bF_\rr)$ in $\C$ under $\bF_\rr^{Q(\rr)}\colon \overline \bH^0 \cup \{\balpha(\bF_\rr)\} \to\overline \bH^0  \cup \{\balpha(\bF_\rr)\}.$

Let $\bW(i)$ be the closed puzzle piece bounded by $\bR^{i-1}$ and $\bR^i$ and containing $\bH^i$, and let $\bW^1$ be the pullback of $\bW\coloneqq \bW(0)$ along $\bF_\rr^{Q(\rr)}\colon \bH^0\to \bH^0$, see {\rm Figure~\ref{Fig:ss:geom pict}}. Then 
\begin{equation}
\label{eq:ren map:par}
\bF^{Q(\rr)}_\rr\colon \bW^1\to \bW.
\end{equation}
is a $2$-to-$1$ map. 
\end{thm}
\noindent We say that~\eqref{eq:ren map:par} is the \emph{primary renormalization map}. It can be thickened to a ``pinched quadratic-like map''.  Before giving the proof of Theorem~\ref{thm:parab prepacm} let us introduce an expanding metric for $\bF_\rr$.

\subsubsection{Expanding metric for $\bF_\rr$}
\label{sss:Frr:exp metric}
Let $B\subset \Parab$ be a forward invariant open topological disk containing $[1/4,1/2)$ such that $\overline B$ is a closed topological disk with 
\[\partial B\cap p_{1/4}(\overline B)=\alpha(p_{1/4})= \overline B \cap \partial \Parab\]
(recall that $1/4$ is the critical value of $p_{1/4}$). We can take $B$ to be an appropriate small neighborhood of $[1/4,1/2)$.

Define $\bB^0\coloneqq \bbh_0^{-1}(B)$ and observe that $\overline \bB^0\subset \bH^0$. Spreading around $\overline \bB^0$, we obtain 
\[\bB\coloneqq \bigcup_{P\ge 0} \bF^P_{\rr}\big(\overline \bB^0\big)\subset \bigcup_{i\in \Z}\bH^i.\] 
 By construction,  $\bB$ contains the postcritical set $\Post(\bF_\rr)$. Consider the hyperbolic surface $\bX\coloneqq \C\setminus \overline \bB$. Then $\bF_\rr^{-T}(\bX)\subsetneq \bX$ and 
\begin{equation}
\label{eq:bX:covering}
\bF_\rr^{T}\colon \bF_\rr^{-T}(\bX)\to \bX
\end{equation}
is a covering for all $T$. By Schwartz's lemma,~\eqref{eq:bX:covering} expands the hyperbolic metric.

\begin{figure} 
\centering{\begin{tikzpicture}[scale=0.25,xscale=3]
\draw[red] (-5.5,-18.5)-- (-5.5,-3.2);
\draw[red] (-5.5,-3.2).. controls(-5.41088, -0.47673) and (-4.62753, 3.35108) .. (-2.875,4.5);
\draw[red] (-0.25,-3.2).. controls(+5.41088-2*2.875, -0.47673) and (4.62753-2*2.875, 3.35108) .. (-2.875,4.5);

\draw[blue] (-4.2,-0.2).. controls (-4.4,-0.2) and (-4.6,0.91).. (-4.9,0.91);
\node[blue, right] at  (-4.2,-0.2) {$\delta_{5/8}$};
\node[ above left] at (-4.9,0.91){$\alpha(5/8)$};

\draw[blue] (-1.5,-0.2) ..controls (-1.3,-0.2 ) and  (-1.1, 0.98) .. (-0.85,0.98);
\node[blue, left] at  (-1.5,-0.2) {$\delta_{3/8}$};
\node[ above right] at (-0.85,0.98){$\alpha (3/8)$};

\node[below] at (-2.875,-18.5){$\balpha(\bF_\rr)$};
\draw[blue](-2.875,-15.5) .. controls (-2.6,-16) and (-3.1,-17).. (-2.875,-18.5);
\node[blue,right] at (-2.85,-16) {$\delta_0$};

\node at (-2.875,5.5) {$\alpha(1/2)$};
\draw[blue](-2.875, 1.5) .. controls (-2.6,3) and (-3.1,4).. (-2.875,4.5);
\node[blue,right] at (-2.85,2) {$\delta_{1/2}$};

\node at (0.55,-3.2) {$\alpha (1/4)$};
\draw[blue](-1.5,-3.2) .. controls (-1.2,-3.8) and (-0.8,-2.6).. (-0.25,-3.2);
\node[blue,below] at (-1,-3.2) {$\delta_{1/4}$};

\node at (0.55,-10.2) {$\alpha (1/8)$};
\begin{scope}[shift={(0,-7)}]
\draw[blue](-1.5,-3.2) .. controls (-1.2,-3.8) and (-0.8,-2.6).. (-0.25,-3.2);
\node[blue,below] at (-1,-3.2) {$\delta_{1/8}$};
\end{scope}

\node at (0.6,-17.2) {$\alpha (1/16)$};
\begin{scope}[shift={(0,-14)}]
\draw[blue](-1.5,-3.2) .. controls (-1.2,-3.8) and (-0.8,-2.6).. (-0.25,-3.2);
\node[blue,below] at (-1,-3.2) {$\delta_{1/16}$};
\end{scope}

\node at (-6.25,-3.2) {$\alpha (3/4)$};
\begin{scope}[shift={(-5.75,0)},xscale=-1]
\draw[blue](-1.5,-3.2) .. controls (-1.2,-3.8) and (-0.8,-2.6).. (-0.25,-3.2);
\node[blue,below] at (-1,-3.2) {$\delta_{3/4}$};
\end{scope}

\node at (-6.25,-10.2) {$\alpha (7/8)$};
\begin{scope}[shift={(-5.75,-7)},xscale=-1]
\draw[blue](-1.5,-3.2) .. controls (-1.2,-3.8) and (-0.8,-2.6).. (-0.25,-3.2);
\node[blue,below] at (-1,-3.2) {$\delta_{7/8}$};
\end{scope}

\node at (-6.5,-17.2) {$\alpha (15/16)$};
\begin{scope}[shift={(-5.75,-14)},xscale=-1]
\draw[blue](-1.5,-3.2) .. controls (-1.2,-3.8) and (-0.8,-2.6).. (-0.25,-3.2);
\node[blue,below] at (-1.1,-3.2) {$\delta_{15/16}$};
\end{scope}

\draw[red] (-0.25,-3.2)-- (-0.25,-18.5);
\draw[red] (-5.5,-18.5)-- (-5.5,-3.2);

\node at (-2.875, -10.58232){$\bH^0$};

\node[red,left] at (-5.5, -12.97131) {$\partial \bH^0$};
  \end{tikzpicture}}
\caption{Alpha-points of $\partial \bH^0$ are defined at the landing points of curves in $\bH^0$ (compare with Figure~\ref{Fig:ss:geom pict}).}
 \label{Fig:par pacm:alpha points}
\end{figure}

\begin{figure} 
\centering{\begin{tikzpicture}
\node at (0,0){\includegraphics[scale=0.8]{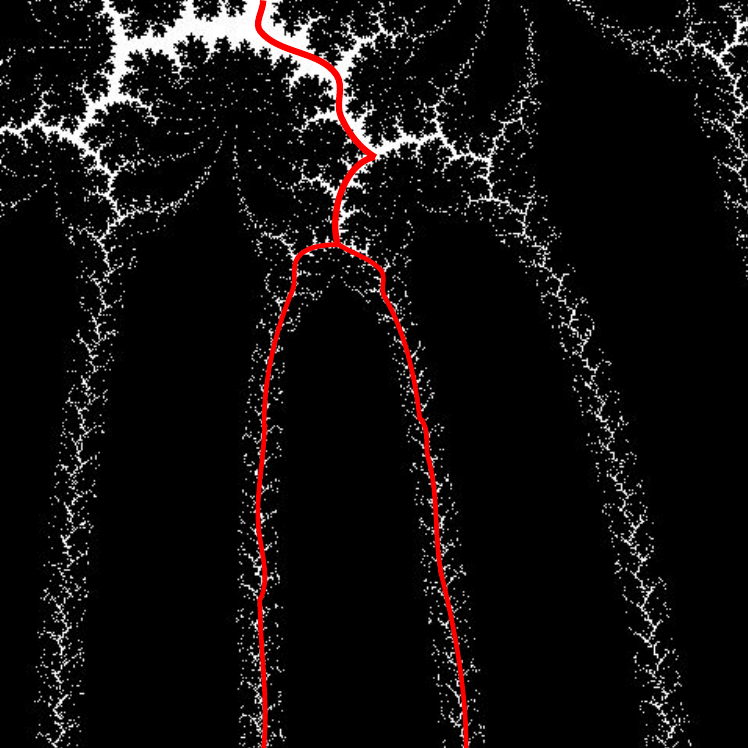}};
  \node at (0,-10.5){\includegraphics[scale=0.557]{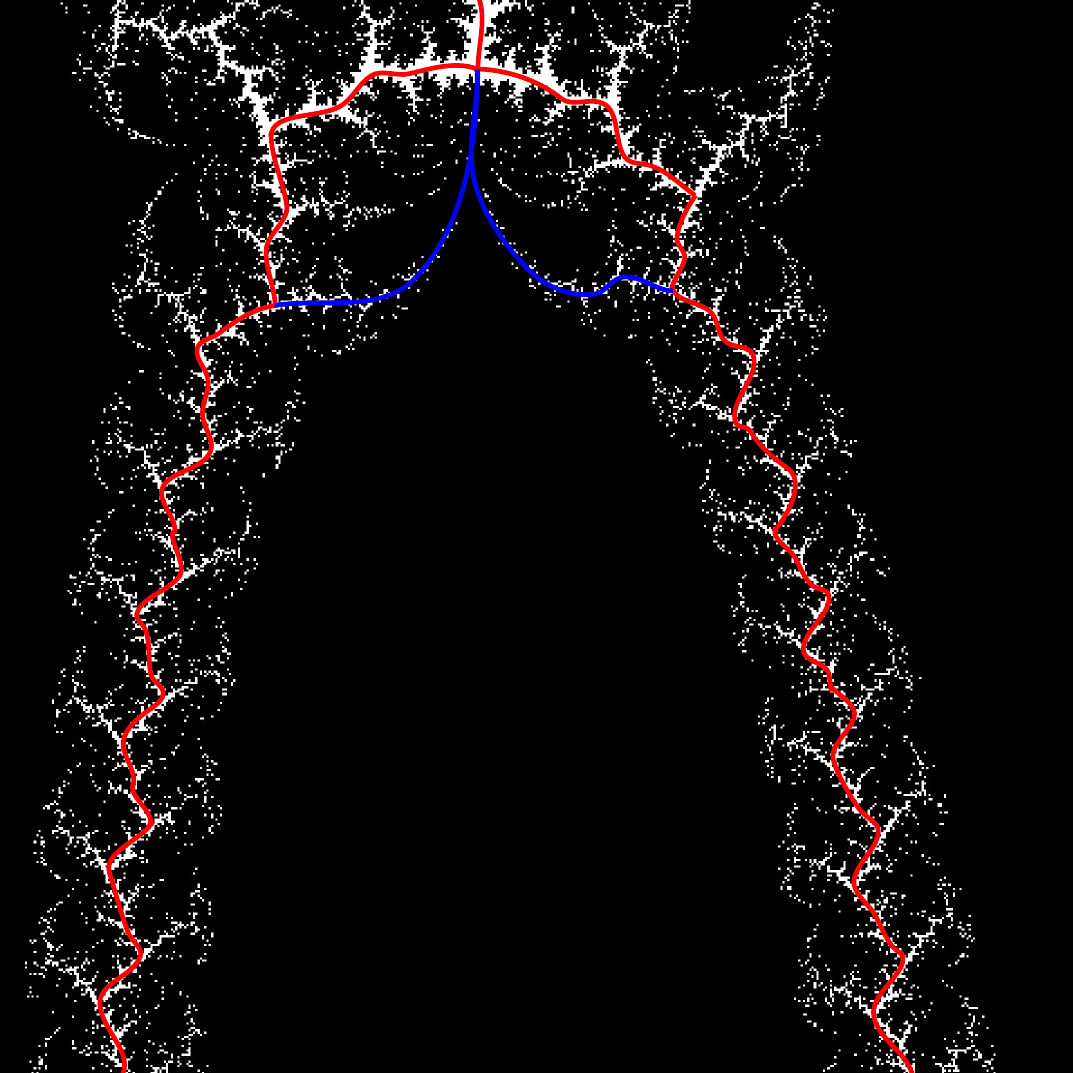}};
  
    \node[white] at(-0.4,-1.6) {$\bH^0$};
    
    \node[red] at (-2.3,-3){$\bR^{-1}$};
    \node[red,right] at (1.2,-3){$\bR^{0}$};
    \node[white] at(-0.4,-1.95) {$\cap$};
     \node[white] at(-0.4,-2.3) {$W(0)$};
        
    \node[white] at(1.9,-1) {$\bH^1$};   
    \node[white] at(4.2,0.5) {$\bH^2$};  
     \node[white] at(-2.6,-1) {$\bH^{-1}$};   
    \node[white] at(-4.5,1) {$\bH^{-2}$};   

  \node[white] at(-0.5,-10) {$W^1(0)$};
  
    \node[white] at(0.3,-7.7) {$W^b$};

     \node[white] at(-1.3,-7.73) {$W^a$};
 
 \end{tikzpicture}}
\caption{Parabolic maximal prepacmen: attracting petals $\bH^i$ and puzzle pieces $\bW=\bW^1\cup \bW^a\cup \bW^b$ (where $\bW^1$ is bounded by two red and two blue curves).}
\vspace{128in} \label{Fig:ss:geom pict}
\end{figure}

\begin{proof}[Proof of Theorem~\ref{thm:parab prepacm}]
Choose $\mu\in (1/4,1/2)$. In the dynamical plane of $p_{1/4}$, define $\delta'_0$ to be the interval $[\mu,1/2)\subset \Parab$ landing at $1/2$. We view $\delta'_0$ as a simple arc parametrized by $[0,1)$. Set $\alpha(k/2^n,p_{1/4})$ to be the landing point of the external ray with angle $k/2^n$, where $k\in \{1,3,5,\dots, 2^n-1\}$. Note that $\alpha(k/2^n,p_{1/4})$ is a preimage of $\alpha$ under $p_{1/4}^n$. We define $\delta'_{k/2^n}$ to be the lift of $\delta'_0$ under $p_{1/4}^n$ landing at $\alpha(k/2^n,p_{1/4})$. Let $\delta_{k/2^n}$  be the lift of $\delta'_{k/2^n}$ under $\bbh_0$.

\emph{Claim $1$:} $\delta_0$ lands at $\balpha(\bF_\rr)$ while $\delta_{k/2^n}$ with $k\in  \{1,3,5,\dots, 2^n-1\}$ land at pairwise different alpha-points of generation $nQ(\rr)$. We denote by $\alpha(k/2^n,\bH^0)$ the landing point of $\delta_{k/2^n}$, see Figure~\ref{Fig:par pacm:alpha points}.
\begin{proof}[Proof of the claim]
Recall that $\UnstLoc$ denotes the set of maximal prepacmen close to $\bF_\str$ that have  the associated pacmen on $\Unst$. By definition, $\bF_\rr\in \Unst$. Applying the $\lambda$-lemma, we will first show that a certain version of Lemma~\ref{lem:preim of alpha} holds in the dynamical plane of $\bG\in  \UnstLoc$.

Since $\delta_0$ is invariant under $\bF_\rr^{Q(\rr)}$, by replacing $\delta_0$ with its subcurve $\delta_0 [t,1)$, we can assume that $\delta_0$ is contained in $\bH^0_0$ (see~\S\ref{sss:par pacmen}). Thus $\delta_0(\bF_\rr)$ projects to the dynamical plane of $f_\rr$, and the projection, call it $\delta_0(f_\rr)$, is disjoint from the critical arc $\gamma_1$, possibly up to a slight rotation of $\gamma_1$ as in~\S\ref{sss:par pacmen}. 

For $g\in \WW^u$, let $T_g$ be the unique translation mapping $\alpha(f_\rr)$ to $\alpha(g)$. For every $\bG\in \UnstLoc$, we define $\delta_0(\bG)$ to be the lift of $\delta_0(g)\coloneqq T_g(\delta_0(f_\rr))$ to $\bS(\bG)\simeq V\setminus \gamma_1$. By construction, $\delta_0(\bG)$ depends holomorphically on $\bG$.

For a small $T\in \PT_{>0}$ and $\bG\in \UnstLoc$, the set of critical values $\CV\big(\bG^T\big)$ is disjoint from $\delta(\bG)$ because  \[\CV\big(\bG^T\big)\setminus \{0\}\subset \bDelta_0(\bG)\setminus \bS\sp\sp \text{ and }\sp\delta(\bG)\subset \bS,\] see \eqref{eq:CV are away from 0}. Therefore,  $\bG^{-T}(\delta)$ moves holomorphically with $\bG\in \UnstLoc$. Applying the $\lambda$-lemma, we obtain the holomorphic motion of $\displaystyle{\overline {\bG^{-T}_{\ }(\delta)}}$ with $\bG\in \UnstLoc$. Since curves in  $\overline {\bF_\str^{-T}(\delta)}$ land at pairwise different alpha-point (Lemma~\ref{lem:preim of alpha}), the same is true for $\overline {\bF_\rr^{-T}(\delta)}$ in the dynamical plane of $\bF_\rr$.

 Let $\bH^i$ be the unique (by~\eqref{eq:ActionOnHp}) periodic preimage of $\bH^0$ under $\bF^T_\rr$. Since $T$ is small, the map $\bF_\rr^T\colon \bH^i\to \bH^0$ is two-to-one. There are two lifts $\delta^T_0$ and $\delta^T_{1/2}$ of $\delta$ under $\bF_\rr^T\colon \bH^i\to \bH^0$. One of them lands at $\balpha(\bF_\rr)$, while the other $\delta^T_{1/2}$ 
  lands at the alpha-point $\alpha(1/2, \bH^i)$ of generation $T$.
  
   Observe that $\overline {\delta^T_{1/2}}=\delta^T_{1/2}\cup \alpha(1/2, \bH^i)$ is contained in $\bX$ (see~\S\ref{sss:Frr:exp metric}). Therefore, for all $P\in  \PT$ all the lifts of $\delta^T_{1/2}$ under $\bF_\rr^P$ land at pairwise different preimages of $ \alpha(1/2, \bH^i)$. Since $\delta_{k/2^n}$ are lifts of $\delta^T_{1/2}$, the claim follows.
\end{proof}

For $\bH^i$ we define $\alpha(k/2^n, \bH^i)$ to be the images of $\alpha(k/2^n, \bH^i)$ under the univalent map $\bF_\rr^S\colon \bH^0\to \bH^i$, where $S<Q(\rr)$.

Let \[\ell_n \coloneqq [\alpha(1/2^n, \bH^0),\alpha((2^n-1)/2^n, \bH^{-1})]\subset \Esc_{n Q(\rr)}(\bF_\rr)\] 
be the unique geodesic (in $\Esc(\bF_\rr)$) connecting the alpha-points, see Figure~\ref{Fig:Curves ell_n}.

\begin{figure} 
\centering{\begin{tikzpicture}

\begin{scope}[scale=0.2,xscale=3]
\draw[] (-5.5,-18.5)-- (-5.5,-3.2);
\draw[] (-5.5,-3.2).. controls(-5.41088, -0.47673) and (-4.62753, 3.35108) .. (-2.875,4.5);
\draw[] (-0.25,-3.2).. controls(+5.41088-2*2.875, -0.47673) and (4.62753-2*2.875, 3.35108) .. (-2.875,4.5);

\node[right] at (-2.875,5.5) {$\alpha(1/2)=a_1$};
\coordinate  (a1) at (-2.875,4.5);

\begin{scope}[shift={(-5.75,0)},xscale=-1]
\node[right]at (-0.2,-3.2) {$\alpha(3/4)=a_2$};
\coordinate  (a2) at (-0.25,-3.2);
\end{scope}

\begin{scope}[shift={(-5.75,-7)},xscale=-1]
\node[right] at (-0.2,-3.2) {$\alpha(7/8)=a_3$};
\coordinate  (a3) at (-0.25,-3.2);
\end{scope}

\begin{scope}[shift={(-5.75,-14)},xscale=-1]
\node[] at (-1.6,-3.2) {$\alpha(15/16)$};
\coordinate  (a4) at (-0.25,-3.2);
\end{scope}

\draw[] (-0.25,-3.2)-- (-0.25,-18.5);
\draw[] (-5.5,-18.5)-- (-5.5,-3.2);

\node at (-1.875, -14.58232){$\bH^0$};
\end{scope}

\begin{scope}[shift={(-5,0.6)},scale=0.2,xscale=3]
\draw[] (-5.5,-18.5)-- (-5.5,-3.2);
\draw[] (-5.5,-3.2).. controls(-5.41088, -0.47673) and (-4.62753, 3.35108) .. (-2.875,4.5);
\draw[] (-0.25,-3.2).. controls(+5.41088-2*2.875, -0.47673) and (4.62753-2*2.875, 3.35108) .. (-2.875,4.5);

\node at (-2.875,5.5) {$\alpha(1/2)$};
\draw[red] (a1)edge node[above]{$\ell_1$}(-2.875,4.5);

\node[left] at (-0.2,-3.2) {$b_2=\alpha(1/4)$};

\draw[red] (a2)edge node[above]{$\ell_2$}(-0.25,-3.2);
\draw[dashed, red] (a1) edge (-0.25,-3.2);

\begin{scope}[shift={(0,-7)}]
\node[left]  at (-0.2,-3.2) {$b_3=\alpha(1/8)$};
\draw[red] (a3)edge node[above]{$\ell_3$}(-0.25,-3.2);
\draw[dashed, red] (a2) edge (-0.25,-3.2);

\end{scope}

\begin{scope}[shift={(0,-14)}]
\node[left]  at (-0.2,-3.2) {$b_4=\alpha(1/16)$};
\draw[red] (a4)edge node[above]{$\ell_4$}(-0.25,-3.2);
\draw[dashed, red] (a3) edge (-0.25,-3.2);

\end{scope}

\draw[] (-0.25,-3.2)-- (-0.25,-21.5);
\draw[] (-5.5,-21.5)-- (-5.5,-3.2);

\node at (-3.875, -13.58232){$\bH^{-1}$};
\end{scope}

  \end{tikzpicture}}
\caption{Curves $\ell_n$ and $[b_n,a_{n+1}]$ (dashed), see also Figure~\ref{Fig:ss:geom pict}.}
 \label{Fig:Curves ell_n}
\end{figure}

\emph{Claim $2$:} $\bF^{Q(\rr)}_\rr$ maps $\ell_{n+1}$ to $\ell_{n}$.  The curves $\ell_n$ converge to $\balpha(\bF_\rr)$, and for $n\gg 0$ the curve $\ell_n$ is contained in a repelling petal between $\bH^{-1}$ and $\bH^0$ 
\begin{proof}[Proof of the claim]
Note that $(-1/4,1/4)\subset \R\cap \Parab$ is $p_{1/4}$-invariant and connects $\alpha(p_{1/4})$ and $\alpha(1/2,p_{1/4})$. Let $\beta'_{+,n}$ be the unique geodesic in $p_{1/4}^{-n+1}( (-1/4,1/4) )$ connecting $\alpha((2^n-1)/2^n,p_{1/4})$ and $\alpha(p_{1/4})$. Set $\beta_{+,n}\coloneqq \bbh^{-1}_0(\beta'_{+,n})$. Similarly, let $\beta'_{-,n}$ be the unique geodesic in $p_{1/4}^{-n+1}( (-1/4,1/4) )$ connecting $\alpha(1/2^n,p_{1/4})$ and $\alpha(p_{1/4})$; set $\beta_{-,n}\coloneqq \bbh^{-1}_{-1}(\beta'_{-,n})$.

Define $\bO_n$ to be the open disk bounded by $\beta_{-,n}\cup \ell_n\cup \beta_{+,n}$ and containing a small repelling petal between $\bH^{-1}$ and $\bH^0$. Then $\bO_{n+1}\subset \bO_n$ and we claim that $\bF_\rr^{Q(\rr)} $ maps $\bO_{n+1}$ univalently onto $\bO_{n}$. Indeed, $\bF_\rr^{Q(\rr)} $  maps  $\beta_{-,n+1}$ and $\beta_{+,n+1}$ to  $\beta_{-,n}$ and $\beta_{+,n}$. Since $\bO_n$  does no contain any critical values of $\bF_\rr^{Q(\rr)}$, the lift of $\bO_n$ attached to $\beta_{-,n+1}$ and $\beta_{+,n+1}$ is attached to the unique lift $\widetilde \ell_{n+1}$ of $\ell_n$. Then $\widetilde \ell_{n+1}$ connects $\alpha(\bH^0,1/2^n)$ and $\alpha(\bH^{-1},(2^n-1)/2^n)$; since $\Esc_{(n+1)Q(\rr)}$ is uniquely geodesic, we obtain $\widetilde \ell_{n+1}=\ell_{n+1}$.

Note that between $\bH^{-1}_0$ and $\bH^0_0$ there is a repelling petal $\bP^{-1}$ emerging from $\balpha$; this repelling petal is obtained from lifting the repelling petal of $\alpha(f_\rr)$ between $H^{-1}$ and $H^0$.  Let $(y_n)_{n\ge 0}$ with $\bF^{Q(\rr)}_{\rr}(y_{n+1})=y_n\in \bX\cap \bO_n$ be the orbit emerging from $\balpha(\bF_\rr)$ in $\bP^{-1}$. Connect $\ell_{n+1}$ and $y_{n+1}$ by a curve in $\bO_n\cap \bX$; then the lift of this curve is a curve connecting $\ell_{n+2}$ and $y_{n+2}.$  Since $\bF_\rr^{Q(\rr)}$ expands the hyperbolic metric of $\bX$ (see~\eqref{eq:bX:covering}), and since $ {\ell_n}\subset \bX$, we obtain that $\ell_n$ shrinks to $\balpha(\bF_\rr)$.
\end{proof}

Since $\Esc(\bF_\rr)$ is uniquely geodesic, there is a unique geodesic $[b_n,a_{n+1}]\subset \Esc(\bF_\rr)$ connecting $\ell_n$ and $\ell_{n+1}$. (In fact, $b_n=\alpha(2^n-1)/2,\bH^{0})$ and $a_n=\alpha(1/2^n,\bH^{-1})$, see Figure~\ref{Fig:Curves ell_n}.) Then $\bF_\rr^{Q(\rr)}$ maps $[b_{n+1},b_{n+2}]$ to $[b_{n},b_{n+1}]$. Since $\bF_\rr^{Q(\rr)}$ is expanding (see~\eqref{eq:bX:covering}), $[b_{n},b_{n+1}]$ shrinks to $\balpha(\bF_\rr)$. We obtain that
\[\bR^{-1}\coloneqq (\binfty,b_1]\bigcup_{n\ge 0}[b_n,b_{n+1}] \]
is a periodic ray landing between $\bH^{-1}$ and $\bH^{0}$.

Similarly, there is a periodic ray $\bR^i$ landing at $\balpha(\bF_\rr)$ between $\bH^i$ and $\bH^{i+1}$. The rays $(\bR^i)_i$ form a periodic cycle.

Let $\bW(i)\supset \bH^i$ be the puzzle piece bounded by $\bR^i\cup \bR^{i+1}$. Recall \eqref{eq:ActionOnHp} that for every $\bH^i$ there is a unique $P\in \PT$ such that $\bF_\rr^P\colon \bH^0\to \bH^i$ is a conformal map.
 
\emph{Claim $3$:} the conformal map $\bF_\rr^P\colon \bH^0\to \bH^i$ extends to a conformal map 
$\bF_\rr^P\colon \bW(0)\to \bW(i)$. Moreover, $\bR^{-1}$ and $\bR^0$ meet at $\alpha(\bH^0,1/2)$.
\begin{proof}[Proof of the Claim]
The critical values of $\bF_\rr^P$ are exactly \[\{\bF_\rr^Q(0)\mid Q<P\}\subset \bigcup_{Q<P} \bF_\rr^Q(\bH^0) \] and this set is disjoint from $\bW(i)$. Therefore, $\bW(0)$ is the conformal pullback of $\bW(i)$ along $\bF_\rr^P\colon \bH^0\to \bH^i$. This shows the first claim

For every $S<Q(\rr)$ there is an $i\in \Z$ such that $\bF_\rr^S\colon \bW(0)\to \bW(i)$ is conformal. Therefore, $\balpha(\bH^i)=\bF_\rr^S(\balpha(\bH^0,1/2))$ and we see that the generation of $\alpha(\bH^0,1/2)$ is exactly $Q(\rr)$. This also implies that $\alpha(\bH^0,1/2)$ has the smallest generation among all the preimages of $\balpha(\bF_\rr)$ in $\bW(0)$. If $\alpha_-$ and $\alpha_+$ are alpha-points in $\bR^{-1}\setminus \bR^0$ and $\bR^{0}\setminus \bR^{-1}$ respectively, then $\alpha_-$ and $\alpha_+$ are $\prec$-separated by  $\alpha(\bH^0,1/2)$ (see~\S\ref{ss:TreeLike str}). By Corollary~\ref{cor:rays meet}, $\bR^{0}\cap  \bR^{-1}=[\alpha(\bH^0,1/2), \binfty)$. 
\end{proof}

 Since $\bW(0)$ contains a unique critical value of $\bF_\rr^{Q(\rr)}$, pulling back $\bW(0)$ we obtain the two-to-one map~\eqref{eq:ren map:par}.

By a standard puzzle argument, $\partial \bH^i$ is a simple arc. Indeed, every alpha-point $\alpha(t, \bH^i)$ is accessible from the interior and the exterior of $\bH^i$. We can cover $\partial \bH^i$ by closed topological disks $(D_j)_{j\in \Z}$ with disjoint interiors whose boundaries intersect $\partial \bH^i$ at exactly two alpha-points: for $j<0$ the disk $D_j$ intersects $\partial \bH^0$ at $\alpha(2^{j})$ and $\alpha(2^{j-1})$, and for $j\ge 0$ the disk $D_j$ intersects $\partial \bH^0$ at $\alpha(1-1/2^{j+1})$ and $\alpha(1-1/2^{j+2})$. Moreover, we can assume that $D_j$ is disjoint from the postcritical set. Taking preimages, we obtain a systems of disks $(D_\eta)_{\eta}$; every $D_\eta$ still intersects $\partial \bH^0$ at exactly two alpha-points. By expansion~\S\ref{sss:Frr:exp metric}, for every $x\in \partial \bH^i$ there is a sequence of disks $D_{i,x}\ni x$ as above shrinking to $x$. This implies that $\partial \bH^0$ is a simple arc, \cite{Na}*{Theorems 6.16 and 6.17}.

Note that $\overline \bH^0$ is the non-escaping set of $\bF^{Q(\rr)}_\rr\colon \bW^1\to \bW.$ Therefore, the escaping set intersects $\partial \bH^i$ at alpha-points. This implies that $\bH^i\cap \bR^k\cap \bH^j=\emptyset$ for all $i\not=j$ and all $k$ because the generation of alpha-point in $\partial \bH^i$ is different from those in $\partial\bH^j$. Since $\bH^i$ and $\bH^j$ are in different puzzle pieces, $\bH^i\cap  \bH^j=\emptyset$ if $i\not=j$. 
 \end{proof}

We will show in Lemma~\ref{lem:bW is tiling} that every $z\in \C$ is contained in some puzzle piece $\bW(i)$.

\subsection{Properties of $\bW(i)$}

\begin{figure}
\begin{tikzpicture}

\begin{scope}[shift={(-8,0)}]

\coordinate (b1) at (2,1.5);
\coordinate (c1) at (-2,1.5);

\draw[] (-0.915,-0.2 ).. 
controls (-2.5,0.5)  and (-4,2.8) ..
(0,3)
..controls (4,2.8)  and (2.5,0.5) ..
(0.915,-0.2 );

 \draw(0,4)--(0,3);

\draw (0,3) .. controls  (1,2 )and (1,-1)..
(1,-4);
\draw (0,3) .. controls  (-1,2 )and (-1,-1)..
(-1,-4);

\node[above] (a1) at (0,-0.35){};

\node[left] at (-1,-3) {$\bR^{-1}$}; 
\node[right] at (1,-3) {$\bR^{0}$};

\node[left] at (-0.3,2.3) {$\bR^{a}$}; 

\node[right] at (0.38,2.3) {$\bR^{b}$};

\node[above right] at (0,3) {$\alpha(1/2,\bH^0)$}; 

\node at(-1.5,1) {$\bW^a$};
\node at(1.5,1) {$\bW^b$};

\node[above] at (0,-1){$\bW^1$};

\end{scope}

\end{tikzpicture}
\caption{The decomposition of $\bW=\bW^1\cup \bW^a\cup \bW^b$. The rays $\bR^a$ and $\bR^b$ land at $\alpha(1/2,\bH^0)$ and bound puzzle pieces $\bW^a$ and $\bW^b$. Compare with Figure~\ref{Fig:ss:geom pict}.}
\label{fig:Decomp of bW}
\end{figure}

\begin{lem}[Decomposition $\bW=\bW^1\cup \bW^a\cup \bW^b$]
\label{lem:decomp:bW(bF)}
The puzzle piece $\bW^1$ is bounded by $4$ external rays; two of them are $\bR^{-1}$ and $\bR^{0}$, we denote the other two by $\bR^a$ and $\bR^b$, see Figure~\ref{fig:Decomp of bW}. The ray $\bR^a$ lands at $\alpha(1/2,\bH^0)\in \bR^a$ and surrounds a puzzle piece $\bW^a$. Similarly, $\bR^b$ lands at $\alpha(1/2,\bH^0)\in \bR^b$ and surrounds a puzzle piece $\bW^b$. The puzzle pieces $\bW^1$, $\bW^a, \bW^b$ have pairwise disjoint interiors, and we have $\bW=\bW^1\cup \bW^a\cup \bW^b$. The map $\bF_\rr^{Q(\rr)}$ maps univalently $\intr \bW^a$ and $\intr \bW^b$ to two different connected components of $\C\setminus (\bW\cup \bR^{0})$.
\end{lem}
\begin{proof}
Since $\bF_\rr^{Q(\rr)}\colon \bW^1\to \bW$ is two-to-one, the puzzle piece $ \bW^1$ is bounded by $4$ rays $\bR^{-1}, \bR^a,\bR^b,\bR^0$, where $\bR^{-1}, \bR^b$ are preimages of $\bR^{-1}$ while $\bR^a,\bR^0$ are preimages of $\bR^0$. Since $\bR^{-1},\bR^0$ land at $\balpha$, we obtain that $\bR^a$ and $\bR^b$ land at $\alpha(1/2,\bH^0)$. As a consequence, $\bR^a$ and $\bR^b$ bound puzzle pieces $\bW^a$ and $\bW^b$. 

Observe that $\bR^{-1}$ and $\bR^a$ meet at $\alpha(1/4,\bH^0)$ while  $\bR^0$ and $\bR^b$ meet at $\alpha(3/4,\bH^0)$. This implies that $\bW=\bW^1\cup \bW^a\cup \bW^b$ and $\bF^{Q(\rr)}_\rr$ maps $\intr \bW^a\cup \intr \bW^a$ to $\C\setminus  (\bW\cup \bR^{0})$ as required.
\end{proof}

Fix an open neighborhood $\bO\coloneqq \Disk(\eta)$ of $\bF_\str$ containing $\bF_\rr$. Recall from~\S\ref{ss:FatComp of bFstr}  that $c_s(\bF_\str)\in \bF_\str^{-S}(0)$ denotes the unique critical point in $\partial \bZ_\str$ of generation $S$.  By Lemma~\ref{lem:discr of dyn}, $\bG^P(0)$ does not collide with $0$ for small $P$ and for $\bG\in \bO$. Therefore, for small $P$ the set $\displaystyle{\bigcup_{S\le P} \bG^{-S}(0)}$ moves holomorphically with $\bG\in \bO$. For $\bG\in \bO$, we denote by $c_S(\bG)$ the image of $c_S(\bF_\str)$ under the holomorphic motion.

For every $S<Q(\rr)$, there is a unique $i=i(S)$ such that $\bF^{Q(\rr)-S}_\rr$ maps univalently $ \bW$ onto $\bW(i)$. We set $\bW^1(i)\coloneqq \bF_\rr ^{Q(\rr)-S}(\bW^1)$. Then 
\begin{equation}
\label{eq:bW^1 i to bW}
\bF^S_\rr\colon \bW^1\big(i(S)\big) \to \bW
\end{equation} is a two-to-one map. 
 
\begin{lem}
\label{lem:c_S is correct}
For a small $S>0$, the point $c_S(\bF_\rr)$ is the unique critical point of~\eqref{eq:bW^1 i to bW}.
\end{lem}
\begin{proof}
The proof will be similar to the argument in Claim 1 of the proof of Theorem~\ref{thm:parab prepacm}. Applying $\RR$, we can assume that $\bF_\rr$ is in a small neighborhood $\bO$ of $\bF_\str$. For $\bG\in \bO$, denote by $L_g$ the unique affine map mapping $\alpha(f_\str)$ and $c_1(f_\str)$ to $\alpha(g)$ and $c_1(g)$ respectively.

Let $J(f_\str)\colon (0,1)\to Z_\str$ be a simple arc connecting the critical value $c_1(f_\str)$ and $\alpha$ (i.e.~$J$ lands at $c_1$ and $\alpha$) such that
\begin{enumerate}

\item $L_{f_\rr}(J(f_\str))$ intersected with a small neighborhood of $\alpha(f_\rr)$ is within $H^0_0$; and\label{cond:1:lem:c_S is correct}
\item $J(f_\str)\subset V\setminus \gamma_1$, possibly up to a slight rotation of $\gamma_1.$
\end{enumerate}
 Let $\bJ_0(\bF_\str)$ be the lift of $J(f_\str)$ via $\bS\simeq V\setminus \gamma_1$. Then $\bJ_0(\bF_\str)$ is an arc connecting from $0$ and $\balpha$ such that $\bJ_0\subset \bZ_\str$. For $g$ close to $f_\str$, we set $J(g)\coloneqq  L_g (J(f_\str))$. And we define $\bJ_0(\bG)$ to be the lift of $J(g)$ via $\bS\simeq V\setminus \gamma_1$. We obtain a holomorphic motion of $\bJ_0(\bG)$ in $ \bO$. By~\eqref{eq:CV are away from 0}, $\bJ_0(\bG)$ does not collide with $\CV\big(\bG^S\big)$ for a small $S$. Therefore, we also have a holomorphic motion of \[\widetilde \bJ(\bG) \coloneqq\displaystyle{ \bigcup_{S\le P}} \bG^{-S} \big( \bJ_0(\bG) \big).\]
 
 By Lemma~\ref{lem:preim of alpha}, there is a unique lift $\bJ_S(\bF_\str)\subset \widetilde \bJ$ of $\bJ_0(\bF_\str)$ under $\bF^S_\str$ such that $\bJ_S$ ends at $\balpha$, where $S$ is small. Moreover, $\bJ_S(\bF_\str)\subset\bZ_\str$ and $\bJ_S(\bF_\str)$ starts at $c_S$. Let $\bJ_S(\bF_\rr)$ be the image of $\bJ_S(\bF_\str)$ under the holomorphic motion.  We need to show that $\bJ_S(\bF_\rr)$  starts at the unique critical point of~\eqref{eq:bW^1 i to bW}. This follows from the following claim (using the homotopy lifting property).

\emph{Claim.} Let $N$ be a small neighborhood of $\balpha\cup \{0\}$. For a small $S>0$ the arc $\bJ_0(\bF_\rr)$ is homotopic rel $N\cup \CV\big(\bF_\rr^S\big)$ to a curve within $\bW$.
\begin{proof}[Proof of the Claim]
It follows from Condition~\eqref{cond:1:lem:c_S is correct} that $\bJ_S(\bF_\rr)\cap N\subset \bW$ for a small neighborhood $N$ of $\balpha\cup \{0\}$. Let $\bJ'_S$ be a simple arc in $\bW$ such that $\bJ'_S\cap N=\bJ_S \cap N$. For a small $S$, the set $\CV\big(\bF_\rr^S\big)\setminus \{0\}$ is far from $0$, thus $\bJ_S$ is homotopic to $\bJ'_S$ as required.
\end{proof}

This completes the proof of Theorem~\ref{thm:parab prepacm}.
\end{proof}

\subsection{Secondary parabolic prepacman $\bF_{\rr,\ss}$}
\label{ss:SecParPacm}
Since $f_\rr$ has $\qq$ attracting petals at $\alpha$, there is a small neighborhood $\UU$ of $f_\rr$ such that $\alpha(f_\rr)$ splits into the fixed point $\alpha(g)$ and a $\qq$-periodic cycle $\gamma (g)$ for $g\in \UU\setminus \{f_\rr\}$. Moreover, we can assume that the multiplier of $\gamma(g)$ parametrizes $\UU$, possibly by shrinking $\UU$. We also denote by \[\bUU\coloneqq \{\bG\mid g\in \UU\}\subset \Unst\]
the corresponding neighborhood of $\bF_\rr$.

For $\bG\in \bUU$ we denote by $\bgamma(\bG)$ the full lift of $\gamma(g)$ to the dynamical plane of $\bG$. More precisely, choose a point $\bgamma_0(g)$ in $\bgamma(g)$ and let $\bgamma_0(\bG)$ be the lift of $\gamma_0(g)$ to $\bS\simeq V\setminus \gamma_1$. Then $\bgamma(\bG)$ is the full orbit of $\bgamma_0$. Every point in  $\bgamma(\bG)$ is $Q(\rr)$-periodic.

Let us consider a path $g_{t}\in \UU$ with $t\ge 0$ emerging from $f_\rr=g_0$ such that $\gamma(g_t)$ is attracting and $\alpha(g_t)$ is on the boundary of the immediate attracting basin of $\gamma(g_t)$ for $t>0$. All $\bG_t$ with $t>0$ are conjugate by qc maps that are conformal on the Julia sets. We denote by $\HH_\rr$ the set of $\bG\in \Unst$ obtained by a qc deformation of $\bG_t$ changing the multiplier of $\bgamma$, see~\S\ref{ss:QC deform}. We say that $\HH_\rr$ is the \emph{primary satellite hyperbolic component attached to $\bF_\rr$}.

The external rays $\bR^i$ (see Theorem~\ref{thm:parab prepacm}) still land at $\balpha$ for all $\bG_t$. Therefore, the first renormalization map~\eqref{eq:ren map:par} exists for all $\bG_t$ as well as for all $\bG\in \HH_\rr$.

 By appropriately shrinking $\UU$ (and respectively $\bUU$) we can assume that the set of pacmen $g\in \UU$ with non-repelling $\gamma(g)$ is connected. In particular, every $g\in \UU$ with attracting $\gamma(g)$ is contained in $\HH_\rr$. 
 
 Consider a rational number $\ss=\pp_{\ss}/\qq_{\ss}>0$. If $\ss$ is close to $0$, then there is a unique parabolic prepacman $\bF_{\rr,\ss}\in \bUU\cap \partial \HH_\rr$ such that the multiplier of $\gamma(f_{\rr,\ss})$ is $\ee(\ss)$.

Consider a point $\bgamma_0$ in the periodic cycle $\bgamma(\bF_{\rr,\ss})$. An \emph{attracting flower} at $\bgamma_0$ is an open set $\bUpsilon_0$ such that
\begin{itemize}
\item $\bUpsilon_0 \cup \{\bgamma_0\}$ is connected; 
\item $\bF_{\rr,\ss}^{Q(\rr)}(\bUpsilon_0 )\subset \bUpsilon_0 $; and
\item all points in $\bUpsilon_0 $ are attracted by $\bgamma_0$ under the iterations of $\bF_{\rr,\ss}^{Q(\rr)}$. 
\end{itemize}

A connected component of $\bUpsilon_0 $ is called a \emph{petal}. Petals are permuted by $\bF_{\rr,\ss}^{Q(\rr)}$ and every petal is $Q(\rr,\ss)\coloneqq \qq_\ss Q(\rr)$ periodic. The flower $\bUpsilon_0 $ contains $m\qq_s$ petals; we will show in Lemma~\ref{lem:enum per comp of Frs} that $m=1$. 

We denote by $\bH$ the full orbit of $\bUpsilon_0$. Clearly, every connected component of $\bH$ is a Fatou component, see Figure~\ref{Fig:LimRabbit}.

\begin{figure} 
\centering{\begin{tikzpicture}
\node at (0,0){\includegraphics[scale=0.56]{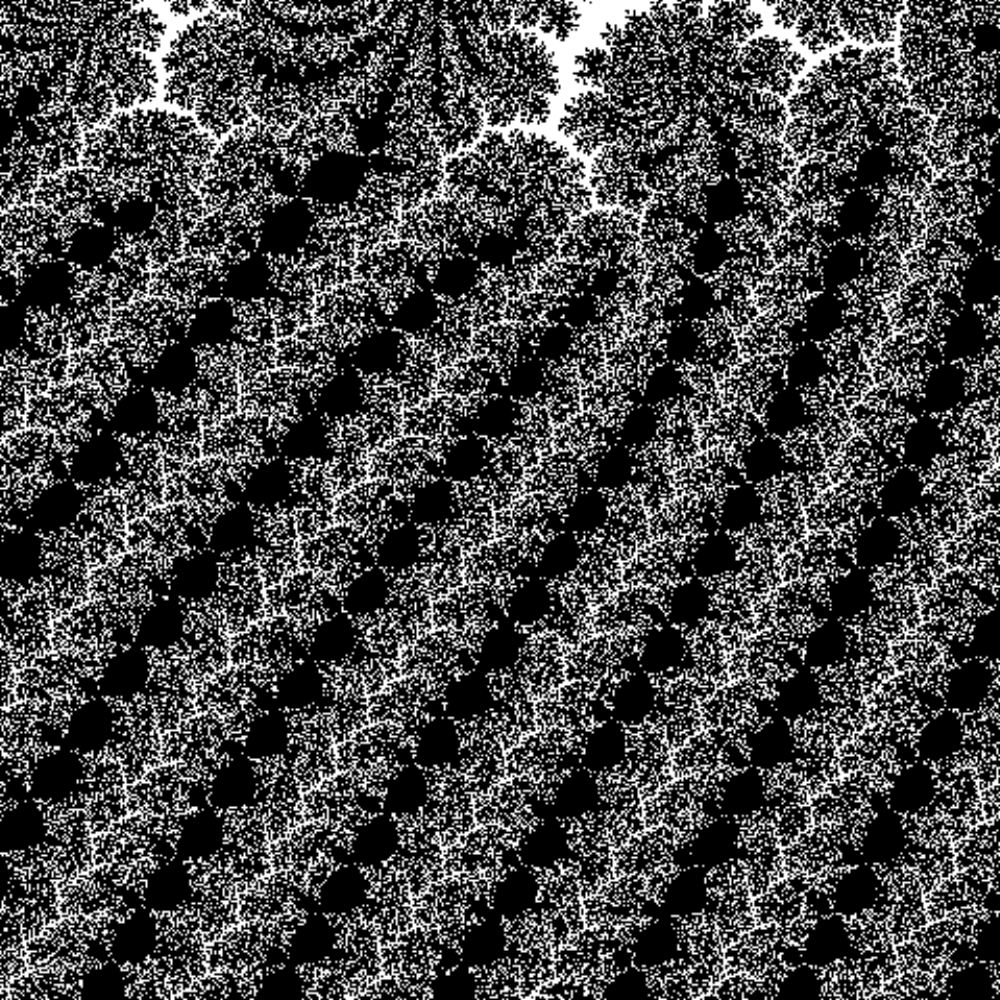}};
  \node at (0,-10.5){\includegraphics[scale=0.56]{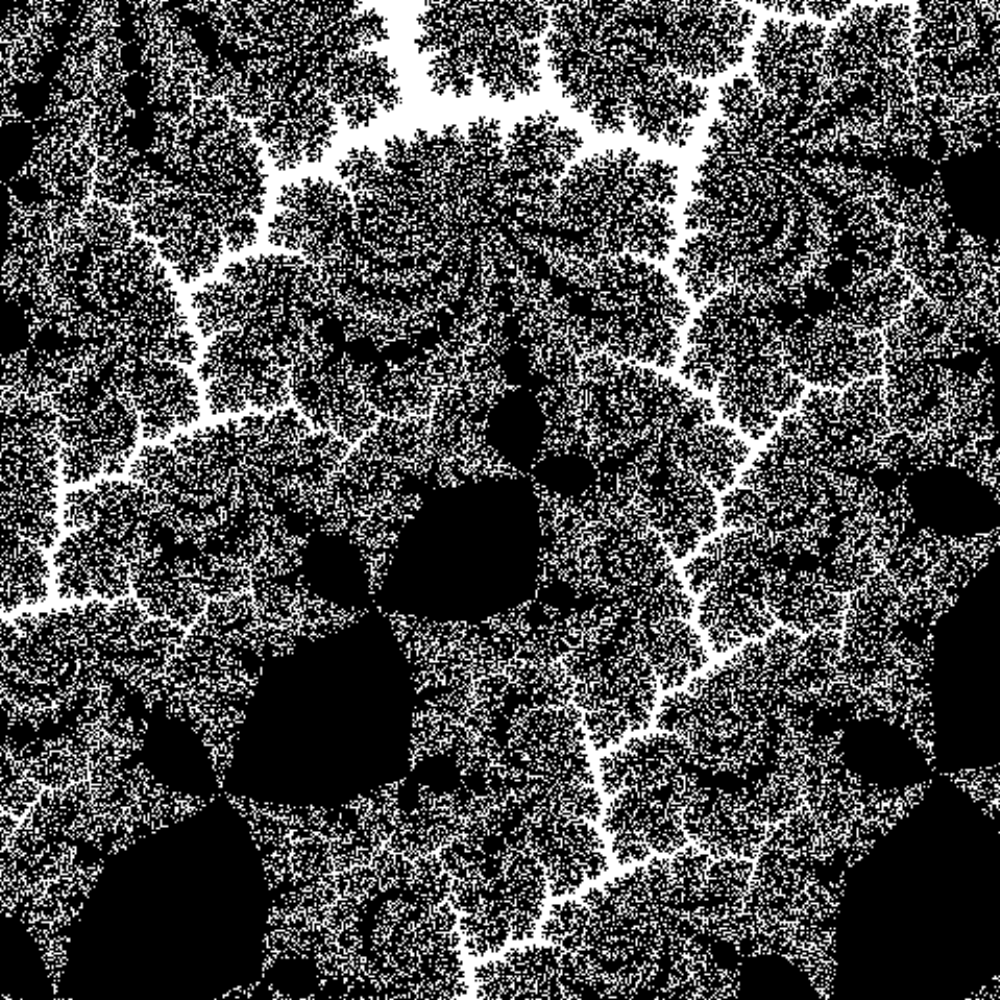}};

\begin{scope}[ line width=1pt,blue]
\draw[red]  (-0.26,-11.2) edge[->,bend right]  (0.65,-10.4)
 (0.65,-10.23) edge[->,bend right]  (0.28,-9.7)  
 (0.1,-9.7) edge[->,bend right] (-0.4,-11.) ; 
 
\node[red,right,scale=1.6] at (0.65,-10.4){$\bH_{\mathbf 0}$};

 \end{scope}
 \end{tikzpicture}}
\caption{The Rabbit maximal prepacman. There are three Fatou components attached to $\bgamma_0$. These components are cyclically permuted under the first return map. Note that there is a ``spiraling'' at the $\balpha=\text{``-$\infty i$''}$-fixed point.}
\vspace{128in} \label{Fig:LimRabbit}
\end{figure}

\begin{lem}
\label{lem:enum per comp of Frs}
The set $\bH$ has a periodic component $\bH^0$ containing $0$.  Moreover, $\bH$ has a unique cycle of periodic components. 

By re-enumerating points in $\bgamma$, we assume that $\bH^0$ is attached to $\bgamma_0\in \bgamma$. There are $\qq_s$ components $\bH^0,\bH^1,\dots, \bH^{\qq_\ss-1}$ of $\bH$ such that $\bH^i$ are attached to $\bgamma_0$ counting counterclockwise. The components $\bH^i$ are cyclically $\pp_s/\qq_s$-permuted under $\bF^{Q(\rr)}_{\rr,\ss}$.
 \end{lem}
\begin{proof}
Follows from a classical argument. Let $\bH'$ be a periodic component of $\bH$. If the forward orbit of $\bH'$ does not contain $0$, then the local Fatou coordinates of $\bH'$ can be extended to a conformal map between $\bH'$ and $\C$. This contradicts to $\bH'\subsetneqq\C$.

As a consequence, $\bH$ has a unique cycle of periodic components. The second statement follows from the fact that $\ee(\ss)$ is the multiplier of $\bgamma$.
\end{proof}

\begin{figure}[tp!]  
\centering{\begin{tikzpicture}
\begin{scope}[scale=0.7]
\draw (-2.,0.)-- (-2.,3.);
\draw (-2.,3.)-- (1.,3.);
\draw (1.,3.)-- (1.,0.);
\draw (1.,3.)-- (1.,3.58);
\draw (1.,3.58)-- (3.58,3.56);
\draw (3.58,3.56)-- (3.62,0.);
\draw (3.58,3.56)-- (3.58,3.98);
\draw (3.58,3.98)-- (6.38,4.);
\draw (6.38,4.)-- (6.42,0.);
\draw (-2.,3.)-- (-2.017159667882821,3.3905028845801137);
\draw (-2.017159667882821,3.3905028845801137)-- (-4.24,3.38);
\draw (-4.345349585700933,0.)-- (-4.24,3.38);
\draw (-4.24,3.38)-- (-4.296845629079723,4.231238132681097);
\draw (-4.296845629079723,4.231238132681097)-- (-7.207083026352363,4.2797420893023075);
\draw (-7.207083026352363,4.2797420893023075)-- (-7.207083026352363,0.);
\draw[red] (-0.6913848535697298,2.549767636479132)-- (-0.6913848535697298,5.330661149428534);
\draw[red] (-0.6913848535697298,5.330661149428534)-- (-2.69621506057977,4.554597843489166);
\draw[red] (-2.69621506057977,4.554597843489166)-- (-3.0842467135494553,2.905463318368009);
\draw[red] (-0.6913848535697298,5.330661149428534)-- (0.24635830777367618,6.721107905903234);
\draw[red] (0.24635830777367618,6.721107905903234)-- (1.8308208907332242,5.233653236186113);
\draw[red] (1.8308208907332242,5.233653236186113)-- (2.2835244858645236,2.67911152080236);
\draw[red] (0.24635830777367618,6.721107905903234)-- (0.3110302499352904,8.305570488862777);
\draw[red] (0.3110302499352904,8.305570488862777)-- (4.547042461521021,5.68635683131741);
\draw[red] (4.547042461521021,5.68635683131741)-- (5.258433825298777,2.970135260529623);
\draw[red] (0.3110302499352904,8.305570488862777)-- (0.3110302499352904,9.404993505610214);
\draw[red] (0.3110302499352904,9.404993505610214)-- (-5.15374886272111,6.2684043107719365);
\draw[red] (-5.15374886272111,6.2684043107719365)-- (-5.6711244000140235,3.131815115933658);
\draw[red] (0.3110302499352904,9.404993505610214)-- (1.7014770064099958,10.084048898307161);
\draw (6.38,4.)-- (6.390192813127025,4.683941727812393);
\draw (7.748303598520923,4.683941727812393)-- (6.390192813127025,4.683941727812393);
\node[red] at (-0.3,2.5) {$y_0$};
\node at (-0.5,1.6) {$\Delta_0(0)$};
\node at (-3.24,2.2) {$\Delta_0(-1)$};
\node at (-5.8,2.5) {$\Delta_0(-2)$};
\node at (2.2, 2.2)  {$\Delta_0(1)$};
\node at (5.2, 2.4)  {$\Delta_0(2)$};

\draw[blue] 
(4.5,5)
 .. controls (4.4, 3.7) and (4.4, 4.3)..
(4, 3) .. controls (3, 0.8) and (1, 1.5).. (0,0); 
\node[blue] at (2,0.8) {$R_i$};

\end{scope}
\end{tikzpicture}  }

\caption{Illustration to the proof of Theorem~\ref{thm:RatRays land}. The cycle of periodic rays $R_y$ (red) lands at a repelling periodic cycle $y$ that is within the triangulation $\bDelta_0$. The ray segment $R_i$ (blue) does not intersect $R_y$. Either the hyperbolic diameter of $R_i$ decreases, or $R_i$ shrinks to the cycle $\byy$, or $R_i$ tends to $\balpha$. }
\label{fig:lem:RatRays land}
\end{figure}

\subsection{Landing of dynamic rays}  
\begin{thm}
\label{thm:RatRays land}
Suppose $\bG\in \Unst$ has a parabolic or attracting periodic point ${x_0\in \C}$ and suppose that $\alpha(g)$ is repelling. Then 
every rational ray of $\bG$ lands. 
\end{thm}
\begin{proof}
Consider a rational ray $\bR$. Let us first give a sketch of the argument. If $\bR$ does not go to infinity, then $\bR$ lands by the expansion of $\bG$.  The infinity in $\C\setminus \bDelta_0$ is blocked by another periodic ray cycle (the red cycle in Figure~\ref{fig:lem:RatRays land}) that exists in a neighborhood of $\bF_\str$. Since $(\bbg_-,\sp \bbg_+)\mid \bDelta_0$ is a pair of almost translations, if $\bR$ goes to infinity in $\bDelta_0$, then $\bR$ tends to $\balpha$; in this case $\bR$ lands at $\balpha$.

We assume that $x_0$ is parabolic; the case when $x_0$ is attracting is similar. Denote by $\bxx$ the periodic cycle of $x_0$. Let $\bH$ be the attracting basin of $\bxx$. Since the Fatou coordinates in $\bH$ capture critical points of $\bG^{\ge 0}$, there is a unique periodic component $\bH^0$ of $\bH$ containing $0$. By re-enumerating points in $\bxx$, we can assume that $\bH^0$ is attached to $x_0$.

By Lemma~\ref{lem:unique conj}, the first return map $\bG\colon \bH^0\to \bH^0$ is conformally conjugate to $p_{1/4}\colon \Parab\to \Parab$, say via $\bbh_0\colon \bH^0\to \Parab$, where $p_{1/4}(z)=z^2+1/4$ and $ \Parab$ is the attracting basin of $\alpha(p_{1/4})$. As in~\S\ref{sss:Frr:exp metric}, let $B\subset \Parab$ be a forward invariant open topological disk containing $[1/4,1/2)$ such that $\overline B$ is a closed topological disk with 
\[\partial B\cap p_{c}(\overline B)=\alpha(p_{1/4})= \overline B \cap \partial \Parab.\] 
Set $\bB^0\coloneqq \bbh^{-1}_0(B)$; spreading around $\overline \bB^0$, we obtain
\[\bB\coloneqq \bigcup_{P\ge 0} \bG^P\big(\overline \bB^0\big) \subset \bH\cup \bxx.\]

 Consider the hyperbolic surface $\bX\coloneqq \C\setminus  \bB$. Since $\bB$ contains the postcritical set $\Post(\bG)$, we have $\bG^{-T}(\bX)\subsetneq \bX$ and 
\[
\bG^{T}\colon \bG^{-T}(\bX)\to \bX
\]expands the hyperbolic metric of $\bX$ for all $T\in \PT_{>0}$.

Let us assume that the rational ray $\bR$ is periodic with period $N\in \PT$; the preperiodic case follows from the periodic case. Let us decompose $\bR$ as a concatenation of ray segments \[ \dots  \cup \bR_{2}\cup \bR_{1}\cup \bR_0\]
such that $\bG^N$ maps $\bR_{i+1}$ to $\bR_{i}$.

Consider a renormalization triangulation $\bDelta_0(\bF_\str)$ and choose a periodic repelling point $y_0\in \Delta_0(0,1)$ such that $y_0$ does not escape $\Delta_0(0,1)$ under
\[\bF_\str^{A_0}\colon \Delta_0(0)\to \bS,\sp\sp \bF_\str^{B_0}\colon \Delta_0(1)\to \bS.\] 
(To existence of $y_0$ follows from the existence of a periodic point $y'_0$ in the dynamical plane of $f_\str$, see~\S\ref{sss:SiegPacmen}; then $y_0$ is the lift of $y'_0$.) Let $\byy(\bF_\str)\coloneqq \displaystyle\bigcup_{P\ge 0}\bF^{P}_\str
\{y_0\}$ be the periodic cycle of $y_0$. By Corollary~\ref{cor:ext rays land}, there is a cycle of periodic rays $\bR_\byy$ landing at $\byy(\bF_\str)$, see Figure~\ref{fig:lem:RatRays land}. 

Since the landing of a periodic ray at a repelling periodic point is stable,  there is a neighborhood $\YY$ of $\bF_\str$ such that 
\begin{itemize}
\item $\overline{\bR_\byy}(\bG)$ exists and depends continuously on $\bG\in \YY$; and 
\item $\byy\in \bDelta_0(\bG)$ for all $\bG\in \YY$.
\end{itemize}
By replacing $\bG$ with its antirenormalization, we can assume that $\bG\in \YY$.

Choose a sufficiently big $n\gg 0$ such that $\bR_n$ is disjoint from $\bR_\byy$. (If $\bR_n$ intersects $\bR_\byy$ for all $n$, then $\bR\subset \bR_\byy$ and the claim is trivial.) Let $\bR'_n$ be a rectifiable curve in $\bX\setminus \bR_\byy$ connecting the endpoints of $\bR_n$ such that $\bR'_n$ is homotopic (in $\bX\setminus \bR_\byy$) to $\bR_n$ rel the endpoints. For $j\ge 0$ we define $\bR'_{n+j}$ to be the lift of $\bR'_n$ under $\bG^{jN}$ such that $\bR'_{n+j}$  connects the endpoints of $\bR_{n+j}$. Clearly, $|\bR'_i|\le  |\bR'_{i-1}|,$ where ``$|~|$'' denotes the hyperbolic length in $\bX$. 

\begin{claim2} 
\label{cl2:bO of balpha}There is a neighborhood $\bO$ of $\balpha$ with the following property. If $\bR'_n\subset \bO$ for some $n$, then $\bR$ lands at $\balpha$.
\end{claim2}
\begin{proof}
Recall from \S\ref{ss:balpha:wall topoloy} that $\C\sqcup \{\balpha\}$ is endowed with the wall topology, where a neighborhood $\bO$ of $\balpha(\bG)$ is the full lift of a neighborhood $O$ of $\alpha(f)$. Since $\alpha(f)$ is repelling, we can choose a small open disk $O$ around $\alpha(f)$ such that
\begin{itemize}
\item $\partial O$ intersects the curves $\gamma_0\cup \gamma_1$ at two points; and
\item $O_2\coloneqq f(O)\Supset O$ and $O_2$ also intersects $\gamma_0\cup \gamma_1$ at two points, 
\end{itemize}
 see Figure~\ref{Fig:balpha:neigbb}. The pair $\gamma_0\cup \gamma_1$ cuts $O$ into two sectors. Lifting these sectors to the dynamical plane of $\bG$ and spreading the lifts around, we obtain a neighborhood $\bO$ of $\balpha(\bG)$ such that $\bO$ is backward  invariant. If $\bR'_n\subset \bO$, then $\bR'_{n+i}\subset \bO$ for all $i$, and, moreover, $\bR'_m$ tends to $\balpha$ as $m\to +\infty$.
 \end{proof}

\begin{figure}
\centering{\begin{tikzpicture}

\begin{scope}[shift={(0,4)}]

\draw[red](0.,0.) circle (2cm);
\draw[](0.,0.) circle (1.5cm);

\draw [rotate around={30:(0,0)}] (0,0)-- (2,2)node[left]{$\gamma_0$};
\draw[rotate around={160:(0,0)}]  (0,0)  --(2.1,2.1) node[above]{$\gamma_1$};
\node[left] at (1.5,0) {$O$};
\node[right, red] at (2,0) {$O_2$};
\node[below  ] at (0,0){$\alpha(f)$}; 

\draw (0.9,0.9) edge [->,bend right]  node[above]{$f$}(-1.3,1.3);
\end{scope}

\draw (-0.5,0) -- (1.5,0);

\draw (-0.5,0)--(-0.5,-3);
\draw (-2,0)--(-0.5,0);

\node[below] at (-0.5,-3) {$\balpha(\bG)= \text{``$-\infty i$''}$};

\draw[red] (1.5,-3)--(1.5,1)--(-2,1)--(-2,-3);

\begin{scope}[shift={(2,1)}]
\draw (-0.5,0) -- (1.5,0);
\end{scope}

\draw (3.5,1)--(3.5,-3);
\draw (3.5,1)--(4.5,1);

\begin{scope}[shift={(-3.5,0)}]

\draw (-0.5,0) -- (1.5,0);
\end{scope}

\draw (-4,0)--(-4,-3);
\draw (-4,0)--(-6,0);
\draw (-5.5,0)--(-5.5,-3);

\draw (-1.2,-1.1) edge[->, bend left ]  (0.8,0.5);
\node[left]at (-1,-0.8) {$\bbf_+$};
\draw (0.4,-1.2) edge[->, bend right ]  (-1.2,0.5);
\node[left]at (0.35,-0.8){$\bbf_-$};

\node[left]at (-2.8,-2) {$\bO$};

\end{tikzpicture}}
\caption{Top: a neighborhood $O$ of $\alpha(f)$ and its image $O_2\coloneqq f(O)$.  Bottom: the neighborhood $\bO$ of $\balpha(f)$ is obtained by lifting and spreading around $O(f)$. Note that the points in $\bO$ stay in $\bO$ under the backward iteration of the cascade $\bG^{\ge 0}$.}
\label{Fig:balpha:neigbb}
\end{figure}

\begin{claim2} 
\label{cl2:bM}
 Let $\bM$ be a sufficiently small neighborhood of $\bxx\cup\{\balpha\}$. Then there is an $\varepsilon>0$ such that the following holds. If $\bR'_i$ intersects $\bDelta_0\setminus \bM$, then  
\[|\bR'_i |  \le \max \left\{ |\bR'_{i-1}|-\varepsilon , \frac{|\bR'_{i-1}|}{1+\varepsilon} \right\} .\]
\end{claim2} 
\begin{proof}
By increasing $N$, we can assume that $N\ge 2\max\{A_0,B_0\}$. Choose a point $z\in (\bDelta_0\setminus \bM)\cap \bR'_i$. Let $A\le \max\{A_0,B_0\} $ be the first moment such that $\bG^A(z)\in \bS$.

If $\bG^A(z)\in \Delta_0(0,1)$, then either \[\bbf_+\colon \Delta_0(0)\to \bS\sp \text{ or }\sp \bbf_-\colon \Delta_0(1)\to \bS\] expands the hyperbolic length of $\bG^A(\bR'_i)$.

If $\bG^A(z)\in \bS\setminus \Delta_0(0,1)$, then $ \bS\setminus \Delta_0(0,1)\subset \Delta_{-1}(0,1)$ (see~\eqref{eq:S in Delta-1}), and either
\[\bbf^\#_{-1,+}\colon \Delta_{-1}(0)\to \bS_{-1}^\#\sp \text{ or }\sp \bbf^\#_{-1,-}\colon \Delta_{-1}(1)\to \bS^\#_{-1}\] 
expands the hyperbolic length of $\bG^A(\bR'_i)$.
\end{proof}

As a consequence of Claim \ref{cl2:bM}, either the diameters of the $\bR'_i$ shrink, or the $\bR'_i$
 are eventually in $\C\setminus (\bDelta_0\setminus \bM)$. Suppose $\bR'_i\subset \C\setminus (\bDelta_0\setminus \bM)$ for all $i\ge n$. If $\bR'_n$ is in a small neighborhood of $\balpha$, then $\bR$ lands at $\balpha$ by Claim \ref{cl2:bO of balpha}. If $\bR'_n$ is in a small neighborhood of $\bxx$, then $\bR$ lands at $\bxx$. 
 The remaining case is when all $\bR'_i$ are in some connected component of $\C\setminus (\bDelta_0\cup \bR_\byy)$ for all $i\gg 0$. Every connected component of  $\C\setminus (\bDelta_0\cup \bR_\byy)$ is bounded; by hyperbolic contraction, the hyperbolic diameters of the $\bR'_i$ tend to $0$.
\end{proof}

\begin{cor}
\label{cor:FRM exists in bVV}
The primary renormalization map~\eqref{eq:ren map:par}
\begin{equation}
\label{eq:FRM:bG}
\bG^{Q(\rr)}\colon \bW^1\to \bW.
\end{equation}
 exists for all $\bG$ close to $\bF_{\rr,\ss}$. In particular, $\bR^{-1}(\bG)$ and $\bR^0(\bG)$ land at $\balpha$.
\end{cor}
\begin{proof}
 In the dynamical plane of $\bF_{\rr,\ss}$, let us consider the cycle of rays from Theorem~\ref{thm:parab prepacm}. By Theorem~\ref{thm:RatRays land} the cycle $(\bR_i)$ lands at a periodic cycle of points $\delta$. We need to show that $\delta=\balpha$. 
 
 Observe that $\delta\not= \bgamma$ because the period $Q(\rr,\ss)$ of repelling petals at $\bgamma$ is greater than the period $Q(\rr)$ of $(\bR_i)_{i\in \Z}$. Therefore, $\delta$ is repelling. Since the landing at a repelling periodic cycle is stable under a small perturbation, $(\bR_i(\bG))_{i\in \Z}$ lands at $\delta(\bG)$ for $\bG$ in a small neighborhood of $\bF_{\rr,\ss}$. We obtain that $\delta(\bG)=\balpha(\bG)$ for $\bG\in \HH_\rr$, see~\S\ref{ss:SecParPacm}. Therefore, $\delta(\bG)=\balpha(\bG)$ for all $\bG$ close to $\bF_{\rr,\ss}$, and~\eqref{eq:ren map:par} exists for all $\bG$ close to $\bF_{\rr,\ss}$.
\end{proof}

\begin{figure}

\centering{\begin{tikzpicture}[red]

\node[ black,shift={(5,0)},scale=1] at(0.3,-0.25) {$\bgamma_0$};

\begin{scope}[shift={(5,0)},rotate=144, yscale=1.1, xscale =1.4,scale =0.55,blue]

\node at (4,2) {$\bR_\ell$};
\node at (4,-2) {$\bR_\rho$};

\draw (0,0) .. controls (1,1.2) and (3,1.5).. (5,1.5) .. controls (5.5,1) and (5.5,-1).. (5,-1.5) .. controls (3,-1.5) and (1,-1.2).. (0,0);

\draw[green] (2.4,1.26) 
.. controls (2.6,1) and (2.8,0.5).. 
 (2.81,0)
.. controls (2.8,-0.5) and (2.6,-1) ..  
 (2.4,-1.26) ; 
\draw (5.35,0) -- (7,0);

\node(A)[black] at (2.1,-0.7) {$\bW_{\rr,\ss}^1$};

\node(B)[black] at (4.7,1) {$\bW_{\rr,\ss}$}; 

\draw[black] (A) edge[->,bend right] node[above] {$\bF_{\rr,\ss}^{Q(\rr,\ss)}$} (B);
 
\end{scope}

\begin{scope}[shift={(5,0)},rotate=0,scale =0.4]
\draw[fill, fill opacity=0.1] (0,0) .. controls (1,0.4) and (3,0.7).. (5,0.7) .. controls (5.5,0.3) and (5.5,-0.3).. (5,-0.7) .. controls (3,-0.7) and (1,-0.4).. (0,0);
\node at (4,0){$\bH_{3}$};

\end{scope}

\begin{scope}[shift={(5,0)},rotate=72,scale =0.4]
\draw[fill, fill opacity=0.1] (0,0) .. controls (1,0.4) and (3,0.7).. (5,0.7) .. controls (5.5,0.3) and (5.5,-0.3).. (5,-0.7) .. controls (3,-0.7) and (1,-0.4).. (0,0);

\node at (4,0){$\bH_{4}$};
\end{scope}

\begin{scope}[shift={(5,0)},rotate=144,scale =0.4,yscale=0.8]
\draw[fill, fill opacity=0.1] (0,0) .. controls (1,0.4) and (3,0.7).. (5,0.7) .. controls (5.5,0.3) and (5.5,-0.3).. (5,-0.7) .. controls (3,-0.7) and (1,-0.4).. (0,0);

\node at (4,0){$\bH_0$};
\end{scope}

\begin{scope}[shift={(5,0)},rotate=216,scale =0.4]
\draw[fill, fill opacity=0.1] (0,0) .. controls (1,0.4) and (3,0.7).. (5,0.7) .. controls (5.5,0.3) and (5.5,-0.3).. (5,-0.7) .. controls (3,-0.7) and (1,-0.4).. (0,0);

\node at (4,0){$\bH_1$};
\end{scope}

\begin{scope}[shift={(5,0)},rotate=-72,scale =0.4]
\draw[fill, fill opacity=0.1] (0,0) .. controls (1,0.4) and (3,0.7).. (5,0.7) .. controls (5.5,0.3) and (5.5,-0.3).. (5,-0.7) .. controls (3,-0.7) and (1,-0.4).. (0,0);
\node at (4,0){$\bH_{2}$};

\end{scope}

\end{tikzpicture}}
\caption{The secondary renormalization map $\bF^{Q(\rr,\ss)}\colon \bW^1_{\rr,\ss}\to \bW_{\rr,\ss}$ in the dynamical plane of $\bF=\bF_{\rr,\ss}$. Petals in the attracting flower $(\bH_i)$ around $\bgamma_0$ are enumerated counterclockwise so that $\bH_0\ni 0$. The puzzle piece $\bW_{\rr,\ss}\supset \bH^0$ is bounded by the characteristic ray pair $\bR_\ell, \bR_\rho$. The puzzle piece $\bW^1_{\rr,\ss}$ is bounded by $\bR_\ell, \bR_\rho$ and there preimages (marked green). When $\bgamma_0$ becomes repelling (for example if $\bF\in \HH_{\rr,\ss}$), the secondary renormalization map can be thickened to a quadratic-like map.} \label{Fig:Flower around gamma 0}
\end{figure}

Recall from Lemma~\ref{lem:enum per comp of Frs} that $\bH^0$ denotes the periodic Fatou component of $\bF_{\rr,\ss}$ containing $0$, and recall that $\bH^0$ is attached to $\bgamma_0$, see Figure~\ref{Fig:Flower around gamma 0}. 
\begin{lem}
In the dynamical plane of $\bF_{\rr,\ss}$ there is a characteristic pair $\bR_\ell,\bR_\rho$ of periodic rays landing at $\bgamma_0$: it is the unique ray pair such that $\overline {\bR_\ell\cup \bR_{\rho}}$ separates $\bH^0$ from all $\bH^i$ with $i\not=0$.
\end{lem}
\begin{proof}
The parabolic point $\bgamma_0$ has exactly $\qq_\ss$ local repelling petals, and we need to show that there is a periodic ray landing at every repelling petal of $\bgamma_0$. Recall from Corollary~\ref{cor:FRM exists in bVV} that the first renormalization map
 \begin{equation}
 \label{eq:FRM:bF rr ss}
\bF_{\rr,\ss}^{Q(\rr)}\colon \bW^1\to \bW
\end{equation}
exists.

Denote by $\frakX$ the non-escaping set of~\eqref{eq:FRM:bF rr ss}, and note that $\frakX$ contains $\bgamma_0$, $\bH^0$, and $0$. Define inductively $\bW^{n+1}$ to be the unique degree two preimage of $\bW^n$ under~\eqref{eq:FRM:bF rr ss}. By induction (the case $n=1$ is in Lemma~\ref{lem:decomp:bW(bF)}), $\bW^{n}\setminus \bW^{n-1}$ consists of $2^n$ connected components, $\bW^{n+1}\subset \bW^n$, and $\displaystyle{\bigcap_{n\ge0} \bW^n=\frakX}$. Note that the components of $\bW\setminus \bW^1$ have bounded diameters with respect to the hyperbolic metric of $\C\setminus \overline \Post(\bF_{\rr,\ss})$, see Figure~\ref{fig:Decomp of bW}. Since $\bF^{Q(\rr)}_{\rr,\ss}$ expands the hyperbolic metric of $\C\setminus \overline \Post(\bF_{\rr,\ss})$, the spherical diameter of the components of $\bW^{n+1}\setminus \bW^n$ tends to $0$. Therefore, for every local repelling petal $P$ of $\bgamma_0$, there is a sequence $\bT_k\subset P$, $k\gg0$ shrinking to $\bgamma_0$ such that $\bT_k$ is a component of $\bW^{k}\setminus \bW^{k-1}$ and $\bF_{\rr,\ss}^{Q(\rr,\ss)}$ maps $\bT_{k+1}$ to $\bT_k$. This implies that there is a unique geodesic $\bR$ in the escaping set $\Esc (\bF_{\rr,\ss})$ intersecting every $\bT_k$; thus $\bR$ is a periodic ray landing in $P$ and $Q(\rr,\ss)$ is the minimal period of $\bR$.
\end{proof}

 We denote by $\bW_{\rr,\ss}$ the puzzle piece bounded by $\bR_\ell\cup \bR_{\rho}$ and containing $\bH^0$. Let  $\bW^1_{\rr,\ss}\subset \bW_{\rr,\ss} $ be the two-to-one pullback of $ \bW_{\rr,\ss}$ along $\bF^{Q(\rr,\ss)}_{\rr,\ss}\colon \bH^0\to \bH^0$. We call 
\begin{equation}
\label{eq:SRM:rr ss}
\bF^{Q(\rr,\ss)}_{\rr,\ss}\colon \bW_{\rr,\ss}^1 \to \bW_{\rr,\ss}
\end{equation}
a \emph{secondary renormalization map}.

\subsection{Secondary parabolic bifurcation}
Since $\bF_{\rr,\ss}$ has $\qq_\ss$ attracting petals at every point of $\bgamma$ (see~\S\ref{ss:SecParPacm}), there is a small neighborhood $\bVV$ of $\bF_{\rr,\ss}$ such that $\bgamma(\bF_{\rr,\ss})$ splits into the $Q(\rr)$-periodic cycle $\bgamma(\bG)$ and a $Q(\rr,\ss)$-periodic cycle  $\bdelta(\bG)$ for $\bG\in \bVV\setminus \{\bF_{\rr,\ss}\}$. By shrinking $\bVV$,  we can assume that 
\begin{itemize}
\item the multiplier of $\bdelta(\bG)$ parametrizes $\bVV$;
\item the primary renormalization map~\eqref{eq:ren map:par} exists for all $\bG\in \bVV$ (by Corollary~\ref{cor:FRM exists in bVV}).
\end{itemize}

Moreover, there is a small path $\bG_t$, $t\in [0,1]$ emerging from $\bF_{\rr,\ss}=\bG_0$ such that $\bdelta(\bG_t)$ is attracting and $\bgamma(\bG_t)$ is on the boundary of the immediate attracting basin of $\bdelta(\bG_t)$ for all $t\in (0,1]$. Moreover, we can assume that the characteristic ray pair $\bR_\ell(\bG_t), \bR_\rho(\bG_t)$ lands at $\bgamma_0$ for all $t>0.$

Since the escaping set $\Esc(\bG)$ moves holomorphically for all $\bG $ in a small neighborhood of $\{\bG_t\mid t>0\}$, we obtain that all $\bG_t$ with $t>0$ are qc conjugate. We denote by $\HH_{\rr,\ss}\supset \{\bG_t\mid t>0\}$ the set of $\bG\in \bVV$ obtained through a qc deformation changing the multiplier of $\bdelta$, see~\S\ref{ss:QC deform}. By appropriately shrinking $\bVV$, we can also assume that  $\HH_{\rr,\ss}\cap \bVV$ has a single connected component attached to $\bF_{\rr,\ss}$, and: 
\begin{equation}
\label{eq:delta is attr}
\HH_{\rr,\ss}\cap  \bVV=\{\bG\in \bVV: \bdelta(\bG)\text{ is attracting}\}.
\end{equation}

Since the rays $\bR_\ell$ and $\bR_\rho$ land at $\bgamma_0$ for all $\bG\in \HH_{\rr,\ss}$, the secondary renormalization \eqref{eq:SRM:rr ss}
\begin{equation}
\label{eq:SRM}
\bG^{Q(\rr,\ss)}\colon \bW_{\rr,\ss}^1 \to \bW_{\rr,\ss}
\end{equation}
exists for $\bG\in \HH_{\rr,\ss}$.

\begin{figure}[tp!]
\begin{tikzpicture}
\draw[white,fill=black, fill opacity=0.1](-2.3,0)--(2.3,0)--(2.3,-2)--(-2.3,-2);
\draw (-2.3,0)--(2.3,0);
\filldraw(0.,0.) circle (0.04cm);
\node[below] at (0,0) {$\bF_\str$}; 

\node[below] at (-1.6,-0.2) {$\HH$};
\node[below] at (-1,-0.7) {$\left(\begin{matrix}
       \balpha(\bG)\\
 \text{is attracting}
        \end{matrix} \right)$};
        
 \draw[fill=black, fill opacity=0.1] (1,0.5) arc (180:360:0.5);       
 
 \node[above] at (1.5,0.1) {$\HH_\rr$};

\filldraw(1.5,0.) circle (0.04cm);
\node[below]at(1.5,0){$\bF_\rr$}; 
\draw[dashed,blue](1.5,0.) circle (0.75cm);

 \node[below,blue] at(1.5,-0.75){$\bUU$}; 
 
 \begin{scope}[shift={(1.02,0.05)},scale=0.2]
\draw[red] (0,0)--(1,0)--(1,1)--(0,1)--(0,0); 
\coordinate  (w0) at (1,0.6);
 \end{scope}

 \begin{scope}[shift={(2.8,-1)},scale=2]
 

 \draw[fill=black, fill opacity=0.1] (0.2,1) .. controls (0.45,0.6) and (0.6,0.45) .. (1,0.2) --(1,1);  

\draw[red](w0) --(0,0.6);

\filldraw(0.6,0.49) circle (0.02cm);

\node[above right ] at (0.6,0.49){$\bF_{\rr,\ss}$};

\draw[dashed,blue](0.6,0.49) circle (0.3cm);
\node[below,blue] at (0.6,0.21) {$\bVV$};

\draw[fill=black, fill opacity=0.1] (0.4,0.47) 
.. controls (0.45,0.51) and (0.55,0.52) ..
(0.6,0.49)
.. controls (0.63,0.45) and (0.62,0.35) ..
(0.61,0.3) ;

 \begin{scope}[shift={(0.58,0.35)},scale=0.08]
\draw[red] (0,0)--(1,0)--(1,1)--(0,1)--(0,0); 
\coordinate  (w1) at (1,0.6);
 \end{scope}
 
\draw[red] (0,0)--(1,0)--(1,1)--(0,1)--(0,0); 

 \end{scope}

 \begin{scope}[shift={(5.86,-1)},scale=2]
 
\draw [fill=black, fill opacity=0.1] (0,0)--(0,1)--(0.5,1) --(0.5,0);

\draw[red](w1) --(0,0.6);

\draw[rotate around={135:(0.5,0.5)}, fill=black, fill opacity=0.1,shift={(-0.1,0.01)}] (0.4,0.44) 
.. controls (0.45,0.51) and (0.55,0.52) ..
(0.6,0.49)
.. controls (0.67,0.45) and (0.65,0.35) ..
(0.61,0.247) ;

 \draw[dashed,blue](0.6,0.49) circle (0.3cm);
\node[below,blue] at (0.6,0.21) {$\bWW$};
\filldraw(0.5,0.51) circle (0.02cm);

\node[left] at  (0.5,0.51) {$\bF_{\rr,\ss,\ff}$}; 

 \begin{scope}[shift={(0.475,0.33)},scale=0.08]
\draw[red] (0,0)--(1,0)--(1,1)--(0,1)--(0,0); 
\coordinate  (w2) at (0,0);
 \end{scope}

\draw[red] (0,0)--(1,0)--(1,1)--(0,1)--(0,0);

 \end{scope}

 \begin{scope}[shift={(4.5,-3.5)},scale=2]
 
\draw [fill=black, fill opacity=0.1] (0,0)--(0,1)--(0.2,1) --(0.2,0); 
 
 \draw[red] (0,0)--(1,0)--(1,1)--(0,1)--(0,0); 

\draw[red](w2) --(0.6,1);

\draw[dashed,fill=black, fill opacity=0.1] (0.5,0.5) ellipse (0.3cm and 0.15cm);
 
 \node at (0.5,0.5){$\bMM_0$};

 \end{scope}

\end{tikzpicture}
\caption{Recognizing  a small copy on the unstable manifold. A parabolic prepacman $\bF_\rr$ has a small neighborhood $\bUU$ parametrized by the multiplier of $\bgamma$. A secondary parabolic prepacman $\bF_{\rr,\ss}$ has a small neighborhood $\bVV\subset \bUU$ where the primary renormalization map~\eqref{eq:FRM:bG} exists. A ternary parabolic prepacman $\bF_{\rr,\ss,\ff}$ has a neighborhood $\bWW\subset \bVV$ where the secondary renormalization map~\eqref{eq:SRM} exists. The set $\bWW$ contains a ternary small copy $\bMM_0$.}
\label{Fig:Rec Small copy}
\end{figure}

\subsection{A ternary small copy $\bMM_{0}$}
\label{ss: ternary sat copy}
Choose a rational number $\ff=\pp_\ff/\qq_\ff$ close to $0$ and let $\bF_{\rr,\ss,\ff}\in \partial \HH_{\rr,\ss}$ be a parabolic prepacman such that the multiplier of $\bdelta(\bF_{\rr,\kk})$ is $\ee(\pp_\ff/\qq_\ff)$. By Theorem~\ref{thm:RatRays land}, rational rays land in the dynamical plane of $\bF_{\rr,\ss,\ff}$. The same argument as in Corollary~\ref{cor:FRM exists in bVV} shows that the secondary renormalization~\eqref{eq:SRM} exists in a small neighborhood $\bWW$ of $\bF_{\rr,\ss,\ff}$.

Observe that the secondary renormalization~\eqref{eq:SRM} can be slightly thickened to a quadratic-like map, compare with~\cites{DH, Mi}. Indeed, since the rays $\bR_\ell, \bR_\rho$ forming $\partial \bW_{\rr,\ss}^1$ land at a (regular) repelling periodic point $\bgamma_0\in\Dom\big(\bG^{Q(\rr,\ss)}\big)$, see Figure~\ref{Fig:Flower around gamma 0}, we thicken $\bW_{\rr,\ss}^1$ in a small neighborhood of $\bgamma_0$ respecting its linear coordinates, and then extend the thickening for finitely many iterations along $\bR_\ell, \bR_\rho$. Let $\brho$ be the Douady-Hubbard straightening map associated with~\eqref{eq:SRM} ($\brho$ is $\strai$ from \S\ref{ssss:QL anal ren} appropriately restricted). Then $\brho$ is a homeomorphism in a small neighborhood of $\bF_{\rr,\ss,\ff}$. Observe that $p_\ff\coloneqq \brho(\bF_{\rr,\ss,\ff})$ is the quadratic polynomial whose $\alpha$-fixed point has multiplier $\ee(\ff)$. Choose $\kk\in \Q$ close to $\ff$.  By the Yoccoz inequality in the quadratic family, the primary satellite small copy $\Mandel_\kk$ (see notations in \S\ref{sss:not:SatCopies}) is in a small neighborhood of $p_\ff$. We obtain that $\bMM_{\rr,\ss,\kk}\coloneqq \brho^{-1}(\Mandel_\kk)$ is a full small copy of the Mandelbrot set (see Figure~\ref{Fig:Rec Small copy}):

\begin{thm}
\label{thm:small copy M1}
In a small neighborhood of $\bF_{\rr,\ss,\ff}$ there is a ternary satellite copy $\bMM_0= \bMM_{\rr,\ss,\kk}$ of the Mandelbrot set.\qed
\end{thm}

\noindent We say that $\bF_\rr$ is the \emph{root of the limb containing} $\bMM_0$.

\begin{rem}
The fact that $\bMM_{\rr,\ss,\kk}$ is a bounded subset of $\Unst$ follows from the Yoccoz inequality. Note that we use the Yoccoz inequality in the quadratic family after applying the straightening map associated with the second renormalization~\eqref{eq:SRM}. By applying the Yoccoz inequality directly in $\Unst$ (after slight thickening of external rays), it is possible to construct a \emph{secondary} satellite copy $\bMM_{\rr,\ss'}$ of the Mandelbrot set in a small neighborhood of $\bF_{\rr,\ss}$. This would imply the scaling law for secondary satellite copies as in~\eqref{eq:thm:main}.

The case of primary copies of the Mandelbrot set is more delicate because the first renormalization map~\eqref{eq:FRM:bG} is pinched and the pinching can not be resolved:~\eqref{eq:FRM:bG}  maps $\alpha(1/2)\in \partial \bW^1$ to $\balpha\in \partial \bW$ and $\alpha(1/2)\in \Esc_{Q(\rr)}$ is an essential singularity for $\bG^{Q(\rr)}$ by Lemma~\ref{lem:Acc Of preim}; i.e.~$\bG^{Q(\rr)}$ has no extension through $\Esc_{Q(\rr)}$. Moreover, the map $\mRRc^\mm$ in~\eqref{eq:thm:main} cannot be quasiconformal because $\Mandel_{p/q}$ and $\Mandel_{p'/q'}$ are not qc homeomorphic if $q\not=q'$~\cite{LP}. However, the pictures suggest that the scaling law~\eqref{eq:thm:main} for primary copies may still be valid.
\end{rem}

\section{The Valuable Flower Theorem}
\label{s:val flow}

Let us fix a copy $\bMM_0\subset \Unst$ from Theorem~\ref{thm:small copy M1}. We set $\bMM_{n}\coloneqq \RR^{n}(\bMM_0)$.

\begin{figure}
\begin{tikzpicture}[xscale=1.2]

\begin{scope}[shift={(6,0)}, line width=3pt,blue,scale =1.2]

\draw (0.7161732054529941,2.479411658020722)-- (0.8094459499434268,4.377758721340394);

\draw[ line width=4pt,black, shift={(0.54, 2.5)},scale=0.3] (0.7161732054529941,2.479411658020722)-- (0.8094459499434268,4.377758721340394);

\draw[black] (0.5344819805156678,3.5617366185225294)-- (1.0489306975095394,3.535127202126295);

\draw(0.7161732054529941,2.479411658020722)-- (2.4414901555791584,2.298676320247921);

\draw[ line width=4pt,black, shift={(1.22, 1.65)},scale=0.3] (0.7161732054529941,2.479411658020722)-- (2.4414901555791584,2.298676320247921);

\draw[black] (1.689330652112256,2.767002048821652)-- (1.6325638971336218,2.0006508566100925);

\draw (0.7161732054529941,2.479411658020722)-- (-0.4145872042833716,3.3932103146797097);

\draw [ line width=4pt,black, shift={(0.05, 2.1)},scale=0.3](0.7161732054529941,2.479411658020722)-- (-0.4145872042833716,3.3932103146797097);

\draw[black] (-0.05092518020149673,2.8078031539625465)-- (0.2329085946916739,3.144855761648186);

\draw (-0.5742437026607801,1.9474320238176246)-- (0.7161732054529941,2.479411658020722);

\draw[ line width=4pt,black, shift={(-0.05, 1.5)},scale=0.3] (-0.5742437026607801,1.9474320238176246)-- (0.7161732054529941,2.479411658020722);

\draw[black] (-0.08640440206314305,2.3731826861573797)-- (0.06438229084885384,2.0095206620755057);

\draw (0.7161732054529941,2.479411658020722)-- (0.8804043936667194,0.17347093073531167);

\draw[ line width=4pt,black, shift={(0.56, 0.8)},scale=0.3] (0.7161732054529941,2.479411658020722)-- (0.8804043936667194,0.17347093073531167);

\draw[black] (0.3748254821382592,1.166889142861407)-- (1.1553683630944784,1.228977781119288);

\node (v0) at (0, 3.05) {};

\draw[red,line width=1pt,dashed, rotate around={50:(v0)}]  (v0) ellipse (0.3 and 1.1);

\draw[red,line width=1pt,dashed, rotate around={50:(v0)}]  (v0) ellipse (0.45 and 1.25);

\draw[line width=1pt,black] (0.4, 3.05) edge[bend right=70, ->]  (-0.9, 3.8);

\node[above,black ] at (-0.9, 3.9){$\bG^{Q(\rr,\ss)}\colon {\color{red} \bO}\to {\color{red} \bO'}$}; 
\end{scope}

\begin{scope}[scale=0.9,line width=0.3mm]

\begin{scope}[red]
\draw (-1.223793823139081,-4.202696996018645)edge[ ->] (-1.162207568263142,2.7257566775244775);
\draw (3.087244018176641,-4.202696996018645)edge[->] (2.8716921261108554,3.187653589094019);
\draw (-1.162207568263142,2.7257566775244775)edge[-] (-0.46,5.06);
\draw (-0.46,5.06)edge[->](0.7161732054529941,6.4517250975187785);
\draw (0.7161732054529941,6.4517250975187785)edge[-<] (2.0094845578477107,5.035241235372185);

\draw (2.0094845578477107,5.035241235372185)-- (2.8716921261108554,3.187653589094019);
\draw (-0.46,5.06)-- (0.7161732054529941,2.479411658020722);
\draw (0.7161732054529941,2.479411658020722)-- (-1.162207568263142,2.7257566775244775);

\draw (0.7161732054529941,2.479411658020722)-- (2.4008595320400254,4.196580576388654);

\draw 
 (0.7161732054529941,2.479411658020722)edge[->] (-1.1838311866418052,0.2930996099248926);

\draw (0.7161732054529941,2.479411658020722)edge[->] (2.9864482592175015,-0.7468424031338423);
\end{scope}

\begin{scope}[ line width=3pt,blue]

\draw (0.7161732054529941,2.479411658020722)-- (0.8094459499434268,4.377758721340394);

\draw[ line width=4pt,black, shift={(0.54, 2.5)},scale=0.3] (0.7161732054529941,2.479411658020722)-- (0.8094459499434268,4.377758721340394);

\draw[black] (0.5344819805156678,3.5617366185225294)-- (1.0489306975095394,3.535127202126295);

\draw(0.7161732054529941,2.479411658020722)-- (2.4414901555791584,2.298676320247921);

\draw[ line width=4pt,black, shift={(1.22, 1.65)},scale=0.3] (0.7161732054529941,2.479411658020722)-- (2.4414901555791584,2.298676320247921);

\draw[black] (1.689330652112256,2.767002048821652)-- (1.6325638971336218,2.0006508566100925);

\draw (0.7161732054529941,2.479411658020722)-- (-0.4145872042833716,3.3932103146797097);

\draw [ line width=4pt,black, shift={(0.05, 2.1)},scale=0.3](0.7161732054529941,2.479411658020722)-- (-0.4145872042833716,3.3932103146797097);

\draw[black] (-0.05092518020149673,2.8078031539625465)-- (0.2329085946916739,3.144855761648186);

\draw (-0.5742437026607801,1.9474320238176246)-- (0.7161732054529941,2.479411658020722);

\draw[ line width=4pt,black, shift={(-0.05, 1.5)},scale=0.3] (-0.5742437026607801,1.9474320238176246)-- (0.7161732054529941,2.479411658020722);

\draw[black] (-0.08640440206314305,2.3731826861573797)-- (0.06438229084885384,2.0095206620755057);

\draw (0.7161732054529941,2.479411658020722)-- (0.8804043936667194,0.17347093073531167);

\draw[ line width=4pt,black, shift={(0.56, 0.8)},scale=0.3] (0.7161732054529941,2.479411658020722)-- (0.8804043936667194,0.17347093073531167);

\draw[black] (0.3748254821382592,1.166889142861407)-- (1.1553683630944784,1.228977781119288);

\end{scope}

\begin{scope}[ line width=3pt,green]
\draw (0.8804043936667194,0.17347093073531167)-- (0.9364484013943152,-2.303019436185083);
\draw (0.30544804802662356,-0.9915285056561595)-- (1.5055859750200762,-0.9667833937593874);
\draw (0.9364484013943152,-2.303019436185083)-- (0.9302564191127121,-4.035815749377829);
\draw (0.46500940066983343,-3.191025110626291)-- (1.4689634930992033,-3.166538425445087);
\end{scope}

\node[green] at (-0.01248, -1.01171) {$\frakZ_1$};
\node[green] at (0.01201, -3.20327){$\frakZ_2$};
\node[blue] at (-0.02472, 1.14312){$\frakY_j$};

\node[blue] at (-0.73484, 1.57164) {$\frakY_{j-1}$};

\node[blue] at (2.28927, 1.79202) {$\frakY_{j+1}$};

\draw (1.57094, -1.04844) edge[->,bend right ] node[below right]{$\bG^{Q(\rr)}$}  (1.35877, 0.76357) ;

\draw(1.48524, -3.39103) edge[->,bend right ] node[below right]{$\bG^{Q(\rr)}$} (1.57094, -1.24844) ;
\node[red] at (-0.63919, 0.44871) {$\bR_-$};
\node[red] at (-0.63919, 5.3) {$\bR_-$};

\node[red] at(2.02963, 1.02851) {$\bR_+$};
\node[red] at(2.02963, 5.4) {$\bR_+$};

\node[red,right] at(0.8,6.5) {$\alpha(1/2)$};

\node at (1.2801, -4.50312){$\balpha=\infty$};
\node[above right] at (0.8617, 2.47941){$\bgamma_0$};

\draw[red] (0.71617, 6.45173)edge [->] (0.71617, 6.45173+0.25); 
\draw[red] (0.71617, 6.45173)edge (0.71617, 6.45173+0.5);

\end{scope}
\end{tikzpicture}

\caption{Left: the valuable petal $\bX^0(\bG)$ consists of a periodic cycle of secondary small filled-in Julia sets $\frakY_i$ (blue) and secondary preperiodic small filled-in Julia sets $\frakZ_i$ (green) converging to $\balpha$. The ternary Julia sets are marked black. External rays landing at $\bgamma_0$ and $\balpha$ are marked red. Arrows indicate the direction of $\bR_-,\bR_+$ from $\bgamma_0$ to $\infty$. Right: the quadratic-like map~\eqref{eq:QL domain bO} realizing the secondary renormalization.}
\label{Fig:Val Flower X(G)}
\end{figure}

\subsection{Valuable flowers of prepacmen}
\label{ss:ValFlower:prep}
Consider a pacman $\bG\in \bMM_0$. We denote by $\frakX^0(\bG)$ the non-escaping set (the ``little Julia set'') of the primary renormalization map~\eqref{eq:FRM:bG}. Similarly, let $\frakY_0$ be the non-escaping set of the secondary renormalization map~\eqref{eq:SRM}. We say that $\frakY_0$ is a \emph{secondary small filled-in Julia set} (recall that~\eqref{eq:SRM} can be thickened to a quadratic-like map). Applying $\bG^{Q(\rr)}$ to $\frakY_0$, we obtain a periodic cycle $(\frakY_i)_{i}$ of small filled-in Julia sets, see Figure~\ref{Fig:Val Flower X(G)}. By construction, all $\frakY_i$ are attached to $\bgamma_0$ (which is the $\beta$-fixed point of $\frakY_i$); we enumerate $\frakY_i$ counterclockwise around $\bgamma_0$.

Choose an index $j$ such that $\frakY_j$ contains a unique critical point of the primary renomalization map~\eqref{eq:FRM:bG}. Since~\eqref{eq:FRM:bG} is a branched covering of degree $2$ (a pinched quadratic-like map), the following holds because the combinatorics of lifts for~\eqref{eq:FRM:bG}  is the same as for quadratic polynomials. There are external rays $\bR_-,\bR_+$ landing at $\bgamma_0$ and separating $\frakY_j$ from all remaining $\frakY_i$, see Figure~\ref{Fig:Val Flower X(G)}. Then $\frakY_j$ and $\balpha$ are on the same side of $\overline{\bR_-\cup\bR_+}$: the rays $\overline{\bR_-\cup\bR_+}$ bound a closed topological disk containing $\displaystyle{\bigcup_{i\not=j} \frakY_i}$. Let $\frakZ_1$ be the unique preimage of $\frakY_j$ under~\eqref{eq:FRM:bG} intersecting $\frakY_j$; and similarly, let $\frakZ_{i+1}$ be the unique $\bG^{Q(\rr)}$-preimage of $\frakZ_i$ intersecting $\frakZ_i$.

\begin{lem}
\label{lem:frakZ_i land balpha}
The sets $\frakZ_i$ converge to $\balpha$.
\end{lem}
\begin{proof}
Recall from Corollary~\ref{cor:FRM exists in bVV} that the first renormalization map is written as ${\bG^{Q(\rr)}\colon \bW^1\to \bW}$. Let $\bW'$ be the closure of the connected component of ${\bW\setminus \overline {\bR_-\cup \bR_+}}$ containing all $\frakZ_{i}$. There is a unique connected component $\bW''$ of $\bG^{-Q(\rr)}(\bW')$ such that $\bW''\subset \bW'$. The map 
\begin{equation}
\label{eq:prf:lem:frakZ_i land balpha}
\bG^{Q(\rr)}\colon \bW''\to \bW'
\end{equation}
is univalent, and $\balpha$ is the only point in $\partial \bW'$ that does not escape under the iterations of~\eqref{eq:prf:lem:frakZ_i land balpha}. Therefore, $\balpha$ is the Denjoy--Wolff fixed point of the inverse of~\eqref{eq:prf:lem:frakZ_i land balpha}: under the backward iterations of~\eqref{eq:prf:lem:frakZ_i land balpha} every point in $\bW'$ converges to $\balpha$; the convergence is with respect to the wall topology because $\partial \bW\subset \bR^{-1}\cup  \bR^0$ and the rays $\bR^{-1}, \bR^0$ land at $\balpha$, see Corollary~\ref{cor:FRM exists in bVV}.
\end{proof}

The \emph{valuable petal} $\displaystyle{\bX^0(\bG)\subset \frakX^0(\bG)}$ is the union $\displaystyle{\bigcup_i \frakY_i\cup \bigcup_{i\ge 1}\frakZ_i}$. The \emph{upper part} of $\bX^0(\bG)$ is $\displaystyle\bX^0_\up(\bG)\coloneqq \bigcup_i \frakY_i$. By construction, both $\bX^0(\bG)$ and $\bX^0_\up(\bG)$ are $\bG^{Q(\rr)}$-invariant.  Spreading around $\bX^0(\bG)$, we obtain the \emph{valuable flower}: \[\bX(\bG) \coloneqq {\bigcup_{P<Q(\rr)}\bG^P\big(\bX^0(\bG)\big)}.\]
Similarly, $\bX_\up \subset \bX(\bG)$ is obtained by spreading around $\bX_\up^0$. We enumerate petals of $\bX(\bG)$ from left-to-right as $\bX^i$ so that $\bX^i\subset \bW(i,\bG)$, where $\bW(i)$ are puzzle pieces from Theorem~\ref{thm:parab prepacm}.

Recall that the secondary renormalization~\eqref{eq:SRM} admits a thickening to a quadratic-like map. For every $\bG\in \bMM_0$, the quadratic-like germ of~\eqref{eq:SRM} can be presented as a quadratic-like map 
\begin{equation}
\label{eq:QL domain bO}
\bG^{Q(\rr,\ss)}\colon \bO \to \bO'
\end{equation} 
such that
\begin{itemize}
\item $\bO'\Subset \bW$;
\item $\overline \bO'\setminus \intr \bW_{\rr,\ss}$ is in a small neighborhood of $\bgamma_0$,
\item $\bO'$ depends continuously on $\bG\in \bMM_0$; and
\item the \emph{unbranched condition} holds:
\begin{equation}
\label{eq:unbr cond:bMM}
 \Post(\bG)\cap \bO'\subset \frakY^0,
 \end{equation} 
\end{itemize}
where~$\frakY^0$ is the secondary small filled-in Julia set containing $0$, see~\S\ref{ss:ValFlower:prep}.
(For the unbranched condition, observe that $\Post(\bF)$ is within the cycle of ternary filled-in Julia sets; these sets are disjoint from $\bgamma_0$.)  

The \emph{enlarged valuable flower} is 
\[\widetilde \bX(\bG) \coloneqq \bX(\bG)\cup \bigcup_{P\le Q(\rr,\ss)} \bG^P(\bO).\] 
Since different  sets in $(\bG^P(\bO))_{P\le Q(\rr,\ss)}$ intersect only in a small neighborhood of $\bgamma$, different petals of  $\bX(\bG)$ are in different petals of  $\widetilde \bX(\bG)$. Petals in $\widetilde\bX$ are enumerated as $\widetilde \bX^i, i\in \Z,$ so that $\bX^i\subset\widetilde \bX^i\subset \bW(i)$. Similarly, we set \[\widetilde \bX_\up\coloneqq \bX_\up(\bG)\cup \bigcup_{P\le Q(\rr,\ss)} \bG^P(\bO)\sp \text{ and }\sp \widetilde \bX_\up^i\coloneqq \widetilde \bX_\up\cap \widetilde \bX^i.\]

For $\bG_n\coloneqq \RR^n(\bG)\in \bMM_n$ we define $ \bX(\bG_n)$ and  $\widetilde \bX(\bG_n)$ to be the $A_\str^{-n}$-images of $\bX(\bG)$ and  $\widetilde \bX(\bG)$. Since $\bO'(\bG_n)$ and $\bO(\bG_n)$ are rescalings of $\bO'(\bG)$ and $\bO(\bG)$, there is an $\varepsilon>0$ such that 
\[\mod\big( \bO'(\bG_n) \setminus \bO(\bG_n) \big)\ge \varepsilon>0\] for all $n$ and $\bG_n\in \bMM_n$. The upper parts $\bX_\up(\bF) $ and $\widetilde \bX_\up (\bF)$ are defined accordingly.

\subsection{Valuable flowers of pacmen} 
\label{ss:ValFlow:pacman}
Consider a pacman $f\in \BB$ from a Banach neighborhood of $f_\str$ where the pacman renormalization $\RR\colon \BB\dashrightarrow \BB$ is defined, see~\S\ref{sss:HypSelfOper}.  By a \emph{flower} we mean a connected set $T\ni \alpha$ such that $T\setminus \{\alpha\}$ has finitely many connected components, called \emph{petals}.  We say that the flower $T$ is \emph{nice} if $f$ has a Siegel triangulation $\bDelta$ (see~\S\ref{sss:SiegTriang}) with a wall approximating $\partial Z_{f_\str}$ such that different petals of $T$ are in different triangles of $\bDelta$. As a consequence if $f=\RR f_{-1}$ (or more generally, $f=\RR_{\Sieg} f_{-1}$ for an operator $\RR_\Sieg$ as in~\S\ref{sss:from SM to SP}), then the flower $T$ admits a full lift $T_{-1}$ to the dynamical plane of $f_{-1}$, see Lemma~\ref{lem:SiegTriangLifting}.

\begin{figure}
\begin{tikzpicture}
\node at (0,0){\includegraphics[scale=0.6]{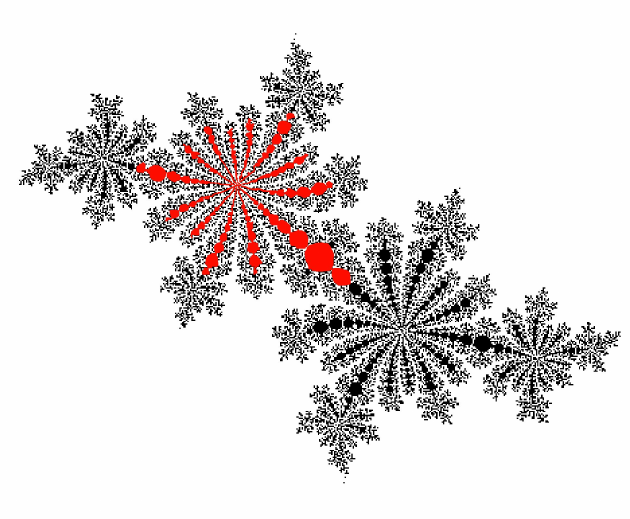}};
\node at (6,0){\includegraphics[scale=0.725]{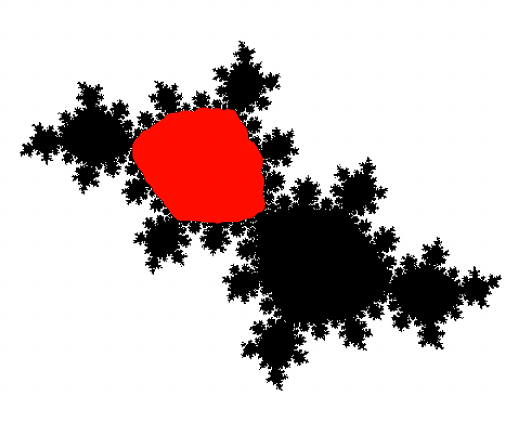}};

\end{tikzpicture}

\caption{Illustration to Theorem~\ref{thm:proj of val flow}: the valuable flower (red) of the $5/13$ Rabbit tuned with the Basilica approximates the golden Siegel disk (also red).}
\label{Fig:thm:proj of val flow}
\end{figure}

\begin{thm}
\label{thm:proj of val flow}
Consider the dynamical plane of $f_\str$ and fix a small open neighborhood $N_\str$ of $\overline Z_\str$. For $n\ll 0$ and every $\bG_n\in \bMM_n$ the flowers $\widetilde \bX(\bG_n)$ and $\bX(\bG_n)$ projects to the dynamical plane of $g_n$. Moreover, these projections $\widetilde X(g_n)$ and $X(g_n)$ are nice flowers within $N_\str$.

More precisely, $\bG_n$ has a fundamental domain $\bS^\new$ such that if a petal $\widetilde \bX^i$ intersects $\bS^\new$, then $\widetilde \bX^i\subset \bS^\new$. The projection $\widetilde X$ of $\widetilde \bX\cap \bS^\new$ is within the Siegel triangulation $\bDelta(g_n)$ of $g_n$, and the projection $\widetilde X_\up$ of $\widetilde \bX_\up$ is within the wall $\bPi(g_n)$ of $\bDelta(g_n)$. The triangulation $\bDelta(g_n)$ has a full lift to $\bDelta(g_m)$ for $m\le n$; for big $m\ll n<0$, the triangulation $\bDelta(g_m)$ approximates $\overline Z_\str$. 
\end{thm}
\begin{rem}The valuable flower $X(g_n)$ is a forward invariant set containing the postcritical set of $g_n$. The purpose of $X(g_n)$ is to encode the hybrid class of $g_n$, i.e.~the combinatorial rotation numbers of dividing periodic cycles, as well as hybrid classes of secondary small filled-in Julia sets. 
\end{rem}
\begin{proof}
Let us first give an outline of the proof. We also slightly simplify the notations in the outline. In the actual proof, $\bZ_\str$ is replaced by the renormalization triangulation $\bDelta_{-n}(\bG_n)$.

In the dynamical plane of $\bG\in \bMM_0$, we re-enumerate the puzzle pieces $\bW(i,\bG), i\in \Z$, from Theorem~\ref{thm:parab prepacm} by the ``landing time.'' Namely, for every $\bW(i)$ there is a unique minimal $P(i)\in \PT$ such that $\bG^{P(i)}\big(\bW(i) \big)\supset \bW(0)$, see details in \S\ref{sss:Decomp bW_P}. We write 
\begin{equation}
\label{eq:P(i)}
\bW_{P(i)} \coloneqq \bW(i)\sp\sp \text{ and }\sp\sp i\big( P(i)\big)\coloneqq i.
\end{equation}

For $n\in \Z$, we have $\bW_P (\bG_n)= A_\str ^{-n}\bW_{\tt^n P}(\bG)$ by~\eqref{eq:ren:F:F_n}. Recall from~\S\ref{ss:wakes} that $\bW_P(\bF_\str)$ denotes the primary wake of generation $P$. In Lemma~\ref{lem:bW:bG to bFstr}, we will show that $\bW_P(\bG_n)\setminus \bZ_\str$ converges to $\bW_P(\bF_\str)$ for every fixed $P>0$. Combing with Lemma~\ref{lem:wake shrinks}, we obtain that if $P$ is sufficiently big, then $\bW_P(\bG_n)\setminus \bZ_\str$ is small in the spherical metric of $\widehat \C$.

The main step is to show that $\widetilde \bX^i\setminus \bZ_\str$ is uniformly small. If $P(i)$ is big, then $\widetilde \bX^i\setminus \bZ_\str$ is small because $\bW_{P(i)}\setminus \bZ_\str$ is small. If $P(i)$ is bounded but $\bW_{P(i)}$ is far from $0$ in $\C$, then $\widetilde \bX^i\subset \bW_{P(i)}$ is small in the spherical metric of $\widehat \C$ (in fact, this case can be ignored). There are finitely many $i$ in the remaining case; i.e.~when $P(i)$ is bounded and $\bW_{P(i)}$ is not far from $0$ in $\C$. Let us fix such $i$. For $n\ll 0$, we have $\widetilde \bX^i_\up(\bG_n)  \ni c_{P(i)}(\bG_n)$, where $c_{P(i)}(\bF_\str)\in \partial \bZ_\str$ is the unique critical point of $\bF_\str^{\ge 0}$ in $\partial \bZ_\str$ of generation $P(i)$. We have a map \[\bG_n^{P(i)} \colon  \widetilde\bX^i_\up\to  \widetilde\bX^0_\up\ni 0 .\]
Since $\widetilde \bX^0_\up(\bG_n)=A_\str^{-n} \big(\widetilde\bX^0_\up(\bG_0)\big)$ shrinks to $0$ as $n\to -\infty$, we obtain that $\widetilde \bX^i_\up(\bG_n)$ shrinks to $c_{P(i)}$. This allows us to deduce that  $\widetilde \bX^i\setminus \bZ_\str$ is small because the remaining part of $\widetilde \bX^i$ is ``below'' $\widetilde \bX^i_\up$, compare with Figure~\ref{fig:Decomp of bW_P}.

Recall that the parabolic prepacman $\bF_\rr$ is the root of the limb containing $\bMM_0$, see Figure~\ref{Fig:Rec Small copy}. For $n\le 0$, the parabolic pacman 
\begin{equation}
\label{eq:f rr_n}
f_{\rr_n}\coloneqq \RR^n f_\rr
\end{equation}
 has rotation number $\rr_n=\pp_n/\qq_n\in\Q$ satisfying $\cRRc^{-n \mm }(\rr_n)=\rr$, see Lemma~\ref{lem:RR:rot numb act}. Then $\bF_{\rr_n}$ is the root of the limb containing $\bMM_n$ and $\rr_n$ is the combinatorial rotation number of $g_n$. Since  $\widetilde \bX^i\setminus \bZ_\str$ is small for every $i$, we can adjust the fundamental domain $\bS(\bG_n)$ (see~\S\ref{ss:fund domain for bF}) so that the new $\bS^\new(\bG_n)$ contains the petals $\widetilde \bX^i$ for $i\in \{-\pp_n+1,\pp_n+2,\dots , \qq_n-\pp_n\}$. By Theorem~\ref{thm:quot of fund domain}, $\widetilde \bX$ projects to the dynamical plane of $g_n$ and $\bS^\new$ projects into $\bDelta_\Sieg\coloneqq \bDelta_0^\new(g_n)$; applying antirenormalizations, we can assume that $\bDelta_\Sieg$ approximates $\overline Z_\str$. 
 
 Let us now provide the details.
\subsubsection{The escaping set} Up to replacing $\bMM_0$ with its antirenormalization, we can assume that ${\bMM_0\subset \UnstLoc}$. In particular, every $\bG\in \bMM_0$ has a renormalization triangulation $\bDelta_0(\bG)$ with the wall $\bPi_0(\bG)$ that bounds $\bQ_0=\bDelta_0\setminus \bPi_0$, see~\S\ref{ss:MP:walls}.

Fix a big $T\in \PT$. For $n\ll0$ sufficiently big,  the holomorphic motion $\tau$ from  Lemma~\ref{lem:esc set moves hol} induces an equivariant map 
\begin{equation}
\label{eq:h_n:close to 1}
h_n \colon \Esc_T(\bG_n)\to \Esc_T(\bF_\str).
\end{equation}

Applying the $\lambda$-lemma, we obtain:
\begin{lem}
\label{lem:Esc bG_n is close to Fstr}
For $n\ll 0$ sufficiently big (depending on $T$), \eqref{eq:h_n:close to 1} is close to the identity with respect to the spherical distance.\qed
\end{lem}

\subsubsection{Decomposition $\bW_P=\bW^a_P\cup \bW_P^1\cup \bW^b_P$} 
\label{sss:Decomp bW_P}
Recall from Theorem~\ref{thm:parab prepacm} that $\bW(i,\bF_\rr)$ denotes the puzzle piece bounded by $\bR^{i-1}$ and $\bR^i$. Since $\bMM_0\subset \bVV$ (see Figure~\ref{Fig:Rec Small copy}), the puzzle pieces $\bW(i, \bG)$ exist (in the sense of~\S\ref{ss:puzzles}) in the dynamical plane of $\bG\in \bMM_0$. Let us first discuss the combinatorics of $\bW(i,\bG)$. For $i\not= 0$, the \emph{generation of $\bW(i,\bG_0)$} is the unique ``landing time'' $P(i)<Q(\rr)$ such that 
\begin{equation}
\label{eq:bWi to bW0}
\bG^{Q(\rr)-P(i)}\colon \bW(0)\to \bW(i)
\end{equation}
 is a univalent map.  Let us re-label the puzzle pieces by their landing time:
   \[\bW_{P(i)}(\bG)\coloneqq \bW(i,\bG)\sp\sp \text{ and }\sp\sp \bW_0(\bG)\coloneqq \bW(0,\bG).\]
Then~\eqref{eq:bWi to bW0} takes the form 
\begin{equation}
\label{eq:bW_0 to bW_P}
\bG^{Q(\rr)-P}\colon \bW_0\to \bW_P.
\end{equation}

Recall from Lemma~\ref{lem:decomp:bW(bF)}  that $\bW=\bW^1\cup \bW^a \cup \bW^b$.  Let $\bW^1_P$, $\bW^a_P$, and $\bW^b_P$ be the images of $\bW^1, \bW^a ,$ and $\bW^b$  under~\eqref{eq:bW_0 to bW_P} respectively. We have a two-to-one map 
\begin{equation}
\label{eq:bW_P to bW_0}
\bG^P\colon \bW^1_P\to \bW_0,
\end{equation}
 while $\bG^P$ maps univalently $\intr \bW^a_P$ and $\intr \bW^b_P$ to two different components of $\C\setminus (\bW\cup \bR^{0})$, see Figure~\ref{fig:Decomp of bW_P}.

\begin{figure}
\begin{tikzpicture}

\begin{scope}[shift={(-8,0)}]

\draw[white, fill=black, opacity =0.1]
(0.86,0.3)
.. controls  (0.9,-0.1 )and  (0.93,-0.6 )..
(0.95,-1) 
.. controls  (0,-1 )and  (0,-1 )..
(-0.95,-1)
.. controls   (-0.93,-0.6 ) and (-0.9,-0.1 )..
(-0.86,0.3) ;

\draw (0.86,0.3) -- (-0.86,0.3)
(-0.95,-1)--(0.95,-1);

\coordinate (b1) at (2,1.5);
\coordinate (c1) at (-2,1.5);

\draw[] (-0.915,-0.2 ).. 
controls (-2.5,0.5)  and (-4,2.8) ..
(0,3)
..controls (4,2.8)  and (2.5,0.5) ..
(0.915,-0.2 );

\draw[red] (-2,-0.35)--(2,-0.35)
 (-2,-0.45)--(2,-0.45);

 \draw(0,4)--(0,3);

\draw (0,3) .. controls  (1,2 )and (1,-1)..
(1,-4);
\draw (0,3) .. controls  (-1,2 )and (-1,-1)..
(-1,-4);

\node[above] (a1) at (0,-0.35){};

\node at(0,1.4) {$\bW_P^+$};

\node at(-1.5,1) {$\bW_P^a$};
\node at(1.5,1) {$\bW_P^b$};

\node[above] at (0,-1){$\bW_P^\bullet$};

\node at(0,-3) {$\bW_P^-$};

\end{scope}

\begin{scope}
\draw(0,4)--(0,0);

\draw[white, fill=black, opacity =0.1] (0,0) .. controls  (0.5,-0.3 )and  (0.6,-0.7 )..
 (0.7,-1)
.. controls  (0.2,-1 )and  (-0.2,-1 )..
 (-0.7,-1)
 .. controls   (-0.6,-0.7 ) and (-0.5,-0.3 ).. 
 (0,0);

\draw  (0.7,-1)-- (-0.7,-1);

\draw[red] (-2,-0.35)--(2,-0.35)
 (-2,-0.45)--(2,-0.45);

\draw (0,0) .. controls  (1,-0.5 )and (1,-3)..
(1,-4);
\draw (0,0) .. controls  (-1,-0.5 )and (-1,-3)..
(-1,-4);

\node[above] (a2)at (0,-1){$\bW_0^\bullet$};

\node at (0,-3){$\bW_0^-$};

\draw[blue] (a1) edge[->,bend left] node[below]{$2:1$} (-0.1,-0.3);

\draw[blue] (b1) edge[->,bend left] (-0.5,1);
\draw[blue] (c1) edge[->,bend left] (1,1.5);

\node[below,red] at (2,0.15){$\bPi_0$};

\node[] at (0.4,3){$\bR^\bullet$};
\end{scope}

\end{tikzpicture}
\caption{The decomposition of $\bW_P$, see also Figure~\ref{fig:Decomp of bW}. The regions $\bW_P^\bullet$ and $\bW_P^-\cup \bW_P^+$ are the preimages of $\bW_0^\bullet$ and $\bW_0^-$ under $\bG^P\colon \bW^1_P\to \bW_0$ respectively. The region $\intr\big(\bW_P^a\cup \bW_P^b\big)$ is the preimages of $\C\setminus \big(\bW_0\cup \bR^\bullet\big)$ under $\bG^P\colon \bW_P\to \C$. We define $\bW_P^\out\coloneqq \bW^\bullet_P\cup \bW^+_P\cup \bW^a _P\cup \bW^b_P$.}
\label{fig:Decomp of bW_P}
\end{figure}

\subsubsection{Decomposition $\bW^1_P=\bW^+_P\cup \bW_P^\bullet \cup \bW^-$} 
We still consider the dynamical plane of $\bG\in \bMM_0$. Choose an auxiliary univalent $2$-wall $\bA\subset \bQ$, see~\S\ref{sss:walls bPi} and~\S\ref{ss:MP:walls}. We denote by $\bQ'$ the connected component of $\bQ\setminus \bA$ attached to $\balpha$.  Let us fix a closed topological disk $\bW^\bullet_0 $ in $ \bW_0$ such that
\begin{itemize}
\item[(A)] $\bW_0 \setminus \bQ'\subset  \bW^\bullet_0$;
\item[(B)] $\displaystyle \bW^\bullet_0 \supset \bX^0_\up $ (recall that $\bX^0_\up =\bigcup_i \frakY_i$).
\end{itemize}
We set $\bW^-_0\coloneqq \bW_0\setminus \bW^\bullet_0$.

Since $\bW^\bullet_0$ contains the unique critical value of~\eqref{eq:bW_P to bW_0}, the preimage of $\bW^\bullet_0$ under~\eqref{eq:bW_P to bW_0} consists of a single connected component, call it  $\bW^\bullet_P$. On the other hand, the preimage of $\bW^-_0$ under~\eqref{eq:bW_P to bW_0}  consists of two connected components, we denote them by $\bW_P^-$ and $\bW_P^+$  specified so that $\bW_P^-$ is attached to $\balpha$, see Figure~\ref{fig:Decomp of bW_P}.

Finally, we define $\bW^\out_0\coloneqq \bW^\bullet _0$ and \[\bW_P^\out\coloneqq \bW^\bullet_P\cup \bW^+_P\cup \bW^a _P\cup \bW^b_P\sp\sp \text{ for }P>0.\]

For $n\in \Z$ and $P<Q(\rr)$, we define:
\begin{equation}
\label{eq:bW^hash _P}
 \bW_P(\bG_n) \coloneqq \bW_{ P} (\bG_n)\coloneqq A_\str ^{-n}\bW_{\tt^{n} P}( \bG).
\end{equation}
 The regions $\bW_P^1(\bG_n),\bW_P^\bullet(\bG_n), \bW_P^a(\bG_n), \bW_P^b(\bG_n), \bW_P^-(\bG_n),\bW_P^+(\bG_n), \bW_P^\out(\bG_n)$ are defined accordingly using~\eqref{eq:bW^hash _P}.

As in~\S\ref{ss:MP:walls}, $\bQ'_{-n}(\bG_n)$ is the rescaling of $\bQ'(\bG)$. It follows from Condition~(A) that
\[\bW_P(\bG_n) \setminus \bQ'_{-n}\subset  \bW^\out_P(\bG_n) \sp\sp\text{ and } \sp\sp \bW^-_P(\bG_n) \subset  \bQ'_{-n}(\bG_n).\]
Condition (B) implies 
\begin{equation}
\label{eq:bX^i is in bWbullet}
\widetilde \bX^i(\bG_n) \subset \bW_{P(i)}^\bullet\cup \bW_{P(i)}^- (\bG_n)\sp\sp \text{ and }\sp \sp   \widetilde\bX^i_\up(\bG_n) \subset \bW_{P(i)}^\bullet(\bG_n).
\end{equation}

\subsubsection{$\bR^\bullet(\bG_n)$ converges to $\bR^\str(\bF_\str)$ }
We say that arcs $\beta,\gamma\subset \wC$ are \emph{$C^0$-close} if, up to re-parameterization, the functions $\beta, \gamma\colon[0,1]\to \C$ are close with respect to the spherical metric of $\wC$. The $C^0$-closeness for closed curves is defined in the same way. 

Two topological closed disks $D_1,D_2\subset \wC$ are \emph{$C^0$-close} if $\partial D_1,\partial D_2$ are $C^0$-close closed curves. Equivalently, viewing $D_1,D_2$ as injective functions ${D_1,D_2\colon \ovDisk\to \wC}$, the disks $D_1,D_2$ are $C^0$-close if, up to re-parameterization, the corresponding functions are close with respect to the spherical metric of $\wC$.

Let us write:
\[\bR^{\bullet}(\bG_0)\coloneqq \bR^{-1}\cap \bR^0(\bG_0)\sp \text{ and }
\sp \bR^{\bullet}(\bG_n)\coloneqq A_\str^{-n} \bR^{\bullet}(\bG_0).\]

Recall that $\bR^\str(\bF_\str)$ is the ray landing at $0$, see~\S\ref{ss:ExtRayFstr}.
\begin{lem}
If $n\ll0$ is sufficiently big, then $\bR^\bullet(\bG_n)$ is close to $\bR^\str(\bF_\str)$ with respect to the $C^0$-distance.
\end{lem}
\begin{proof}
Fix a small $\varepsilon>0$. We will show that $\bR^\bullet(\bG_n)$ is $\varepsilon$-close to $\bR^\str(\bF_\str)$ for $n\ll0$.

Recall from~\eqref{eq:bRstr:decomp} that $\bR^\str(\bF_\str)$ decomposes as a concatenation $\displaystyle \bigcup_{j\in \Z} \bI_j$, where $\bI_j$ is a ray segment in the closure of $\Esc_{\tt^j P}\setminus \Esc_{\tt^{j-1} P}(\bF_\str)$ such that $A_\str \bI_j =\bI_{j+1}$. 

By Corollary~\ref{cor:rays meet}, every two rays eventually meet at an alpha-point. Suppose that $\bR^\bullet (\bG_n)$ meets $\bR^\str(\bG_n)$ at $\alpha_n$, where  $\bR^\str(\bG_n)$ is the counterpart of $\bR^\str(\bG_n)$. Then $\alpha_n=A_\str\alpha_{n+1} $ and the $\alpha_n$ tends to $0$ as $n\to -\infty$.

Choose first a big $k\gg 0$, then a sufficiently big $T>0$. For a sufficiently big $n\ll 0$, we can decompose 
\[\bR^\bullet(\bG_n)= \bL^n\cup \bigcup_{j\le k} \bI_j (\bG_n)\]
such that $\displaystyle  \bigcup_{j\le k} \bI_k \subset \Esc_T (\bG_n)$. By Lemma~\ref{lem:Esc bG_n is close to Fstr}, $ \bI_j (\bG_n)$ is $\varepsilon$-close to $\bI_j (\bF_\str)$ with respect to the spherical metric for all $j\le k$. 

Applying $A_\str$, we obtain the decomposition  
\[\bR^\bullet(\bG_{n-1})= \bL^{n-1}\cup \bigcup_{j\le k+1} \bI_j (\bG_{n-1})\]
where \[\bL^{n-1}(\bG_{n-1})=A_\str \bL^n (\bG_n)\sp \text { and }\sp \bI_{k+1}(\bG_{n-1})=A_\str \bI_k(\bG_n).\]
By Lemma~\ref{lem:Esc bG_n is close to Fstr}, for $j\le k$ the ray segment $\bI_j(\bG_{n-1})$ is $\varepsilon$-close to $\bI_j(\bF_\str).$
 Recall that $\bI_k(\bG_n)$ is $\varepsilon$-close to $\bI_k(\bF_\str)$ which is close to $0$ because $k\gg0$. Since $A_\str$ contracts the spherical metric in a neighborhood of $0$, we see that  $\bI_{k+1}(\bG_{n-1})$ is $\varepsilon$-close to $\bI_{k+1}(\bF_{\str})$. 
 
Continuing the process, we obtain the decomposition  
\[\bR^\bullet(\bG_{m})= \bL^{m}\cup \bigcup_{j\le k + n-m} \bI_j (\bG_{m})\]
for $m\le n$, where $\displaystyle \bigcup_{j\le k + n-m} \bI_j (\bG_{m})$ is $\varepsilon$-close to $\displaystyle\bigcup_{j\le k + n-m} \bI_j (\bF_\str)$. Since $\bL^{m}(\bG_m)=A_\str^{n-m} \bL^{n}(\bG_n)$, the chain $\bL^{m}(\bG_m)$ is eventually in a small neighborhood of $0$, and the claim follows.
\end{proof}

\subsubsection{$\bW^\out_P (\bG_n)$ approximates $\bW_P(\bF_\str)$}
Clearly, $\bW^\bullet_0(\bG_n)=A_\str^{-n}\bW_0^\bullet(\bG_0)$ shrinks to $0$  as $n\to -\infty$.

\begin{lem} 
\label{lem:bW:bG to bFstr}
For every $\varepsilon>0$ and $P\in \PT$ the following holds. If $n\ll 0$ is sufficiently big, then $\bW^\out_S(\bG_n)$ is $\varepsilon$-close to $\bW_S(\bF_\str)$ with respect to $C^0$-metric for every $S\le P$. Moreover, $\bW_S^\bullet$ is in the $\varepsilon$-neighborhood of $c_s(\bF_\str)$.
\end{lem}
\begin{proof}
By Lemma~\ref{lem:discr of dyn}, we can choose a sufficiently small disk $D_0\coloneqq \ovDisk(\zeta)$ around $0$ and a sufficiently small neighborhood $\bOO$ of $\bF_\str$ such that $D_0$ is disjoint from $\CV(\bG^S)\cup \bG^S(D_0)$ for all $S\le P$. In particular, 
\[ \widetilde D(\bG)\coloneqq \bigcup_{S\le P} \bG^{-S}(D_0) \]
depends holomorphically on $\bG\in \bOO$ and every connected component of $\widetilde D(\bG)$ is a degree two preimage of $D_0$.

For $S\le P$, let $D_S(\bF_\str)$ be the lift of $D$ along $\bF_\str^S\colon c_S\mapsto 0$. For $\bF\in \bOO$, we define $D_S(\bF)$ to be the lift of $D_0$ under $\bF^S$ such that  $D_S(\bF)$ depends holomorphically on $\bF$. Let $c_S(\bF)$ be the unique preimage of $0$ under $\bF^S\colon D_S\to D_0$.
We claim that for a sufficiently big $n\ll 0$, the point $c_S(\bG_n)$ is the unique critical point of $\bG^S_n\colon \bW^1_S\to \bW_0$ for all $S\le P$. Indeed, Lemma~\ref{lem:c_S is correct} asserts such statement for $\bF_\rr$ and for sufficiently small $S$; perturbing $\bF_\rr$ to $\bG\in \bMM_0$ and passing to the antirenormalization, we obtain the required claim.

Choose a sufficiently big $T\gg P$ such that \[\bR^\bullet (\bG_n) \setminus  \Esc_{T-P}(\bG_n) \subset D_0\]
for all $n\ll 0$.

Consider $\bW^\out_S(\bG_n)$ for $S\le P$. We can decompose $\partial \bW^\out_S = \beta_1\cup \beta_2$, where $\beta_1$ and $\beta_2$ are simple arcs satisfying \[\beta_2\coloneqq \partial \bW^\out_S\cap \Esc_T(\bG_n) \sp \text{ and }\sp \beta_1\subset D_S(\bG_n).\] By Lemma~\ref{lem:Esc bG_n is close to Fstr}, $\beta_2(\bG_n)$ is close to \[\beta_2(\bF_\str)=h_n \big(\beta_2(\bG_n) \big)\subset \bW_S(\bF_\str).\] Since \[\beta_1(\bF_\str)\coloneqq \partial \bW_S(\bF_\str)\setminus \beta_2(\bF_\str)\subset D_S(\bF_\str) \] and since $D_S(\bG_n)$ is close to $D_S(\bF_\str)$, we obtain that $\bW^\out_S(\bG_n)$ is close to $\bW_S(\bF_\str)$.
\end{proof}

Combining with Corollary~\ref{cor:wakes are close to bJ}, we obtain:
 
\begin{cor}
\label{cor:contr bW'}
Fix a small $\varepsilon>0$ and then a sufficiently big $n\ll 0$. Then the set $  \bW^\bullet_S(\bG_n)\setminus \bQ'_{-n}(\bG_n)$ is within the spherical  $\varepsilon$-neighborhood of $c_S(\bF_\str)$ for every $S\in \PT$. 
\end{cor}
\noindent Since  $\widetilde \bX^{i}\setminus \bQ'_{-n}(\bG_n)\subset  \bW^\bullet_S(\bG_n)\setminus \bQ'_{-n}(\bG_n)$ (see~\eqref{eq:bX^i is in bWbullet}), the set  $\widetilde \bX^{i}\setminus \bQ'_{-n}(\bG_n)$ is also in a small neighborhood of $c_S(\bF_\str)$.
\begin{proof}
Choose a sufficiently big $P\gg0$. By Lemma~\ref{lem:bW:bG to bFstr},  $\bW^\bullet_S(\bG_n)\setminus \bQ'_{-n}(\bG_n)$ is within a small neighborhood of $c_S(\bF_\str)$ for every $S\le P$ and every sufficiently big $n\ll0$.

By Corollary~\ref{cor:wakes are close to bJ} and Lemma~\ref{lem:bW:bG to bFstr}, every connected component of 
\begin{equation}
\label{eq:gaps between bWout}
\C\setminus \left( \bQ'_{-n}(\bG_n) \cup \bigcup_{S\le P} \bW^\out_S(\bG_n) \right)
\end{equation} is $\varepsilon$-small. Note that for $S>P$ and $T_1,T_2\le S$, the puzzle $\bW_S(\bG_n)$ is between $\bW_{T_1}(\bG_n)$ and  $\bW_{T_2}(\bG_n)$ with respect to the left-right order if and only if $c_S(\bF_\str)$ is between $c_{T_1}$  and $c_{T_2}$ with respect to $\partial \bZ_\str$. Therefore, for $S>P$, the sets  ${\bW_S(\bG_n)\setminus \bQ'_{-n}}$ and $\bW_S(\bF_\str)$ are in the closures of  $\varepsilon/2$-close components of ~\eqref{eq:gaps between bWout} and~\eqref{eq:gaps between bW:bFstr} respectively. This proves the corollary for $S> P$.
 \end{proof}

As a byproduct, we also obtain:
\begin{lem}
\label{lem:bW is tiling}
For every $\bG\in \bMM_0$, the puzzle pieces $\big(\bW(i,\bG)\big)_{i\in \Z}$ form a partition of $\C$: every $z\in \C$ belongs to some $\bW(i)$.
\end{lem}
\begin{proof} 
Recall that the rays $\bR^i$ land at $\balpha$, and the union $\bigcup_{i\in \Z}\bR^i$ is a tree in $\C$ (follows from Lemma~\ref{lem:over R is graph} and Theorem~\ref{thm:parab prepacm}). For $i<j$, let $\bW_{i,j}$ be the unique component of $\C\setminus \big(\bR^i\cup \bR^j\big)$ attached to $\balpha$ in the following sense: $\bW_{i,j}$ without a small neighborhood of $\balpha$ is precompact in $\C$. Since \[\overline {\bW_{i,j}}= \bigcup_{i<k\le j} \bW(k),\]
it is sufficient to show that $\displaystyle{\bigcup _{i<j}\bW_{i,j}= \C}$.

The case $z\in \bQ$ is straightforward. For $z\not \in\bQ$, we can surround $A_\str^{-n} z$ by $\bQ'_{-n} \cup \bW^\out _S\cup \bW_T^\out(\bG_n)$, where  $S,T$ are certain fixed power-triples and $n\ll0$.
\end{proof}

\begin{figure}[t!]
\[\begin{tikzpicture}

\draw (-6,0)--(5,0);

\node[above] at (-5.5,0) {$\partial \bZ_\str$};

\draw[red] (-3,-2)--(-3,3)--(3,3)--(3,-2);
\node[below,red] at (-2,2.5) {$\bS$};

\node[below] at (0,0) {$0$}; 
\filldraw (0,0) circle (1.5pt);

\node[below left] at (-3,0) {$\bF^L_\str(0)$};
\filldraw (-3,0) circle (1.5pt);

\node[below left] at (3,0) {$\bF^R_\str(0)$};
\filldraw (3,0) circle (1.5pt);

\node[above , blue] at (-2.5,0) {$c_A$};
\filldraw[blue] (-2.5,0) circle (1.5pt);

\node[above , blue] at (3.5,0) {$c_B$};
\filldraw[blue] (3.5,0) circle (1.5pt);

\end{tikzpicture}\]
\caption{The critical points $c_A$ and $c_B$ are near $\bF^R_\str(0)$ and $\bF^L_\str(0)$ respectively.}
\label{Fg:c_A c_B}
\end{figure}

\subsubsection{Fundamental domain}
Using Corollary~\ref{cor:contr bW'}, we can now construct a fundamental domain $\bS^\new$ as required.  Write $L=(0,1,0)$, $R=(0,0,1)$, and note that $\bS(\bF_\str)\cap \partial \bZ_\str$ is the arc $J\coloneqq \big[\bF_\str^L(0), \bF_\str^R(0)\big]\subset \partial \bZ_\str.$ Choose $c_A$ and $c_B$ close to $\bF_\str^L(0)$ and $\bF_\str^R(0)$ respectively such that $A+R=B+L$, see Figure~\ref{Fg:c_A c_B}. We can assume that $c_A\in J$ while $c_B\not\in J$. 

Since $\widetilde \bX^i\setminus \bQ'_{-n}(\bG_n)$ is in a small neighborhood of $c_{P(i)} (\bF_\str)$ (using notations from~\eqref{eq:P(i)}), we can adjust $\bS(\bG_n)$ such that the new fundamental domain $\bS^\new(\bG_n)$ (see~\S\ref{ss:fund domain for bF}) contains $\widetilde \bX^i$ for all $i\in \{i(A),i(A)+1,\dots , i(B)-1\}\eqqcolon I$ and $\bS^\new(\bG_n)$ is disjoint from $\widetilde \bX^i$ for all $i\not\in I$.  (To satisfy Conditions~\eqref{cond:2:fund domain} and~\eqref{cond:3:fund domain} in ~\S\ref{ss:fund domain for bF}, we can assume that $\lambda(\bS^\new)$ follows the ray $\bR^{i(A)}$ and  $\rho(\bS^\new)$ follows the ray  $\bR^{i(B)-1}$.) We also have \[\widetilde \bX^i_\up(\bG_n) \subset \bigcup_{P} \bW^\bullet_P (\bG_n)\subset \bPi_0(\bG_n). \]
 \end{proof}

\section{Proof of the main results}
\label{s:proof:main thms}
By Theorem~\ref{thm:small copy M1}, the unstable manifold $\Unst$ of $\bF_\str$ contains a sequence $(\bMM_n)_{n\le 0}$ of copies of the Mandelbrot sets such that $ \RR \bMM_{n-1}=\bMM_n$. Every $\bMM_n$ naturally corresponds to a small copy $\Mandel_n\subset \Mandel$, see~\eqref{eq:chi_Pacm:comb}. We will show that $(\Mandel_n)_{n\le m}$ for $m\ll 0$ is a sequence satisfying Theorem~\ref{thm:main}.

\subsection{Stable lamination} 
\label{ss:StabLamin}
For $n\ll0$, we define $\MM_n$ to be the set of pacmen $g_n\in \WW^u$ with $\bG_n\in \bMM_n$. By Theorem~\ref{thm:proj of val flow}, every $g_n\in \MM_n$ has a nice valuable flower $X(g_n)$ and a nice extended valuable flower $\widetilde X(g_n)$ in a small neighborhood of $\overline Z_\str$. Since the flowers are nice, $\widetilde X(g_{n-1})$ is a full lift of $\widetilde X(g_{n})$. 
 
For $g_n\in \MM_n$, we denote by $\frakY^0(g_n)$ and $O(g_n)$ the projections of $\frakY^0(\bG_n)$ and $\bO(\bG_n)$, see~\S\ref{ss:ValFlower:prep}. Then $\frakY^0(g_n)$ is the non-escaping set of the quadratic-like map
\begin{equation}
\label{eq:q-l:g}
g^{\aa_n}_n\colon O\to g_n^{\aa_n}(O) = O',
\end{equation} 
where $\rr_n= \pp_n/\qq_{n}$ is the combinatorial rotation number of $\alpha(g)$ and $\aa_n\coloneqq \qq_{n}\qq_{\ss}$. By construction, $g_n^i(O)\subset \widetilde X(g)$ for $i\le \aa_n$. Since $\widetilde X(g_n)$ is the projection of $\widetilde \bX(\bG_n)$, the unbranched condition~\eqref{eq:unbr cond:bMM}  implies the unbranched condition for~\eqref{eq:q-l:g}: 
\begin{equation}
\label{eq:unbr cond:MM} 
\Post(g_n)\cap O'\subset \frakY^0(g_n). 
\end{equation}
By Lemma~\ref{lem:RR:rot numb act}, 
\begin{equation}
\label{eq:rr_n:charact}
\cRRc^\mm (\rr_{n-1}) =\rr_n.
\end{equation}

Recall from \S\ref{sss:not:SatCopies} that $\Mandel_{\rr(n),\ss,\kk}$ denotes the ternary satellite copy of $\Mandel$ with rotation parameters $\rr(n),\ss,\kk$. Let us consider the canonical homeomorphism  
\[\strai_\Pacm\colon \MM_n\to \Mandel_{\rr(n),\ss,\kk}\eqqcolon \Mandel_n\] between two copies of the Mandelbrot sets. Then
\begin{equation}
\label{eq:chi_Pacm:comb}
\mRRc^\mm \circ \strai_{\Pacm}(g)=\strai_{\Pacm} \circ \RR(g),\sp\sp g\in \bigcup_{n\ll0 } \MM_n,
\end{equation}
where $\mRRc$ is the molecule map, see~\S\ref{ss:MolecMap}. We write $\mRR\coloneqq \mRRc^\mm$; then $ \strai_{\Pacm}$ conjugates $\RR $ to $\mRR$.

\begin{rem}
\label{rem: pacman str map}
We believe that $\chi_\Pacm$ on $\displaystyle{\bigcup_{n\ll 0} \MM_n}$ can be extended to a ``pacman straightening map'' defined on the connectedness locus of the space of pacmen. Such a result is related to the Full Hyperbolicity of neutral renormalization and Conjecture~\ref{conj:P-P relation}.
\end{rem}

\subsubsection{Lamination $\bFol_\bbn$}
Fix a big $\bbn\ll0$ and consider $g\in \MM_\bbn$. For a pacman $f\in \BB$ close to $g$, we set $O'(f)\coloneqq O'(g)$ and we set $O(f)$ to be the lift of $O'(f)$ under $f^{\aa_\bbn}$ along the orbit of $f^{\aa_\bbn}\colon \frakY^0(f)\to \frakY^0(f)$; this is well defined because $f$ is close to $g$. We obtain a quadratic-like map
\begin{equation}
\label{eq:q-l:forall f}
\RR_{\QL}^{\bullet 2} (f)\coloneqq f^{\aa_\bbn}\colon O(f)\to  O'(f).
\end{equation} Let $\NN_\bbn$ be a small Banach neighborhood of $\MM_\bbn$. Since $\MM_\bbn$ is within a partial secondary copy of the Mandelbrot set (see Figure~\ref{Fig:Rec Small copy}), \eqref{eq:q-l:forall f} defines analytic operators into the space of quadratic-like germs $\QG$, see~\S\ref{ssss:QL anal ren}: 
\[\RR_{\QL}^{\bullet 2}\colon \NN_\bbn \to \QL \sp\sp \text{ and }\sp\sp \RR_{\QL}^{\bullet 3}\coloneqq \RR_{\QL}\circ \RR_{\QL}^{\bullet 2}\colon \NN_\bbn \to \QL,\]
where $\bchi\circ \RR_{\QL}^{\bullet 3} (\MM_\bbn) =\Mandel$ and $\bchi\circ \RR_{\QL}^{\bullet 2} (\MM_\bbn)=\Mandel_\kk\subset \Mandel$ -- see Remark~\ref{rem:RR_QL^2:strai} below.

We denote by $\frakY^0(f)$ the non-escaping set of~\eqref{eq:q-l:forall f}. Slightly shrinking $O'(f)$, if necessary, we can assume that the unbranched a priori bounds holds for~\eqref{eq:q-l:forall f}:
\[ \mod \left( O'(f)\setminus  O(f)\right)\ge \varepsilon_{\MM_\bbn}=\varepsilon_{\bMM_0},\sp\sp\sp\sp O(f)\cap \Post(f)\subset \Kfilled \left(\RR_{\QL}^{\bullet 2} (f)\right)\]
 because they hold for~\eqref{eq:q-l:g}.

\begin{rem}
\label{rem:RR_QL^2:strai}
The operator $\RR_{\QL}^{\bullet 2} $ should be viewed as the composition $\RR_\QL\circ \RR_{\QL}^{\bullet}$, where $ \RR_{\QL}^{\bullet} f$ is a degenerate quadratic-like map associated with the first renormalization map~\eqref{eq:FRM:bG}. It follows that  
\begin{equation}
\label{eq:RR_QL^2:strai}
 \mRR_{\QL}^2\circ \strai_{\Pacm}(g)= \strai \circ \RR_\QL^{\bullet 2}(g)\in \Mandel_\kk,\sp\sp g\in \MM_\bbn
\end{equation}
where $\strai$ is the quadratic-like straightening map, see~\S\ref{ssss:QL anal ren}.
\end{rem}

 Recall that we often identify a lamination with its support.
 \begin{lem}
 \label{lem:Foliat}
 For $g\in \MM_\bbn$ and $p\coloneqq \strai_{\Pacm}(g)$, 
define $\Fol_{p}$ to be the set of $f$ close to $g$ such that~\eqref{eq:q-l:forall f} is hybrid equivalent to~\eqref{eq:q-l:g}.
 The set  
  \[\bFol_\bbn\coloneqq \{\Fol_{p}\}_{p\in \Mandel_\bbn}\]
 forms a codimension-one lamination with complex-analytic leaves in a small neighborhood of $\MM_\bbn$.
 \end{lem}
\begin{proof}
Recal from~\S\ref{ssss:QL anal ren} that the hybrid classes form a codimension-one lamination $\bFol_{\QL}$ of the connectedness locus $\Conn$ with complex codimension-one analytic leaves. We need to show that the pullback $\bFol_\bbn=\left(\RR_{\QL}^{\bullet 2}\right)^*\left(\bFol_{\QL} \right)$ is again a lamination of the same type. We will use an argument from~\cite{ALM}*{Theorem 4.9}.

Consider a local leaf $\Fol'$ of $\bFol_\QL$ intersecting $\RR_{\QL}^{\bullet 2}(\MM_\bbn)$. In a small neighborhood of  $\RR_{\QL}^{\bullet 2}(\MM_\bbn)$, the leaf $\Fol'$ is the zero set of some analytic function $\phi\colon \QL\dashrightarrow \C$. Note that $\Fol'$ intersects $\RR_{\QL}^{\bullet 2}(\MM_\bbn)$ at a single point. By the Inverse Function Theorem, 
\[\left(\RR_{\QL} ^{\bullet 2} \right)^{-1}(\Fol')= (\phi\circ \RR_{\QL}^{\bullet 2})^{-1}(0)\]
is a codimension-one analytic manifold in a small neighborhood of $\MM_\bbn$. Varying the $\Fol'$ we obtain the pairwise disjoint closed leaves $\left(\RR_{\QL} ^{\bullet 2} \right)^{-1}(\Fol')$ that are transverse to $\MM_\bbn$; these leaves form a lamination by the $\lambda$-lemma.
\end{proof}

    Since the local dynamics is structurally stable at $\alpha(g)$, by shrinking a neighborhood of $\MM_\bbn$, we can assume that the flower $X(f)$ exists and depends holomorphically on $f\in \Fol_{p}$; i.e.~certain preimages of $\frakY^0(f)$ assembly into the flower $X(f)$ in the same pattern as certain preimages of $\frakY^0(g)$ assemble into the flower $X(g)$. Indeed, by continuity, every preimage $\frakZ'\subset X(g)$ of $\frakY^0(g)$ of bounded generation has the corresponding preimage $\frakZ'(f)$ of $\frakY'(f)$ such that $\frakZ'(f)$ is close to $\frakZ'(g)$. Since the linear coordinate at $\alpha$ is structurally stable, every preimage of $\frakZ'\subset X(g)$ of $\frakY^0(g)$ has a counterpart $\frakZ'(f)$ if $\frakZ'$ is close to $\alpha$.
    
The \emph{extended flower} $\widetilde X(f)$ is $X(f)\cup \bigcup_{i=0}^{\aa_n} f^i(O)$.  Since $f$ is close to $g$, the flower $\widetilde X(f)$ is also in a small neighborhood of $\overline Z_\str$. 
\subsubsection{Lamination $\bFol$} For a fixed $\bbn\ll0$, we defined the lamination $\bFol_\bbn$ in Lemma~\ref{lem:Foliat}. For $m\le \bbn$ and $p\in \Mandel_m$, we define
\[\Fol_{p}\coloneqq \{f\in \BB\mid \RR^{\bbn-m}(f)\in \Fol_{\mRR^{\bbn-m}(p)}\},\] 
Since $\RR$ is hyperbolic,
 \begin{equation}
\label{eq:ScThm:bfol} 
\bFol_m\coloneqq\{\Fol_{p}\mid p\in \Mandel_m\} \sp \text{ and } \sp \bFol\coloneqq  \{\WW^s\}\cup  \bigcup_{m\le \bbn}\bFol_{m}
\end{equation} form codimension-one stable laminations in a neighborhood of $f_\str$. A pacman $f\in \Fol_{p}$ with $p\in \MM_m$ has nice flowers $X(f)$ and $\widetilde X(f)$ that are the full lifts of $X(\RR^{\bbn-m} f)$ and $\widetilde X(\RR^{\bbn-m} f)$ respectively. The flowers $X(f)$  and $\widetilde X(f)$ satisfy the same conditions as $X(\RR^{\bbn-m} f)$ and $\widetilde X(\RR^{\bbn-m} f)$. In particular, $X(f)$ and $\widetilde X(f)$ are in a small neighborhood of $\overline Z_\str$; and all pacmen in $\Fol_{p}$ are hybrid conjugate in neighborhoods of their valuable flowers.  

Let us write \[p_{\str}=p_{c(\theta_\str)}\eqqcolon \strai_\Pacm (f_\str)\sp \text{  and }\Fol_{p_\str}=\Fol_{\str}\coloneqq \WW^s,\]
where $p_{\str}=p_{c(\theta_\str)}\in \Mandel$ is the unique quadratic polynomial on the boundary of the main hyperbolic component with rotation number $\theta_\str$. Since $\theta_\str$ has bounded type, $p_{\str}$  is hybrid conjugate to $f_\str$ on neighborhoods of their closed Siegel disks, see~\S\ref{sss:from SM to SP}.
We obtain the parameterization of leaves of $\bFol$ by \[\Mandel'\coloneqq \{p_\str\}\cup \bigcup_{m\le \bbn}\Mandel_m\] such that
\begin{equation}
\RR(\Fol_p) \subset \Fol_{\mRR^\mm(p)},\sp\sp p\in \Mandel'.
\end{equation}

\begin{proof}[Proof of the {\bf scaling} theorem (the first part of Theorem~\ref{thm:main})]Since $p_\str$ and $f_\str$ are hybrid conjugate in neighborhoods of their Siegel disks, there is a compact analytic renormalization operator $\RR_2\colon\UU\dashrightarrow \BB$ from a small neighborhood of $p_\str$ in the space of quadratic polynomials to a small neighborhood of $f_\str$, see~\S\ref{sss:from SM to SP}. Since maps in a small neighborhood of $p_\str$ have different multipliers at their $\alpha$-fixed points, the image of the slice $\UU$ is transverse to the lamination $\bFol$ in a small neighborhood of $f_\str$. (Otherwise, the multiplier map in $\UU$ will be a covering near the Siegel value.)


The operator $\RR_2$ acts on the rotation angles of indifferent maps as $\cRRc^{k\mm}$ for some $k\in \N$. Recall that $\mRRc$ denotes the molecule map,~\S\ref{ss:MolecMap}. We claim that for every $p,\  \widetilde p\coloneqq \mRRc^{k\mm}( p)\in \Mandel'$, we have $\RR_2( p) \in \bFol_{\widetilde p}$. Indeed, let us define $g_{p}$ to be the unique intersection of $\Fol_{p}$ with $\RR_2(\UU)$. We define $\widetilde p\in \UU$ to be the preimage of $g_p$ via $\RR_2$. The nice flower $\widetilde X(g_p)$ lifts to the dynamical plane of $\widetilde p$; we denote the lift by $\widetilde X(\widetilde p)$. Since $\widetilde p$ is a quadratic polynomial, the valuable flower $\widetilde X(\widetilde p)$ uniquely determines $\widetilde p$. 
 Comparing the combinatorial rotation numbers at the $\alpha$-fixed points, we obtain $\widetilde p = \mRRc^{-k\mm} (p).$

 Since the holonomy along $\bFol$ is asymptotically conformal \cite[Appendix 2, The $\lambda$-lemma (quasi- conformality)]{L:FCT}, the hyperbolicity of $\RR$ and the holonomy along $\bFol$ imply the scaling result.
\end{proof}

\subsection{Homoclinic configuration}
  \label{ss:HomConf}

Recall that the operator $\RR_{\QL}^{\bullet 2}\colon \NN_\bbn\to \QL$ and $\RR_{\QL}^{\bullet 3}\colon \NN_\bbn\to \QL$ on a neighborhood of $\bFol_\bbn$ are defined by~\eqref{eq:q-l:forall f}. Set 
\begin{itemize}
\item $  \RR_\QL^{\bullet 2} (g)\coloneqq  \RR_\QL^{\bullet 2} \circ \RR^{\bbn-m} (g)$ for $g\in \bFol_m$ with $m\le \bbn$; 
\item $\RRR=\RR_\QL^{\bullet 3}\coloneqq \RR_\QL\circ \RR_\QL^{\bullet 2}$, and
\item $\MM_\cp\setminus \{\text{cusp}\}\coloneqq \RR_\QL^{\bullet 3} \big(\MM_\bbn\setminus \{\text{cusp}\}\big)=\RR_\QL^{\bullet 3} \big(\MM_m\setminus \{\text{cusp}\}\big)$,
\end{itemize}
 where $\RR_\QL$ is the quadratic-like operator, see~\S\ref{ssss:QL anal ren}. By Theorem~\ref{thm:small copy M1}, $\MM_\cp$ is a copy of the Mandelbrot set, and $\strai\colon \MM_\cp\to \Mandel$ is the canonical straightening homeomorphism. Note that the renormalization change of variables of $\RR_\QL^{\bullet 3}\mid \bFol_\bbn$ is linear, but the renormalization change of variables of $\RR_\QL^{\bullet 3}\mid \bFol_m$ is non-linear for $m<\bbn$.

  By construction and~\eqref{eq:RR_QL^2:strai}, 
\begin{equation}
 \strai \circ \RR_{\QL}^{\bullet 3}(g) = \mRR_\QL^3\circ  \strai_\Pacm(g),\sp\sp g\in \bFol.
\end{equation}

\begin{figure}
\begin{tikzpicture}

\begin{scope}[line width=1.2pt]
\draw (0,-2) -- (0,5)
(-2,0) -- (3,0)
(6,0)--(8,0);


\end{scope}

\draw[blue,line width=2.5pt ] (-0,3.5)--(-0,4.5);

\node(Mn)[ right,red] at  (-0,4) {$\NN_\bbn$};
\node [ left,blue] at  (-0,4) {$\MM_\bbn$};

\draw[line width=2pt,blue] (7,-1.5) -- (7,2.5);

\node[ right,blue] at (7,1){$\MM_\cp$};

\draw[dashed,red]  (7,0.5) ellipse (1.3 and 2.3);

\draw[dashed,red]  (7,0) ellipse (0.5 and 0.6);

\node(QL)[ left,red] at (7,2){$\QL$};

\draw (Mn) edge[->,bend left] node[above right]{$\RRR=\RR_\QL^{\bullet 3}$} (QL);

\filldraw (7,0) circle (1.5pt);
\node[below right ] at (7,0){$g_\str$};
\node(AA)[above left,red ] at (7,0){$\AA$};

\filldraw (2,0) circle (1.5pt);

\draw[dashed,red]  (2,0) ellipse (0.4 and 1);
\draw[blue,dashed](2,-0.8) --(2,0.9);

\draw (AA) edge[->,bend right]  (2.2,0.5);
\node[above] at(4.7,1.1) {$\RR_\Sieg\eqqcolon \RR^\tau$};

\draw[dashed,red ]  (0.2,1.8) ellipse (3.5 and 4.5);
\node[red] at (-2.3,2.6){(pacmen)};
\node[red,above] at (-2.3,2.8){$\BB$};

\draw[dashed,red] (0,4) ellipse (0.75 and 0.85);

\filldraw (0,0) circle (1.5pt);
\node[below left ] at (0,0){$f_\str$};

\node[left] at (-2,0){$\WW^s$};

\node[above] at (0,5){$\WW^u$};
\end{tikzpicture}
\caption{The space of pacmen $\BB$, the quadratic-like renormalization operator $\RR_\QL^{\bullet 3}\colon \NN_\bbn\to \QL$ defined on a neighborhood of $\MM_\bbn\setminus \{cusp\}$, and a Siegel renormalization operator $\RR_\Sieg\colon \AA\to \BB$ defined on a neighborhood $\AA$ of $g_\str$.} 
\label{fig:R Sieg}
\end{figure}

\subsubsection{Extension of $\bFol$}
Denote by $g_\str\in \MM_\cp$ the unique Siegel map on the main hyperbolic component of $\MM_\cp$ such that $g_\str$ is hybrid equivalent to $f_\str$. Equivalently, $\strai(g_\str)=p_\str=\strai_\Pacm(f_\str)$. By~\S\ref{sss:from SM to SP}, there is a compact analytic renormalization operator $\RR_{\Sieg}\colon \AA\to \BB$, where $\AA$ is a Banach neighborhood of $g_\str$, see Figure~\ref{fig:R Sieg}.

\begin{lem}
The stable lamination $\bFol$ admits a pullback via $\RR_\Sieg$. For all sufficiently big $m<0$, all leaves of the lamination $\RR_\Sieg^*( \bFol_m)$  transversally intersect $\MM_\cp$.
\end{lem}
\begin{proof}
By the same argument as in Lemma~\ref{lem:Foliat}, $\RR_\Sieg^*(\bFol)$ is a lamination with complex-analytic leaves. 

Let $\WW'$ be a small neighborhood of $g_\str$ in $\RR_\QL^{\bullet 3}(\WW^u)$. Then $\WW^s$ transversally intersects $\RR_{\Sieg} (\WW')$ at $\RR_{\Sieg}  (g_\str)$. Since $\bFol$ forms a lamination, all the leaves of $\bFol_m$ transversally intersect $\RR_{\Sieg} (\WW')$ for $m\ll \bbn$. Taking the pullback, we obtain that the leaves of $\RR_\Sieg^*( \bFol_m)$ transversally intersect $\WW'$. Clearly, all the points in the intersections are within the non-escaping set $\MM_\cp$. 
\end{proof}

Let us extend the lamination $\bFol$ by adding $\RR_\Sieg^*  \bFol$ to $\bFol$. The operator $\RR_{\Sieg}$ acts on the rotation numbers of indifferent pacmen as  $\cRRc^{\mm\btau}$ for some $\btau\ge 1$. Let us view $\RR_{\Sieg}$ as $\RR^\btau$. We factorize $\RR_\Sieg$ as a composition of $\btau$ operators, each operator acts on the rotation numbers of indifferent maps as  $\cRRc^{\mm}$. With this convention, the lamination $\bFol$ naturally extends to $\AA$. Namely, for every $\Fol_{p}$ in $\bFol_{m}$ with $m\le \bbn-\btau$, we define $\Fol'_p$ to be the preimage of $\Fol_{\mRRc^{\mm\btau}(p)}$ under $\RR_{\Sieg}$, and we set 
\begin{equation}
\label{eq:fol:ext}
\Fol_p\coloneqq \Fol_p\cup \Fol'_p.
\end{equation}
Similarly, $\bFol$ is extended to $\RR^i(\AA)$ for $i<\btau$. The new extended lamination $\bFol$ is still $\RR$-invariant.

\begin{figure}[tp!]
\begin{tikzpicture}

\begin{scope}[line width=1.2pt]
\draw (0,-2) -- (0,5)
(-2,0) -- (8,0);

\end{scope}

\draw[white, fill=red, opacity=0.3] (5,2) --(7,2)--(7,1.6)--(5,1.6)--(5,2);

\draw[white, fill=orange, opacity=0.5] (-0.5,4)--(-0.5,4.4)--(0.5,4.4)--(0.5,4)--(-0.5,4)--(-0.5,4);

\draw[line width=0.0001pt, fill=orange, opacity=0.5]  (-1.5,1.9) --(7.5,1.9)--(7.5,1.8)--(-1.5,1.8)--(-1.5,1.9);

\draw (-0.5,3.5)--(-0.5,4.5)--(0.5,4.5)--(0.5,3.5)--(-0.5,3.5)--(-0.5,3.5); 

\draw[blue,line width=2.5pt ] (-0,3.5)--(-0,4.5);

\node[above right,blue] at  (-0,4.5) {$\MM_\bbn$};

\node[ left] at  (-0.5,4) {$\bFol_\bbn$};

\draw[blue,line width=2.5pt ] (0,2)--(0,1.6);

\node[above right,blue] at  (0,2) {$\MM_m$};

\node[ left] at  (-1.5,1.8) {$\bFol_m$};

\draw[blue,line width=2.5pt ] (0,0.9)--(0,0.6);

\node[above right,blue] at  (0,0.9) {$\MM_{m-1}$};

\node[ left] at  (-1.5,0.75) {$\bFol_{m-1}$};

\draw[line width=2pt,blue] (6,-1.5) -- (6,2.5);

\node[above right,blue] at (6,2.5){$\MM_\cp$};

\filldraw (6,0) circle (1.5pt);
\node[below right ] at (6,0){$g_\str$};

\draw (-1.5,2) --(7.5,2)--(7.5,1.6)--(-1.5,1.6)--(-1.5,2);

\draw (3,0.95) edge[->,bend right] node[right]{$\RR$} (3,1.55);

\draw (3,2.05) edge[->,bend right] node[above right]{$\RR^{\bbn-m}$} (0.55,3.45);

\draw (6.8,1.7) edge[<-,bend right] node[above right]{$\RRR=\RR_\QL^{\bullet 3}$} (0.25,4.2);

\node[below,red] at (6.5,1.7) {$\bFol'_m$};

\node[right,orange] at (7.5,1.9) {$\bFol^{(m)}_m$};

\draw[shift={(0,-1)}] (-1.5,1.9) --(7.5,1.9)--(7.5,1.6)--(-1.5,1.6)--(-1.5,1.9);

\draw[dashed] (5,-1.5)--(7,-1.5)--(7,2.5)--(5,2.5)--(5,-1.5);

\filldraw (5.6,1.85) circle (0.8pt);
\node[below] at (5.6,1.85){$g_m$};

\filldraw (0,0) circle (1.5pt);
\node[below left ] at (0,0){$f_\str$};

\node[left] at (-2,0){$\WW^s$};

\node[above] at (0,5){$\WW^u$};
\end{tikzpicture}
\caption{Homoclinic dynamics: the operator $\RRR$ has a unique hyperbolic fixed point $g_m\in \bFol_m$. (Note that we slightly simplified the picture and made leaves in $\bFol$ connected, see~\eqref{eq:fol:ext}.)}
\label{fig:HomDyn}
\end{figure}

\subsubsection{Hyperbolic horseshoe} 
Let $p$ be a hyperbolic fixed point of a $C^2$-smooth diffeomorphism. If the stable and unstable manifolds $W^s$ and $W^u$ of $p$ intersect, then any point $q\in W^s\cap W^u $ is called \emph{homoclinic} to $p$. If the intersection is transversal, then there is a hyperbolic set in a neighborhood of the orbit of $q$ union $p$, see \cite{PT}*{Chapters 1, 2} for reference. We will now adopt this principle to show that there is a hyperbolic renormalization horseshoe of $\RRR$ near $g_\str\in \WW^s\cap \RRR (\WW^u)$, where  $\RRR (\WW^u)$ can be viewed as an extension of $\WW^u$.
  
 Let $\bFol'_m$ be $\bFol_m$ intersected with a small neighborhood of $\MM_\cp$, see Figure~\ref{fig:HomDyn}.  We define $\bFol^{(m)}_\bbn$ to be the preimage of $\bFol'_m$ under $\RRR\mid \bFol_\bbn$, and we define  $\bFol^{(m)}_k\subset \bFol_k$ to be the preimage of $\bFol^{(m)}_\bbn$ under $\RR^{\bbn-k}$ (extended to a neighborhood of $g_\str$ as above).

\begin{lem}[The second part of Theorem~\ref{thm:main}: {\bf rigidity}]
 For $\bbk\ll \bbn$ the operator 
\begin{equation}
\label{eq:defn:hyp RR_QL}
 \RRR\colon \bigcup_{m,k\le \bbk}\bFol_m^{(k)}\to \bigcup_{m\le \bbk}\bFol'_m
\end{equation}
is uniformly hyperbolic. More precisely, let $\Hors$ be the non-escaping set of~\eqref{eq:defn:hyp RR_QL}; i.e. the set of points with bi-infinite orbit. Then $\Hors$ is a hyperbolic set. Let $\mHors$ be the non-escaping set of 
\[ \mRR_\QL^3\colon \bigcup_{t\le \bbk} \Mandel_t \to  \bigcup_{t\le \bbk} \Mandel_t . \]
Then 
\begin{itemize}
\item  every connected component of $\mHors$ is a singleton; and 
\item $ \RRR\colon \Hors \to \Hors$ is parametrized by the natural extension of $\mRR_\QL^3\colon \mHors\to \mHors$ via $\strai_\Pacm$.
\end{itemize}
\end{lem}
\noindent In particular, $\RRR\colon \bFol^{(m)}_m\to \bFol'_m$ has a unique hyperbolic fixed point $g_m\in \bFol^{(m)}_m\cap \bFol'_m$, see Figure~\ref{fig:HomDyn}. 

\begin{proof} This is a homoclinic configuration for the operator $\RR$ combined with the ``gluing'' operators ${\RR_\QL^{\bullet 3}\colon \NN_\bbn\to \QL}\supset\AA$ and $\RR_\Sieg\colon \AA\to \BB$, see Figure~\ref{fig:R Sieg}. A point $g_\str$ has a forward infinite orbit along $\WW^s$ towards $f_\str$ and it has a backward infinite orbit along $\WW^u$ towards $f_\str$. Therefore, the non-escaping set in a small neighborhood of $\{f_\str\}\cup \orb (g_\str)$ is hyperbolic, see \cite{PT}*{Theorem 1 in \S7, Chapter 2}. (In short, since the orbit of $\bFol_m^{(k)}$ stays in a small neighborhood of $f_\str$ for most of the iterations (Figure~\ref{fig:RenOrbit}), the map $\bFol_m^{(k)}\longrightarrow \bFol'_m$ has a big contraction in the horizontal direction and big expansion in the vertical direction.) Considering the first return map to a neighborhood of $g_\str$, we obtain that~\eqref{eq:defn:hyp RR_QL} is hyperbolic.

We have a natural surjective semi-conjugacy $\pi$ from $ \RRR\colon \Hors \to \Hors$ to the natural extension of $\mRR_\QL^3\colon \mHors\to \mHors$. Since $\varepsilon$-close orbits for a hyperbolic map coincide, $\pi$ is a homeomorphism and every connected component of $\mHors$ is a singleton.
 \end{proof}

\begin{figure}[tp!]
\begin{tikzpicture}

\begin{scope}[line width=1.2pt]
\draw (0,-3) -- (0,4)
(-3,0) -- (7,0);

\end{scope}

\draw[dashed,red] (0,0) circle (2);

\coordinate (b1) at (5.6,0) ;

\coordinate (a1) at (5.6,0.2) ; 
\coordinate (a2) at (1.57,0.25); 
\coordinate (a3) at (1.2,0.28);
\coordinate (a4) at (0.95,0.33);
\coordinate (a5) at (0.7,0.5);
\coordinate (a6) at (0.55,0.8);
\coordinate (a7) at (0.4,1.1);
\coordinate (a8) at (0.3,1.5);

\coordinate (a9) at (0.2,3.5);

\filldraw (b1) circle (1.5pt);
\node[below] at  (b1) {$g_\str$};

\filldraw (a1) circle (1pt);
\node[shift={(0.6,0)}, above] at (a1) {$g=g_m\in \bFol_m$};

\filldraw (a2) circle (1pt);
\node[right ] at (a2) {$\RR^t g\in \bFol_{m+t}$};

\filldraw (a3) circle (1pt);
\filldraw (a4) circle (1pt);
\filldraw (a5) circle (1pt);
\filldraw (a6) circle (1pt);
\filldraw (a7) circle (1pt);
\filldraw (a8) circle (1pt);
\filldraw (a9) circle (1pt);

\node[right] at (a8) {$\RR^{-m+\bbn-t} g\in \bFol_{\bbn-t}$};

\node[right] at (a9) {$\RR^{-m+\bbn} g\in \bFol_{\bbn}$};

\filldraw (0,0) circle (1.5pt);

\node[below left ] at (0,0){$f_\str$};

\node[left] at (-3,0){$\WW^s$};

\node[above] at (0,4){$\WW^u$};
\end{tikzpicture}
\caption{The orbit of $g\in \bFol_m$ stays within a small neighborhood of $f_\str$ for many iterates.}
\label{fig:RenOrbit}
\end{figure}

\begin{proof}[Theorem~\ref{thm:main}: proof of {\bf JLC}] Local connectivity of every map in $\Hors$ follows (see~\S\ref{sss:JLC}) from unbranched \emph{a priory} bounds:  
\begin{lem}
Every map in $\Hors$ has unbranched a priory bounds.  
\end{lem}
\begin{proof}
Recall that we constructed a homoclinic configuration under the renormalization illustrated in Figure~\ref{fig:HomDyn}, a certain neighborhood of $g_\str$, and the hyperbolic horseshoe $ \RRR\colon \Hors \selfmap$ for the first return map to this neighborhood of $g_\str$. Every map $f\in \Hors$ before it returns to $\Hors$ travels through a small neighborhood of $\MM_n$ unbranched a priori bounds holds for the associated quadratic-like map~\eqref{eq:q-l:forall f}. Since the renormalization change of variables is conformal, unbranched a priori bounds descends to all deep scales, see~\S\ref{sss:renorm horsh}; i.e.~ setting $f_{3n}\coloneqq \left(\RRR\right)^n(f)$ we have unbranched a priori bounds for 
\begin{equation}
\label{eq:ql:f_3n+2}
 \RR_\QL^{\bullet 2} (f_{3n})\colon O(f_{3n})\to O'(f_{3n}).
\end{equation} 
\end{proof}
\end{proof}

\subsubsection{Upper semi-continuity of the valuable flower}
\label{sss:contr post set} Along the lines, we also obtained a geometric control of the postcritical sets of maps in $\bFol$. Let us write $ X(f)\coloneqq \overline Z_f$ for a Siegel map $f\in \WW^s$. Then $X(f)$ depends upper semi-continuously on $f\in \bFol$. Moreover, if a sequence $f_n\in \bFol_{n}$ tends to $f\in \WW^s$ as $n\to -\infty$, then $\Post(f_n)$ and $\widetilde X_\up (f_n)$ tend to $\partial  Z_f= \overline \Post(f)$.  Indeed, by Theorem~\ref{thm:proj of val flow}, the valuable flower $\widetilde X(f_n)$ is within a certain Siegel triangulation $\bDelta(f_n)$. And the wall $\bPi(f_n)$ of $\bDelta(f_n)$ contains $ \widetilde X_\up(f_n)$ -- the cycle of secondary small Julia sets. Since  $\bDelta(f_n)$ is a full lift of  $\bDelta(f_{n+1})$, Lemma~\ref{lem:SiegTriangLifting} implies that $\bDelta(f_n)$ tends to $\overline Z_f$ while $\bPi(f_n)$ tends to $\partial Z_f$.

The upper semi-continuity can easily be transferred to a parameter neighborhood of any Siegel map. Indeed, if $g$ is Siegel map, then there is a hyperbolic renormalization operator $\RR$ around a fixed pacman $f_\str$ as above such that a certain renormalization operator $\RR_\Sieg$ maps a neighborhood of $g$ to a neighborhood of $f_\str$. Pulling back the lamination $\bFol$ under $\RR_\Sieg$, we obtain the upper semi-continuity of $X(f)$ for $f\in \RR_\Sieg^* \big(  \bFol \big)$. In~\cite{DLS}*{Appendix C} a much stronger conjecture was stated that $\bFol$ can be extended to a lamination parametrized by a subset of the Mandelbrot set containing the main molecule of the Mandelbrot set. 

\subsection{Positive measure}
\label{ss:PositArea}
In this section we will show that for $m\ll \bbk<0$, the Julia set of $g\coloneqq g_m\colon U\to V$ has positive measure. The result essentially follows from the Koebe-type estimates~in \cite{AL-posmeas}*{\S\S6.6--6.8}, but we need to adjust them to our setting.

 Let us give a {\bf short outline}. By construction, $g_m=\RRR(g_m)$ is a renormalization fixed point (see~\S\ref{sss:renorm horsh}): it is conformally conjugate to its quadratic-like renormalization $g_\bullet = g^{t(m)}\colon U_\bullet \to V_\bullet$. As $m\to -\infty$, the map $g_m\colon  U\to V$ tends to the Siegel map $g_\str\colon U_\str\to V_\str$ so that the valuable flower $X(g_m)$ approximates the Siegel disk $\overline Z(g_\str)$. This will allows us to construct in the dynamical plane of $g_m$ ``trapping disks'' $D_0,\dots, D_\bbs$ with $\bbs\to \infty$ as $m\to -\infty$ so that, see~\S\ref{sss:TrapDisks} and Figure~\ref{Fig:Trapping disks}:
\begin{itemize}
\item[(\RN{1})] if a point $z\in \overline Z(g_\str)$ escapes $U$ under iterations of $g_m$, then the orbit of $z$ passes through all the $D_i$;
\item[(\RN{2})] a definite portion (independent of $m,\bbs$) of $D_i$ returns to $\overline Z(g_\str)$;
\item[(\RN{3})] a definite portion of $D_0$ maps to $U_\bullet$. 
\end{itemize}
Namely, trapping disks satisfying (\RN{1}) and (\RN{2}) exist in the dynamical plane of $f_\str$; let us choose one such disk $D$ close to $Z_{f_\str}$. Since the renormalization orbit $\RR^i g_m$ has many iterations in a small neighborhood of $f_\str$ (Figure~\ref{fig:RenOrbit}), we can lift $D$ from the dynamical plane of $\RR^i g_m$ to the dynamical plane of $g_m$. Different lifts will be in different renormalization scales. Disk $D_0$ is the lift of $D$ from the dynamical plane of $\bar g\coloneqq \RR^{-m+\bbn-t} g\in \bFol_{\bbn -t}$. Since $ \bFol_{\bbn -t}$ is independent of $m$, (\RN{3}) holds for $D$ in the dynamical plane of $\bar g$; lifting $D$ we obtain (\RN{3}) for $D_0(g_m)$.

Consider now the $g_m$-orbit of $z\in \overline Z(g_\str)$ that escapes $U$. Since $\bbs\gg 1$,  $(\RN{1})$ and $(\RN{2})$ imply that the orbit typically passes many times through $D_0$ before the orbit goes through all the $D_i$; in each such visit to $D_0$, the orbit has a definite chance to enter $U_\bullet$ -- by (\RN{3}). Therefore, the probability of $z$ to escape $U$ is much lower than the probability to enter $U_\bullet$. By~\cite{AL-AMS} the Julia set of $g_m$ (and hence of $p_m$) has positive area. 

Properties (\RN{2}) and (\RN{3}) are proven as Properties~\ref{prop:trap disk} and \ref{prop:first trap disk}. With two more ingredients (Properties~\ref{prop:trap disk:hyp diam is bound} and~\ref{property:g D_i is exp}), the Koebe distortion arguments allow us to formally justify the probability viewpoint in the same way as it was done in \cite{AL-posmeas}. 


\begin{rem}
\label{rem:posit area}
We can re-state the positive-area argument as follows. Once a full copy of the Mandelbrot set $\MM_0$ is recognized on the unstable manifold, methods of~\cite{AL-posmeas} imply that the Julia set has positive measure for the parameter associated with a sufficiently big pacman antirenormalization of $\MM_0$. In~\cite{AL-posmeas}, a full primitive copy is constructed by following a certain periodic point whose orbit is close to the Siegel disk. In this paper we use puzzle techniques to recognize a full satellite copy of the Mandelbrot set. (Potentially, puzzle techniques may allow one to recognize all the existing copies on the unstable manifold, see~\S~\ref{ss:param rays}.)          
\end{rem}

\subsubsection{Notations}\label{sss:Notat: posit Area}
By saying that a set $K$ is \emph{well inside}
a domain $D \Subset \C$ we mean that $K \Subset D$ with a definite $\mod(D \setminus K)$. The meaning
of expressions \emph{bounded, comparable}, etc. is similar.

Given a pointed domain $(D, \beta)$, we say that $\beta$ lies in the \emph{middle} of $D$, or equivalently,
that $D$ has a \emph{bounded shape} around $\beta$ if
\[\max _{\zeta\in \partial D} |\beta-\zeta| \le C\min_{\zeta\in \partial D} |\beta-\zeta|.\]

We set 
\begin{itemize}
\item $g \coloneqq g_m \colon U\to V$;
\item $\Jul \coloneqq \Jul(g)$;
\item $g_\bullet=g^{t(m)}\colon U_\bullet \to V_\bullet$ to be the $\RR_\QL^3$-pre-renormalization of $g$ normalized so that $g_\bullet\colon U_\bullet \to V_\bullet$ is conformally conjugate to $g\colon U\to V$ (this is possible because $g=\RRR( g)$; note that the conjugacy is not affine);
\item $Z$ is the Siegel disk of $g_\str$ and $Z'$ is prefixed Siegel disk of $g_\str$;
\item $\Jul _\bullet\coloneqq \Jul (g_\bullet)\subset \Jul$;
\item ``$\diam$'' and ``$\dist$'' denote the Euclidean diameter and distance.
\end{itemize}

By a \emph{hyperbolic metric}, we mean the metric of $V\setminus \ovPost(g)$, unless specified otherwise. Since \[g\colon U\setminus g^{-1}\big(\ovPost(g)\big)\to V\setminus \ovPost(g)\]
is a covering map, while
 \[ U\setminus g^{-1}\big(\ovPost(g)\big)\hookrightarrow  V\setminus \ovPost(g)\]
 is an inclusion, $g$ expands (non-uniformly) the hyperbolic metric.

\subsubsection{Parameters $\eta$ and $\xi$}
 Let us recall from~\cite{AL-AMS} a condition ensuring that the Julia set $\Jul$ of $g=g_m$ has positive area. Define
\begin{itemize}
\item $\eta$ to be the probability for an orbit starting in $U$ (the domain of $g$) to enter $U_\bullet$ (the domain of $g_\bullet$),
\item $\xi$ to be the probability that an orbit starting in $V_\bullet \setminus U_\bullet$ will never come back to $U_\bullet$.
\end{itemize}

There is a constant $C>0$ independent of $m$ such that if $\eta/\xi>C$, then the Julia set of $g=g_m$ has positive area. The constant $C$ depends on geometric bounds (like $\mod (V\setminus U)$, see \cite{AL-AMS}*{\S 2.7}) that are uniform over $m\ll \bbk$. We remark that the renormalization change of variables was assumed to be affine in~\cite{AL-AMS}; but the criterion is easily relaxed for a conformal change of variables by linearizing it, see~\S\ref{sss:renorm horsh}.

\begin{figure}[tp!]
\begin{tikzpicture}

\draw (-3,0)--(0,2)--(3,0)--(0,-2)--(-3,0);

\draw[blue,scale=1.2,dashed] (-3,0)--(0,2)--(3,0)--(0,-2)--(-3,0);
\draw[blue,scale=1.3,dashed] (-3,0)--(0,2)--(3,0)--(0,-2)--(-3,0);

\draw (4.5,-1)--(3,0)--(4.5,1);

\coordinate  (b1) at (3.6,0); 

\node[blue] at (-1.9,1.8) {$A$};

\draw[blue,fill=blue, fill opacity=0.1]  (b1) ellipse (0.41 and 0.6);
\node[blue,above,shift={(-0.1,0.6)}] at (b1) {$D$};

\node[above] at (0,0) {$Z$};
\node[] at (4.5,0) {$Z'$};
\node[ above ] at (2.87,0){$c_0$};

\end{tikzpicture}
\caption{A trapping disk $D\subset \C\setminus \overline Z$: every orbit escaping from the domain surrounded by $A$ passes through $D$. (The annulus $A$ is close to $\partial Z$ while $D$ is close to the critical point $c_0$).}
\label{Fig:Trapping disk}
\end{figure}

\subsubsection{Trapping disks}
\label{sss:TrapDisks}
Consider the dynamical plane of $f_\str$. Below we recall main properties of \emph{trapping disks}; see~\cite{AL-posmeas}*{\S~4.4.4.} for a detailed discussion.
There is an annulus $A\subset \C\setminus \overline Z_\str$ in a small neighborhood of $\overline Z_\str$ and a trapping disk $D\subset  \C\setminus \overline Z_\str$ such that (see Figure~\ref{Fig:Trapping disk})
\begin{itemize}
\item if $z$ is in the bounded component $O$ of $\C\setminus A$, then $f(z)\subset O\cup A$;  
\item if $z\in A$, then $f^i(z)\in D$ with $i\le q$, where $q$ depends on the renormalization scale of $D$ (i.e., on how close $D$ is to the critical point);
\item a definite portion of $D$ is in $\overline Z'_\str$.  
\end{itemize}
 Moreover, all the properties still hold under small shrinking of $A$ and $D$. Therefore, by continuity, the trapping disk $D$ exists in the dynamical plane of a nearby map. 
 
 Consider the orbit of $g$ under the pacman renormalization $\RR$ and note that $\RR^i g$ is close to $f_\str$ if $i\in \{t,t+1,\dots ,-m+\bbn-t\}$ with $-m+\bbn-t\gg t$, see Figure~\ref{fig:RenOrbit}. For such $i$, consider the dynamical plane of $\bar g\coloneqq \RR^i g$. By continuity, $D$ is a trapping disk for $\bar g$. Since $\widetilde X(\bar g)$ is in a small neighborhood of $\overline Z_\str$ (by~\S\ref{sss:contr post set}), $A$ surrounds $\widetilde X(\bar g)$. Let $\psi_0$ be the renormalization change of variables from the dynamical plane of $\bar g$ to the dynamical plane of $g$ normalized so that $\psi_0( c_0(\bar g)) =c_0 (g)$, see~\S\ref{sss:CofV near c_0}. Then $D' \coloneqq \psi_0 (D)$ is a trapping disk for $g$ with the property that every escaping orbit starting in $\widetilde X(g)$ passes through $D'$. Since $i$ can be chosen between $t$ and $-m+\bbn-t$, we can construct pairwise disjoint trapping disks
\[D_0,D_1,\dots , D_\bbs\]
in the dynamical plane of $g$, where $\bbs$ is sufficiently big (if $m\le \bbk \ll \bbn$ is sufficiently big), see Figure~\ref{Fig:Trapping disks}. We assume that $D_0$ is the lift of $D( \RR^{-m+\bbn-t} g)$.

\begin{property}
\label{prop:trap disk}
All trapping disks $D_i$ are within $\C\setminus \Post$. For every $i\le \bbs$, a definite portion of $D_i$ maps to $\overline Z(g_\str)$ under one iteration.\qed
\end{property}

\begin{property}
\label{prop:first trap disk}
There are degree two iterated preimages $U'_\bullet,V'_\bullet\subset D_0$ of $U_\bullet$ and $V_\bullet$ such that $U'_\bullet$ and $V'_\bullet$ occupy a definite  portion (i.e.~independent of $m,$) of $D_0$.
\end{property}
\begin{proof}
Since $\bar g=\RR^{-m+\bbn-t} (g)\in\bFol_{\bbn-t}$ where $\bFol_{\bbn-t}$ is independent of $m$, the property holds in the dynamical plane of $\bar g$ for $D(\bar g)$ and the associated quadratic-like renormalization of $\bar g$. Lifting $D(\bar g)$ to $D_0(g)$, we obtain the property for $g$.
\end{proof}

\subsubsection{Estimating $\eta$} (Similar to \cite{AL-posmeas}*{Proposition 6.22}.)
As a consequence of Property~\ref{prop:first trap disk}, for any point $z$ whose orbit passes through the first trapping
disk $D_0$ under the iterates of $g$, there exist quasidisks $U_\bullet(z) \subset V_\bullet(z)$ with bounded shape whose size is comparable with $\dist(z, V(z))$, and such that
\[f^{n} (U_\bullet(z)) \subset U_\bullet \text{ and }f^{n} (V_\bullet(z))\subset V_\bullet\text{ for some }n=n(z).\]

As a corollary, the landing probability $\eta$ is bounded below uniformly in $m$.
Indeed, it is known that almost every point in $\Jul$ lands in $\Jul_\bullet$,~\cite{L:TypBeh}. Since the Siegel disk $Z(g_\str)$ occupies a certain area, it is sufficient to check that a definite portion of points $z \in Z (g_\str)\setminus \Jul$ land
in $U_\bullet$. But any point $z \in Z (g_\str)\setminus \Jul$ on its way from $Z(g_\str)$ to $V\setminus U$ must pass through the
first trapping disk $D_0$. Since $U_\bullet(z)$ occupies a definite portion of some neighborhood of $z$, the statement follows.

\subsubsection{Expansion of $g\mid\big( D_i\setminus g^{-1}(\Post)\big)$}
In this subsection we will verify the following properties:
\begin{property}[similar to~\cite{AL-posmeas}*{(6.9)}]
\label{prop:trap disk:hyp diam is bound}
The hyperbolic diameter of $D_i$ is uniformly (in $i$ and $m$) bounded.
\end{property}

\begin{property}[similar to~\cite{AL-posmeas}*{\S 6.2.2.}]
\label{property:g D_i is exp}
The map $g\mid \big(D_i\setminus g^{-1}(\ovPost)\big)$ is uniformly expanding with respect to the hyperbolic metric of $V\setminus \ovPost(g)$. 
\end{property}

Property~\ref{property:g D_i is exp} has the following explanation.  Consider the dynamical plane of the Siegel map $g_\str$. Let $x\not\in \overline Z \cup \overline Z'$ be a point close to $c_0$. It was shown by McMullen~\cite{McM3} that if \[\dist \big(x, \overline Z'\big)\le C \dist\big(x,\overline Z\big),\] then $g_\str$ expands the hyperbolic metric of $\C\setminus \overline Z$ by a factor $\lambda>1$ depending only on the constant $C$. Recall from~\S\ref{sss:contr post set} that for a big $m\ll \bbk$, the postcritical set $\Post(g_m)$ approximates $\Post(g_\str)$. Suppose $x$ belongs to the self-similarity scale $t$; i.e.~$\dist(x, c_0)\asymp \mu_\str^t$. One can show that if $-t-m\gg 0$ is sufficiently big, then $\Post(g_m)$ is sufficiently close to $\partial Z$ (relative to $ \mu_\str^t$) and $g=g_m$ has a definite expansion at $x$.  Let us now proceed with the proof of Property~\ref{property:g D_i is exp}. We need the following fact:

\begin{property}
\label{property:hyp metr}
There is a function $\tau \colon \R_{>2}\to \R_{>1}$ such that \[\tau(r)\to 1\sp\sp \text{ as }\sp\sp r\to +\infty,\]
and such that the following property holds. Let $S_1,S_2$ be two closed connected subsets of $\C$ such that 
\begin{itemize}
\item $1\in S_1\cap S_2$ but $0\not\in S_1\cup S_2$; and 
\item $S_1\cap \ovDisk(t r)=S_2\cap \ovDisk(tr)$ for some $t>1$ and $r >2$.
\end{itemize}
Let $\rho_1$ and $\rho_2$ be the hyperbolic densities of $\C\setminus S_1$ and $\C\setminus S_2$ with respect to the Euclidean metric. Then
\[\frac {1}{\tau(r)}\le \frac{\rho_1(z)}{\rho_2(z)}\le \tau(r)\sp \sp \text{ for } z\in \Disk(t).\]\qed
\end{property} 

Consider now a trapping disk $D_i$. Recall from \S\ref{sss:TrapDisks} that $D_i=\psi_0 (D)$, where $D$ is the trapping disk in the dynamical plane of $\bar g\coloneqq \RR^{n(i)}g$ and $\psi_0$ is the renormalization change of variables specified so that $\psi_0 \big(c_0\big(\bar g \big)\big) = c_0(g)$.

\begin{figure}[tp!]
\begin{tikzpicture}

\draw[dashed] (-3,0)--(0,2)--(3,0)--(0,-2)--(-3,0);

\draw[dashed] (6,-2)--(3,0)--(6,2);
\coordinate  (aa) at (-1.37,0.87); 
\draw[rotate around={-45: (aa)},red,fill=red, fill opacity=0.1]  (aa) ellipse (0.5 and 0.3);

\node[red,above,shift={(-0.2,0.4)}] at (aa) {$U_\bullet$};

\coordinate  (b1) at (3.55,0);

\draw[blue,fill=blue, fill opacity=0.1]  (b1) ellipse (0.3 and 0.6);
\node[blue,above,shift={(-0.1,0.6)}] at (b1) {$D_0$};

\draw[red,fill=red, fill opacity=0.6,shift={(0,0.)},red,shift={(0,0.1)}] (3.45,-0.1)--(3.65,-0.1)--(3.65,0.1)--(3.45,0.1)--(3.45,-0.1);

\node[blue,scale =2] at (4.5,0) {$\dots$}; 

\coordinate  (b2) at (5.5,0); 

\draw[blue,fill=blue, fill opacity=0.1]  (b2) ellipse (0.4 and 1.9);
\node[blue,above,shift={(-0.1,1.9)}] at (b2) {$D_n$};

\draw (3.4,0) edge[red ,->,bend left] (-1.07,0.4);


\end{tikzpicture}
\caption{Trapping disks at different scales. A definite portion of the
first trapping disk $D_0$ returns to $U_\bullet$.}
\label{Fig:Trapping disks}
\end{figure}

\begin{property}
\label{st:hyp metr}
Assuming that the trapping disk $D$ from~\S\ref{sss:TrapDisks} is sufficiently close to $c_0(f_\str)$ and $\bar g$ is sufficiently close to $f_\str$, we have:
\begin{enumerate}
\item $\psi_0\mid D$ is almost an isometry with respect to the hyperbolic metrics of $\C\setminus \ovPost(\bar g)$ and $V\setminus \ovPost(g)$;
\label{st:hyp metr:1}
\item $\psi_0\mid D$ is almost an isometry with respect to the hyperbolic metrics of $\C\setminus \bar g^{-1}\big(\ovPost(\bar g)\big)$ and $V\setminus g^{-1}\big(\ovPost(g)\big)$;\label{st:hyp metr:2}
\item on $D$ the hyperbolic metric of $\C\setminus  \ovPost(\bar g)$ is almost the same as the hyperbolic metric of $\C\setminus \ovPost(f_\str)$;\label{st:hyp metr:3}
\item on $D$ the hyperbolic metric of $\C\setminus  \bar g^{-1}\big(\ovPost(\bar g)\big)$ is almost  the same as  (i.e., sufficiently close to)  the hyperbolic metric of $\C\setminus f_\str^{-1}\big(\ovPost(f_\str)\big)$; and \label{st:hyp metr:4}
\item on $D$ the hyperbolic density of $\C\setminus f_\str^{-1}\big(\ovPost(f_\str)\big)$ is by $\lambda>1$ smaller than the hyperbolic density of $\C\setminus \ovPost(f_\str)$.\label{st:hyp metr:5}
\end{enumerate}
\end{property}
\begin{proof}

Claims~\ref{st:hyp metr:1} and~\ref{st:hyp metr:2} follow from Property~\ref{property:hyp metr}: since $D$ and $c_0\in \Post(\bar g)$ are deep in $\Dom \psi_0$ and since $\psi_0$ respects the postcritical sets, $\psi_0$ is almost an isometry.

Claims~\ref{st:hyp metr:3} and~\ref{st:hyp metr:4} follow from $\Post(\bar g)\to \Post(f_\str)$ as $\bar g\to f_\str$, see~\S\ref{sss:contr post set}.

Claim~\ref{st:hyp metr:5} is equivalent to a strict expansion of $f_\str\mid D\setminus \big( f_\str^{-1}(\overline Z'_\str)\big)$ with respect to the hyperbolic metric of  $\C\setminus \ovPost(f_\str)$. It can be proven in the same way as Lemma~\ref{lem:F^P exp}. 
\end{proof}

\begin{proof}[Proof of Properties~\ref{prop:trap disk:hyp diam is bound} and~\ref{property:g D_i is exp} ]
Property~\ref{prop:trap disk:hyp diam is bound} follows from Claim~\ref{st:hyp metr:1} of Property~\ref{st:hyp metr}. 

Applying Property~\ref{st:hyp metr}, we obtain that the hyperbolic metric of $V\setminus g^{-1}\big(\ovPost(g)\big)$ is in $(\lambda-\varepsilon)>1$ smaller than the hyperbolic metric of $V\setminus \ovPost(g)$ uniformly in $D_i$. This implies Property~\ref{property:g D_i is exp}.
\end{proof}

\subsubsection{Porosity} 
A \emph{gap} of radius $r$ in a set $S$  is a round disk of radius $r$ disjoint from $S$. The following lemma asserts that if a set $S$ has density less than $1 -\epsilon$ in many scales, then it has small area.

\begin{lem}[\cite{AL-posmeas}*{Lemma 6.23}]
\label{lem:cond:small area}
For any $\rho\in  (0, 1), C > 1$ and $\epsilon > 0$ there exist $\sigma\in (0, 1)$ and $C_1 > 0$ with the following property. Assume that a measurable set $S \in \ovDisk(r)$ has the
property that for any $z \in S$ there are $n$ disks $\ovDisk(z, r_k)$ with radii
\[C^{-1} \rho^{\ell_k}\le r_k/r \le C \rho^{\ell_k},\sp \sp \ell_k\in \N,\sp \ell_1<\ell_2<\dots <\ell_n, \]
containing gaps in $S$ of radii $\epsilon r_k$. Then $\area S \le  C_1\sigma^nr^2$.
\end{lem}

\subsubsection{Estimating $\xi$: outline}\label{sss:outline: Est xi} A point in $V_\bullet \setminus U_\bullet$ escaping $U$ travels through each trapping disk $D_i$. Every $D_i$ has a definite portion returning to $Z$ (Property~\ref{prop:trap disk}). Therefore, with high probability, a point in $V_\bullet \setminus U_\bullet$ escaping $U$ travels through $D_0$ many times. Since a definite  portion of $D_0$ returns to $U_\bullet$ (Property~\ref{prop:first trap disk}), a point in $V_\bullet \setminus U_\bullet$ returns to $U_\bullet$ with high probability. 

We will use the following ingredients. Since almost every point in the Julia set is eventually in a small Julia set, it is sufficient to estimate $\xi$ for points escaping $U$. By expansion, different passages through $D_i$ create gaps in different scales (Lemma~\ref{lem:pullbacks are in different scales}), thus Lemma~\ref{lem:cond:small area} is applicable. Using area estimates, we obtain that points travel through $D_0$ many times with high probability (Lemma~\ref{lem:Sigma_n is small}). 
\subsubsection{Landing branches}

For any point $z$, let 
\[0\le r_1(z)< r_2(z)<\dots < r_n(z)<\dots\]
be all the landing times of $\orb z$ at $D_i$,~i.e. the moments at which $g^{r_n(z)}(z)\in D_i$. 

Let $P^n(z)$ be the pullback of $D_i$ along $g^{r(n)}\colon z\mapsto g^{r(n)}(z)\in D$. The map \[T_{P^n(z)}=T^n\coloneqq g^{r(n)}\colon P^n(z)\to D_i\] is univalent. Let $\PP(D_i)$ be the family of all domains $P=P^n(z)$. Combing the Koebe Distortion Theorem and Property~\ref{prop:first trap disk}, we obtain: 

\begin{property}[\cite{AL-posmeas}*{Lemma 6.24}] The following properties hold.
\begin{itemize}
\item  Landing branches $T_P\colon P\to D_i$ with $P\in \PP(D_i)$ have uniformly (in $P$ and $D_i$) bounded distortion. The domains $P\in \PP(D_i)$ have a bounded shape and are well inside $\C\setminus \ovPost(g)$. 
\item Each domain $P\in \PP(D_0)$ contains a pullback of $V_\bullet$ of comparable size. 
\end{itemize}
\end{property}

Combining expansion with the fact that the hyperbolic density near $D_i$ is bounded below, we obtain: 

\begin{property}[\cite{AL-posmeas}*{Lemma 6.25}]
\label{prop:P int D:bound}
There is a constant $C_0$ such that the following holds. If $P\in \PP(D_i)$ intersects $D_j$, then
\[\diam P\le C_0 \diam D_j.\]
\end{property}

The following lemma asserts that the intersecting pullbacks of $(D_i)_i$ belong to different scales. The lemma follows from the uniform expansion of $g\mid D_i$ (Property~\ref{property:g D_i is exp}) combined with the uniform boundedness of $D_i$ (Property~\ref{prop:trap disk:hyp diam is bound}) and the fact that the hyperbolic density near $D_i$ is bounded below.
\begin{lem}[\cite{AL-posmeas}*{Lemma 6.26}]
\label{lem:pullbacks are in different scales}
For any $\sigma\in (0,1)$ there is a $\nu \in \N$ with the following property. Consider a point $z$ landing at the $D_{i(t)}$ at moments $r_t$, where \[t\in \{0,1,\dots,\nu\}\sp \text{ and }\sp 0\le r_1< r_2<  \dots <r_\nu ,\] and
let $P^t\ni z$ be the corresponding pullback of the $D_{i(t)}$. Then
\[\diam P^\nu< \sigma \diam P^1.\]
\end{lem}

\subsubsection{Truncated Poincar\'e series}
We need to understand how disks in $\PP(D_0)$ intersect. Let $\scrP$ be the set of $P\in \PP(D_0)$ that intersect $D_0$. And let $\scrP^n$ be the set of domains $P\in \PP(D_0)$ that can be written as $P=P^m(z)$ with $m\le n$. In other words, the smallest landing time of $P$ is less or equal than $n$.

The \emph{truncated Poincar\'e series} is:
\[\phi_n (\xi)\coloneqq\sum _{P\in \scrP^n} \frac{1}{\big|T'_P(\xi_P)\big|^2},\sp \text{ where }\xi_P\in P\text{ and }T_P(\xi_P)=P.\]

The following lemma follows from the Koebe Distortion Theorem, Property~\ref{prop:P int D:bound}, and the observation that  the family $\scrP^n$ has the intersection multiplicity at most $n$; the proof uses area estimates.

\begin{lem}[\cite{AL-posmeas}*{Lemma 6.28}]
\label{lem:phi_n:est}
There is a constant $C>0$ such that ${\phi_n(\xi)\le Cn}$ for all $\xi \in D_0$.
\end{lem}

\subsubsection{Few returns to the base} Let $\Sigma $ be the set of points in $D_0\setminus \Jul$ that under the iterates of $g$
never return back to $D_0$.

\begin{lem}[\cite{AL-posmeas}*{Lemma 6.30}]
\label{lem:Sigma is small}
For any $\sigma\in(0,1)$ and for any natural $\tau\in \N$, if $m\ll\bbk$ is sufficiently big, then 
\[\area \Sigma\le C\sigma^\tau \area D_0 \]
\end{lem}
\begin{proof}
Since the orbit of $z$ escapes, it passes through all trapping disks $D_1,\dots ,D_\bbs$. Each $D_i$ contains a disk $W_i$ of bounded shape that maps to the Siegel disk $Z(g_\str)$. The pullbacks of $W_i$ create gaps of definite size (distortion theorem) and in different scales (Lemma~\ref{lem:pullbacks are in different scales}). By Lemma~\ref{lem:cond:small area}, the area of $\Sigma$ is small.
\end{proof}

Set
\[\Sigma_n \coloneqq \bigcup _{P\in \scrP^n} T_P^{-1}(\Sigma).\]

As a consequence of Lemmas~\ref{lem:Sigma is small} and~\ref{lem:phi_n:est}, we have:
\begin{lem}[\cite{AL-posmeas}*{Lemma 6.31}]
\label{lem:Sigma_n is small}
Under the assumption of Lemma~\ref{lem:Sigma is small}, there is a constant $C>0$  such that for any $n\in \N$ we have
\[\area\ \Sigma_n\le Cn\sigma^\tau \area D_0.\]
\end{lem}

\subsubsection{Many returns to the base}
Set
\[\mbbS^n \coloneqq\bigcup_{P\in \scrP\setminus \scrP^n} P.\]
\begin{lem}[\cite{AL-posmeas}*{Lemma 6.32}]
\label{lem:mbbS is small}
There exist $C > 0$ and $\sigma\in (0, 1) $ such that for any $n\in \N$ the area of the set of points of $\mbbS^n$ that never land in $V_\bullet$ is at
most $C\sigma ^n\area D_0$.
\end{lem} 
\begin{proof}
Each time the orbit of $z$ passes through $D_0$, we have a gap of points of definite size that eventually maps to $U_\bullet$ (Property~\ref{prop:first trap disk}). The gaps are in different scales (Lemma~\ref{lem:pullbacks are in different scales}). By Lemma~\ref{lem:cond:small area}, the area of all such $z$ is small.
\end{proof}

\subsubsection{Estimating $\xi$}

\begin{prop}[\cite{AL-posmeas}*{Proposition 6.33}]
For any $\epsilon>0$, if $m\ll\bbk$ is sufficiently big, then $\xi< \epsilon$.
\end{prop}
\begin{proof}
Let $Y$ be the set of point in $D_0$ that never land at $V_\str$. There are $3$ cases:
\begin{itemize}
\item by~\cite{L:TypBeh}, $\area \big(Y\cap \Jul\big)=0$ (because almost every point in $\Jul$ is eventually in small Julia sets); 
\item   $\area \big(Y\cap  \mbbS^n\big)$ is small by Lemma~\ref{lem:mbbS is small};
\item the area of remaining points in $Y$ is small by Lemma~\ref{lem:Sigma_n is small}.
\end{itemize}

We can now transfer the escaping density estimate for $D_0$ to the escaping density estimate for $V_\bullet\setminus U_\bullet$, see the argument in \cite{AL-posmeas}*{Proposition 6.33}. 
\end{proof}

\section*{Conventions and Notations}

\noindent {\bf Basic conventions.} We set:
\begin{itemize}
\item $\ovDisk(a,r)$ to be the closed disk around $a\in \C$ with radius $r$;
\item $\ovDisk(r)\coloneqq \ovDisk(0,r)$ and $\ovDisk\coloneqq \ovDisk(0,1)$;
\item $T_c  \colon z\mapsto z+c$ is the translation by $c\in \C$;
\item $A_c \colon z\mapsto cz$ is the scaling by $c\in \C\setminus \{0\}$.
\end{itemize}

A \emph{simple arc} is an embedding of a closed interval. We often say that a simple arc $\ell\colon [0,1]\to \C$ \emph{connects} $\ell(0)$ and $\ell(1)$. A \emph{simple closed curve} or a \emph{Jordan curve} is an embedding of the unit circle. A \emph{simple curve} is either a simple closed curve or a simple arc.

A \emph{closed topological disk} is a subset of a plane homeomorphic to the closed unite disk. In particular, the boundary of a closed topological disk is a Jordan curve. A \emph{quasidisk} is a closed topological disk qc homeomorphic to the closed unit disk.

Given a subset $U$ of the plane, we denote by $\intr U $ the interior of $U$.

Let $U$ be a closed topological disk. For simplicity we say that a homeomorphism $f\colon U\to \C$ is \emph{conformal} if $f\mid \intr U $ is conformal. Note that if $U$ is a quasidisk, then such an $f$ admits a qc extension through $\partial U$.

A \emph{closed sector}, or \emph{topological triangle} $S$ is a closed topological disk with two distinguished simple arcs $\gamma_-~,\gamma_+$ in $\partial S$ meeting at the \emph{vertex $v$ of $S$} satisfying $\{v\}=\gamma_-\cap \gamma_+$. Suppose further that $\gamma_-~,\intr S, \gamma_+$ have clockwise orientation at $v$. Then $\gamma_-$ is called the \emph{left boundary of $S$} while $\gamma_+$ is called the \emph{right boundary} of $S$. A closed \emph{topological rectangle} is a closed topological disk with four marked vertices.

Consider a continuous map $f\colon U\to \C$ and let $S\subset \C$ be a connected set. An \emph{$f$-lift} is a connected component of $f^{-1}(S)$. Let \[x_0,x_1,\dots x_n,\sp x_{i+1}=f(x_i)\]
be an $f$-orbit with $x_n\in S$. The connected component of $f^{-n}(S)$ containing $x_0$ is called the \emph{pullback of $S$ along the orbit $x_0,\dots, x_n$.}

Consider two partial maps $f\colon  X\dashrightarrow X$ and $g\colon Y\dashrightarrow Y$. A  homeomorphism $h\colon X\to Y$ is \emph{equivariant} if 
\begin{equation}
\label{eq:defn:equiv}
 h\circ f(x)=g\circ h (x)
 \end{equation}
for all $x$ with $x\in \Dom f$ and $h(x)\in \Dom g$. If~\eqref{eq:defn:equiv} holds for all $x\in T$, then we say that $h$ is \emph{equivariant on $T$.} 

We often write a partial map as $f\colon W\dashrightarrow W$; this means that $\Dom f\cup \Im f \subset W$.

By a \emph{tree} in an open set $U\subset  \C$ we mean an increasing union of finite trees $T_1\subset T_2\subset\dots $ such that $T_i\setminus T_{i-1}$ does not intersect any given compact subset of $U$ for $i\gg0$. A \emph{forest} and a \emph{graph} are defined similarly.

To keep notations simple, we will often suppress indices. For example, we denote
a pacman by $f\colon U_f \to V$, however a pacman indexed by $i$ is denoted as $f_i\colon U_i\to V$ instead of $f_i\colon U_{f_i}\to V$. 

Slightly abusing notations, we will often identify a lamination (or a triangulation) with its support. Given a triangulation $\bDelta$, we denote by $\Delta(i)$ its $i$-th triangle; $\Delta(i,i+1,\dots, i+j)$ denotes the union $\displaystyle \bigcup_{k=0}^j\Delta(i+k)$.

{\bf Renormalizations.} We will usually denote an analytic renormalization operator as ``$\RR$'', i.e.~$\RR f$ is a renormalization of $f$ obtained by an analytic change of variables. A renormalization postcomposed with a straightening will be denoted by ``$\mRR$''; for example, $\mRR_{s} \colon \Mandel_s\to \Mandel$ is the Douady-Hubbard straightening map from a small copy $\Mandel_s$ of $\Mandel$ to the Mandelbrot set. The action of the renormalization operator on the rotation numbers will be denoted by ``$\cRR$''.

 {\bf Rotations.} Combinatorial aspects of rotations are discussed in Section~\ref{s:SectRenorm}. The main notations:  
\begin{itemize}
\item $\ee\colon z\mapsto e^ {2\pi i z}$;
\item $\Lbb_\theta\colon  \ovDisk\to \ovDisk,\sp z\to \ee(\theta) z$,~\eqref{eq:defn:Lbb};
\item $\cRRc\colon \R/\Z\selfmap$ is the prime renormalization on rotation angles~\eqref{eq:R_prm}; 
\item $\mm$ is the renormalization period: $\theta_\str =\cRRc^\mm(\theta_\str)$, \S\ref{sss:renorm:pair of rotat},~\eqref{eq:theta is a per pnt};
\item $\M$ is the antirenormalization matrix,~\eqref{eq:sect to quarter},~\eqref{eq:def:M:2};
\item $\tt$ is the leading eigenvalue of $\M$,~\S\ref{sss:renorm:pair of rotat}; Lemma~\ref{lem:lambda:t*t}: $\lambda_\str=\tt^2=(\cRRc^\mm)'(\theta_\str)$; 
\item $\aa,\ \bb$ renormalization return times \eqref{eq:app:FirstReturnToSector};
\item $\PT$ is the semigroup of power-triples~\eqref{eq:defn:PT},~\S\ref{ss:PowerTriples}. 
\end{itemize}

Let $f : (W, \alpha) \ra (\C, \alpha)$  be a holomorphic map with a distinguished $\alpha$-fixed point. We will usually denote by $\lambda$ the multiplier at the $\alpha$-fixed point. If $\lambda=\ee(\phi)$ with $\phi\in \R$, then $\phi$ is called the \emph{rotation number of $f$}. If, moreover, $\phi=\pp/\qq\in \Q$, then $\pp/\qq$ is also the \emph{combinatorial rotation number}: there are exactly $\qq$ local attracting petals at $\alpha$ and $f$ maps the $i$-th petal to $i+\pp$ counting counterclockwise.

 {\bf Quadratic and quadratic-like families.} (See~\S\ref{ss:QL renorm}.) We denote by 
\begin{itemize}
\item $\Mandel$ is the Mandebrot set, $\Mandel_i\subset \Mandel$ small copies in $\Mandel$;
\item $\HH\subset \Mandel$ the main hyperbolic component;
\item $\QG$ the space of quadratic-like germs~\S\ref{ssss:QL anal ren};
\item $\Conn\subset \QG$ the connectedness locus $\QG$;
\item $\MM_i\subset \Conn$ copies of $\Mandel$ in $\Conn$,~\S\ref{ssss:QL anal ren};
\item $\RR_\QL\colon \QG\dashrightarrow \QG$ a quadratic-like analytic operator~\S\ref{ssss:QL anal ren};
\item $\strai\colon \Conn\to \Mandel$ the straightening;
\item $\mRRc\colon \Mandel \dashrightarrow \Mandel$ the molecule map~\S\ref{ss:MolecMap}; acts as $\cRRc$ on $\partial \HH$. 
\end{itemize}

{\bf Pacmen.} We denote a pacman by a lowercase letter (for example~$f$ or $g$), its bold capital version denotes the corresponding maximal prepacman (resp~$\bF$ or $\bG$). Objects in the dynamical planes of maximal prepacmen are often written in bold script. Objects in the parameter plane are usually written in calligraphic script.

For a pacman $f$ we write $f_n=\RR^n f$ if $f$ is in the domain of $\RR^n$; in particular $f_0=f$. If $f\in \WW^u$, then $f_n$ are well-defined for all $n\le 0$. The corresponding maximal prepacmen are denoted by $\bF_n$.

Renormalization fixed point and associated objects are indicated by ``$\str$'', f.e. $f_\str \colon  U_\str\to V, \sp \bF_\str,\sp \bZ_\str$.  We will suppress ``$\str$'' in~\S\ref{s:Dyn F_str}. For example, $\bF$ denotes the fixed maximal prepacman $\bF_\str$ in~\S\ref{s:Dyn F_str}.

We denote by 
\begin{itemize}

\item $\mu_\str$ and $\lambda_\str$ the dynamical and parameter self-similarity constants with $| \lambda_\str|>1$ and $|\mu_\str|<1$;
\item $A_\str=A_{\mu_\str}\colon z\mapsto \mu_\str z$ the dynamical scaling;
\item $f_\str\in \WW^u$ the fixed pacman: $f_\str=\RR(f_\str)$;
\item $\bF_\str$ is the fixed maximal prepacman; $\bF_\str=\bF$ in Section~\ref{s:Dyn F_str};
\item  $c_0$ and $c_1$ the critical point and the critical value of a pacman $f$;

\item $Z_f$ is the Siegel disk of a Siegel map $f$;

\item $\RR\colon \BB \dashrightarrow \BB$ is an analytic renormalization operator~\S\ref {sss:HypSelfOper}; it acts as $\cRRc^\mm$ on rotation numbers, see Lemma~\ref{lem:RR:rot numb act};
\item $\WW^s,\WW^u\subset \BB$ the stable and unstable manifolds of $\RR$; 
\item $\Unst$ the space of maximal prepacmen;

\item $F=(f_-,\ f_+),\  \bF=(\bbf_-,\ \bbf_+)$ a prepacman and a maximal prepacman~\S\ref{sss:max prepacmen};
\item  $\bF^P,\sp P\in \PT$ is defined in~\S\ref{ss:PowerTriples}; 
\item $\alpha=f(\alpha)$ is a fixed point of a pacman, Figure~\ref{Fg:Pacman};
\item $\balpha=\bF^P(\balpha)$ is a fixed boundary point of a maximal prepacman \S\ref{ss:balpha:wall topoloy};
\item in~\S\ref{s:par pacm} and \S\ref{s:val flow}, $\bgamma(\bG)$ and $\bdelta(\bG)$ denote the periodic cycles characterizing the satellite and the secondary satellite hyperbolic components;

\item $\bDelta_n$ a renormalization triangulation with a wall $\bPi_n$~\S\S\ref{sss:ren trian:pacmen},~\ref{sss:walls bPi},~\ref{ss:RenTriang};  
\item $\CP\big(\bF^P\big)$, $\CV\big(\bF^P\big)$, $\Post(\bF)$ the set of critical values, critical points, the postcritical set 
\item $\Fat(\bF), \Jul(\bF) , \Esc(\bF)$ the Fatou, Julia,  escaping sets~\S\ref{ss:Fat Jul  Esc};
\item $\bZ=\bZ_\str$ the invariant Siegel disk (half-plane) of $\bF_\str$;
\item $c_S(\bF_\str)$ is the unique critical point on $\partial \bZ_\str$ of generation $S\in \PT$;

\item $\bZ_s$ the lift of $\bZ$ along $\bF_\str^S\colon c_s \mapsto 0$;
\item  $\bL_s$ and $\bW_s$ the limb and wake centered at $\bZ_s$;

\item $\bG^{Q(\rr)}\colon \bW^1\to \bW$ a primary renormalization map~\eqref{eq:FRM:bG};
\item $\bW(i,\bG)$ primary wakes; see Theorem~\ref{thm:parab prepacm} and \S\ref{ss:ValFlower:prep};
\item $\bF^{Q(\rr,\ss)}_{\rr,\ss}\colon \bW_{\rr,\ss}^1 \to \bW_{\rr,\ss}$ a secondary renormalization map~\eqref{eq:SRM:rr ss}, 
\item $\frakY_j, \frakZ_i$ periodic and certain preperiodic small Julia sets of the secondary renormalization, see Figure~\ref{Fig:Val Flower X(G)};
\item $\bMM_0=\bMM_{\rr,\ss,\kk}\subset \Unst$ a ternary satellite copy~\S\ref{ss: ternary sat copy}; 
\item $\bMM_{n}\coloneqq \RR^{n}(\bMM_0)\subset \Unst$ the renormalization orbit of $\bMM_0$;
\item $\widetilde \bX(\bG), \bX(\bG), \sp \bG\in \Unst$ the (enlarged) valuable flower of a maximal prepacman~\S\ref{ss:ValFlower:prep};
\item $\widetilde X(g), X(g)$ the (enlarged) valuable flower of a pacman $g\in \MM_n\simeq \bMM_n, \sp {n\ll 0}$, Theorem~\ref{thm:proj of val flow}.
\end{itemize}

\end{document}